\documentclass[a4paper,leqno]{article1}   %10pt default
\usepackage{amsfonts}
\usepackage{fancyhdr}
\usepackage[centertags]{amsmath}
\usepackage{amssymb,amsthm,stmaryrd,enumerate,multirow,bigstrut,xcolor}
\usepackage[british]{babel}
\usepackage[T1]{fontenc}
\newcommand{\HRule}{\rule{\linewidth}{0.5mm}}

\usepackage{array}
\makeatletter
\newcommand{\thickhline}{%
    \noalign {\ifnum 0=`}\fi \hrule height 1pt
    \futurelet \reserved@a \@xhline
}
\newcolumntype{"}{@{\hskip\tabcolsep\vrule width 1pt\hskip\tabcolsep}}
\makeatother
\renewcommand{\labelenumi}{\emph{(\roman{enumi})}}

\usepackage[left=2.38cm,right=2.38cm,top=2.66cm,bottom=2.66cm,bindingoffset=1cm]{geometry}

\newcommand{\thmref}[1]{The\-o\-rem~\ref{#1}}
\newcommand{\propref}[1]{Prop\-o\-si\-tion~\ref{#1}}
\newcommand{\lemref}[1]{Lem\-ma~\ref{#1}}
\newcommand{\corref}[1]{Co\-rol\-la\-ry~\ref{#1}}
\newcommand{\dfnref}[1]{Def\-i\-ni\-tion~\ref{#1}}

\newlength{\defbaselineskip}
\newcommand{\setlinespacing}[1]%
           {\setlength{\baselineskip}{#1 \defbaselineskip}}

\numberwithin{equation}{section}
\newtheorem{thm}{Theorem}[section]
\newtheorem{prop}[thm]{Proposition}
\newtheorem{cor}[thm]{Corollary}
\newtheorem{lem}[thm]{Lemma}

\theoremstyle{definition}
\newtheorem{dfn}{Definition}[section]
\newtheorem{rem}[dfn]{Remark}

\newcommand{\tr}{{\rm tr}}

\newcommand{\II}{\mathcal{I}}
\newcommand{\R}{\mathbb{R}}

\newcommand{\KKK}{\mathcal{K}}
\newcommand{\C}{\mathbb{C}}
\newcommand{\HC}{\mathcal{H}}
\newcommand{\HCC}{\mathcal{HC}}

\newcommand{\z}{\mathfrak{z}}
\newcommand{\F}{\mathcal{F}}
\newcommand{\T}{\mathcal{T}}
\newcommand{\U}{\mathcal{U}}
\newcommand{\PP}{\mathcal{P}}
\newcommand{\CC}{\mathcal{C}}
\newcommand{\LL}{\mathcal{L}}
\newcommand{\LLL}{\mathfrak{L}}
\newcommand{\GG}{\mathcal{G}}
\newcommand{\KK}{\mathcal{K}}
\newcommand{\MM}{\mathcal{M}}
\newcommand{\SSS}{\mathcal{S}}
\newcommand{\N}{\widehat{N}}
\newcommand{\f}{\varphi}
\newcommand{\Id}{{\rm Id}}
\newcommand{\tg}{\widetilde{g}}
\newcommand{\h}{\widetilde{h}}
\newcommand{\tPP}{\widetilde{F}}
\newcommand{\DDD}{\operatorname{D}}%\hspace{-2pt}
\newcommand{\M}{(\MM,\allowbreak{}\f,\allowbreak{}\xi,\allowbreak{}\eta,\allowbreak{}g)}
\newcommand{\Span}{\mathrm{span}}
\newcommand{\D}{{\rm d}}
\newcommand{\ddt}{\dfrac{\D}{\D t}}

\newcommand{\pd}{\partial}
\newcommand{\ddx}{\dfrac{\pd}{\pd x^i}}
\newcommand{\ddy}{\dfrac{\pd}{\pd y^i}}
\newcommand{\ddu}{\dfrac{\pd}{\pd u^i}}
\newcommand{\ddv}{\dfrac{\pd}{\pd v^i}}
\newcommand{\dda}{\dfrac{\pd}{\pd a}}
\newcommand{\ddb}{\dfrac{\pd}{\pd b}}
\newcommand{\ddc}{\dfrac{\pd}{\pd c}}

\newcommand{\iim}{{\rm im}}
\newcommand{\ddr}{\frac{\D}{\D r}}%\newcommand{\ddr}{\tfrac{\D}{\D r}}
\newcommand{\HH}{\mathcal{H}}
\newcommand{\VV}{\mathcal{V}}
\newcommand{\tD}{\widetilde{\DDD}}%\thinspace
\newcommand{\nSvK}{\n^{\circ}}
\newcommand{\tnSvK}{\widetilde{\n}^{\circ}}

\newcommand{\Alt}{\mathrm{Alt}}
\newcommand{\Ric}{\mathrm{Ric}}
\newcommand{\Scal}{\mathrm{Scal}}

\newcommand{\tn}{\widetilde\nabla}

\newcommand{\tR}{\widetilde{R}}

\newcommand{\tF}{\widetilde{F}}
\newcommand{\tS}{\widetilde{S}}

\newcommand{\tA}{\widetilde{A}}
\newcommand{\tP}{\widetilde{\Phi}}
\newcommand{\tN}{\widetilde{[J,J]}}
\newcommand{\wP}{\widetilde{P}}

\newcommand{\Div}{{\rm div \hspace{0.01in}}}

\newcommand{\norm}[1]{\left\Vert#1\right\Vert}

\newcommand{\const}{\operatorname{const}}

\newcommand{\dd}{{\rm d}}
\newcommand{\grad}{{\rm grad}\hspace{0.01in}}
\newcommand{\X}{\mathfrak{X}}
\newcommand{\W}{\mathcal{W}}
\newcommand{\TT}{\mathcal{T}}

\newcommand{\n}{\nabla}
\newcommand{\nJJ}{\norm{\DDD\hspace{-2pt}  J}^2}
\newcommand{\DDJ}{\norm{\operatorname{D}J}^2}
\newcommand{\nJ}[1]{\norm{\DDD J_{#1}}^2}

\newcommand{\nn}{\widetilde{\nabla}}
\newcommand{\cn}{\check{\nabla}}
\newcommand{\cg}{\check{g}}
\newcommand{\g}{\widetilde{g}}

\newcommand{\ee}{\end{equation}}
\newcommand{\be}[1]{\begin{equation}\label{#1}}
\newcommand{\al}{\alpha}
\newcommand{\bt}{\beta}
\newcommand{\gm}{\gamma}
\newcommand{\om}{\omega}
\newcommand{\lm}{\lambda}
\newcommand{\ta}{\theta}
\newcommand{\ep}{\varepsilon}
\newcommand{\ea}{\varepsilon_\alpha}
\newcommand{\eb}{\varepsilon_\bt}
\newcommand{\eg}{\varepsilon_\gamma}
\newcommand{\s}{\mathfrak{S}}
\newcommand{\sx}{\mathop{\mathfrak{S}}\limits_{x,y,z}}

\newcommand{\wh}[1]{\widehat{#1}}
\newcommand{\ii}{\bar{i}}
\newcommand{\jj}{\bar{j}}
\newcommand{\tnJ}{\Vert\tD J\Vert^2}

\newcommand{\xia}{\xi_{\alpha}}
\newcommand{\etaa}{\eta_{\alpha}}
\newcommand{\Ja}{J_{\alpha}}

\def\dfrac#1#2{\displaystyle\frac{#1}{#2}}

\begin{document}
%\magstep2
% \Large{ \tableofcontents }
%\newpage

%\lipsum[1-20]

%%\include{Man-zagl}

%\include{Man-head2}
%%%%%%%%%%%%%%%%%%%%%%%%%%%%%%%%%%%%%%%%%%%%%%%%%%%%%%%%%%%%%%%%%%%%%%%%%%%%%%%

\thispagestyle{empty}
%\selectlanguage{british}

\LARGE{
\begin{center}
\textsc{University of Plovdiv Paisii Hilendarski \\
Faculty of Mathematics and Informatics\\
Department of Algebra and Geometry}\\[6pt]\hrule
\end{center}
}

\
\vskip 1cm

\LARGE{
\begin{center}
\textsc{Mancho Hristov Manev}
\end{center}
}

\vskip 1cm

%\begin{center}
%\Huge{\textsc{\textbf{
%On Differential Geometry \\
%of Smooth Manifolds \\
%with Almost Complex\\
%and Almost Contact Structures\\
%and Their Corresponding 3-Structures\\
% Equipped with Metrics\\
% of Norden Type
%}}}
%\end{center}
%\begin{center}
%\Huge{\textsc{\textbf{
%On Geometry of Manifolds \\
%with Almost Contact Complex Structures\\
%and Their Corresponding 3-Structures\\
%with Metrics of Norden Type
%}}}
%\end{center}
\begin{center}
\Huge{\textsc{\textbf{
On Geometry of Manifolds \\
with Some Tensor Structures \\
and
Metrics of Norden Type%\\
%
%On Geometry of Manifolds \\
%with Almost Complex \\
%and Almost Contact Structures,\\
%Their Corresponding 3-Structures\\
%and Metrics of Norden Type%\\
%
%On Geometry of Manifolds \\
%with Metrics of Norden Type
}}}
\end{center}

%\
%\\[-12pt]

\
\\
\smallskip

\LARGE{
\begin{center}
\textsc{Dissertation} \\[12pt]
Submitted for the Scientific Degree: \emph{Doctor of Science}\\[12pt]
Area of Higher Education: 4. Natural Sciences, Mathematics and Informatics\\%[12pt]
Professional Stream: 4.5. Mathematics \\%[12pt]
Scientific Specialty: Geometry and Topology
\\
\end{center}
}

\vspace{100pt}
\vfill

\Large{
\begin{center}\hrule\ \\
\textsc{Plovdiv, 2017}
\end{center}
}

%\newpage   Премахване на 2-ра празна страница
%
%
%
%\thispagestyle{empty}
%
%$ \phantom{\quad} $
%

%

%\include{Man-str}
%\pagestyle{empty}
%\count0 = 3

%\newpage

\label{str}

\Large{

\
\\[6pt]
\bigskip

\
\\[6pt]
\bigskip

\thispagestyle{empty}

\lhead{\emph{}}

% \noindent  {\LARGE\bf Contents}
\noindent  {\hfill\LARGE{\emph{To my wife Rositsa}}
}%\\[6pt]\vskip2pt}

\newpage

\Large{

\
\\[6pt]
\bigskip

\
\\[6pt]
\bigskip

\lhead{\emph{Structure of the Dissertation}}

% \noindent  {\LARGE\bf Contents}
\noindent  {\Huge\bf Structure of the Dissertation
}%\\[6pt]\vskip2pt}

\vskip 1cm

The present dissertation consists of an introduction, a main body, a conclusion and a bibliography.
The introduction consists of two parts: a scope of the topic and the purpose of the dissertation.
The main body includes two chapters containing 15 sections.
The conclusion provides a brief summary of the main contributions of the dissertation, a list of the publications on the results given in the dissertation, a declaration of originality and acknowledgements.
The bibliography contains a list of \ref{count-cite} publications used in the text.

%%%%%%%%%%%%%%%%%%%%%%%%%%%%%%%%%%%%%%%%%%%%%%%%%%%%%%%%%%%%%%%%%%%%%%%%%%%%%%%%%%%%%%

\vspace{20pt}

\begin{center}
$\divideontimes\divideontimes\divideontimes$
\end{center}

%\                   премахване на 4. празна страница
%\newpage
%\thispagestyle{empty}
%
%$ \phantom{\quad} $

%\include{Man-cont}
%\pagestyle{empty}
%\count0 = 3

\newpage

\Large{

\
\\[6pt]
\bigskip

\
\\[6pt]
\bigskip

\lhead{\emph{Contents}}

% \noindent  {\LARGE\bf Contents}
\noindent  {\Huge\bf Contents
}%\\[6pt]\vskip2pt}

\vskip 1cm

%\noindent
%Structure of the Dissertation
%\hrulefill\hbox to 2pc{\hfill \pageref{str}}\\[-9pt]

\noindent
Introduction
\hrulefill\hbox to 2pc{\hfill \pageref{intro}}\\[-9pt]

\emph{Scope of the Topic}
\dotfill\hbox to 2pc{\hfill \pageref{intro}}\\[-9pt]

\emph{Purpose of the Dissertation}
\dotfill\hbox to 2pc{\hfill \pageref{purp}}\\[-9pt]

%\emph{Glossary of Symbols}
%\dotfill\hbox to 2pc{\hfill \pageref{symb}}\\[-9pt]

%\emph{Brief Statement of the Results}
%\dotfill\hbox to 2pc{\hfill \pageref{stat}}\\[-9pt]
%\\[-6pt]

%\noindent
%\textsc{Part A.} \textsc{On Almost Contact Complex Riemannian  } \\
%\noindent\phantom{\textsc{Part A.}}
%\textsc{Manifolds}
%\hfill\hbox to 2pc{\hfill \pageref{partA}}\\[-9pt]
%\\
%\noindent
%\phantom{Глава I. }
%многообразия

%\noindent
%Chapter I. \textsc{On almost complex structures with Norden } \\
%\noindent\phantom{Chapter I.}
%\textsc{metrics}
%\hrulefill\hbox to 2pc{\hfill \pageref{chapI}}\\[-9pt]
%%\\

\noindent
Chapter I. \textsc{On manifolds with almost complex structures } \\
\noindent\phantom{Chapter I.}
\textsc{and almost contact structures, equipped with }\\
\noindent\phantom{Chapter I.}
\textsc{metrics of Norden type}
\hrulefill\hbox to 2pc{\hfill \pageref{chap1}}\\[-9pt]
%\\
%\noindent
%\phantom{Глава I. }
%многообразия

\indent
\S1. Almost complex manifolds with Norden metric
\dotfill\hbox to 2pc{\hfill \pageref{par:2n}}\\[-9pt]

\indent
\S2. Invariant tensors under the twin interchange of Norden \\
\phantom{\indent\S2. }metrics on almost complex manifolds
\dotfill \hbox to 2pc{\hfill \pageref{par:inv}}\\[-9pt]

\indent
\S3. Canonical-type connections on almost complex manifolds\\
\phantom{\indent\S3. }with Norden metric
\dotfill \hbox to 2pc{\hfill \pageref{par:conn}}\\[-9pt]

%\noindent
%Chapter II. \textsc{On almost contact structures with metrics } \\
%\noindent\phantom{Chapter II.}
%\textsc{of Norden type}
%\hrulefill\hbox to 2pc{\hfill \pageref{chapII}}\\[-9pt]
%\\
%\noindent
%\phantom{Глава I. }
%многообразия

\indent
\S4. Almost contact manifolds with B-metric
\dotfill \hbox to 2pc{\hfill \pageref{par:2n+1}}\\[-9pt]

\indent
\S5. Canonical-type connection on almost contact manifolds\\
\phantom{\indent\S5. }with B-metric
\dotfill \hbox to 2pc{\hfill \pageref{par:can.conn}}\\[-9pt]

\indent
\S6. Classification of affine connections on almost contact \\
\phantom{\indent\S6. }manifolds with B-metric
\dotfill \hbox to 2pc{\hfill \pageref{par:classT}}\\[-9pt]

%\indent
%\S7. On canonical-type connections on almost contact manifolds\\
%\phantom{\indent\S7. }with B-metric
%\dotfill \hbox to 2pc{\hfill \pageref{par:can.conn2}}\\[-9pt]

\indent
\S7.  Pair of associated Schouten-van Kampen connections adapted\\
\phantom{\indent\S7. }to an almost contact B-metric structure
\dotfill \hbox to 2pc{\hfill \pageref{par:pair}}\\[-9pt]

\indent
\S8. Sasaki-like almost contact complex Riemannian manifolds
\dotfill \hbox to 2pc{\hfill \pageref{par:sas}}\\[-9pt]
\\[-6pt]

%\newpage
%\noindent
%\textsc{Part B.} \textsc{On Almost Hypercomplex Structures and } \\
%\noindent\phantom{\textsc{Part B.}}
%\textsc{Almost Contact 3-Structures with Metrics }\\
%\noindent\phantom{\textsc{Part B.}}
%\textsc{of Hermitian-Norden Type}
%\hfill\hbox to 2pc{\hfill \pageref{partB}}\\[-9pt]
%
%\noindent
%Chapter III. \textsc{On almost hypercomplex structures with } \\
%\noindent\phantom{Chapter III.}
%\textsc{metrics of Hermitian-Norden type}
%\hrulefill\hbox to 2pc{\hfill \pageref{chapIII}}\\[-9pt]

\newpage
\noindent
Chapter II. \textsc{On manifolds with almost hypercomplex } \\
\noindent\phantom{Chapter II.}
\textsc{structures and almost contact 3-structures, }\\
\noindent\phantom{Chapter II.}
\textsc{equipped with metrics of Hermitian-Norden}\\
\noindent\phantom{Chapter II.}
\textsc{type}
\hrulefill\hbox to 2pc{\hfill \pageref{chap2}}\\[-9pt]

\indent
\S9. Almost hypercomplex manifolds with Hermitian-Norden \\
\phantom{\indent\S9. }metrics
\dotfill \hbox to 2pc{\hfill \pageref{par:4n}}\\[-9pt]

\indent
\S10. Hypercomplex structures with Hermitian-Norden metrics \\
\phantom{\indent\S10. }on 4-dimensional Lie algebras
\dotfill \hbox to 2pc{\hfill \pageref{par:4Lie}}\\[-9pt]

\indent
\S11. Tangent bundles with complete lift of the base metric and\\
\phantom{\indent\S11. }almost hypercomplex Hermitian-Norden structure
\dotfill \hbox to 2pc{\hfill \pageref{par:bund}}\\[-9pt]

\indent
\S12. Associated Nijenhuis tensors on almost hypercomplex \\
\phantom{\indent\S12. }manifolds with Hermitian-Norden metrics
\dotfill \hbox to 2pc{\hfill \pageref{par:assNij}}\\[-9pt]

\indent
\S13. Quaternionic K\"ahler manifolds with Hermitian-Norden\\
\phantom{\indent\S13. }metrics
\dotfill \hbox to 2pc{\hfill \pageref{par:quatK}}\\[-9pt]

%
%\noindent
%Chapter IV. \textsc{On almost contact 3-structures with } \\
%\noindent\phantom{Chapter IV.}
%\textsc{metrics of Hermitian-Norden type}
%\hrulefill\hbox to 2pc{\hfill \pageref{chapIV}}\\[-9pt]

\indent
\S14. Manifolds with almost contact 3-structure and metrics of\\
\phantom{\indent\S14. }Hermitian-Norden type
\dotfill \hbox to 2pc{\hfill \pageref{par:4n+3}}\\[-9pt]

\indent
\S15. Associated Nijenhuis tensors on manifolds with almost \\
\phantom{\indent\S15. }contact 3-structure and metrics of Hermitian-Norden type
\dotfill \hbox to 2pc{\hfill \pageref{par:4n+3assNij}}\\[-9pt]

%
%\indent
%\S16. Natural connections with totally skew-symmetric torsion \\
%\phantom{\indent\S16. }on manifolds with almost contact 3-structure and metrics\\
%\phantom{\indent\S16. }of Hermitian-Norden type
%\dotfill \hbox to 2pc{\hfill \pageref{par:4n+3conn}}\\[-9pt]
%

\noindent
Conclusion
\hrulefill\hbox to 2pc{\hfill\ \pageref{zakl}}\\[-9pt]

\emph{Contributions of the Dissertation}
\dotfill\hbox to 2pc{\hfill \pageref{zakl}}\\[-9pt]

\emph{Publications on the Dissertation}
\dotfill\hbox to 2pc{\hfill \pageref{publ}}\\[-9pt]

\emph{Declaration of Originality}
\dotfill\hbox to 2pc{\hfill \pageref{dekl}}\\[-9pt]

\emph{Acknowledgements}
\dotfill\hbox to 2pc{\hfill \pageref{blag}}\\[-9pt]

\noindent
Bibliography
\hrulefill\hbox to 2pc{\hfill\ \pageref{bib}}\\[-9pt]

%\begin{tabular}{lr}
%\hbox to 30pc{} &\phantom{333}\cr \parbox[t]{30pc}{Аааааааааа
%ббббб \dotfill} &\parbox[t]{3pc}{\ \\221} \cr
%asdfasdf sdjkghasd & 2
%\end{tabular}

%%%%%%%%%%%%%%%%%%%%%%%%%%%%%%%%%%%%%%%%%%%%%%%%%%%%%%%%%%%%%%%%%%%%%%%%%%%%%%%%%%%%%%

\vspace{20pt}

\begin{center}
$\divideontimes\divideontimes\divideontimes$
\end{center} 

\newpage

%%\setcounter{section}{1}
%\addtocounter{section}{1}\setcounter{subsection}{0}\setcounter{subsubsection}{0}
%
%\setcounter{thm}{0}\setcounter{equation}{0}

\label{intro}

 \Large{

\
\\[6pt]
\bigskip

\
\\[6pt]
\bigskip

\lhead{\emph{Introduction. Scope of the Topic
}}

\noindent  {\Huge\bf Introduction
}%\\[6pt]\vskip2pt}

\vskip 1cm
%
%\begin{quote}
%\begin{large}
%In the present section we ...
%\end{large}
%\end{quote}
%
%
%
%
%
%\vskip 0.2in \addtocounter{subsection}{1}
%
%\noindent  {\Large\bf \thesubsection. Introduction}
%
%\vskip 0.15in
%

\vskip 0.2in

\noindent  {\huge\bf \emph{Scope of the Topic}
}

\vskip 1cm

Among additional tensor structures on a smooth manifold, one of the most studied is \emph{almost complex structure}, i.e. an endomorphism of the tangent bundle whose square, at each point, is minus the identity. The manifold must be even-dimensional, i.e. $\dim=2n$.
Usually it is equipped with a \emph{Hermitian metric} which is a (Riemannian or pseudo-Riemannian) metric that preserves the almost complex structure, i.e. the almost complex structure acts as an isometry with respect to the (pseudo-)Riemannian metric. The associated (0,2)-tensor of the Hermitian metric is a %associated (K\"ahler)
2-form and hence the relationship with symplectic geometry. %The almost Hermitian manifolds are well known.

A relevant counterpart is the case when the almost complex structure acts as an anti-isometry regarding a pseudo-Riemannian metric. Such a metric is known as an anti-Hermitian metric or a \emph{Norden metric} (first studied by and named after A.\thinspace{}P. Norden \cite{N1,N2}).
The Norden metric is necessary pseudo-Riemannian of neutral signature whereas the Hermitian metric can be Riemannian or pseudo-Riemannian of signature $(2n_1,2n_2)$, $n_1+n_2=n$.
The associated (0,2)-tensor of any Norden metric is also a Norden metric. So, in this case we dispose with a pair of mutually associated Norden metrics, known also as twin Norden metrics.
This manifold can be considered as an $n$-dimensional manifold with a complex Riemannian metric whose real and imaginary parts are the twin Norden metrics.
Such a manifold is known as
a \emph{generalized B-manifold} \cite{GriMekDje85a,Mek85c,GriMekDje85d,GriDjeMek85b},
an \emph{almost complex manifold with Norden metric} \cite{v,GaGrMi85,GaBo,BonCasHer,CasHerRio,KO,Olsz05,KS},
an \emph{almost complex manifold with B-metric} \cite{GaGrMi87,GaMi87},
an \emph{almost complex manifold with anti-Hermitian metric} \cite{BoFeFrVo99,BorFraVol,DraFra}
or a \emph{manifold with complex Riemannian metric} \cite{LeB83,Manin,GaIv92,Low}.

%%%%%%%%%%%%%%%%%%%%%%%%%%%%%%%%%%%%%%%%%%%%%%%%%%%%

Supposing a manifold is of an odd dimension, i.e. $\dim=2n+1$, then there exists a contact structure. The codimension one contact distribution %, generated by the contact 1-form,
can be considered as the horizontal distribution of the sub-Riemannian manifold.
This distribution allows an almost complex structure which is
the restriction of a contact endomorphism on the contact distribution.
%
%Since any Norden metric on the contact distribution, any B-metric can be considered as an odd-dimensional counterpart of the corresponding Norden metric, or the B-metric is a pseudo-Riemannian metric of Norden type on an odd-dimensional differentiable manifold.
%
The vertical distribution is spanned by the corresponding Reeb vector field.
%Then its dual 1-form $\eta$ determines a codimension one distribution $\HC = \ker(\eta)$ endowed with an
%almost complex structure $\f$.
Then the odd-dimensional manifold is equipped with an \emph{almost contact
structure}.

If we dispose of a Hermitian metric on the contact distribution then the almost contact
manifold is called an \emph{almost contact metric manifold}.
In another case, when a Norden metric is available on the contact distribution
then we have an \emph{almost contact manifold with B-metric}. Any
B-metric as an odd-dimensional counterpart of a Norden metric is a
pseudo-Riemannian metric of signature $(n+1,n)$.

%%%%%%%%%%%%%%%%%%%%%%%%%%%%%%%%%%%%%%%%%%%%%%%%%%%

A natural generalization of almost complex structure is almost hypercomplex structure.
 %and almost quaternionic structure.
An almost hypercomplex manifold is a manifold which tangent bundle equipped with an action by the algebra of quaternions in such a way that the unit quaternions define almost complex structures.
Then an \emph{almost hypercomplex structure} on a $4n$-dimensional manifold is a triad
%H = (J?), ? = 1, 2, 3,
of anti-commuting almost complex structures whose triple composition is minus the identity.
%
%A quaternionic Riemannian manifold is the analogue of complex Riemannian manifold.
%

It is known that, if the almost hypercomplex manifold is equipped with a Hermitian metric, the
derived metric structure is a hyper-Hermitian
structure. It consists of the given Hermitian metric with
respect to the three almost complex structures and the three associated  K\"ahler forms \cite{AlMa}.
Almost hypercomplex structures (under the terminology of C. Ehresmann of almost quaternionic structures) with Hermitian metrics %were %first introduced by C. Ehresmann [7] in 1947 and then
were studied in many works, e.g. \cite{Ehr,Oba,Bon,YaAk,So,Boy,AlMaPo}.

%Hypercomplex manifolds were studied by C. Boyer in 1988 \cite{Boy}.
%He also proved that in real dimension 4, the only compact hypercomplex manifolds are the complex torus {\displaystyle T^{4}} T^{4}, the Hopf surface and the K3 surface.
%In 1955 M. Obata \cite{Oba} studied affine connection associated with
%His construction leads to what Edmond Bonan called the Obata connection[5][6] which is torsion free, if and only if, "two" of the almost complex structures {\displaystyle I,J,K} {\displaystyle I,J,K} are integrable and in this case the manifold is hypercomplex.

An object of our interest is
a metric structure on the almost hypercomplex manifold derived by a Norden metric.
Then the existence of a Norden metric with respect to one of the three almost complex structures implies the existence of one more Norden metric and a Hermitian metric with respect to the other two almost complex structures.
Such a metric is called a \emph{Hermitian-Norden metric} on an almost hypercomplex manifold.
Furthermore,
the derived metric structure contains the given metric and three (0,2)-tensors associated by
the almost hypercomplex structure -- a K\"ahler form and two Hermitian-Norden metrics
for which the roles of the almost complex structures change cyclically.
%On the other hand, a quaternionic inner product generates
%in a natural way the latter four bilinear forms such that by the following
%decomposition: < ;  >= g + ig1 + jg2 + kg3.
%
Thus, the derived manifold is called an
\emph{almost hypercomplex manifold with Hermitian-Norden metrics}.
Manifolds of this type are studied in several papers with the author participation \cite{GriManDim12,Ma05,ManSek,GriMan24,Man28,ManGri32}.

%
%Recently, manifolds with neutral metrics and various tensor
%structures have been object of interest in theoretical physics.

%%%%%%%%%%%%%%%%%%%%%%%%%%%%%%%%%%%%%%%%%%%%%%%%%%%%%%%%

The notion of \emph{almost contact 3-structure} is introduced by Y.\thinspace{}Y. Kuo in \cite{Kuo} and independently under the name \emph{almost coquaternion structure} by C. Udri\c{s}te in \cite{Udr}. Later, it is studied by several authors, e.g. \cite{Kuo,KuoTac,TacYu,YaIsKo}. It is well known that the product of a man\-i\-fold with almost contact 3-structure and a real line admits an almost hypercomplex structure
%(also known as an \emph{almost quaternion structure}) (cf.
\cite{Kuo,AlMa}.
All authors have previously considered the case when there exists a Riemannian metric compatible with each of the three structures in the given almost contact 3-structure. Then the object of study is the so-called \emph{almost contact metric 3-structure} (see also \cite{BaikBla95,BoyGalMann,BoyGal,BoyGal08}).

Compatibility of an almost contact 3-structure with a B-metric
%(the odd-dimensional counterpart of Norden metric)
is not yet considered before the publications that are part of this dissertation. % according to the knowledge of the author. % by other authors.
In the present work we launch such a metric structure on a manifold with almost contact 3-structure.

\vspace{20pt}

\begin{center}
$\divideontimes\divideontimes\divideontimes$
\end{center}

\newpage

%\setcounter{section}{1}
%\addtocounter{section}{1}\setcounter{subsection}{0}\setcounter{subsubsection}{0}
%
%\setcounter{thm}{0}\setcounter{equation}{0}

\label{purp}

 \Large{

\
\\[6pt]
\bigskip

\
\\[6pt]
\bigskip

\lhead{\emph{Introduction. Purpose of the Dissertation
}}

\noindent  {\huge\bf \emph{Purpose of the Dissertation}
}%\\[6pt]\vskip2pt}

\vskip 1cm

%
%
%\vskip 0.2in %\addtocounter{subsection}{1}
%
%\noindent  {\Large\bf Purpose of the Dissertation}
%
%\vskip 0.15in

The object of study in the present dissertation are some topics in differential geometry of smooth manifolds with additional tensor structures and metrics of Norden type.
%These metrics are inspired from the Norden metric in the even-dimensional case and its corresponding metrics in the other three cases: the odd-dimension, the multiple-of-four-dimension and the dimension of one less than multiple of four.
%
There are considered four cases depending on dimension of the man\-i\-fold: $2n$, $2n+1$, $4n$ and $4n+3$.
%for an arbitrary natural number $n$.
%
The studied tensor structures, which are counterparts in the different related dimensions, are: the almost complex structure, the almost contact structure, the almost hypercomplex structure and the almost contact 3-structure.
The considered metric on the $2n$-dimen\-sion\-al manifold is the Norden metric.
The metrics on the manifolds in the other three cases are generated by the Norden metric and they are: the B-metric, the Hermitian-Norden metric and the metric of Hermitian-Norden type, respectively.
The four types of tensor structures with metrics of Norden type %former and latter pairs of structures
are considered in their interrelationship.

\vskip12pt

%According to the author can be formulated the following
The purpose of the dissertation is to carry out the following:
%\noindent
%\textbf{Main Purposes of the Dissertation:}
\begin{enumerate}[1.]
  \item
        Further investigations of almost complex manifolds with Norden metric and, in particular, studying of natural connections with conditions for their torsion and invariant tensors under the twin interchange of Norden metrics.
  \item
        Further investigations of almost contact manifolds with B-metric including studying of natural connections with conditions for their torsion and associated Schouten-van Kampen connections as well as classification of affine connections.
  \item
        Introducing and studying of Sasaki-like almost contact complex Rie\-mann\-ian manifolds.
  \item
        Further investigations of almost hypercomplex manifolds with Her\-mit\-ian-Norden metrics including: studying of integrable structures of the considered type on 4-dimensional Lie algebra and tangent bundles with complete lift of the base metric; introducing and studying of associated Nijenhuis tensors in relation with natural connections having totally skew-symmetric torsion as well as quaternionic K\"ahler manifolds with Hermitian-Norden metrics.
  \item
        Introducing and studying of manifolds with almost contact 3-structures and metrics of Hermitian-Norden type and, in particular, associated Nijenhuis tensors and their relationship with natural connections having totally skew-symmetric torsion.
\end{enumerate}

%%%%%%%%%%%%%%%%%%%%%%%%%%%%%%%%%%%%%%%%%%%%%%%%%%%%%%%%%%%%%%%%%%%%%%%%%%%%%%%%%%%%%%

\vspace{20pt}

\begin{center}
$\divideontimes\divideontimes\divideontimes$
\end{center}

%\               премахване на празна страница
%\newpage
%\thispagestyle{empty}
%
%$ \phantom{\quad} $

%\include{Man-symb}
%\include{Man-stat}

%\include{Man-partA}
%\include{Man-chapI}
%\include{Man-chap1}
%%%%%%%%%%%%%%%%%%%%%%%%%%%%%%%%%%%%%%%%%%%%%%%%%%%%%%%%%%%%%%%%%%%%%%%%%%%%%%%
\newpage
\thispagestyle{empty}

%\count0 = 19 % номер начална страница

\label{chap1}

\Large{

\
\\
\bigskip\
\\
\bigskip\
\\
\bigskip

\begin{center}
\Huge{
\textbf{
Chapter I. \\[6pt]
}}
\HRule
\Huge{
\textbf{\\[12pt]
                                \textsc{On manifolds with almost \\[12pt]
                                        complex structures and almost \\[12pt]
                                        contact structures, equipped\\[12pt]
                                        with metrics of Norden type}}}
\end{center}

}

%\               премахване на празна страница
%\newpage
%\thispagestyle{empty}
%
%$ \phantom{\quad} $

%\include{Man-J} % par 1
\newpage

\addtocounter{section}{1}\setcounter{subsection}{0}\setcounter{subsubsection}{0}
\label{par:2n}

\setcounter{thm}{0}\setcounter{equation}{0}

 \Large{

\
\\[6pt]
\bigskip

\
\\[6pt]
\bigskip

\lhead{\emph{Chapter I $|$ \S\thesection. Almost complex manifolds with
Norden metric
}}
%\thispagestyle{empty}

%\noindent  {\Huge\bf \S\thesection. Almost complex manifolds \\[12pt]
%\phantom{\S\thesection. }with Norden metric
%}%\\[6pt]\vskip2pt}

\noindent
\begin{tabular}{r"l}
  %\hline
\hspace{-6pt}{\Huge\bf \S\thesection.}  & {\Huge\bf Almost complex manifolds} \\[12pt]
                             & {\Huge\bf with Norden metric}
  %\hline
\end{tabular}

\vskip 1cm

\begin{quote}
\begin{large}
In the present section we recall some notions and knowledge for the almost complex manifolds with
Norden metric \cite{GriMekDje85a,GaBo,GaGrMi87,MekMan}.
\end{large}
\end{quote}

%\vskip 0.2in \addtocounter{subsection}{1}
%
%\noindent  {\Large\bf \thesubsection. Introduction}

\vskip 0.15in

%
%
%Among smooth manifolds with additional structures one of the most popular is \emph{almost complex manifold}. Such a manifold is even-dimensional and it is equipped with an almost complex structure, i.e. an endomorphism on the manifold's tangent bundle that squares to the minus identity.
%
%In comparison, the action of the almost complex structure with
%respect to the Hermitian metric (respectively, the Norden metric -- first studied by and named after Aleksandr Petrovich Norden \cite{N1,N2})
%on the tangent spaces of the almost complex manifold is an
%isometry (respectively, an anti-isometry). The latter manifolds
%are known as \emph{generalized B-manifolds} \cite{GriMekDje85a} or
%\emph{almost complex manifolds with Norden metric} \cite{GaGrMi85,GaBo}, almost
%complex manifolds with B-metric \cite{GaGrMi87,GaMi87} or
%\emph{complex Riemannian manifolds} \cite{LeB83,Manin,GaIv92,BoFeFrVo99}.
%%\texttt{(Izlishno ili da mine v Introduction)}
%
%

%%%%%%%%%%%%%%%%%%%%%%%%%%%%%%%%%%%%%%%%%%%%%%%%%%%%%%%%%%%%%%%%%%%%%%%%%%%%%%%%%1
\vskip 0.2in \addtocounter{subsection}{1} \setcounter{subsubsection}{0}

\noindent  {\Large\bf \thesubsection. Almost complex structure and Norden metric}%\\[6pt]\vskip2pt}

\vskip 0.15in
%\section{Almost complex manifolds with Norden metric}

Let $(\MM,J,{g})$ be a $2n$-dimensional almost complex manifold with
Norden metric or briefly an \emph{almost Norden manifold}.
This means that $J$ is an almost complex structure and ${g}$ is
a pseudo-Riemannian metric on $\MM$ such that
\begin{equation}\label{2.1}%\label{str}
J^2x=-x,\quad {g}(Jx,Jy)=-{g}(x,y).
\end{equation}

Here and further, $x$, $y$, $z$, $w$ will stand for
arbitrary differentiable vector fields on the considered manifold (or vectors in its
tangent space at an arbitrary point of the manifold).

On any almost Norden manifold, there exists an \emph{associated metric} $\g$ of its metric ${g}$ defined by
\[
\g(x,y)={g}(Jx,y).
\]
It is also a Norden
metric since $\g(Jx,Jy)=-\g(x,y)$ and the manifold $(\MM,J,\g)$ is an almost Norden manifold, too. Both metrics are necessarily of neutral signature $(n,n)$.

The elements of the pair of Norden metrics of an almost Norden manifold are also known as \emph{twin Norden metrics} on $\MM$ because of the associated metric of $g$ is $\g$ and the associated metric of $\g$ is $-g$, i.e.
\begin{equation}\label{twin}
 \g(x,y)=g(Jx,y),\qquad \g(Jx,y)=-g(x,y).
\end{equation}

Let the Levi-Civita connections of ${g}$ and $\g$ be denoted by $\DDD$ and $\tD$, respectively.

The structure group $\GG$ of almost Norden manifolds is determined, according to \cite{GaBo}, by the following way
\begin{equation}\label{GG}
\GG=\mathcal{GL}(n;\mathbb{C})\cap\mathcal{O}(n,n),
\end{equation}
i.e. it is the intersection of the general linear group of degree $n$ over the set of complex numbers and the indefinite orthogonal group for the neutral signature $(n,n)$.
Therefore, $\GG$ consists of the real square matrices of
order $2n$ having the following type
\[
\left(%
\begin{array}{r|c}
  A & B \\ \hline
  -B & A
\end{array}%
\right),
\]
such that the matrices $A$ and $B$ belongs to $\mathcal{GL}(n;\mathbb{R})$ and their corresponding transposes $A^\top$ and $B^\top$ satisfy the following identities
\begin{equation}\label{AB}
\begin{split}
&
  A^\top A-B^\top B=I_n,\\[6pt]
&
  B^\top A+A^\top B=O_n,
\end{split}
\end{equation}
where $I_n$ and $O_n$ are the
unit matrix and the zero matrix of size $n$, respectively.

%%%%%%%%%%%%%%%%%%%%%%%%%%%%%%%%%%%%%%%%%%%%%%%%%%%%%%%%%%%%%%%%%%%%%%%%%%%%%%%%%1
\vskip 0.2in \addtocounter{subsection}{1} \setcounter{subsubsection}{0}

\noindent  {\Large\bf \thesubsection. First covariant derivatives}%\\[6pt]\vskip2pt}

%\vskip 0.15in
%\section

\vskip 0.2in \addtocounter{subsubsection}{1}

\noindent  {\Large\bf{\emph{\thesubsubsection. Fundamental tensor $F$}}}

\vskip 0.15in
%\subsection

The fundamental $(0,3)$-tensor $F$ on $\MM$ is defined by
\begin{equation}\label{F'}
F(x,y,z)={g}\bigl( \left(\DDD_x J \right)y,z\bigr).
\end{equation}
It has the following basic properties: \cite{GriMekDje85a}
\begin{equation}\label{F'-prop}%\label{2.3}
F(x,y,z)=F(x,z,y)=F(x,Jy,Jz)
\end{equation}
and their consequence
\[
F(x,Jy,z)=-F(x,y,Jz).
\]

Let $\{e_i\}$ ($i=1,2,\dots,2n$) be an arbitrary
basis of the tangent space of $\MM$ at
any its point and let ${g}^{ij}$ be the corresponding components
of the inverse matrix of ${g}$.
Then, the corresponding Lee forms of $F$ with respect to ${g}$ and
$\widetilde {g}$ are defined by
\[
\ta(z)={g}^{ij}F(e_i,e_j,z),\qquad
\widetilde \ta(z)=\widetilde g^{ij}F(e_i,e_j,z),
\]
respectively. They imply
the relation
\begin{equation}\label{wtataJ}
\widetilde \ta=\ta\circ J
\end{equation}
because of
\[
\widetilde g^{ij}F(e_i,e_j,z)=-g^{ij}F(e_i,Je_j,z)=g^{ij}F(e_i,e_j,Jz).
\]

Somewhere, instead of $\widetilde \ta$ it is used the 1-form ${\ta^*}$ associated with $F$, which is defined by
\begin{equation}\label{ta'}
{\ta^*}(z)={g}^{ij}F(e_i,Je_j,z).
\end{equation}
Using $\g$, we have the following
\[
\begin{split}
  {\ta^*}(z)  &={g}^{ij}F(e_i,Je_j,z)=J^j_k {g}^{ik}F(e_i,e_j,z) \\[6pt]
                &=-\g^{ij}F(e_i,e_j,z).
\end{split}
\]
Then the identity
\begin{equation}\label{ta*taJ}
{\ta^*}=-\widetilde{\ta}
\end{equation}
holds by means of \eqref{F'-prop}, because of \eqref{wtataJ} and
\[
\begin{split}
{\ta^*}(z)&={g}^{ij}F(e_i,Je_j,z)=-{g}^{ij}F(e_i,e_j,Jz)\\[6pt]
            &=-{\ta}(Jz).
\end{split}
\]
%There exists an associated
%Lee form $\ta^*$ defined by $\ta^*(z)=g^{ij}F(e_i,Je_j,z)

The almost Norden manifolds are classified into
basic classes $\W_i$ $(i=1,2,3)$ with respect to $F$ by G.~Ganchev and A.~Borisov in \cite{GaBo}.
All classes are determined  as follows:
\begin{subequations}\label{Wi}
\begin{equation}%\label{Wi}%\label{class}
\begin{split}
\W_0:\quad &F(x,y,z)=0;\\[6pt]
\W_1:\quad &F(x,y,z)=\dfrac{1}{2n}\bigl\{{g}(x,y)\ta(z)%\\[6pt]
%&\phantom{F(x,y,z)=\dfrac{1}{2n}\bigl\{}
+\g(x,y)\widetilde\ta(z)\\[6pt]
&\phantom{F(x,y,z)=\dfrac{1}{2n}\bigl\{}
+{g}(x,z)\ta(y)\\[6pt]
&\phantom{F(x,y,z)=\dfrac{1}{2n}\bigl\{}
+\g(x,z)\widetilde\ta(y)\bigr\};
\\[6pt]
\W_2:\quad &\sx F(x,y,Jz)=0,\quad \ta=0;\\[6pt]
\W_3:\quad &\sx F(x,y,z)=0;
\\[6pt]
\W_1\oplus\W_2:\quad &\sx F(x,y,Jz)=0;
\\[6pt]
\W_1\oplus\W_3:\quad &\sx F(x,y,z)=\\[6pt]
&\phantom{}%F(x,y,z)=\dfrac{1}{2n}\bigl\{}
=\dfrac{1}{n}\sx\bigl\{{g}(x,y)\ta(z)%\\[6pt]
%&\phantom{=\dfrac{1}{n}\sx\bigl\{}
+\widetilde{g}(x,y)\widetilde\ta(z)\bigr\};
\\[6pt]
\end{split}
\end{equation}
\begin{equation}%\label{class}
\begin{split}
\W_2\oplus\W_3:\quad &\ta=0;\\[6pt]
\W_1\oplus\W_2\oplus\W_3:\quad &\textrm{no conditions}.
\end{split}
\end{equation}
\end{subequations}
%In the latter equalities the notation $\sX$ means the cyclic sum by the three arguments $x$, $y$, $z$.

The class $\W_0$ is a special class that belongs to any other class, i.e. it is their intersection. It contains K\"ahler manifolds with Norden metric (known also as \emph{K\"ahler-Norden manifolds}, a K\"ahler manifolds with B-metric or a holomorphic
complex Riemannian manifolds).

%The \emph{square norm} $\nJ$ of $\DDD^g J$ is defined in  \cite{MeMa}
%by
%\begin{equation}\label{2.9}
 %   \nJ=g^{ij}g^{ks}
 %   g\bigl((\DDD^\g_{e_i} J)e_k,(\DDD^\g_{e_j}
 %   J)e_s\bigr)
%\end{equation}
%and the considered manifold with $\nJ=0$ is called an
%\emph{isotropic-K\"ah\-ler manifold with Norden metric}. It is
%clear that every K\"ahler manifold with Norden metric is
%isotropic-K\"ahler, but the inverse implication is not always
%true.

%\subsection{Preliminaries}\label{sec-prelim}

%Citirana e paper\#22

\vskip 0.2in \addtocounter{subsubsection}{1}

\noindent  {\Large\bf{\emph{\thesubsubsection. Isotropic K\"ahler-Norden manifolds}}}

\vskip 0.15in
%\subsection{Isotropic K\"ahler manifolds}\label{sec-iK}

The square norm $\DDJ$
 of $\DDD{J}$ with respect to the metric $g$ is defined in \cite{GRMa} as follows
 \[
    \DDJ=g^{ij}g^{kl}
        g\bigl((\DDD_{e_i} J)e_k,(\DDD_{e_j} J)e_l\bigr).
\]
By means of \eqref{F'} and \eqref{F'-prop}, we obtain the following equivalent formula,
which is given in terms of the components of $F$
\begin{equation}\label{snorm}
    \DDJ=g^{ij}g^{kl}g^{st}(F)_{iks}(F)_{jlt},
\end{equation}
where $(F)_{iks}$ denotes $F(e_i,e_k,e_s)$.

An almost Norden manifold satisfying the
condition $\DDJ=0$ is called an isotropic K\"ahler manifold
with Norden metrics \cite{MekMan} or an \emph{isotrop\-ic K\"ahler-Norden manifold}.

Let us remark that if a manifold belongs to $\W_0$, then
it is an isotrop\-ic K\"ahler-Norden manifold but the inverse statement is not
always true. It can be noted that %https://en.wikipedia.org/wiki/Wikipedia:%22Note_that%22_is_unnecessary
the class of isotrop\-ic K\"ahler-Norden manifolds is
the closest larger class of almost Norden manifolds containing the class of K\"ahler-Norden manifolds.

\vskip 0.2in \addtocounter{subsubsection}{1}

\noindent  {\Large\bf{\emph{\thesubsubsection. Fundamental tensor $\Phi$}}}

\vskip 0.15in
%\subsection{The pair of the Nijenhuis tensors}

Let us consider the tensor $\Phi$ of type (1,2) defined in \cite{GaGrMi87} as the difference of the Levi-Civita connections $\tD$ and $\DDD$ of the corresponding Norden metrics as follows
\begin{equation}\label{Phi}
    \Phi(x,y)=\tD_x y-\DDD_x y.
\end{equation}
This tensor is known also as the \emph{potential} of $\tD$ regarding $\DDD$ because of the formula
\begin{equation}\label{tn=nPhi}
    \tD_x y=\DDD_x y+\Phi(x,y).
\end{equation}
Since both the connections are torsion-free, then $\Phi$ is symmetric, i.e. $\Phi(x,y)=\Phi(y,x)$ holds.
Let the corresponding tensor of type $(0,3)$ with respect to $g$ be defined by
\begin{equation}\label{Phi03}
    \Phi(x,y,z)=g(\Phi(x,y),z).
\end{equation}
     By virtue of properties \eqref{F'-prop}, the following interrelations between $F$ and $\Phi$ are valid \cite{GaGrMi87}
\begin{equation}\label{PhiJFJ}
\begin{split}
      \Phi(x,y,z) =\dfrac{1}{2}\bigl\{F(Jz,x,y) &-F(x,y,Jz)\\[6pt]
                                                    &-F(y,x,Jz)\bigr\},
\end{split}
\end{equation}
\begin{equation}\label{FPhi'}
    F(x,y,z)=\Phi(x,y,Jz)+\Phi(x,z,Jy).
\end{equation}

Let the corresponding 1-form of $\Phi$ be denoted by $f$ and let be defined as follows
\[
f(z)=g^{ij}\Phi(e_i,e_j,z).
\]
Using \eqref{PhiJFJ} and \eqref{wtataJ}, we obtain the relation $f=-\widetilde{\ta}$.

An equivalent classification of the given one in \eqref{Wi} is proposed in \cite{GaGrMi87},
where all classes are defined in terms of $\Phi$ as follows:
\begin{equation}\label{class2}
\begin{split}
\W_0:\quad &\Phi(x,y,z)=0;\\[6pt]
\W_1:\quad &\Phi(x,y,z)=\dfrac{1}{2n}\bigl\{g(x,y)f(z)\\
&\phantom{\Phi(x,y,z)=\dfrac{1}{2n}\bigl\{}
+g(x,Jy)f(Jz)\bigr\};\\[6pt]
\W_2:\quad &\Phi(x,y,z)=-\Phi(Jx,Jy,z),\quad f=0;\\[6pt]
\W_3:\quad &\Phi(x,y,z)=\Phi(Jx,Jy,z)\\[6pt]
\W_1\oplus\W_2:\quad &\Phi(x,y,z)=-\Phi(Jx,Jy,z), \\[6pt]
                  &\Phi(Jx,y,z)=-\Phi(x,y,Jz);\\[6pt]
\W_1\oplus\W_3:\quad &\Phi(x,y,z)=\Phi(Jx,Jy,z)\\
&%\phantom{\Phi(x,y,z)=}
+\dfrac{1}{n}\bigl\{{g}(x,y)f(z)
+{g}(x,Jy)f(Jz)\bigr\};\\[6pt]
\W_2\oplus\W_3:\quad &f=0;\\[6pt]
\W_1\oplus\W_2\oplus\W_3:\quad &\textrm{no conditions}.
\end{split}
\end{equation}
                                                            %%%

\newpage
\vskip 0.2in \addtocounter{subsubsection}{1}

\noindent  {\Large\bf{\emph{\thesubsubsection. Pair of the Nijenhuis tensors}}}

\vskip 0.15in
%\subsection{The pair of the Nijenhuis tensors}

As it is well known, the Nijenhuis tensor (let us denote it by the brackets $[\cdot,\cdot]$) of the almost complex structure $J$ is
defined by
\begin{equation}\label{NJ}
\begin{split}
[J,J](x,y) %&= [J, J](x, y)\\[6pt]
        &=\left[Jx,Jy\right]-\left[x,y\right]-J\left[Jx,y\right]-J\left[x,Jy\right].
\end{split}
\end{equation}

Besides it, we give the following definition in analogy to \eqref{NJ},
where the symmetric braces $\{x,y\}=\DDD_xy+\DDD_yx$
are used instead of the antisymmetric brackets $[x,y]=\DDD_xy-\DDD_yx$.
More precisely, the symmetric braces $\{x,y \}$ are determined by
\begin{equation}\label{braces}
\begin{split}
g(\{x,y\},z)&=g(\DDD_xy+\DDD_yx,z)\\[6pt]
&=x\left(g(y,z)\right)+y\left(g(x,z)\right)-z\left(g(x,y)\right)\\[6pt]
&\phantom{=x\left(g(y,z)\right)\,\,}-g([y,z],x)+g([z,x],y).
\end{split}
\end{equation}
\begin{dfn}
The symmetric (1,2)-tensor $\{J,J\}$ defined by
\begin{equation}\label{hatNJ}
%\begin{split}
\{J,J\}(x,y)=\{Jx,Jy\}-\{x,y\}-J\{Jx,y\}-J\{x,Jy\}
%\end{split}
\end{equation}
is called the \emph{associated Nijenhuis tensor of the almost complex structure $J$ on $(\MM,J,g)$}.
\end{dfn}

The Nijenhuis tensor and its associated tensor for
$J$ are determined in terms of the covariant derivatives of $J$ as follows:
\begin{equation}\label{NJ-nabli}
    \begin{array}{l}
      [J,J](x,y)=\left(\DDD_x J\right)J y-\left(\DDD_y
J\right)J x%\\[6pt]
%\phantom{[J,J](x,y)=}
+\left(\DDD_{J x}
J\right)y-\left(\DDD_{J y} J\right)x, \\[6pt]
      \{J,J\}(x,y)=\left(\DDD_x J\right)J y+\left(\DDD_y
J\right)J x%\\[6pt]
%\phantom{\{J,J\}(x,y)=}
+\left(\DDD_{J x}
J\right)y+\left(\DDD_{J y} J\right)x.
    \end{array}
\end{equation}

The tensor $\{J,J\}$ coincides with the associated tensor $\widetilde{N}$ of $[J,J]$ introduced in \cite{GaBo} by the latter equality above.

The pair of the Nijenhuis tensors $[J,J]$ and $\{J,J\}$
plays a fundamental role in the topic
of natural connections (i.e. such connections that $J$ and $g$ are parallel with respect to them) on an almost Norden manifold.
The torsions and the potentials of these connections
are expressed by these two tensors. Because of that
we characterize the classes of the considered manifolds in terms of $[J,J]$ and $\{J,J\}$.

As it is known from \cite{GaBo}, the class $\W_3$ of
the \emph{quasi-K\"ahler manifolds with Norden
metric} is the only basic class of almost Norden manifolds with
non-integrable almost complex structure $J$, because $[J,J]$ is
non-zero there. Moreover, this class is determined by the condition
$\{J,J\}=0$. The class $\W_1\oplus\W_2$ of the
\emph{(integrable almost) complex manifolds with Norden metric}
is characterized by $[J,J]=0$ and $\{J,J\}\neq 0$.
Additionally, the basic classes $\W_1$ and $\W_2$
are distinguish from each other according to the
Lee form $\ta$: for $\W_1$ the tensor $F$ is expressed explicitly by the metric and the Lee form, i.e.
$\ta\neq 0$; whereas for $\W_2$ the equality $\ta=0$ is valid.

The corresponding $(0,3)$-tensors are determined as follows
\[
[J,J](x,y,z)=g([J,J](x,y),z),\quad \{J,J\}(x,y,z)=g(\{J,J\}(x,y),z).
\]
Then, using \eqref{NJ} and \eqref{hatNJ}, the (0,3)-tensors $[J,J]$ and $\{J,J\}$ can be expressed in
terms of $F$ by: %\cite{GaBo}
\begin{equation}\label{NF'}
\begin{split}
[J,J](x,y,z)&=F(x,Jy,z)-F(y,Jx,z)\\[6pt]
&+F(Jx,y,z)-F(Jy,x,z),\\[6pt]%
\end{split}
\end{equation}
\begin{equation}\label{NhatF'}
\begin{split}
\{J,J\}(x,y,z)&=F(x,Jy,z)+F(y,Jx,z)\\[6pt]
                    &+F(Jx,y,z)+F(Jy,x,z)
\end{split}
\end{equation}
or equivalently
\begin{equation}\label{NhatF'2}
\begin{split}
\{J,J\}(x,y,z)&=F(Jx,y,z)+F(Jy,x,z)\\[6pt]
                    &-F(x,y,Jz)-F(y,x,Jz).
\end{split}
\end{equation}
%The tensor $\{J,J\}$ coincides with the tensor $\widetilde N$
%introduced in \cite{GaBo} by an equivalent equality of
%\eqref{NhatF'}.

By virtue of \eqref{2.1}, \eqref{F'-prop}, \eqref{NF'} and \eqref{NhatF'}, we get the following properties
of  $[J,J]$ and $\{J,J\}$:
\begin{equation}\label{N-prop}
\begin{array}{l}
[J,J](x,y,z)=[J,J](x,Jy,Jz)=[J,J](Jx,y,Jz)\\[6pt]
\phantom{[J,J](x,y,z)}
         =-[J,J](Jx,Jy,z),\\[6pt]
[J,J](Jx,y,z)=[J,J](x,Jy,z)=-[J,J](x,y,Jz);
\end{array}
\end{equation}
\begin{equation}\label{hatN-prop}
\begin{array}{l}
\{J,J\}(x,y,z)=\{J,J\}(x,Jy,Jz)=\{J,J\}(Jx,y,Jz)\\[6pt]
\phantom{\{J,J\}(x,y,z)}
                =-\{J,J\}(Jx,Jy,z), \\[6pt]
\{J,J\}(Jx,y,z)=\{J,J\}(x,Jy,z)=-\{J,J\}(x,y,Jz).
\end{array}
\end{equation}

\begin{thm}\label{thm:FN}
The fundamental tensor $F$ of an almost Norden manifold $(\MM,J,g)$
is expressed in terms of the Nijenhuis tensor $[J,J]$ and its associated Nijenhuis tensor $\{J,J\}$ by the formula
\begin{equation}\label{F=NhatN}
\begin{split}
F(x,y,z)&=-\dfrac14\bigl\{[J,J](Jx,y,z)+[J,J](Jx,z,y)\\[6pt]
&\phantom{=-\dfrac14\bigl\{}
+\{J,J\}(Jx,y,z)+\{J,J\}(Jx,z,y)\bigr\}.
\end{split}
\end{equation}
\end{thm}
\begin{proof}
Taking the sum of \eqref{NF'} and \eqref{NhatF'}, we obtain
\begin{equation}\label{s1}
F(Jx,y,z)+F(x,Jy,z)=\dfrac12\bigl\{[J,J](x,y,z)+\{J,J\}(x,y,z)\bigr\}.
\end{equation}
The  identities  \eqref{2.1} and \eqref{F'-prop} imply
\begin{equation}\label{s2}
F(x,z,Jy)=-F(x,y,Jz).
\end{equation}
A suitable combination of \eqref{s1} and \eqref{s2} yields
\begin{equation}\label{FJN}
\begin{split}
F(Jx,y,z)&=\dfrac14\bigl\{[J,J](x,y,z)+[J,J](x,z,y)\\[6pt]
&\phantom{=\dfrac14\bigl\{}
+\{J,J\}(x,y,z)+\{J,J\}(x,z,y)\bigr\}.
\end{split}
\end{equation}
Applying \eqref{2.1} to \eqref{FJN}, we obtain the stated formula.
\end{proof}

As direct corollaries of \thmref{thm:FN}, for the classes of the considered manifolds with vanishing $[J,J]$ or $\{J,J\}$, we have respectively:
\begin{equation}\label{Wi:N0Nhat0}
\begin{split}
\W_1\oplus\W_2: \; &F(x,y,z)=-\dfrac{1}{4}\bigl(\{J,J\}(Jx,y,z)\\
                   &\phantom{F(x,y,z)=-\dfrac{1}{4}\bigl(}+\{J,J\}(Jx,z,y)\bigr),\\[6pt]
\W_{3}: \; &F(x,y,z)=-\dfrac{1}{4}\bigl([J,J](Jx,y,z)\\
           & \phantom{F(x,y,z)=-\dfrac{1}{4}\bigl(}+[J,J](Jx,z,y)\bigr).
\end{split}
\end{equation}

According to \thmref{thm:FN}, we obtain the following relation between the corresponding traces:
\begin{equation}\label{ta=nu}
\ta=\dfrac14\vartheta\circ J,
\end{equation}
where we denote
\[
{\vartheta}(z)={g}^{ij}\{J,J\}(e_i,e_j,z).
\]
For the traces $\widetilde\ta$ and $\widetilde{\vartheta}$ with respect to the associated metric
$\widetilde g$  of $F$ and $\{J,J\}$, i.e.
\[
\widetilde\ta(z)=\widetilde g^{ij}F(e_i,e_j,z), \qquad
\widetilde{\vartheta}(z)=\widetilde g^{ij}\{J,J\}(e_i,e_j,z),
\]
we obtain the following interrelations:
\[
\widetilde\ta=-\dfrac14\vartheta=\ta\circ J, \qquad
\widetilde{\vartheta}=4\ta=\vartheta\circ J.
\]
Then, bearing in mind \eqref{Wi} and the subsequent comments on the pair
of the Nijenhuis tensors, from \thmref{thm:FN} and \eqref{ta=nu} we obtain
immediately the following
\begin{thm}\label{prop:Wi:N}
%The basic classes $\W_i$ $(i=1,2,3,4)$ of almost Norden manifolds and their direct sums
The classes of almost Norden manifolds
are character\-ized by the pair of Nijenhuis tensors $[J,J]$ and $\{J,J\}$ as follows:
%\begin{subequations}\label{Wi:N}
\begin{equation}\label{Wi:N}
\begin{array}{rll}
\W_{0}: \quad &[J,J]=0, \quad  \{J,J\}=0;\\[6pt]
\W_{1}: \quad &[J,J]=0, \quad
\{J,J\}=\dfrac{1}{2n}\bigl\{g\otimes\vartheta+\widetilde{g}\otimes\widetilde{\vartheta}
\bigr\};\\[6pt]
%
%\end{array}
%\end{equation}
%\begin{equation}
%\begin{array}{rll}
\W_{2}: \quad   &[J,J]=0, \quad \vartheta=0;\\[6pt]
\W_{3}: \quad %& \\[6pt]
&\{J,J\}=0;\\[6pt]
\W_{1}\oplus\W_2: \quad %\\[6pt]
&[J,J]=0; \quad \\[6pt]
%&\\[6pt]
%
\W_{1}\oplus\W_3: \quad %\\[6pt]
&\s\{J,J\}=\dfrac{1}{2n}\s\bigl\{g\otimes\vartheta+\widetilde{g}\otimes\widetilde{\vartheta}\bigr\};\\[6pt]
\phantom{\W_{1}\oplus}\W_{2}\oplus\W_3: \quad &\vartheta=0;\\[6pt]
\W_{1}\oplus\W_{2}\oplus\W_3: \quad
&%\multicolumn{2}{l}{\quad \text{no conditions}.}
\textrm{no conditions}.
\end{array}
\end{equation}
%\end{subequations}
\end{thm}

%%%%%%%%%%%%%%%%%%%%%%%%%%%%%%%%%%%%%%%%%%%%%%%%%%%%%%%%%%%%%%%%%%%%%%%%%%%%%%%%%1
\vskip 0.2in \addtocounter{subsection}{1} \setcounter{subsubsection}{0}

\noindent  {\Large\bf \thesubsection. Second covariant derivatives}%\\[6pt]\vskip2pt}

\vskip 0.15in
%\section{Almost complex manifolds with Norden metric}

%\subsection{Curvature properties}\label{sec-curv}

Let $R^{\DDD}$ be the curvature tensor of $\DDD$ defined as usual by
\begin{equation}\label{K}
    R^{\DDD}(x,y)z=\DDD_x \DDD_y z - \DDD_y \DDD_x z -
    \DDD_{[x,y]}z.
\end{equation}
As it is known, the latter formula can be rewritten as
\[
R^{\DDD}(x,y) z={\DDD}^{\thinspace{}2}_{x,y} z - {\DDD}^{\thinspace{}2}_{y,x} z,
\]
using the second covariant derivative given by
\[
{\DDD}^{\thinspace{}2}_{x,y} z = {\DDD}_x {\DDD}_y z - {\DDD}_{\thinspace{}{\DDD}_x y} z.
\]

The corresponding tensor of type $(0,4)$ with respect to the metric $g$ is determined by
\[
R^{\DDD}(x,y,z,w)=g(R^{\DDD}(x,y)z,w).
\]
It has the following properties:
\begin{equation}\label{curv}
\begin{split}%
&R^{\DDD}(x,y,z,w)=-R^{\DDD}(y,x,z,w)=-R^{\DDD}(x,y,w,z),\\[6pt]
&R^{\DDD}(x,y,z,w)+R^{\DDD}(y,z,x,w)+R^{\DDD}(z,x,y,w)=0.
\end{split}%
\end{equation}

Any tensor of type (0,4) satisfying \eqref{curv}
is called a \emph{curvature-like tensor}. %, i.e.
%an arbitrary (0,4)-tensor $L$ having the pro\-per\-ties
%\[
%\begin{split}
%&L(x,y,z,w)=-L(y,x,z,w)=-L(x,y,w,z), \\[6pt]
%&\mathop{\sX} \limits_{x,y,z} L(x,y,z,w)=0 %\quad (\text{the first Bianchi identity})
%%\label{2.12}
%\end{split}
%\] %
%is a curvature-like tensor.
Moreover, if the
curvature-like tensor $R^{\DDD}$ has the property
\[
R^{\DDD}(x,y,Jz,Jw)=-R^{\DDD}(x,y,z,w),
\]
it is called a \emph{K\"ahler tensor} \cite{GaGrMi87}.

The Ricci tensor
$\rho^{\DDD}$ and the scalar curvature $\tau^{\DDD}$ for the curvature tensor of $\DDD$ (and similarly for every curvature-like tensor) are defined as usual by
\[
\rho^{\DDD}(y,z)=g^{ij}R^{\DDD}(e_i,y,z,e_j),\qquad \tau^{\DDD}=g^{ij}\rho^{\DDD}(e_i,e_j).
\]
%where $g^{ij}$ are the corresponding components of the inverse matrix of $g$ with respect to an arbitrary basis $\{e_i\}$ ($i=1,\dots, 2n$)
%of the tangent space of $\MM$ at an arbitrary point $q$
%in $\MM$.

It is well-known that the Weyl tensor $C^{\DDD}$ on a pseudo-Rie\-mann\-ian manifold $(\MM,g)$, $\dim{\MM}=2n\geq 4$,
is given by
\begin{equation}\label{W'}
C^{\DDD}=R^{\DDD}+\dfrac{1}{2(n-1)}g\owedge\rho^{\DDD}-\dfrac{\tau^{\DDD}}{4(n-1)(2n-1)}g\owedge g,
\end{equation}
where $g\owedge\rho^{\DDD}$ is the Kulkarni-Nomizu product of $g$ and $\rho^{\DDD}$, i.e.
\begin{equation}\label{Kulkarni}
\begin{split}
\left(g\owedge\rho^{\DDD}\right)(x,y,z,w)&=g(x,z)\rho^{\DDD}(y,w)-g(y,z)\rho^{\DDD}(x,w)\\[6pt]
&+g(y,w)\rho^{\DDD}(x,z)-g(x,w)\rho^{\DDD}(y,z).
\end{split}
\end{equation}
Moreover, $C^{\DDD}$ vanishes if and only if the manifold $(\MM,g)$
is conformally flat, i.e. it is transformed into a flat manifold by an usual conformal transformation of the metric defined by $\overline{g}=e^{2u}g$ for a differentiable function $u$ on $\MM$.

Let $R^{\tD}$ be the curvature tensor of $\tD$ defined as usually.
Obviously, the corresponding curvature (0,4)-tensor with respect to the metric $\g$ is
\[
R^{\tD}(x,y,z,w)=\g(R^{\tD}(x,y)z,w)
\]
and it has the same properties as in \eqref{curv}.
The Weyl tensor $C^{\tD}$ is generated by $\tD$ and $\g$ by the same way and it has the same geometrical interpretation for the manifold $(\MM,J,\g)$.

\vspace{20pt}

\begin{center}
$\divideontimes\divideontimes\divideontimes$
\end{center}

\newpage

\addtocounter{section}{1}\setcounter{subsection}{0}\setcounter{subsubsection}{0}

\setcounter{thm}{0}\setcounter{dfn}{0}\setcounter{equation}{0}

\label{par:inv}

 \Large{

\
\\[6pt]
\bigskip

\
\\[6pt]
\bigskip

\lhead{\emph{Chapter I $|$ \S\thesection. Invariant tensors under the twin interchange
of Norden metrics \ldots% on almost complex \ldots %mani\-folds
}}
%\thispagestyle{empty}

%\noindent  {\Huge\bf \S\thesection. Invariant tensors under the twin \\[12pt]
%\phantom{\S\thesection. }interchange
%%of the counterparts
%%in the pair
%of Norden metrics\\[12pt]
%\phantom{\S\thesection. }on almost complex manifolds
%}%\\[6pt]\vskip2pt}
\noindent
\begin{tabular}{r"l}
  %\hline
\hspace{-6pt}{\Huge\bf \S\thesection.}  & {\Huge\bf Invariant tensors under the twin} \\[12pt]
                             & {\Huge\bf interchange of Norden metrics} \\[12pt]
                             & {\Huge\bf on almost complex manifolds}
  %\hline
\end{tabular}

\vskip 1cm

\begin{quote}
\begin{large}
The object of study in the present section are almost complex manifolds with a pair of Norden metrics, mutually associated by means of the almost complex structure. More precisely, a torsion-free connection and tensors with certain geometric interpretation are found which are invariant under the twin interchange, i.e. the swap of the counterparts of the pair of Norden metrics and the corresponding Levi-Civita connections. A Lie group depending on four real parameters is considered as an example of a 4-dimensional manifold of the studied type. The mentioned invariant objects are found in an explicit form.

The main results of this section are published in \cite{Man49}.
\end{large}
\end{quote}

%
%\vskip 0.2in \addtocounter{subsection}{1}
%
%\noindent  {\Large\bf \thesubsection. Introduction}

\vskip 0.15in

%%%%%%%%%%%%%%%%%%%%%%%%%%%%%%%%%%%%%%%%%%%%%%%%%%%%%%%%%%%%%%%%%%%%%%%%%%%0
%\section*{Introduction}

%Hermitian metrics on almost complex manifolds are well known. Then the almost complex structure $J$ acts as an isometry with respect to the (Riemannian or pseudo-Riemannian) metric. The associated (0,2)-tensor of the Hermitian metric is a 2-form. Other case is when the almost complex structure acts as an anti-isometry regarding a pseudo-Riemannian metric. Such a metric is called a Norden metric. The associated (0,2)-tensor of any Norden metric is also a Norden metric. So, in this case we dispose with a pair of mutually associated Norden metrics, known also as twin Norden metrics. These \emph{almost Norden manifolds} are studied in the latter three decades, in the beginning under the names generalized B-manifolds \cite{GriMekDje85a}, almost complex manifolds with Norden metric \cite{GaBo} and almost complex manifolds with B-metric \cite{GaGrMi85}.

An interesting problem on almost Norden manifolds is the presence of tensors with some geometric interpretation which are invariant under the so-called twin interchange. This is the swap of the counterparts of the pair of Norden metrics and their Levi-Civita connections. Similar results for the considered manifolds in the basic classes $\W_1$ and $\W_3$ are obtained in \cite{MT06th} and \cite{DjDo}, \cite{MekManGri22}, respectively. The aim here %of the present section
is to solve the problem in general.

The present section is organised as follows. In Subsection~\thesection.1 we present the main results on the topic about the invariant objects and their vanishing. In Subsection~\thesection.2 we consider an example of the studied manifolds of dimension 4 by means of a construction of an appropriate family of Lie algebras depending on 4 real parameters. Then we compute the basic components of the invariant objects introduced in the previous subsection.

%In this section, $x$, $y$, $z$ will stand for
%arbitrary differentiable vector fields on the even dimensional manifold $\MM$ (or vectors in the
%tangent space $T_{q}\MM$ of $\MM$ at an arbitrary point $q\in \MM$).

\newpage
%%%%%%%%%%%%%%%%%%%%%%%%%%%%%%%%%%%%%%%%%%%%%%%%%%%%%%%%%%%%%%%%%%%%%%%%%%%  &2
\vskip 0.2in \addtocounter{subsection}{1}

\noindent  {\Large\bf \thesubsection. The twin interchange
corresponding to the pair of Norden metrics and their Levi-Civita connections}

\vskip 0.15in
%\section{The twin interchange corresponding to the pair of Norden metrics and their Levi-Civita connections}

%\subsection{The relations between corresponding tensors for the twin interchange}

% НЕ СЕ ИЗПОЛЗВА
%Since \eqref{F7-prop} holds, then from \eqref{FPhi} follows the following identity
%\begin{equation}\label{Phi-prop}
%  \Phi'(x,Jy,z)+\Phi'(x,y,Jz)+\Phi'(x,Jz,y)+\Phi'(x,z,Jy)=0.
%\end{equation}

Bearing in mind \eqref{twin}, we give the following
\begin{dfn}%[\cite{Man49}]
The interchange of the Levi-Civita connections $\DDD$ and $\tD$ (and respectively their corresponding Norden metrics $g$ and $\g$) we call the \emph{twin interchange}.
\end{dfn}

\vskip 0.2in \addtocounter{subsubsection}{1}

\noindent  {\Large\bf{\emph{\thesubsubsection. Invariant classification}}}%\\[6pt]\vskip2pt}

\vskip 0.15in
%\subsection{Invariant classification}

Let us consider the potential $\Phi$ of $\tD$ regarding $\DDD$ on $(\MM,J,g)$ defined by \eqref{Phi}.

\begin{lem}\label{lem:Phi}
  The potential $\Phi(x,y)$ is an anti-invariant tensor under the twin interchange, i.e.
\begin{equation}\label{tPsi=-Psi}
    \tP(x,y)=-\Phi(x,y).
\end{equation}
\end{lem}
\begin{proof}
The equalities \eqref{F'-prop},
\eqref{tn=nPhi},  \eqref{Phi03} and \eqref{PhiJFJ} imply the following relation between $F$ and its corresponding tensor $\tPP$ for $(\MM,J,\g)$, defined by $\tPP(x,y,z)=\g\bigl(\bigl(\tD_xJ\bigr)y,z\bigr)$,
\begin{equation}\label{tFF'}
\begin{split}
    \tPP(x,y,z)=\dfrac{1}{2}\bigl\{&F(Jy,z,x)-F(y,Jz,x)\\[6pt]&+F(Jz,y,x)-F(z,Jy,x)\bigr\}.
\end{split}
\end{equation}
Bearing in mind \eqref{PhiJFJ}, we write the corresponding formula for $\tP$ and $\tPP$ as
\begin{equation}\label{tPtF}
\begin{split}
  \tP(x,y,z)=\dfrac{1}{2}\bigl\{\tPP(Jz,x,y) &-\tPP(x,y,Jz) \\[6pt]
    &-\tPP(y,x,Jz)\bigr\}.
\end{split}
\end{equation}
Using \eqref{tFF'} and \eqref{tPtF}, we get an expression of $\tP$ in terms of $F$ and then by \eqref{PhiJFJ} we obtain
\begin{equation}\label{tPPhi}
\tP(x,y,z)=-\Phi(x,y,Jz).
\end{equation}
Taking into account that $\tP(x,y,z)$ is defined by \[\tP(x,y,z)=\g(\tP(x,y),z),\] we accomplish the proof.
\end{proof}

In \cite{GaGrMi85}, for an arbitrary almost Norden manifold, it is given the following identity
\begin{equation}\label{Phi-prop}
\begin{split}
  \Phi(x,y,z) &-\Phi(Jx,Jy,z)-\Phi(Jx,y,Jz) \\[6pt]
    &-\Phi(x,Jy,Jz)=0.
\end{split}
\end{equation}
The associated 1-forms $f$ and $f^*$ of $\Phi$ are defined by
\[
f(z)=g^{ij}\Phi(e_i,e_j,z),\qquad f^*(z)=g^{ij}\Phi(e_i,Je_j,z).
\]
%Obviously, $f(z)=g(\mathrm{tr}{\Phi},z)$ holds.
Then, from \eqref{Phi-prop} we get the identity
\begin{equation}\label{psi-prop}
f(z)=f^*(Jz).
\end{equation}
The latter identity resembles the equality $\ta(z)=\ta^*(Jz)$, equivalent to \eqref{ta*taJ}.
Indeed, there exists a relation between the associated 1-forms of $\Phi$ and $F$. It follows from \eqref{PhiJFJ} and has the form
\begin{equation}\label{fta}
f(z)=\ta^*(z),\qquad f^*(z)=-\ta(z).
\end{equation}

\begin{lem}\label{lem:psi}
The associated 1-forms $f$ and $f^*$ of $\Phi$ are invariant under the twin interchange, i.e.
\begin{equation*}\label{tff}
\widetilde{f}(z)=f(z), \qquad \widetilde f^*(z)=f^*(z).
\end{equation*}
\end{lem}
\begin{proof}
Taking the trace of \eqref{tPPhi} by $\g^{ij}=-J^j_kg^{ik}$ for $x=e_i$ and $y=e_j$, we have $\widetilde{f}(z)=f^*(Jz)$. Then, comparing the latter equality and \eqref{psi-prop}, we obtain the statement for $f$. The relation in the case of $f^*$ is valid because of \eqref{psi-prop}.
\end{proof}

\begin{lem}\label{lem:tata*}
The Lee forms $\ta$ and $\ta^*$ are invariant under the twin interchange, i.e.
\[
\ta(z)=\widetilde{\ta}(z),\qquad \ta^*(z)=\widetilde\ta^*(z).
\]
\end{lem}
\begin{proof} It follows directly from \lemref{lem:psi} and \eqref{fta}.
\end{proof}

\begin{thm}\label{thm:inv.cl}
  All classes $\W_i$ of almost Norden manifolds %according to the classification in \cite{GaBo}
  are invariant under the twin interchange.
\end{thm}
\begin{proof}
We use the classification by $\Phi$ in \cite{GaGrMi87}, the definitions of all classes are given in \eqref{class2}.

Applying \lemref{lem:Phi}, \lemref{lem:psi} as well as equalities \eqref{tPPhi} and \eqref{Phi03}, we get the following conditions for the considered classes in terms of $\tP$:
\[
\begin{split}
&\W_0:\; \tP=0;\\[6pt]
&\W_1:\; \tP(x,y,z)=\dfrac{1}{2n}\left\{\g(x,y)\widetilde{f}(z)+\g(x,Jy)\widetilde{f}(Jz)\right\};\\[6pt]
&\W_2:\;
\tP(x,y,z)=-\tP(Jx,Jy,z),\quad \widetilde{f}=0;
%\\[6pt]
\end{split}
\]
\[
\begin{split}
&\W_3:\; \tP(x,y,z)=\tP(Jx,Jy,z);\\[6pt]
&\W_1\oplus\W_2:\; \tP(x,y,z)=-\tP(Jx,Jy,z);\\[6pt]
&\W_2\oplus\W_3:\; \widetilde{f}=0;\\[6pt]
%\end{split}
%\]
%\[
%\begin{split}
&\W_1\oplus\W_3:\;\tP(x,y,z)=\tP(Jx,Jy,z)\\
&\phantom{\W_1\oplus\W_3:\;\tP(x,y,z)=}
+\dfrac{1}{n}\bigl\{\g(x,y)\widetilde{f}(z)+\g(x,Jy)\widetilde{f}(Jz)\bigr\};\\[6pt]
&\W_1\oplus\W_2\oplus\W_3:\;\emph{no condition}.
\end{split}
\]
Taking into account the latter characteristic conditions of the considered classes,
we obtain the truthfulness of the statement.
\end{proof}

%\end{proof}

Let us note that the invariance of $\W_1$ and $\W_3$ is proved in \cite{GaGrMi85} and \cite{MekManGri22}, respectively.

Actually, by \thmref{thm:inv.cl} we establish that the classification with basic classes $\W_i$ $(i=1,2,3)$ has four equivalent forms: in terms of $F$, $\tPP$, $\Phi$ and $\tP$.

\vskip 0.2in \addtocounter{subsubsection}{1}

\noindent  {\Large\bf{\emph{\thesubsubsection. Invariant connection}}}%\\[6pt]\vskip2pt}

\vskip 0.15in
%\subsection{Invariant connection}

Let us define an affine connection $\DDD^{\diamond}$ by
\begin{equation}\label{hn=nP}
\DDD^{\diamond}_x y=\DDD_x y+\dfrac12\Phi(x,y).
\end{equation}
By virtue of \eqref{tn=nPhi}, \eqref{tPsi=-Psi} and \eqref{hn=nP}, we have the following
\[
\begin{array}{l}
\widetilde{\DDD}^{\diamond}_x y=\tD_x y+\dfrac12\tP(x,y)=\DDD_x y+\Phi(x,y)-\dfrac12\Phi(x,y)\\[6pt]
\end{array}
\]
\[
\begin{array}{l}
\phantom{\widetilde{\DDD}^{\diamond}_x y}
=\DDD_x y+\dfrac12\Phi(x,y)=\DDD^{\diamond}_x y.
\end{array}
\]
Therefore, $\DDD^{\diamond}$ is an invariant connection under the twin interchange.
Bearing in mind \eqref{Phi}, we establish that $\DDD^{\diamond}$ is actually the \emph{average connection} of $\DDD$ and $\tD$, because
\begin{equation}\label{av.con}
\begin{split}
\DDD^{\diamond}_x y&=\DDD_x y+\dfrac12\Phi(x,y)=\DDD_x y+\dfrac12\left\{\tD_x y-\DDD_x y\right\}\\[6pt]
&=\dfrac12\left\{\DDD_x y+\tD_x y\right\}.
\end{split}
\end{equation}
So, we obtain
\begin{prop}\label{prop:inv.conn}
    The average connection $\DDD^{\diamond}$ of $\DDD$ and $\tD$ is an invariant connection under the twin interchange.
\end{prop}

\begin{cor}\label{cor:inv.conn}
    If the invariant connection $\DDD^{\diamond}$ vanishes then $(\MM,J,g)$ and $(\MM,J,\g)$ are K\"ahler-Norden manifolds and $\DDD=\tD$ also vanishes.
\end{cor}
\begin{proof}
Let us suppose that $\DDD^{\diamond}$ vanishes. Then we have the following relations $\DDD=-\tD\ $ and $\Phi=-2\DDD$, be\-cause of \eqref{hn=nP} and \eqref{av.con}. Hence we obtain
\[
[x,y]=\DDD_x y-\DDD_y x=-\dfrac12\{\Phi(x,y)-\Phi(y,x)\}=0
\]
and consequently, using the corresponding Koszul formula for $g$ and $\DDD$
\begin{equation}\label{koszul-g}
\begin{split}
2g(\DDD_x y,z)&=x\left(g(y,z)\right)+y\left(g(x,z)\right)-z\left(g(x,y)\right)\\
&+g([x,y],z)+g([z,x],y)+g([z,y],x),
\end{split}
\end{equation}
we get that $\DDD$ vanishes. Thus, $\tD$ and $\Phi$ vanish as well as $(\MM,J,g)$ and $(\MM,J,\g)$ belong to $\W_0$.
\end{proof}

\vskip 0.2in \addtocounter{subsubsection}{1}

\noindent  {\Large\bf{\emph{\thesubsubsection. Invariant tensors}}}%\\[6pt]\vskip2pt}

\vskip 0.15in
%\subsection{Invariant tensors}

As it is well-known, the Nijenhuis tensor $[J,J]$ of the almost complex structure $J$ is
defined by \eqref{NJ}.
%\begin{equation*}
%[J,J](x,y) = [J, J](x, y)=\left[Jx,Jy\right]-\left[x,y\right]-J\left[Jx,y\right]-J\left[x,Jy\right].
%\end{equation*}
%
Besides $[J,J]$, in \eqref{hatNJ} it is defined the
(1,2)-tensor $\{J,J\}$, called associated
Nijenhuis tensor of $J$ and $g$.
%\[
%\{J,J\}'(x,y)=\{J ,J\}(x,y)=\{Jx,Jy\}-\{x,y\}-J\{Jx,y\}-J\{x,Jy\},
%\]
%where the symmetric braces $\{x,y\}=\nabla_xy+\nabla_yx$
%are used instead of the antisymmetric brackets $[x,y]=\nabla_xy-\nabla_yx$.
%The tensor $\{J,J\}'$ is also called the  \emph{}.

\begin{prop}\label{prop:inv.NhN}
    The Nijenhuis tensor is invariant and the associated Nijenhuis tensor is anti-invariant under the twin interchange, i.e.
    \[
    [J,J](x,y)=\tN(x,y),\qquad \{J,J\}(x,y)=-\widetilde{\{J,J\}}(x,y).
    \]
\end{prop}
\begin{proof}
The relations of $[J,J]$ and $\{J,J\}$ with $\Phi$ are given in \cite{GaGrMi85} as follows
\begin{alignat}{2}
  &[J,J](x,y,z) = 2\Phi(z,Jx,Jy)-2\Phi(z,x,y),\label{NPhi}
\\[6pt]
  &\{J,J\}(x,y,z) = 2\Phi(x,y,z)-2\Phi(Jx,Jy,z).\label{wNPhi}
\end{alignat}
Using \eqref{tPPhi}, the latter equalities imply the following
\begin{gather}
  \tN(x,y,z)=-[J,J](x,Jy,z),\label{tNN}\\[6pt]
  \widetilde{\{J,J\}}(x,y,z)=-\{J,J\}(x,y,Jz).\label{twNwN}
\end{gather}
In \eqref{N-prop}, it is given the property $[J,J](x, y,z) = [J,J](x, Jy, Jz)$ which is equivalent to $[J,J](x,Jy,z) = -[J,J](x,y, Jz)$. Then \eqref{tNN} gets the form
\begin{gather}
  \tN(x,y,z)=[J,J](x,y,Jz).\label{tNN2}
\end{gather}
The equalities \eqref{tNN2} and \eqref{twNwN} yield the relations in the statement.
\end{proof}

The following relation between the curvature tensors of $\DDD$ and $\tD$, related by \eqref{tn=nPhi}, is
well-known:
\begin{equation}\label{tRRS}
    R^{\tD}(x,y)z=R^{\DDD}(x,y)z+P(x,y)z,
\end{equation}
where
\begin{equation}\label{P'}
\begin{split}
    P(x,y)z=& \left(\DDD_x \Phi\right)(y,z)- \left(\DDD_y \Phi\right)(x,z)\\
    &+\Phi\left(x,\Phi(y,z)\right)-\Phi\left(y,\Phi(x,z)\right).
\end{split}
\end{equation}

Let us consider the following tensor $A$, which is part of $P$:
\begin{equation}\label{A'}
A(x,y)z=\Phi(x,\Phi(y,z))-\Phi(y,\Phi(x,z)).
\end{equation}
\begin{lem}\label{lem:A}
  The tensor $A(x,y)z$ is invariant under the twin interchange, i.e.
\begin{equation}\label{A=tA}
A(x,y)z=\tA(x,y)z.
\end{equation}
\end{lem}
\begin{proof}
Since \eqref{tPsi=-Psi} is valid, we obtain immediately
\begin{equation*}\label{nPtnP}
\begin{split}
  &\Phi(x,\Phi(y,z)) -\Phi(y,\Phi(x,z)) \\[6pt]
    & =\tP(x,\tP(y,z))-\tP(y,\tP(x,z)),
\end{split}
\end{equation*}
which yields relation \eqref{A=tA}.
\end{proof}

\begin{lem}\label{lem:P}
  The tensor $P(x,y)z$ is anti-invariant under the twin interchange, i.e.
\begin{equation}\label{tS=-P}
    \wP(x,y)z=-P(x,y)z.
\end{equation}
\end{lem}
\begin{proof}
For the covariant derivative of $\Phi$ we have
\[
\left(\DDD_x\Phi\right)(y,z)=\DDD_x\Phi(y,z)-\Phi(\DDD_x y,z)-\Phi(y,\DDD_x z).
\]
Applying \eqref{tn=nPhi} and \eqref{tPsi=-Psi}, we get
\[
\begin{split}
\left(\DDD_x\Phi\right)(y,z)=&-(\tD_x\tP)(y,z)-\tP(x,\tP(y,z))\\
&+\tP(y,\tP(x,z))+\tP(z,\tP(x,y)).
\end{split}
\]
As a sequence of the latter equality and \eqref{A'} we obtain
\begin{equation}\label{nP}
\begin{split}
    &\left(\DDD_x\Phi\right)(y,z)-\left(\DDD_y\Phi\right)(x,z)\\[6pt]
    &=-(\tD_x\tP)(y,z)+(\tD_y\tP)(x,z)-2\widetilde{A}(x,y)z.
\end{split}
\end{equation}
Then, \eqref{P'}, \eqref{A'}, \eqref{A=tA} and \eqref{nP} imply relation \eqref{tS=-P}.
\end{proof}

\begin{prop}\label{prop:inv.tensor2}
    The curvature tensor $R^{\DDD^{\diamond}}$ of the average connection $\DDD^{\diamond}$ for $\DDD$ and $\tD$ is an invariant tensor under the twin interchange, i.e. $\widetilde{R^{\DDD^{\diamond}}}(x,y)z=R^{\DDD^{\diamond}}(x,y)z$.
\end{prop}
\begin{proof}
From \eqref{hn=nP}, using the formulae
\eqref{tn=nPhi}, \eqref{tRRS}, \eqref{P'} and \eqref{A'}, we get the following relation
\begin{equation*}\label{hRRS}
\begin{split}
    R^{\DDD^{\diamond}}(x,y)z=R^{\DDD}(x,y)z&+\dfrac12\left(\DDD_x\Phi\right)(y,z)\\[6pt]
                    &-\dfrac12\left(\DDD_y\Phi\right)(x,z)
    +\dfrac14 A(x,y)z,
\end{split}
\end{equation*}
which is actually
\begin{equation}\label{hRRSA}
    R^{\DDD^{\diamond}}(x,y)z=R^{\DDD}(x,y)z+\dfrac12P(x,y)z
    -\dfrac14 A(x,y)z.
\end{equation}
By virtue of \eqref{tRRS}, \eqref{A=tA}, \eqref{tS=-P} and \eqref{nP}, we establish the relation $\widetilde{R^{\DDD^{\diamond}}}=R^{\DDD^{\diamond}}$.
\end{proof}

As a consequence of \eqref{hRRSA}, we obtain the following
\begin{cor}\label{cor:inv.tensor2}
    The invariant tensor $R^{\DDD^{\diamond}}$ vanishes if and only if the following equality is valid
    \[
    R^{\DDD}(x,y)z=-\dfrac12 P(x,y)z+\dfrac14 A(x,y)z.
    \]
\end{cor}

Let us consider the average tensor $B$ of the curvature tensors $R^{\DDD}$ and $R^{\tD}$, respectively, i.e.
$B(x,y)z=\dfrac12\{R^{\DDD}(x,y)z+R^{\tD}(x,y)z\}$.
Then by \eqref{tRRS} we have
\begin{equation}\label{bR=RS}
    B(x,y)z=R^{\DDD}(x,y)z+\dfrac12 P(x,y)z.
\end{equation}

\begin{prop}\label{prop:inv.tensor}
    The average tensor $B$ of $R^{\DDD}$ and $R^{\tD}$ is an invariant tensor under the twin interchange, i.e. $B(x,y)z=\widetilde{B}(x,y)z$.
\end{prop}
\begin{proof}
Using \eqref{tRRS}, \eqref{bR=RS} and \eqref{tPsi=-Psi},  we have the following
\begin{equation*}\label{tbR=bR}
\begin{split}
    \widetilde{B}(x,y)z&=R^{\tD}(x,y)z+\dfrac12 \wP(x,y)z\\[6pt]
    \phantom{\widetilde{B}(x,y)z}
    &=R^{\DDD}(x,y)z+P(x,y)z-\dfrac12 P(x,y)z\\[6pt]
    \phantom{\widetilde{B}(x,y)z}
    &=R^{\DDD}(x,y)z+\dfrac12 P(x,y)z=B(x,y)z.
\end{split}
\end{equation*}
\vskip-2em
\end{proof}

Immediately from \eqref{bR=RS} we obtain the next
\begin{cor}\label{cor:bR=0}
    The invariant tensor $B$ vanishes if and only if the following equality is valid
    \[
    R^{\DDD}=-\dfrac12P.
    \]
\end{cor}

By virtue of \eqref{hRRSA} and \eqref{bR=RS}, we have the following relation between the invariant tensors $R^{\DDD^{\diamond}}$, $B$ and $A$
\begin{equation}\label{wRbRPhi}
    R^{\DDD^{\diamond}}(x,y)z=B(x,y)z
    -\dfrac14 A(x,y)z.
\end{equation}

\begin{thm}\label{thm:inv.tensors}
    Any linear combination of the average tensor $B$ of the curvature tensors $R^{\DDD}$ and $R^{\tD}$ and the curvature tensor $R^{\DDD^{\diamond}}$ of the average connection $\DDD^{\diamond}$ for $\DDD$ and $\tD$  is an invariant tensor under the twin interchange.
\end{thm}
\begin{proof}
It follows from \propref{prop:inv.tensor2} and \propref{prop:inv.tensor}.
\end{proof}

\vskip 0.2in \addtocounter{subsubsection}{1}

\noindent  {\Large\bf{\emph{\thesubsubsection. Invariant connection and invariant tensors on the manifolds in the main class}}}%\\[6pt]\vskip2pt}

\vskip 0.15in
%\subsection{Invariant connection and invariant tensors on the manifolds in the main class}

Now, we consider an arbitrary manifold $(\MM,J,g)$ belonging to the basic class $\W_1$. This class is known as the main class in the classification in \cite{GaBo}, because it is the only class where the fundamental tensor $F$ and the potential $\Phi$ are expressed explicitly by the metric. In this case, we have the form of $F$ and $\Phi$ in \eqref{Wi} and \eqref{class2}, respectively.
%and \cite{GaGrMi85}
%\begin{equation}
%\begin{array}{l}
%  \Phi(x,y,z)=\dfrac{1}{2n}\bigl\{g(x,y)f(z)+g(x,Jy)f(Jz)\bigr\}.\label{W'1:Phi}
%\end{array}
%\end{equation}
Taking into account \eqref{tFF'}, \eqref{Wi} and \eqref{F'-prop},
we obtain the following form of  $F$ %and $\Phi$
under the twin interchange
%, we obtain
\begin{equation*}\label{W'1:tF}
\begin{array}{l}
  \tPP(x,y,z)=-\dfrac{1}{2n}\bigl\{
  g(x,y)\ta(Jz)+g(x,z)\ta(Jy)\\[6pt]
  \phantom{\tPP(x,y,z)=-\dfrac{1}{2n}\bigl\{}
  -g(x,Jy)\ta(z)-g(x,Jz)\ta(y)\bigr\}.
\end{array}
\end{equation*}
%By \eqref{tPPhi} and \eqref{W'1:Phi} we have
%\begin{equation}\label{W'1:tPhi}
%\begin{array}{l}
%  \tP(x,y,z)=-\dfrac{1}{2n}\bigl\{
%  g(x,y)f(Jz)-g(x,Jy)f(z)\bigr\}.
%\end{array}
%\end{equation}
Therefore, we get the following relation for a $\W_1$-manifold, i.e. an almost Norden manifold belonging to the class $\W_1$,
\begin{equation*}\label{W'1:tFF}
  \tPP(x,y,z)=F(Jx,y,z).
\end{equation*}

%Since $\g(x,y)=g(x,Jy)$, $\g(x,Jy)=-g(x,y)$ and $\widetilde{\ta}(z)=\ta(z)$, the relation \eqref{W'1:tF} can be rewritten as
%\begin{equation}\label{W'1:tFtg}
%\begin{array}{l}
%  \tPP(x,y,z)=\dfrac{1}{2n}\bigl\{
%  \g(x,y)\widetilde{\ta}(z)+\g(x,z)\widetilde{\ta}(y)\\[6pt]
%  \phantom{\tPP(x,y,z)=\dfrac{1}{2n}\bigl\{}
%  +\g(x,Jy)\widetilde{\ta}(Jz)+\g(x,Jz)\widetilde{\ta}(Jy)\bigr\},
%\end{array}
%\end{equation}
%which is the definition condition the manifold $(\MM,J,\g)$ to be in $\W_1$. Therefore, the class $\W_1$ is invariant under the twin interchange.
%
%MOZHE PO-NAKRATKO za W'1!
%

The invariant connection $\DDD^{\diamond}$ has the following form on a $\W_1$-manifold, applying the definition from \eqref{class2} in \eqref{hn=nP},
\begin{equation*}\label{W'1:inv-n}
\DDD^{\diamond}_x y=\DDD_x y+\dfrac{1}{4n}\left\{g(x,y)f^{\sharp}+g(x,Jy)Jf^{\sharp}\right\},
\end{equation*}
where $f^{\sharp}$ is the dual vector of the 1-form $f$ regarding $g$, i.e. \[f(z)=g(f^{\sharp},z).\]

The presence of the first equality in \eqref{class2}, the explicit expression of $\Phi$ in terms of $g$ for the case of a $\W_1$-manifold,
gives us a chance to find a more concrete form of $P$ and $A$ defined by \eqref{P'} and  \eqref{A'}, respectively. This expression gives results in the corresponding relations between $R^{\DDD}$ and $R^{\tD}$, $R^{\DDD^{\diamond}}$, $B$, given in \eqref{tRRS}, \eqref{hRRSA}, \eqref{bR=RS}, respectively. A relation between $R^{\DDD}$ and $R^{\tD}$ for a $\W_1$-manifold is given in \cite{MT06th} but using $\ta$.

\begin{prop}\label{prop:W'1_K}
If $(\MM,J,g)$ is an almost Norden manifold belonging to the class $\W_1$, then the tensors $P$ and $A$ have the following form, respectively:
\[%\begin{equation*}\label{W'1:P}
\begin{array}{l}
P(x,y)z=\dfrac{1}{2n}\bigl\{g(y,z)p(x)+g(y,Jz)Jp(x)\\[6pt]
\phantom{P(x,y)z=\dfrac{1}{2n}}
-g(x,z)p(y)-g(x,Jz)Jp(y)\bigr\}, \\[6pt]
%\end{array}
%\end{equation*}
%\begin{equation*}\label{W'1:A}
%\begin{array}{l}
A(x,y)z=\dfrac{1}{4n^2}\bigl\{g(y,z)a(x)+g(y,Jz)a(Jx)\\[6pt]
\phantom{A(x,y)z=\dfrac{1}{4n^2}}
-g(x,z)a(y)-g(x,Jz)a(Jy)\bigr\},
\end{array}
\]%\end{equation*}
where
\[
\begin{array}{l}
p(x)=\DDD_x f^{\sharp}+\dfrac{1}{2n}\{f(x)f^{\sharp}-f(f^{\sharp})x-f(Jf^{\sharp})Jx\}
\\[6pt]
 a(x)=f(x)f^{\sharp}+f(Jx)Jf^{\sharp}.
\end{array}
\]
\end{prop}
\begin{proof}
The formulae follow by direct computations, using \eqref{Wi}, \eqref{class2}, \eqref{P'} and \eqref{A'}.
\end{proof}

%%%%%%%%%%%%%%%%%%%%%%%%%%%%%%%%%%%%%%%%%%%%%%%%%%%%%%%%%%%%%%%%%%%%%%%%%%%

%%%%%%%%%%%%%%%%%%%%%%%%%%%%%%%%%%%%%%%%%%%%%%%%%%%%%%%%%%%%%%%%%%%%%%%%%%%    &3
\vskip 0.2in \addtocounter{subsection}{1}  \setcounter{subsubsection}{0}

\noindent  {\Large\bf \thesubsection. Lie group as a manifold from the main class and the invariant connection and the invariant tensors on it}

\vskip 0.15in
%\section{Lie group as a manifold from the main class and the invariant connection and the invariant tensors on it}
\label{sec_3}

In this subsection we consider an example of a 4-dimensional Lie
group as a $\W_1$-manifold given in \cite{MT06ex}.

Let $\LL$ be a 4-dimensional real connected Lie group, and let
$\mathfrak{l}$ be its Lie algebra with a basis
$\{x_{1},x_{2},x_{3},x_{4}\}$.

We introduce an almost complex structure
$J$ and a Norden metric by
\begin{equation}\label{Jdim4}
\begin{array}{llll}
Jx_{1}=x_{3}, \quad & Jx_{2}=x_{4}, \quad & Jx_{3}=-x_{1},
\quad &
Jx_{4}=-x_{2},
\end{array}
\end{equation}
\begin{equation}\label{gL4}
\begin{array}{c}
  g(x_1,x_1)=g(x_2,x_2)=-g(x_3,x_3)=-g(x_4,x_4)=1, \\[6pt]
  g(x_i,x_j)=0,\; i\neq j \quad (i,j=1,2,3,4).
\end{array}
\end{equation}
Then, the associated Norden metric $\g$ is determined by its non-zero components
\begin{equation}\label{tg}
\begin{array}{c}
  \g(x_1,x_3)=\g(x_2,x_4)=-1.
\end{array}
\end{equation}

Let us consider $(\LL,J,g)$ with the Lie algebra $\mathfrak{l}$
determined by the following nonzero commutators:
\begin{equation}\label{lie-w1-2}
\begin{array}{l}
\left
[x_{1},x_{4}\right]=[x_{2},x_{3}]=\lm_{1}x_{1}+\lm_{2}x_{2}+\lm_{3}x_{3}+\lm_{4}x_{4},\\[6pt]
\left[x_{1},x_{3}\right]=[x_{4},x_{2}]=\lm_{2}x_{1}-\lm_{1}x_{2}+\lm_{4}x_{3}-\lm_{3}x_{4},
\end{array}
\end{equation}
where $\lm_i\in\R$ ($i=1,2,3,4$).
%Obviously, $[J x_i,J x_j]=[x_i,x_j]$ holds, i.e. $J$ is an Abelian structure for $\mathfrak{l}$.

In \cite{MT06ex}, it is proved that $(\LL,J,g)$ is a $\W_1$-manifold.
%Moreover, the components of $R^{\DDD}$ are computed and some curvature properties with respect to $\DDD$ are obtained.
Since the class $\W_1$ is invariant under the twin interchange, according to \thmref{thm:inv.cl}, it follows that $(\LL,J,\g)$ belongs to  $\W_1$, too.

%Next, we state the following
\begin{thm}\label{thm:W'10-W'}
Let $(\LL,J,g)$ and $(\LL,J,\g)$ be the pair of $\W_1$-manifolds, determined by
\eqref{Jdim4}--\eqref{lie-w1-2}. Then both the manifolds:
\begin{enumerate}\renewcommand{\labelenumi}{\emph{(\roman{enumi})}}
    \item belong to the class of the locally conformal
        K\"ahler-Norden manifolds if and only if
        \begin{equation}\label{lm}
        \lm_1^2-\lm_2^2+\lm_3^2-\lm_4^2=\lm_1\lm_2+\lm_3\lm_4=0;
        \end{equation}
    \item   are %the twin Norden metrics are
            locally conformally flat by usual conformal transformations
            and the curvature tensors $R^{\DDD}$ and $R^{\tD}$ have the following form, respectively:
        \begin{equation}\label{Rform3}
            \begin{array}{l}
                R^{\DDD}=-\dfrac{1}{2}g\owedge\rho^{\DDD}+\dfrac{1}{12}\tau^{\DDD} g\owedge g,\\[9pt]
                R^{\tD}=  -\dfrac{1}{2}\g\owedge\rho^{\tD}
                                +\dfrac{1}{12}\tau^{\tD}\g\owedge\g.
            \end{array}
        \end{equation}
%    \item   $(\LL,J,g)$ is scalar flat and isotropic K\"ahlerian
%            if and only if %the structural constants from \eqref{lie-w1-2} satisfy the following condition
%        \begin{equation}\label{llll}
%        \lm_{1}^{2}+\lm_{2}^{2}-\lm_{3}^{2}-\lm_{4}^{2}=0.%,
%        \end{equation}
    \item   are scalar flat and isotropic K\"ahlerian
            if and only if the following conditions are satisfied, respectively:%the structural constants from \eqref{lie-w1-2} satisfy the following condition
        \begin{equation*}\label{llll2}
        \lm_{1}^{2}+\lm_{2}^{2}-\lm_{3}^{2}-\lm_{4}^{2}=0,\qquad
        \lm_{1}\lm_{3}+\lm_{2}\lm_{4}=0.%,
        \end{equation*}
\end{enumerate}
\end{thm}
\begin{proof}
%More concretely,
According to \eqref{twin}, \eqref{Jdim4}, \eqref{gL4}, \eqref{lie-w1-2} and the Koszul formula for $g$, $\DDD$ and $\g$, $\tD$, we get the following nonzero components of $\DDD$ and $\tD\thinspace$:
\begin{subequations}\label{nabla3}
\begin{equation}
\begin{array}{ll}
\DDD_{x_{1}}x_{1} = \DDD_{x_{2}}x_{2} = \tD_{x_{3}}x_{3} = \tD_{x_{4}}x_{4} =\lm_{2}x_{3} +\lm_{1}x_{4},
\\[6pt]
\DDD_{x_{1}}x_{3} = \DDD_{x_{4}}x_{2} =-\tD_{x_{2}}x_{4} = -\tD_{x_{3}}x_{1} = \lm_{2}x_{1} -\lm_{3}x_{4},
\\[6pt]
\DDD_{x_{1}}x_{4} = -\DDD_{x_{3}}x_{2} =\tD_{x_{1}}x_{4} = -\tD_{x_{3}}x_{2} = \lm_{1}x_{1} +\lm_{3}x_{3},
\\[6pt]
\DDD_{x_{2}}x_{3} =- \DDD_{x_{4}}x_{1} =\tD_{x_{2}}x_{3} =- \tD_{x_{4}}x_{1} =\lm_{2}x_{2} +\lm_{4}x_{4},
%\\[6pt]
\end{array}
\end{equation}
\begin{equation}
\begin{array}{ll}
\DDD_{x_{2}}x_{4} = \DDD_{x_{3}}x_{1} =-\tD_{x_{1}}x_{3} = -\tD_{x_{4}}x_{2} = \lm_{1}x_{2} -\lm_{4}x_{3},
\\[6pt]
\DDD_{x_{3}}x_{3} = \DDD_{x_{4}}x_{4} =\tD_{x_{1}}x_{1} = \tD_{x_{2}}x_{2} = -\lm_{4}x_{1} -\lm_{3}x_{2}.
\end{array}
\end{equation}
\end{subequations}

The components of $\DDD J$ and $\tD J$ follow from \eqref{nabla3} and \eqref{Jdim4}. Then, using \eqref{gL4}, \eqref{tg} and \eqref{F'}, we
get the following nonzero components
$(F)_{ijk}=F(x_{i},x_{j},x_{k})$ and \[(\tPP)_{ijk}=\tPP(x_{i},x_{j},x_{k})=\g((\tD_{x_i}J)x_j,x_k)\]
of $F$ and $\tPP$, respectively:
%\begin{subequations}
\begin{equation}\label{lambdi}
\begin{array}{l}
\lm_{1}=(F)_{112}=(F)_{121}=(F)_{134}=(F)_{143}\\[6pt]
\phantom{\lm_{1}}
=\dfrac{1}{2}(F)_{222}=\dfrac{1}{2}(F)_{244}\\[6pt]
\phantom{\lm_{1}}
=(F)_{314}=-(F)_{323}=-(F)_{332}=(F)_{341},\\[6pt]
\lm_{2}=\dfrac{1}{2}(F)_{111}=\dfrac{1}{2}(F)_{133}\\[6pt]
\phantom{\lm_2}
=(F)_{212}=(F)_{221}=(F)_{234}=(F)_{243}\\[6pt]
\phantom{\lm_2}
=-(F)_{414}=(F)_{423}=(F)_{432}
=-(F)_{441},\\[6pt]
%\end{array}
%\end{equation}
%\begin{equation}
%\begin{array}{l}
\lm_{3}=(F)_{114}=-(F)_{123}=-(F)_{132}=(F)_{141}\\[6pt]
\phantom{\lm_{3}}=-(F)_{312}=-(F)_{321}=-(F)_{334}=-(F)_{343}\\[6pt]
\phantom{\lm_{3}}=-\dfrac{1}{2}(F)_{422}=-\dfrac{1}{2}(F)_{444},\\[6pt]
\lm_{4}=-(F)_{214}=(F)_{223}=(F)_{232}=-(F)_{241}\\[6pt]
\phantom{\lm_{4}}
=-\dfrac{1}{2}(F)_{333}=-\dfrac{1}{2}(F)_{311}\\[6pt]
\phantom{\lm_{4}}=-(F)_{412}=-(F)_{421}=-(F)_{434}
=-(F)_{443};
\end{array}
\end{equation}
%\end{subequations}
\begin{subequations}\label{lambdi-tF}
\begin{equation}
\begin{array}{l}
\lm_{1}=(\tPP)_{114}=-(\tPP)_{123}=-(\tPP)_{132}=(\tPP)_{141}\\[6pt]
\phantom{\lm_{1}}
=-(\tPP)_{312}=-(\tPP)_{321}=-(\tPP)_{334}=-(\tPP)_{343}
\\[6pt]
\phantom{\lm_{1}}
=-\dfrac12(\tPP)_{422}=-\dfrac12(\tPP)_{444},
\\[6pt]
\lm_{2}=-(\tPP)_{214}=(\tPP)_{223}=(\tPP)_{232}=-(\tPP)_{241}
\\[6pt]
\phantom{\lm_2}
=-\dfrac12(\tPP)_{311}=-\dfrac12(\tPP)_{333}
\\[6pt]
\phantom{\lm_2}
=-(\tPP)_{412}=-(\tPP)_{421}=-(\tPP)_{434}
=-(\tPP)_{443},
%\\[6pt]
\end{array}
\end{equation}
\begin{equation}
\begin{array}{l}
\lm_{3}=-(\tPP)_{112}=-(\tPP)_{121}
=-(\tPP)_{134}=-(\tPP)_{143}\\[6pt]
\phantom{\lm_{3}}
=-\dfrac12(\tPP)_{222}=-\dfrac12(\tPP)_{244}
\\[6pt]
\phantom{\lm_{3}}
=-(\tPP)_{314}=(\tPP)_{323}=(\tPP)_{332}
=-(\tPP)_{341},\\[6pt]
\lm_{4}=-\dfrac12(\tPP)_{111}=-\dfrac12(\tPP)_{133}\\[6pt]
\phantom{\lm_{4}}
=-(\tPP)_{212}=-(\tPP)_{221}=-(\tPP)_{234}=-(\tPP)_{243}\\[6pt]
\phantom{\lm_{4}}
=(\tPP)_{414}=-(\tPP)_{423}=-(\tPP)_{432}
=(\tPP)_{441}.
\end{array}
\end{equation}
\end{subequations}

Applying \eqref{snorm} for the components in \eqref{lambdi} and \eqref{lambdi-tF}, we obtain the square norms of $\DDD\hspace{-2pt} J$ and $\tD J$:
\begin{equation}\label{nJ-exa}
\begin{split}
\nJJ&=16\left(\lm_1^2+\lm_2^2-\lm_3^2-\lm_4^2\right),\\
\tnJ&=\lm_1\lm_3+\lm_2\lm_4.
\end{split}
\end{equation}

Let us consider the conformal transformations of the metric $g$ defined by
\begin{equation}\label{transf}
    \overline g= e^{2u}\left\{\cos{2v}\ g+\sin{2v}\ \widetilde g\right\},
\end{equation}
where $u$ and $v$ are differentiable functions on
$\MM$~ \cite{GaGrMi87}.
Then, the associated metric $\g$ has the following image
\[
\overline{\tg}=e^{2u}(\cos{2v}\ \tg-\sin{2v}\ g)
\]
The manifold $(\MM,J,\overline{g})$ is
again an almost Norden  manifold. If $v=0$, we obtain the usual conformal transformation.
 Let us remark that the conformal transformation for $u=0$ and $v=\pi/2$ maps the pair $(g,\g)$ into $(\g,-g)$.
%??? DALI TAKAVA ILI OBIKNOVENA SAMO???

According to \cite{GaGrMi85}, a $\W_1$-manifold is locally conformal equivalent to a K\"ahler-Norden manifold if and only if its Lee forms $\ta$ and $\ta^*$ are closed. Moreover, the used conformal transformations are such that the 1-forms $\D u\circ J$ and $\D v\circ J$ are closed.

Taking into account \lemref{lem:tata*}, we have
\[
(\ta)_k=(\widetilde{\ta})_k,\qquad
(\ta^*)_k=(\widetilde{\ta^*})_k
\]
for the corresponding components with respect to $x_k$, $(k=1,2,3,4)$.
Furthermore, the same situation is for  $\D{\ta}=\D\widetilde{\ta}$ and $\D{\ta}^*=\D\widetilde{\ta^*}$.
By \eqref{ta'}, \eqref{ta*taJ} and \eqref{lambdi}, we obtain
$(\ta)_k$ and $(\ta^*)_k$ and thus we get the following
\\
\begin{equation}\label{theta123}
\begin{split}
(\ta)_{2}=(\ta^*)_{4}=(\widetilde{\ta})_{2}=(\widetilde{\ta}^*)_{4}&=4\lm_{1}, \quad \\[6pt]
(\ta)_{1}=(\ta^*)_{3}=(\widetilde{\ta})_{1}=(\widetilde{\ta}^*)_{3}&=4\lm_{2}, \quad \\[6pt]
(\ta)_{4}=-(\ta^*)_{2}=(\widetilde{\ta})_{4}=-(\widetilde{\ta}^*)_{2}&=4\lm_{3},
\quad \\[6pt]
(\ta)_{3}=-(\ta^*)_{1}=(\widetilde{\ta})_{3}=-(\widetilde{\ta}^*)_{1}&=4\lm_{4}.
\end{split}
\end{equation}
Using \eqref{lie-w1-2} and \eqref{theta123}, we compute the components of
$\D\ta$ and $\D\ta^*$ with respect to the basis
$\{x_{1},x_{2},x_{3},x_{4}\}$.
We obtain that $\D\ta^*=\D\widetilde{\ta^*}=0$ and the nonzero components of $\D\ta=\D\widetilde{\ta}$ are
\begin{equation*}\label{dta}
\begin{array}{l}
  (\D\ta)_{13}=(\D\ta)_{42}=(\D\widetilde{\ta})_{13}=(\D\widetilde{\ta})_{42}=4(\lm_1^2-\lm_2^2+\lm_3^2-\lm_4^2),\\[6pt]
  (\D\ta)_{14}=(\D\ta)_{23}=(\D\widetilde{\ta})_{14}=(\D\widetilde{\ta})_{23}=-8(\lm_1\lm_2+\lm_3\lm_4).
\end{array}
\end{equation*}

Therefore $(\LL,J,g)$ and $(\LL,J,\g)$ are
locally conformal K\"ahler-Norden manifolds if and only if conditions
\eqref{lm} are valid. Then, the statement (i) holds.

By virtue of \eqref{gL4}, \eqref{lie-w1-2} and \eqref{nabla3}, we get the basic components $R^{\DDD}_{ijkl}=R^{\DDD}(x_{i},x_{j},x_{k},x_{l})$ and $R^{\tD}_{ijkl}=R^{\tD}(x_{i},\allowbreak{}x_{j},\allowbreak{}x_{k},x_{l})$ of the curva\-tu\-re tensors
for $\DDD$ and $\tD$, respectively. The nonzero ones of them are determined by \eqref{curv} and the following:
\begin{equation}\label{K3}
\begin{array}{l}
\begin{array}{ll}
R^{\DDD}_{1221} = \lm_{1}^{2} + \lm_{2}^{2}, \qquad & %
R^{\DDD}_{1331} = \lm_{4}^{2} - \lm_{2}^{2}, \\[6pt]
R^{\DDD}_{1441} = \lm_{4}^{2} - \lm_{1}^{2}, \qquad & %
R^{\DDD}_{2332} = \lm_{3}^{2} - \lm_{2}^{2}, \\[6pt]
R^{\DDD}_{2442} = \lm_{3}^{2} - \lm_{1}^{2}, \qquad & %
R^{\DDD}_{3443} = -\lm_{3}^{2} - \lm_{4}^{2},\\[6pt]
%\end{array}
%\\[6pt]
%\begin{array}{ll}
R^{\DDD}_{1341}=R^{\DDD}_{2342} = -\lm_{1}\lm_{2}, \qquad & %
R^{\DDD}_{2132}=-R^{\DDD}_{4134} = -\lm_{1}\lm_{3},
\\[6pt]
R^{\DDD}_{1231}=-R^{\DDD}_{4234} = \lm_{1}\lm_{4}, \qquad & %
R^{\DDD}_{2142}=-R^{\DDD}_{3143} = \lm_{2}\lm_{3},
\\[6pt]
R^{\DDD}_{1241}=-R^{\DDD}_{3243} = -\lm_{2}\lm_{4}, \qquad & %
R^{\DDD}_{3123}=R^{\DDD}_{4124} = \lm_{3}\lm_{4};
\end{array}
\end{array}
\end{equation}
%
%\begin{subequations}
\begin{equation}\label{tR3}
\begin{array}{l}
\begin{array}{ll}
R^{\tD}_{1241} = -\lm_{3}^{2}, \qquad & %
R^{\tD}_{2132} = -\lm_{4}^{2}, \\[6pt]
R^{\tD}_{3243} = -\lm_{1}^{2}, \qquad & %
R^{\tD}_{4134} = -\lm_{2}^{2}, \\[6pt]
R^{\tD}_{1331} = 2\lm_{2}\lm_{4},\qquad & %
R^{\tD}_{2442} = 2\lm_{1}\lm_{3},\\[6pt]
%\phantom{R^{\tD}_{1341}=R^{\tD}_{4124} = \lm_{2}\lm_{3},\qquad}&
%\phantom{R^{\tD}_{1231}=R^{\tD}_{2142} = -\lm_{3}\lm_{4},}
%\end{array}
%\end{array}
%\end{equation}
%\begin{equation}
%\begin{array}{l}
%\begin{array}{ll}
R^{\tD}_{1341}=R^{\tD}_{4124} = \lm_{2}\lm_{3},\qquad&
R^{\tD}_{1231}=R^{\tD}_{2142} = -\lm_{3}\lm_{4},\\[6pt]
R^{\tD}_{2342}=R^{\tD}_{3123} = \lm_{1}\lm_{4},\qquad&
R^{\tD}_{3143}=R^{\tD}_{4234} = -\lm_{1}\lm_{2},\\[6pt]
\end{array}
\\
\begin{array}{l}
R^{\tD}_{1234} = R^{\tD}_{2341} = \lm_{1}\lm_{3} + \lm_{2}\lm_{4},%\phantom{\dfrac12}
\end{array}
\end{array}
\end{equation}
%\end{subequations}
Therefore, the components of the Ricci tensors
%$\rho^{\DDD}$, $\widetilde{\rho\hspace{2pt}}^{\DDD}$
and the values of the scalar curvatures
%$\tau^{\DDD}$, $\widetilde{\tau}'$
for $\DDD$ and $\tD$ are:
%\begin{subequations}
\begin{equation*}%\label{Ricci3}
\begin{array}{ll}
\rho^{\DDD}_{11}=2\big( \lm_{1}^{2} + \lm_{2}^{2} - \lm_{4}^{2} \big), \qquad &
\rho^{\DDD}_{22}=2\big( \lm_{1}^{2} + \lm_{2}^{2} - \lm_{3}^{2} \big), \\[6pt]
\rho^{\DDD}_{33}=2\big( \lm_{4}^{2} + \lm_{3}^{2} - \lm_{2}^{2} \big), \qquad &
\rho^{\DDD}_{44}=2\big( \lm_{4}^{2} + \lm_{3}^{2} - \lm_{1}^{2} \big),
\\[6pt]
\rho^{\tD}_{11}=2\lm_{3}^{2}, \qquad &
\rho^{\tD}_{22}=2\lm_{4}^{2}, \qquad \\[6pt]
\rho^{\tD}_{33}=2\lm_{1}^{2},  \qquad &
\rho^{\tD}_{44}=2\lm_{2}^{2},
%\\[6pt]
\end{array}
\end{equation*}
\begin{equation*}%\label{Ricci3}
\begin{array}{ll}
\rho^{\tD}_{13}=2\big(\lm_{1}\lm_{3}+2\lm_{2}\lm_{4}\big),\qquad &
\rho^{\tD}_{24}=2\big(2\lm_{1}\lm_{3}+\lm_{2}\lm_{4}\big),
\\[6pt]
\rho^{\DDD}_{12}=\rho^{\tD}_{12}=-2\lm_{3}\lm_{4}, \quad & %
\rho^{\DDD}_{23}=\rho^{\tD}_{23}=2\lm_{1}\lm_{4}, \\[6pt]
\rho^{\DDD}_{13}=-2\lm_{1}\lm_{3}, \quad &
\rho^{\DDD}_{34}=\rho^{\tD}_{34}=-2\lm_{1}\lm_{2}, \\[6pt] %
\rho^{\DDD}_{14}=\rho^{\tD}_{14}=2\lm_{2}\lm_{3}, \quad &
\rho^{\DDD}_{24}=-2\lm_{2}\lm_{4};
\end{array}
\end{equation*}
%\end{subequations}
\begin{equation}\label{tau3}
\begin{array}{ll}
\tau^{\DDD}=6\big(\lm_{1}^{2}+\lm_{2}^{2}-\lm_{3}^{2}-\lm_{4}^{2}\big),\qquad &
\tau^{\tD}=-12\big(\lm_{1}\lm_{3}+\lm_{2}\lm_{4}\big).
\end{array}
\end{equation}

Applying \eqref{W'} for the corresponding quantities of $\DDD$ and $\tD$, we compute that the respective Weyl tensors $C$ and $\widetilde{C}$ vanish. Then, we obtain the identities in \eqref{Rform3}.
Furthermore, when the Weyl tensor vanishes then the corresponding manifold is conformal equivalent to a flat manifold by a usual conformal transformation. This completes the proof of (ii).

The truthfulness of (iii) follows immediately from \eqref{tau3} and the values of the square norms in \eqref{nJ-exa}. %
\end{proof}

Let us remark that the results in the latter theorem with respect to $\DDD$ %, $F$, $\ta$, $R^{\DDD}$, $\rho^{\DDD}$, $\tau^{\DDD}$ and $W'$
are given in \cite{MT06ex} besides (i), where it is shown a particular case of conditions \eqref{lm}.
%%%%%%%%%%%%%%%%%%%%%%%%%%%%%%%%

\vskip 0.2in \addtocounter{subsubsection}{1}

\noindent  {\Large\bf{\emph{\thesubsubsection. The invariant connection and invariant tensors under the twin interchange}}}%\\[6pt]\vskip2pt}

\vskip 0.15in
%\subsection{The invariant connection and invariant tensors under the twin interchange }

We compute the basic components $B_{ijk}=B(x_i,x_j)x_k$ of the invariant tensor $B$, using that this tensor is the average tensor of $R^{\DDD}$ and $R^{\tD}$ and equalities \eqref{K3}, \eqref{tR3}.
Thus, we get the components $B_{ijkl}=g\left(B(x_i,x_j)x_k,x_l\right)$.
The nonzero of them are the following and the rest are determined by the properties \eqref{curv} for $B$:
\begin{subequations}\label{bR3}
\begin{equation}
%\begin{array}{l}
\begin{array}{ll}
\dfrac12\lm_{1}^{2} = B_{3421} = -B_{2341}, \; &
\dfrac12\lm_{2}^{2} = -B_{3412} = -B_{1432},\\[6pt]
\dfrac12\lm_{3}^{2} = -B_{1243} = -B_{1423},\; &
\dfrac12\lm_{4}^{2} = B_{1234} = -B_{2314}, %\\[6pt]
\end{array}
\end{equation}
%\\[6pt]
\begin{equation}
\begin{array}{l}
\begin{array}{ll}
\dfrac12\bigl(\lm_{1}^{2}+\lm_{2}^{2}+\lm_{3}^{2}\bigr) = -B_{1212}, \ &
\dfrac12\bigl(\lm_{1}^{2}+\lm_{2}^{2}+\lm_{4}^{2}\bigr) = B_{1221},\\[6pt]
\dfrac12\bigl(\lm_{1}^{2}+\lm_{3}^{2}-\lm_{4}^{2}\bigr) = B_{1414}, \ &
\dfrac12\bigl(\lm_{1}^{2}-\lm_{2}^{2}-\lm_{4}^{2}\bigr) = -B_{1441},\\[6pt]
\dfrac12\bigl(\lm_{2}^{2}-\lm_{3}^{2}+\lm_{4}^{2}\bigr) = B_{2323}, \ &
\dfrac12\bigl(\lm_{1}^{2}-\lm_{2}^{2}+\lm_{3}^{2}\bigr) = B_{2332},\\[6pt]
\dfrac12\bigl(\lm_{1}^{2}+\lm_{3}^{2}+\lm_{4}^{2}\bigr) = B_{3434}, \ &
\dfrac12\bigl(\lm_{2}^{2}+\lm_{3}^{2}+\lm_{4}^{2}\bigr) = -B_{3443},%\\[6pt]
%\\[6pt]
\end{array}
\\[6pt]
\begin{array}{l}
\dfrac12\bigl(\lm_{2}^{2}-\lm_{4}^{2}\bigr) = B_{1313} = -B_{1331}, \\[6pt]
\dfrac12\bigl(\lm_{1}^{2}-\lm_{3}^{2}\bigr) = B_{2424} = -B_{2442}, \\[6pt]
\dfrac12\bigl(\lm_{1}\lm_{2}+\lm_{3}\lm_{4}\bigr) = B_{1234} = B_{1332} = B_{2423} = B_{2441}, \\[6pt]
\dfrac12\bigl(\lm_{2}\lm_{3}-\lm_{1}\lm_{4}\bigr) = B_{1312} = -B_{1334} = -B_{2421} = B_{2443},\\[6pt]
\dfrac12\lm_{1}\lm_{2} = -\dfrac12B_{1341} = B_{1413} =-B_{1431} =B_{2324} =-B_{2342}\\[6pt]
\phantom{\dfrac12\lm_{1}\lm_{2} }
=-\dfrac12B_{2432} =B_{3411} =-B_{3422} =B_{3433} =-B_{3444},
\\[6pt]
\dfrac12\lm_{3}\lm_{4} = -B_{1211} = B_{1222} =-B_{1233} =B_{1244} =-\dfrac12B_{1323}\\[6pt]
\phantom{\dfrac12\lm_{3}\lm_{4} }
=-B_{1424} =B_{1442} =-B_{2313} =B_{2331} =-\dfrac12B_{2414},\\[6pt]
\dfrac12\lm_{1}\lm_{3} = -B_{1223} = B_{1241} =B_{1421} =B_{1443} =B_{2321}
\\[6pt]
\phantom{\dfrac12\lm_{1}\lm_{3} }
=B_{2343} =\dfrac12B_{2422} =\dfrac12B_{2444} =-B_{3423} = -B_{3441},\\[6pt]
\dfrac12\lm_{2}\lm_{4} = B_{1214} = -B_{1232} =\dfrac12B_{1311} =\dfrac12B_{1333} =B_{1412}\\[6pt]
\phantom{\dfrac12\lm_{2}\lm_{4} }
=B_{1434} =B_{2312} =B_{2334} =B_{3414} = -B_{3432},
\\[6pt]
\dfrac12\lm_{1}\lm_{4} = -B_{1213} = B_{1231} = \dfrac12B_{1321} =B_{2311} =B_{2322} \\[6pt]
\phantom{\dfrac12\lm_{1}\lm_{4} }
=B_{2333}=B_{2344} =\dfrac12B_{2434} =B_{3424} =-B_{3442},\\[6pt]
%\end{array}
%\end{equation}
%\begin{equation}
%\begin{array}{l}
%\end{array}\end{array}
%\end{equation}
%\begin{equation}
%\begin{array}{l}\begin{array}{l}
\dfrac12\lm_{2}\lm_{3} = B_{1224} = -B_{1242} = \dfrac12B_{1343} =B_{1411} =B_{1422} \\[6pt]
\phantom{\dfrac12\lm_{2}\lm_{3} }
=B_{1433}=B_{1444} =\dfrac12B_{2412} =-B_{3413} =B_{3431}.
\end{array}
\end{array}
\end{equation}
\end{subequations}
The rest components are determined by the property $B_{ijk}=-B_{jik}$. Let us remark that $B$ is not a curvature-like tensor.

Obviously, $B=0$ if and only if the corresponding Lie algebra is Abelian and $(\LL,J,g)$ is a K\"ahler-Norden manifold.

Using \eqref{Phi}, \eqref{Jdim4}, \eqref{gL4}, \eqref{nabla3},    we get the components $(\Phi)_{ijk}=\Phi(x_i,\allowbreak{}x_j,x_k)$ of $\Phi$ as  well as the components $(f)_{k}=f(x_k)$ and $(f^*)_{k}=f^*(x_k)$ of its associated 1-forms. The nonzero of them are the following and the rest are obtained by the property $(\Phi)_{ijk}=(\Phi)_{jik}$:
%\begin{subequations}\label{Phi-ex}
\begin{equation}\label{Phi-ex}
\begin{split}
-\lm_1&=-(\Phi)_{114} = -(\Phi)_{224} =(\Phi)_{334}\\[6pt]
&=(\Phi)_{444}=(\Phi)_{132} =(\Phi)_{242} =\dfrac14 (f)_4=-\dfrac14 (f^*)_2,
\\[6pt]
-\lm_2&=-(\Phi)_{113} = -(\Phi)_{223} =(\Phi)_{333}\\[6pt]
%\end{split}
%\end{equation}
%\begin{equation}
%\begin{split}
&=(\Phi)_{443}=(\Phi)_{131} =(\Phi)_{241} =-\dfrac14 (f)_3=\dfrac14 (f^*)_1,
\\[6pt]
-\lm_3&=(\Phi)_{112} = (\Phi)_{222} =-(\Phi)_{332}\\[6pt]
&=-(\Phi)_{442}=(\Phi)_{134} =(\Phi)_{244} =\dfrac14 (f)_2=\dfrac14 (f^*)_4,
\\[6pt]
-\lm_4&=(\Phi)_{111} = (\Phi)_{221} =-(\Phi)_{331}\\[6pt]
&=-(\Phi)_{441}=(\Phi)_{133} =(\Phi)_{243} =\dfrac14 (f)_1=\dfrac14 (f^*)_3.
\end{split}
\end{equation}
%\end{subequations}

The Nijenhuis tensor vanishes on $(\LL,J,g)$ and $(\LL,J,\g)$ as on any $\W_1$-manifold.
According to \cite{GaGrMi85}, $[J,J]=0$ is equivalent to
\[
\Phi(x_i,x_j)=-\Phi(Jx_i,Jx_j).
\]
Then, by means of \eqref{wNPhi} we obtain for the components of the associated Nijenhuis tensor  $\{J,J\}_{ijk}=4(\Phi)_{ijk}$, where the components of $\Phi$ are given in \eqref{Phi-ex}.

Bearing in mind \eqref{hn=nP}, \eqref{nabla3} and \eqref{Phi-ex}, we get the components of the invariant connection $\DDD^{\diamond}$ as follows
%\begin{subequations}
\begin{equation}\label{hn-ex}
\begin{array}{l}
\DDD^{\diamond}_{x_{1}}x_{1} = \DDD^{\diamond}_{x_{2}}x_{2} =\DDD^{\diamond}_{x_{3}}x_{3} = \DDD^{\diamond}_{x_{4}}x_{4} \\[6pt]
\phantom{\DDD^{\diamond}_{x_{1}}x_{1} }=
 -\dfrac12(\lm_{4}x_{1}+\lm_{3}x_{2}-\lm_{2}x_{3}-\lm_{1}x_{4}),\\[6pt]
\DDD^{\diamond}_{x_{1}}x_{3}=-\DDD^{\diamond}_{x_{2}}x_{4}=-\DDD^{\diamond}_{x_{3}}x_{1}= \DDD^{\diamond}_{x_{4}}x_{2} \\[6pt]
\phantom{\DDD^{\diamond}_{x_{1}}x_{3} }=
 \dfrac12(\lm_{2}x_{1}-\lm_{1}x_{2}+\lm_{4}x_{3}-\lm_{3}x_{4}),
\\[6pt]
%\end{array}
%\end{equation}
%\begin{equation}
%\begin{array}{l}
\DDD^{\diamond}_{x_{1}}x_{4} = -\DDD^{\diamond}_{x_{3}}x_{2} =
\lm_{1}x_{1} +\lm_{3}x_{3},
\\[6pt]
\DDD^{\diamond}_{x_{2}}x_{3} =-\DDD^{\diamond}_{x_{4}}x_{1} =
\lm_{2}x_{2} +\lm_{4}x_{4}.
\end{array}
\end{equation}
%\end{subequations}

After that we compute the basic components $R^{\DDD^{\diamond}}_{ijk}=R^{\DDD^{\diamond}}\hspace{-3pt}(x_i,x_j)x_k$ of the invariant tensor $R^{\DDD^{\diamond}}$ under the twin in\-ter\-change, using \eqref{wRbRPhi}, \eqref{bR3} and \eqref{Phi-ex}.
In other way, $R^{\DDD^{\diamond}}_{ijk}$ can be computed directly from \eqref{hn-ex} as the curvature tensor of $\DDD^{\diamond}$.

Then, we obtain for the basic components $K_{ijkl}$, determined by the equality $K_{ijkl}=g(R^{\DDD^{\diamond}}\hspace{-3pt}(x_i,x_j)x_k,x_l)$, the following quantities:
%The nonzero components are the following and the rest are determined by the properties \eqref{curv} for $B$:
%\begin{subequations}\label{wR3}
\begin{equation*}
\begin{array}{l}
\begin{array}{l}
\lm_{1}^{2} = K_{2424},\quad
\lm_{2}^{2} = K_{1313},\quad
\lm_{3}^{2} = K_{2442}, \quad
\lm_{4}^{2} = K_{1331},
\\[6pt]
\dfrac12\lm_{1}\lm_{3} = -K_{1223} = K_{1241} =K_{1421} =K_{1443} =K_{2321}\\[6pt]
\phantom{\dfrac12\lm_{1}\lm_{3} }
=K_{2343} =\dfrac12K_{2422} =\dfrac12K_{2444} =-K_{3423} =K_{3441}, \\[9pt]
\dfrac12\lm_{2}\lm_{4} = K_{1214} = -K_{1232} =\dfrac12K_{1311} =\dfrac12K_{1333} =K_{1412}\\[6pt]
\phantom{\dfrac12\lm_{2}\lm_{4} }
=K_{1434} =K_{2312} =K_{2334} =K_{3414} =-K_{3432},\\[6pt]
\lm_{1}\lm_{2}+\lm_{3}\lm_{4} = K_{1314} = K_{1332} =K_{2423} =K_{2441},\\[6pt]
\dfrac14(\lm_{1}\lm_{2}-\lm_{3}\lm_{4}) = K_{1211} = -K_{1222} =K_{1424} =-K_{1431}\\[6pt]
\phantom{\dfrac14(\lm_{1}\lm_{2}-\lm_{3}\lm_{4})}
=K_{2313} =-K_{2342} =K_{3433} =-K_{3444},\\[6pt]
\dfrac14(\lm_{1}\lm_{4}+\lm_{2}\lm_{3}) = -K_{1213} = K_{1224} =K_{1422} =K_{1433}\\[6pt]
\phantom{\dfrac14(\lm_{1}\lm_{4}+\lm_{2}\lm_{3})}
=K_{2311} =K_{2344} =K_{3431} =-K_{3442},\\[6pt]
\dfrac14(\lm_{1}\lm_{2}+3\lm_{3}\lm_{4}) = -K_{1233} = K_{1244} =K_{1442} =K_{2331},\\[9pt]
\dfrac14(3\lm_{1}\lm_{2}+\lm_{3}\lm_{4}) = K_{1413} = K_{2324} =K_{3411} =-K_{3422},\\[9pt]
\dfrac14(\lm_{1}\lm_{4}-3\lm_{2}\lm_{3}) = K_{1242} = -K_{1411} = -K_{1444} =K_{3413},\\[9pt]
\end{array}
\\[6pt]
\begin{array}{l}
\dfrac14(3\lm_{1}\lm_{4}-\lm_{2}\lm_{3}) = K_{1231} = K_{2322} =K_{2333} =K_{3424},\\[9pt]
\end{array}
\\[9pt]
\begin{array}{ll}
\dfrac14(\lm_{1}^2-2\lm_{2}^2+\lm_{3}^2) = K_{1432} = K_{3412}, &
\lm_{1}\lm_{2} = -K_{1341} = -K_{2432},\\[9pt]
\dfrac14(\lm_{2}^2-2\lm_{3}^2+\lm_{4}^2) = K_{1243} = K_{1423}, &
\lm_{3}\lm_{4} = -K_{1323} = -K_{2414}, \\[9pt]
\dfrac14(2\lm_{1}^2-\lm_{2}^2-\lm_{4}^2) = -K_{2341} = K_{3421}, &
\lm_{1}\lm_{4} = K_{1321} = K_{2434},\\[9pt]
\dfrac14(\lm_{1}^2+\lm_{3}^2-2\lm_{4}^2) = -K_{1234} = K_{2314}, &
\lm_{2}\lm_{3} = K_{1343} = K_{2412},\\[9pt]
\dfrac14(\lm_{1}^2+2\lm_{2}^2+\lm_{3}^2) = -K_{1212},\quad &
\dfrac14(2\lm_{1}^2+\lm_{2}^2+\lm_{4}^2) = K_{1221},%\\[9pt]
\end{array}
\end{array}
\end{equation*}
%\\[9pt]
\begin{equation*}
\begin{array}{l}
\begin{array}{ll}
\dfrac14(\lm_{1}^2+\lm_{3}^2+2\lm_{4}^2) = K_{3434},\quad &
\dfrac14(\lm_{2}^2+2\lm_{3}^2+\lm_{4}^2) = -K_{3443},\\[9pt]
\dfrac14(3\lm_{1}^2+3\lm_{3}^2-2\lm_{4}^2) = K_{1414},\quad &
\dfrac14(2\lm_{1}^2-3\lm_{2}^2-3\lm_{4}^2) = -K_{1441},\\[9pt]
\dfrac14(3\lm_{2}^2-2\lm_{3}^2+3\lm_{4}^2) = K_{2323},\quad &
\dfrac14(3\lm_{1}^2-2\lm_{2}^2+3\lm_{3}^2) = K_{2332}.
\end{array}
\end{array}
\end{equation*}
%\end{subequations}
The rest components are determined by property $K_{ijk}=-K_{jik}$. Let us remark that $K$ is not a curvature-like tensor.

Obviously, $K=0$ if and only if the corresponding Lie algebra is Abelian and $(\LL,J,g)$ is a K\"ahler-Norden manifold.

%%%%%%%%%%%%%%%%%%%%%%%%%%%%%%%%%%%%%%%%%%%%%%%%%%%%%%%%%%%%%%%%%%%%%%%%%%%

\vspace{20pt}

\begin{center}
$\divideontimes\divideontimes\divideontimes$
\end{center} 

\newpage

\addtocounter{section}{1}\setcounter{subsection}{0}\setcounter{subsubsection}{0}

\setcounter{thm}{0}\setcounter{equation}{0}

\label{par:conn}

 \Large{

\
\\[6pt]
\bigskip

\
\\[6pt]
\bigskip

\lhead{\emph{Chapter I $|$ \S\thesection. Canonical-type connections
on almost complex manifolds with Norden metric
}}
%\thispagestyle{empty}

%\noindent  {\Huge\bf \S\thesection. %\label{sect:48f}
%Canonical-type connections\\[12pt]
%\phantom{\S\thesection. }on almost Norden manifolds
%}%\\[6pt]\vskip2pt}

\noindent
\begin{tabular}{r"l}
  %\hline
\hspace{-6pt}{\Huge\bf \S\thesection.}  & {\Huge\bf Canonical-type connections} \\[12pt]
                             & {\Huge\bf on almost complex manifolds} \\[12pt]
                             & {\Huge\bf with Norden metric}
  %\hline
\end{tabular}

\vskip 1cm

\begin{quote}
\begin{large}
In the present section we give a survey with additions of
results on differential geometry of canonical-type connections
(i.e. metric connections with torsion satisfying a certain
algebraic identity) on the considered manifolds.

The main results of this section are published in \cite{Man48}.
\end{large}
\end{quote}

%
%\vskip 0.2in \addtocounter{subsection}{1}
%
%\noindent  {\Large\bf \thesubsection. Introduction}

\vskip 0.15in

%%%%%%%%%%%%%%%%%%%%%%%%%%%%%%%%%%%%%%%%%%%%%%%%%%%%%%%%%%%%%%%%%%%%%%%%%%%0
%\section{Introduction}\label{sec:intro}

The differential geometry of affine connections with special requirements for their torsion on almost Hermitian manifolds $(\MM^{2n},J,g)$ is well
developed. As it is known, P.~Gauduchon  gives in \cite{Gau97} a
unified presentation of a so-called \emph{canonical} class of
(almost) Hermitian connections, considered by
P.~Libermann in \cite{Lib54}.

Let us recall, an affine connection $\DDD^*$
is called \emph{Hermitian} if it preserves the Hermitian metric
$g$ and the almost complex structure $J$, i.e. the following identities are valid $\DDD^* g=\DDD^* J=0$.

The
\emph{potential} of $\DDD^*$ (with respect to the Levi-Civita
connection $\DDD$), denoted by $Q$, is defined by the difference
$\DDD^*-\DDD$. The connection $\DDD^*$ preserves the metric and therefore is
completely determined by its \emph{torsion} $T$.
According to  \cite{Car25,%, p.51,
TV83,Sal89}, the two spaces of all torsions and of all potentials
are isomorphic as $\mathcal{O}(n)$ representations and an equivariant
bijection is the following
\begin{gather}
T (x,y,z) = Q(x,y,z) - Q(y,x,z) ,\label{TQ}\\[6pt]
2Q(x,y,z) = T (x,y,z) - T (y,z,x) + T (z,x,y).\label{QT}
\end{gather}

Following E. Cartan \cite{Car25}, there are studied the algebraic types of the torsion
tensor for a metric connection, i.e. an affine connection preserving the metric.

On an almost Hermitian manifold, a Hermitian connection
is called \emph{canonical} %denoted $\n^c$,
if its torsion $T$ satisfies the following conditions: \cite{Gau97}

1) the component of $T$ satisfying the Bianchi identity and having
the property $T (J \cdot, J \cdot) = T (\cdot, \cdot)$ vanishes;

2) for some real number $t$, it is valid
$(\s T)^+=(1-2t)(\D\Omega)^+(J\cdot,J\cdot,J\cdot)$, where
%$\s$ denotes the cyclic sum by three arguments and
$(\D\Omega)^+$ is the part of type
$(2,1)+(1,2)$ of the diffe\-ren\-tial  $\D\Omega$ for the K\"ahler
form $\Omega=g(J\cdot,\cdot)$.

%$T (J \cdot, J \cdot) = - T (\cdot, \cdot)$.
This connection is known also as the \emph{Chern connection}
\cite{Chern,Yano,YaKon}. %An example of the natural Hermitian
%connection is the first canonical connection of Lichnerowicz
%$\n^L$ \cite{}.

According to  \cite{Gau97}, there exists an one-parameter family
$\{\n^t\}_{t\in\R}$ of canonical Hermitian connections $\n^t =
t\n^1 + (1-t)\n^0$, where $\n^0$ and $\n^1$ are the
\emph{Lichnerowicz first and second canonical
connections} \cite{Li1,Li2,Lih62}, respectively.

%On an almost Hermitian manifold there exists a unique natural
%connection $\n^C$ with a torsion $T$ which has the property
%$T(J\cdot,J\cdot)=-T(\cdot,\cdot)$ with respect to the almost
%complex structure $J$.

%In
%\cite{Gaud} it is noted that among the natural connections the
%first canonical connection is the one whose torsion has minimal
%possible norm.

The
connection $\n^t$ obtained for $t=-1$ is called the \emph{Bismut
connection} or the \emph{KT-connection}, which is characterized with a
totally skew-symmetric torsion \cite{Bis}. The latter connection
has applications in heterotic  string theory and in 2-dimensional
supersymmetric $\sigma$-models as well as in type II string theory
when the torsion 3-form is closed \cite{GHR,Stro,IvPapa,IvIv}.

In \cite{Fri-Iv2} and \cite{Fri-Iv}, all
almost contact metric, almost Hermitian and $G_2$-structures
admitting a connection with totally skew-symmetric torsion tensor
are described.

Similar problems are studied on almost hypercomplex manifold \cite{Gaud,HoPa}
and Riemannian almost product manifolds in
\cite{StaGri,Dobr11-1,MekDobr,Dobr11-2}.

An object of our interest in this section are almost Norden manifolds.
The goal of the present section is to survey the research on
canonical-type connections in the case of Norden-type metrics
as well as some additions and generalizations are made.
%In the present section %~\thesection.1
%we consider the even-dimensional case.
% and
%in Section~\thesection.2 --- the odd-dimensional one.

%%%%%%%%%%%%%%%%%%%%%%%%%%%%%%%%%%%%%%%%%%%%%%%%%%%%%%%%%%%%%%%%%%%%%%%%%%%%%%%%%1
\vskip 0.2in \addtocounter{subsection}{1} \setcounter{subsubsection}{0}

\noindent  {\Large\bf \thesubsection. Natural connections on an almost Norden manifold}

\vskip 0.15in

%\subsection{Natural connections on an almost Norden manifold}
%on an almost complex manifold with Norden metric}

Let $\DDD^*$ be an affine connection with a torsion $T^*$ and a
potential $Q^*$ with respect to the Levi-Civita connection $\DDD$,
i.e.
\[
T^*(x,y)=\DDD^*_x y-\DDD^*_y x-[x,y],\quad Q^*(x,y)=\DDD^*_x y-\DDD_x y.
\]
The corresponding (0,3)-tensors are defined by
\[
T^*(x,y,z)=g(T^*(x,y),z),\quad Q^*(x,y,z)=g(Q^*(x,y),z).
\]
These tensors have the same mutual relations as in \eqref{TQ} and
\eqref{QT}.

In  \cite{GaMi87}, it is considered the space
$\T$ of all torsion (0,3)-tensors $T$ (i.e. satisfying
$T(x,y,z)=-T(y,x,z)$) on an almost Norden manifold
$(\MM,J,g)$. There, it is given a partial decomposition of $\T$  in the following form
\[
\T=\T_1\oplus\T_2
\oplus\T_3\oplus\T_4.
\]
The components $\T_i$
$(i=1,2,3,4)$ are invariant orthogonal subspaces with respect to
the structure group $\GG$ given in \eqref{GG} and they are determined as follows
\begin{equation*}%\label{Ti}
  \begin{array}{l}
    \T_1:\quad T(x,y,z)=-T(Jx,Jy,z)=-T(Jx,y,Jz);\\[6pt]
    \T_2:\quad T(x,y,z)=-T(Jx,Jy,z)=T(Jx,y,Jz);\\[6pt]
    \T_3:\quad T(x,y,z)=T(Jx,Jy,z),\quad \sx T(x,y,z)=0,\\[6pt]
    \T_4:\quad T(x,y,z)=T(Jx,Jy,z),\quad \sx T(Jx,y,z)=0.
  \end{array}
\end{equation*}
Moreover, in  \cite{GaMi87} there are explicitly given the
components $T_i$ of $T\in\T$ in $\T_i$
$(i=1,2,3,4)$ as follows
\begin{equation}\label{Ti}
\begin{split}
  T_1(x,y,z)&=\dfrac14\bigl\{T(x,y,z)-T(Jx,Jy,z)\\
            &\phantom{\dfrac14\bigl\{\quad}-T(Jx,y,Jz)-T(x,Jy,Jz)\bigr\},\\[6pt]
  T_2(x,y,z)&=\dfrac14\bigl\{T(x,y,z)-T(Jx,Jy,z)\\
            &\phantom{\dfrac14\bigl\{\quad}+T(Jx,y,Jz)+T(x,Jy,Jz)\bigr\},\\[6pt]
  T_3(x,y,z)&=\dfrac18\bigl\{2T(x,y,z)-T(y,z,x)-T(z,x,y)\\
            &\phantom{\dfrac14\bigl\{\quad}+2T(Jx,Jy,z)-T(Jy,z,Jx)\\
            &\phantom{\dfrac14\bigl\{\quad}-T(z,Jx,Jy)-T(Jy,Jz,x)\\
            &\phantom{\dfrac14\bigl\{\quad}-T(Jz,Jx,y)+T(y,Jz,Jx)\\
            &\phantom{\dfrac14\bigl\{\quad}+T(Jz,x,Jy)\bigr\},\\[6pt]
  T_4(x,y,z)&=\dfrac18\bigl\{2T(x,y,z)+T(y,z,x)+T(z,x,y)\\
            &\phantom{\dfrac14\bigl\{\quad}+2T(Jx,Jy,z)+T(Jy,z,Jx)\\
            &\phantom{\dfrac14\bigl\{\quad}+T(z,Jx,Jy)+T(Jy,Jz,x)\\
            &\phantom{\dfrac14\bigl\{\quad}+T(Jz,Jx,y)-T(y,Jz,Jx)\\
            &\phantom{\dfrac14\bigl\{\quad}-T(Jz,x,Jy)\bigr\}.
\end{split}
\end{equation}

An affine connection $\DDD^*$ on an almost Norden manifold $(\MM,J,g)$ is called a \emph{natural connection} if the structure tensors $J$ and $g$ are parallel with respect to this connection, i.e.
$\DDD^* J=\DDD^* g=0$. These conditions are equivalent to $\DDD^* g=\DDD^*
\widetilde{g}=0$. The connection $\DDD^*$ is natural if and only if
the following conditions for its potential $Q^*$ are valid:
\begin{equation}\label{F_J=Q*}
\begin{split}
    &F(x,y,z) = Q^*(x,y,Jz)-Q^*(x,Jy,z),\quad\\[6pt]
%\label{3.6}
    &Q^*(x,y,z) = -Q^*(x,z,y).
\end{split}
\end{equation}

In terms of the components $T_i$, an affine connection with
torsion $T$ on $(\MM,J,g)$ is
natural if and only if
\[
\begin{split}
    &T_2(x,y,z) = \dfrac{1}{4}[J,J](x,y,z),\quad\\[6pt]
    &T_3(x,y,z) = \dfrac{1}{8}\bigl(\{J,J\}(z,y,x)-\{J,J\}(z,x,y)\bigr).
\end{split}
\]
The former condition
is given in  \cite{GaMi87} whereas the latter one
follows immediately by \eqref{TQ}, \eqref{QT},
\eqref{NhatF'}, \eqref{Ti} and \eqref{F_J=Q*}.

%%%%%%%%%%%%%%%%%%%%%%%%%%%%%%%%%%%%%%%%%%%%%%%%%%%%%%%%%%%%%%%%%%%%%%%%%%%%%%%%%1
\vskip 0.2in \addtocounter{subsection}{1} \setcounter{subsubsection}{0}

\noindent  {\Large\bf \thesubsection. The B-connection and the canonical connection}

\vskip 0.15in
%\subsection{The B-connection and the canonical connection}
%on a quasi-K\"ahler manifold with Norden metric}

In \cite{GaGrMi87}, it is introduced the \emph{B-connection} $\DDD'$ only for the manifolds from the class $\W_1$ by relation
\begin{equation}\label{Tb}
\DDD'_xy=\DDD_xy-\dfrac{1}{2}J\left(\DDD_xJ\right)y.
\end{equation}
Obviously, the B-connection is a natural connection on $(\MM,J,g)$
and it exists in any class of the considered manifolds.
The B-connection coincides with the Levi-Civita connection only on a $\W_0$-manifold (i.e. a K\"ahler-Norden manifold% or an almost Norden manifold belonging to the class $\W_0$
).

By virtue of \eqref{TQ}, \eqref{N-prop}, \eqref{hatN-prop}, \eqref{F=NhatN},
from \eqref{Tb} we express the torsion $T'$ of the B-connection $\DDD'$ in the following way
\begin{equation}\label{TbNhatN}
\begin{split}
T'(x,y,z)=\dfrac18 \bigl(&[J,J](x,y,z)+\sx [J,J](x,y,z)\\[6pt]
&+\{J,J\}(z,y,x)
-\{J,J\}(z,x,y)\bigr).
\end{split}
\end{equation}

A natural connection $\DDD''$ with a torsion $T''$ on an almost Norden
manifold $(\MM,J,g)$ is called a
\emph{canonical connection} if $T''$ satisfies the following
condition \cite{GaMi87}
\begin{equation}\label{4.1}
    T''(x,y,z)+T''(y,z,x)-T''(Jx,y,Jz)-T''(y,Jz,Jx)=0.
\end{equation}

In  \cite{GaMi87}, it is shown that \eqref{4.1} is equivalent to the
condition
$T''_1=T''_4=0$,
i.e. $T''\in\T_2\oplus\T_3$. Moreover,
there it is proved that on every almost Norden manifold
 exists a unique canonical
connection $\DDD''$. We express its torsion in terms of $[J,J]$
and $\{J,J\}$ as follows
\begin{equation}\label{4.3}
    T''(x,y,z)=\dfrac{1}{4}[J,J](x,y,z)
    +\dfrac{1}{8}\bigl(\{J,J\}(z,y,x)-\{J,J\}(z,x,y)\bigr).
\end{equation}

Taking into account \eqref{4.3} and \eqref{TbNhatN}, it is easy to
conclude that $\DDD''$ coincides with $\DDD'$
if and only if the condition $[J,J]=\s [J,J]$ holds. The latter equality is equivalent
to the vanishing of $[J,J]$. In other words, on a complex Norden manifold, i.e.
$(\MM,J,g)\in\W_1\oplus\W_2$, the
canonical con\-nect\-ion and the B-connection coincide.

Now, let $(\MM,J,g)$ be in the class $\W_1$ containing the conformally equivalent manifolds of the K\"ahler-Norden manifolds. The conformal equivalence is made with respect to the general conformal transformations of the metric $g$
defined by \eqref{transf}. These transformations form the general group $C$.
An important its subgroup is the group $C_0$ of the \emph{holomorphic
conformal transformations}, defined by the condition: $u + iv$ is
a holomorphic function, i.e. the equality $\D u = \D v \circ J$ is valid.

Then the torsion of
the canonical connection is an invariant of $C_0$, i.e. the relation
$\overline{T''}(x,y)=T''(x,y)$ holds with respect to any
transformation of $C_0$.
It is
proved that the curvature tensor of the canonical connection is a
K\"ahler tensor if and only if $(\MM,J,g)\in\W_1^0$, i.e. a $\W_1$-manifold
with closed Lee forms $\ta$ and $\widetilde \ta$.
Moreover, there are studied conformal invariants of the canonical connection in $\W_1^0$.

%The associated Bochner curvature tensor of $\DDD'$ is defined
%by
%\[
%B(\dot{R}')=\dot{R}'-\dfrac{1}{2(n-2)}\left\{\psi'_1-\psi'_2\right\}(\dot{\rho}')
%+\dfrac{1}{4(n-1)(n-2)}\left\{\dot{\tau}'(\pi'_1-\pi'_2)+\widetilde{\dot{\tau}'}\pi'_3\right\},
%\]
%where $\dot{R}'$, $\dot{\rho}'$ and $\dot{\tau}'$,
%$\widetilde{\dot{\tau}'}$ are the curvature tensor, the Ricci
%tensor and the scalar curvatures of $\DDD'$, respectively;
%$\psi'_1(\dot{\rho}')=-g\owedge \dot{\rho}'$,
%$\psi'_2(\dot{\rho}')=-\widetilde g\owedge
%\widetilde{\dot{\rho}'}$,
%$\widetilde{\dot{\rho}'}(\cdot,\cdot)=\dot{\rho}'(\cdot,J\cdot)$
%and $\owedge$ stands for the Kulkarni-Nomizu product of two
%(0,2)-tensors;
%
%$\psi'_1(s)(x,y,z,W)=-g\owedge s(x,W)-g(x,z)s(y,W)+s(y,z)g(x,W)-s(x,z)g(y,W)$,
%$\psi'_2(s)(x,y,z,W)=g(y,Jz)s(x,JW)-g(x,Jz)s(y,JW)+s(y,Jz)g(x,JW)-s(x,Jz)g(y,JW)$ for (0,2)-tensor $s$,
%
% besides $2\pi'_1=\psi'_1(g)$, $2\pi'_2=\psi'_2(g)$,
% $\pi'_3=-\psi'_1(\widetilde g)=\psi'_2(\widetilde g)$.

%According to  \cite{GaGrMi87}, there are valid the following
%statements for conformal invariants for the canonical connection:

%1) The Bochner curvature tensor of the canonical connection is an
%invariant of the group $C_0$ on any $\W_1^0$-manifold.

%2) Any $\W_1^0$-manifold with $\dim M'\geq 8$ and a vanishing
%Bochner curvature tensor of the canonical connection is
%conformally related by a transformation from $C_0$ to a
%$\W_1^0$-manifold whose canonical connection is flat.

Bearing in mind the conformal invariance of both the basic classes and
the torsion $T''$ of the canonical
connection,
the conditions for $T''$ are used in  \cite{GaMi87} for other characteristics of all classes
of the almost Norden manifolds as follows:
\begin{subequations}\label{Wi:Tc}
\begin{equation}
\begin{split}
\W_0:\quad &T''(x,y)=0;\\[6pt]
\W_1:\quad &T''(x,y)=\dfrac{1}{2n}\left\{t''(x)y-t''(y)x\right.\\%[1pt]
&\phantom{T''(x,y)=\dfrac{1}{2n}\left\{\right.}
\left.+t''(Jx)Jy-t''(Jy)Jx\right\};\\[6pt]
\W_2:\quad &T''(x,y)=T''(Jx,Jy),\quad t''=0;\\[6pt]
\W_3:\quad &T''(Jx,y)=-JT''(x,y);\\[6pt]
\W_1\oplus\W_2:\quad &T''(x,y)=T''(Jx,Jy),\\[6pt]
                      &\sx T''(x,y,z)=0;%\\[6pt]
\end{split}
\end{equation}
\begin{equation}
\begin{split}
\W_1\oplus\W_3:\quad &T''(Jx,y)+JT''(x,y)\\[6pt]
&\phantom{}
=\dfrac{1}{n}\bigl\{
t''(Jy)x-t''(y)Jx\bigr\};
\\[6pt]
\W_2\oplus\W_3:\quad  &t''=0;\\[6pt]
\W_1\oplus\W_2\oplus\W_3:\;  &\text{no conditions},
\end{split}
\end{equation}
\end{subequations}
where the torsion form $t''$ of $T''$ is determined by $t''(x)=g^{ij}T''(x,e_i,e_j)$.
The special class $\W_0$ is characterized by the condition $T''(x,y)=0$ and then $\DDD''\equiv \DDD$ holds.

The torsion $T''$ is known as a \emph{vectorial torsion}, because of its form on a $\W_1$-manifold.
Let the subclass of $\T_3$ with vectorial torsions be denoted by
$\T_3^1$ whereas  $\T_3^0$ be the subclass of $\T_3$ with vanishing torsion forms $t''$.

The classes of the almost Norden manifolds are determined  with respect to the Nijenhuis tensors in
\eqref{Wi:N}. The same classes
are characterized by conditions for
the torsion of the canonical
connection in \eqref{Wi:Tc}. By virtue of these results we obtain the following
\begin{thm}\label{thm:Wi:NhatN}
The classes of the almost Norden manifolds $(\MM,J,g)$ are characterized
by an expression of the torsion $T''$ of the canonical
connection in terms of the Nijenhuis tensors $[J,J]$ and $\{J,J\}$ as follows:
\begin{subequations}\label{Wi:Tc:NhatN}
\begin{equation}
\begin{split}
\W_1:\quad &T''(x,y,z)=\dfrac{1}{16n}\bigl(
    \vartheta(x)g(y,z)-\vartheta(y)g(x,z)\\[6pt]
    &\phantom{T''(x,y,z)=\dfrac{1}{16n}\bigl(}
    +\vartheta(Jx)g(y,Jz)\\[6pt]
    &\phantom{T''(x,y,z)=\dfrac{1}{16n}\bigl(}
    -\vartheta(Jy)g(x,Jz)
    \bigr);
\\[6pt]
\W_2:\quad &T''(x,y,z)=\dfrac{1}{8}\bigl(
    \{J,J\}(z,y,x)-\{J,J\}(z,x,y)\bigr),\\[6pt]
    &
    t''=\vartheta=0;\\[6pt]
\W_3:\quad &T''(x,y,z)=\dfrac{1}{4}[J,J](x,y,z);
\\[6pt]
\W_1\oplus\W_2:\quad    &T''(x,y,z)=\dfrac{1}{8}\bigl(\{J,J\}(z,y,x)\\[6pt]
                        &\phantom{T''(x,y,z)=\dfrac{1}{8}\bigl(}
                            -\{J,J\}(z,x,y)\bigr);%\\[6pt]
\end{split}
\end{equation}
\begin{equation}
\begin{split}
\W_1\oplus\W_3:\quad    &T''(x,y,z)=\dfrac14[J,J](x,y,z)\\[6pt]
                        &\phantom{T''(x,y,z)=}
                        +\dfrac{1}{16n}\bigl(\vartheta(x)g(y,z)\\[6pt]
                        &\phantom{T''(x,y,z)=+\dfrac{1}{16n}\bigl(}
                                    +\vartheta(Jx)g(y,Jz)\\[6pt]
                        &\phantom{T''(x,y,z)=+\dfrac{1}{16n}\bigl(}
                                    -\vartheta(y)g(x,z)\\[6pt]
                        &\phantom{T''(x,y,z)=+\dfrac{1}{16n}\bigl(}
                        -\vartheta(Jy)g(x,Jz)
    \bigr);
\\[6pt]
\W_2\oplus\W_3:\quad  &T''(x,y,z)=\dfrac{1}{4}[J,J](x,y,z)\\[6pt]
    &\phantom{T''(x,y,z)=}
    +\dfrac{1}{8}\bigl(\{J,J\}(z,y,x)\\[6pt]
    &\phantom{T''(x,y,z)=+\dfrac{1}{8}\bigl(}
    -\{J,J\}(z,x,y)\bigr),\\[6pt]
&    t''=\vartheta=0.
\end{split}
\end{equation}
\end{subequations}
The special class $\W_0$ is characterized by $T''=0$ and
the whole class $\W_1\oplus\W_2\oplus\W_3$ --- by \eqref{4.3} only.

Moreover, bearing in mind the classifications for almost Norden manifolds $(\MM,J,g)$   with respect to the
tensor $F$ and the torsion $T''$ in \cite{GaBo} and \cite{GaMi87},
respectively, we have:
\[
\begin{array}{ll}
(\MM,J,g)\in\W_1 \quad & \Leftrightarrow \quad T''\in\T_3^1;\\[6pt]
(\MM,J,g)\in\W_2 \quad & \Leftrightarrow \quad   T''\in\T_3^0;\\[6pt]
(\MM,J,g)\in\W_3 \quad & \Leftrightarrow \quad T''\in\T_2;\\[6pt]
(\MM,J,g)\in\W_1\oplus\W_2 \quad & \Leftrightarrow \quad T''\in\T_3;\\[6pt]
(\MM,J,g)\in\W_1\oplus\W_3 \quad & \Leftrightarrow \quad T''\in\T_2\oplus\T_3^1;\\[6pt]
(\MM,J,g)\in\W_2\oplus\W_3 \quad & \Leftrightarrow \quad T''\in\T_2\oplus\T_3^0;\\[6pt]
(\MM,J,g)\in\W_1\oplus\W_2\oplus\W_3 \quad & \Leftrightarrow \quad T''\in\T_2\oplus\T_3.
\end{array}
\]

%\begin{enumerate}
%\item $(\MM,J,g)\in\W_1\oplus\W_2$ if and only if $T''\in\T_3$;
%\item $(\MM,J,g)\in\W_1$ if and only if $T''\in\T_3^1$;
%\item $(\MM,J,g)\in\W_2$ if and only if   $T''\in\T_3^0$;
%\item $(\MM,J,g)\in\W_3$ if and only if $T''\in\T_2$;
%\item $(\MM,J,g)\in\W_1\oplus\W_3$ if and only if $T''\in\T_2\oplus\T_3^1$;
%\item $(\MM,J,g)\in\W_2\oplus\W_3$ if and only if $T''\in\T_2\oplus\T_3^0$;
%\item $(\MM,J,g)\in\W_1\oplus\W_2\oplus\W_3$ if and only if $T''\in\T_2\oplus\T_3$.
%\end{enumerate}

%\begin{center}
%\begin{tabular}{|c|c|}
%  \hline
%  % after \\: \hline or \cline{col1-col2} \cline{col3-col4} ...
%  $\W_1$ & $\T_3^1$    \\
%  \hline
%  $\W_2$ & $\T_3^0$    \\
%  \hline
%  $\W_3$ & $\T_2$    \\
%  \hline
%  $\W_1\oplus\W_2$ & $\T_3$    \\
%  \hline
%  $\W_1\oplus\W_3$ & $\T_2\oplus\T_3^1$    \\
%  \hline
%  $\W_2\oplus\W_3$ & $\T_2\oplus\T_3^0$    \\
%  \hline
%  $\W_1\oplus\W_2\oplus\W_3$ & $\T_2\oplus\T_3$    \\
%  \hline
%\end{tabular}
%\end{center}

\end{thm}
\begin{proof}
Let $(\MM,J,g)$ be a complex Norden manifold, i.e. $(\MM,J,g)\in\W_1\oplus\W_2$.
According to \eqref{4.3} and $[J,J]=0$ in this case, we have
$T''=T''_3$, i.e.
    $T''\in\T_3$ and the expression
    \[
    T''(x,y,z)=\dfrac{1}{8}\bigl(
    \{J,J\}(z,y,x)-\{J,J\}(z,x,y)\bigr)
    \] is obtained.
Applying \eqref{Wi:N} to the latter equality, we determine the basic classes
$\W_1$ and $\W_2$ as it is given in \eqref{Wi:Tc:NhatN} and the corresponding
subclasses $\T_3^1$ and $\T_3^0$, respectively.
Taking into account the relation between the corresponding traces
$\vartheta=8t''$, which is a consequence of the equality for
$\W_1\oplus\W_2$, we obtain the  characterization for these two basic classes in \eqref{Wi:Tc}.

Let $(\MM,J,g)$ be a quasi-K\"ahler manifold with Norden metric, i.e. $(\MM,J,g)\in\W_3$.
By virtue of \eqref{4.3} and $\{J,J\}=0$ for such a manifold,
we have $T''=T''_2$, i.e.
    $T''\in\T_2$ and therefore we give
$T''=\dfrac{1}{4}[J,J]$.
Obviously, the form of $T''$ in the latter equality satisfies
the condition for $\W_3$ in \eqref{Wi:Tc}.

In a similar way we get for the remaining classes $\W_1\oplus\W_3$ and $\W_2\oplus\W_3$.
The conditions of these two classes, given in \eqref{Wi:Tc},
are consequences of the corresponding equalities in \eqref{Wi:Tc:NhatN}.
The case of the whole class $\W_1\oplus\W_2\oplus\W_3$ was discussed above.
\end{proof}

The canonical connections on quasi-K\"ahler manifolds with Norden
metric are considered in more details in  \cite{Mek09}. There are given the following formulae for
the potential $Q''$ and the torsion $T''$ on a
$\W_3$-manifold:
\[
\begin{split}
    &Q''(x,y) = \dfrac{1}{4}\left\{\left(\DDD_y J\right)Jx-\left(\DDD_{Jy} J\right)x
    +2\left(\DDD_{x}
    J\right)Jy\right\},\\[6pt]
    &T''(x,y) = \dfrac{1}{2}\left\{\left(\DDD_x J\right)Jy+\left(\DDD_{Jx} J\right)y\right\}.
\end{split}
\]
%
%According to  \cite{Mek09}, on a quasi-K\"ahler manifold with
%Norden metric, the scalar curvatures $\ddot{\tau}'$ and $\tau'$
%for the canonical connection $\DDD''$ and the Levi-Civita
%connection $\DDD$, respectively, are related by
%    \begin{equation}\label{4.13}
%        \ddot{\tau}'=\tau'-\dfrac{1}{4}\nJ.
%    \end{equation}
%Then a quasi-K\"ahler manifold with Norden metric is
%isotropic-K\"ah\-ler\-ian if and only if the scalar curvatures for
%the canonical connection and the Levi-Civita connection are equal.
Moreover, some properties for the curvature and the torsion of the  canonical connection are obtained.

%%%%%%%%%%%%%%%%%%%%%%%%%%%%%%%%%%%%%%%%%%%%%%%%%%%%%%%%%%%%%%%%%%%%%%%%%%%%%%

%Moreover, from  \cite{Mek09} we know that if the canonical
%connection has a parallel torsion on a quasi-K\"ahler manifold
%with Norden metric then the manifold is iso\-tropic-K\"ahlerian.%

%%%%%%%%%%%%%%%%%%%%%%%%%%%%%%%%%%%%%%%%%%%%%%%%%%%%%%%%%%%%%%%%%%%%%%%%%%%%%%%%%1
\vskip 0.2in \addtocounter{subsection}{1} \setcounter{subsubsection}{0}

\noindent  {\Large\bf \thesubsection. The KT-connection}

\vskip 0.15in

%\subsection{The KT-connection}

In  \cite{Mek-10}, it is proved that a natural connection $\DDD'''$
with totally skew-symmetric torsion, called a
\emph{KT-connection}, exists on an almost Norden manifold
$(\MM,J,g)$ if and only if $(\MM,J,g)$ belongs to $\W_3$, i.e. the
manifold is quasi-K\"ahlerian with Norden metric. Moreover,
the KT-connection is unique and it is determined by its potential
\begin{equation}\label{QQQ3}
    Q'''(x,y,z)=-\dfrac{1}{4}\sx F(x,y,Jz).
\end{equation}

As mentioned above, the
canonical con\-nect\-ion and the B-connection coincide on $(\MM,J,g)\in\W_1\oplus\W_2$
whereas the
KT-connection does not exist there.

The following natural connections on $(\MM,J,g)$ are studied on a quasi-K\"ahler
manifold with Norden metric: the B-connection $\DDD'$
\cite{Mek-1}, the
canonical connection $\DDD''$ \cite{Mek09} and the KT-connection $\DDD'''$  \cite{Mek-08,Mek-10}.

Relations \eqref{4.3} and \eqref{TbNhatN} of $T'$ and $T''$ in terms of the pair of Nijenhuis tensors
 are specialized for a $\W_3$-manifold in the following way
\begin{equation}\label{TcTbNW3}
\begin{split}
&T'(x,y,z)=\dfrac{1}{8}\bigl([J,J](x,y,z)+\sx [J,J](x,y,z)\bigr),\quad\\[6pt]
&T''(x,y,z)=\dfrac{1}{4}[J,J](x,y,z).
\end{split}
\end{equation}
The equalities \eqref{TQ} and \eqref{QQQ3} yield
\begin{equation}\label{TKT=F}
T'''(x,y,z)=
-\dfrac{1}{2}\sx F(x,y,Jz),
\end{equation}
which by
\eqref{Wi:N0Nhat0} for $\W_3$ and \eqref{N-prop} implies
\begin{equation}\label{TkNW3}
T'''(x,y,z)=\dfrac{1}{4}\sx [J,J](x,y,z).
\end{equation}

Then from \eqref{TcTbNW3} and \eqref{TkNW3} we have  the relation
\[
T'=\dfrac{1}{2}\left(T''+ T'''\right),
\]
which by
\eqref{QT} is equivalent to
\[
Q'=\dfrac{1}{2}\left(Q''+ Q'''\right).
\]
Therefore, as it is shown in  \cite{Mek09}, the B-connection
is the \emph{average} connection for the canonical connection and the
KT-connection on a quasi-K\"ahler manifold with Norden metric,
%(in the sense of the one-parameter family of Hermitian connections of Gauduchon)
i.e.
\[
\DDD'=\dfrac{1}{2}\left(\DDD''+ \DDD'''\right).
\]

\vspace{20pt}

\begin{center}
$\divideontimes\divideontimes\divideontimes$
\end{center}

\newpage

\addtocounter{section}{1}\setcounter{subsection}{0}\setcounter{subsubsection}{0}

\setcounter{thm}{0}\setcounter{dfn}{0}\setcounter{equation}{0}

\label{par:2n+1}

 \Large{

\
\\[6pt]
\bigskip

\
\\[6pt]
\bigskip

\lhead{\emph{Chapter I $|$ \S\thesection. Almost contact manifolds with
B-metric
}}
%\thispagestyle{empty}

%\noindent  {\Huge\bf \S\thesection. Almost contact manifolds \\[12pt]
%\phantom{\S\thesection. }with B-metric
%}%\\[6pt]\vskip2pt}
\noindent
\begin{tabular}{r"l}
  %\hline
\hspace{-6pt}{\Huge\bf \S\thesection.}  & {\Huge\bf Almost contact manifolds} \\[12pt]
                             & {\Huge\bf with B-metric}
  %\hline
\end{tabular}

\vskip 1cm

\begin{quote}
\begin{large}
In the present section we recall some notions and knowledge for the almost contact manifolds with
B-metric which
are %investigated and
studied
in \cite{GaMiGr,Man4,Man31,ManGri1,ManGri2,ManIv38,ManIv36,NakGri2}.
\end{large}
\end{quote}

%
%\vskip 0.2in \addtocounter{subsection}{1}
%
%\noindent  {\Large\bf \thesubsection. Introduction}

\vskip 0.15in

%%%%%%%%%%%%%%%%%%%%%%%%%%%%%%%%%%%%%%%%%%%%%%%%%%%%%%%%%%%%%%%%%%%%%%%%%%%%%%%%%1
\vskip 0.2in \addtocounter{subsection}{1} \setcounter{subsubsection}{0}

\noindent  {\Large\bf \thesubsection. Almost contact structures with B-metric}

\vskip 0.15in

%\section{Almost Contact Manifolds with B-Metric}\label{sec:1}

Let $(\MM,\f,\xi,\eta)$ be an almost contact manifold,  i.e.  $\MM$ is
a differentiable manifold of dimension $(2n+1)$, provided  with an almost
contact structure $(\f,\xi,\eta)$ consisting of an endomorphism
$\f$ of the tangent bundle, a vector field $\xi$ and its dual
1-form $\eta$ such that the following algebraic relations are
satisfied: \cite{Blair}
\begin{equation}\label{str1}
\f\xi = 0,\quad \f^2 = -I + \eta \otimes \xi,\quad
\eta\circ\f=0,\quad \eta(\xi)=1,
\end{equation}
where $I$ denotes the identity.

Further, let us equip the almost contact manifold
$(\MM,\f,\xi,\eta)$ with a pseudo-Riemannian metric $g$ of signature
$(n+1,n)$ determined by
\begin{equation}\label{str2}
g(\f x, \f y ) = - g(x, y ) + \eta(x)\eta(y).
\end{equation}
Then $(\MM,\f,\xi,\eta,g)$ is called an almost contact manifold with
B-metric or an \emph{almost contact B-metric manifold}
\cite{GaMiGr}.

%
%More precisely, it consists of real square matrices of order
%$2n+1$ of the following type
%\[
%\left(%
%\begin{array}{r|c|c}
%  A & B & \vartheta^T\\[6pt] \hline
%  -B & A & \vartheta^T\\[6pt] \hline
%  \vartheta & \vartheta & 1 \\[6pt]
%\end{array}%
%\right),\qquad %
%\begin{array}{l}
%  AA^T-BB^T=I_n,\\[6pt]%
%  AB^T+BA^T=O_n,
%\end{array}%
%\quad A, B\in \mathcal{GL}(n;\mathbb{R}),
%\]
%where $\vartheta$ and its transpose $\vartheta^T$ are the zero row
%$n$-vector and the zero column $n$-vector; $I_n$ and $O_n$ are the
%unit matrix and the zero matrix of size $n$, respectively.

The associated metric $\widetilde{g}$ of $g$ on $\M$ is defined by the equality
\[\widetilde{g}(x,y)=g(x,\f y)\allowbreak+\eta(x)\eta(y).\] Both
metrics $g$ and $\widetilde{g}$ are necessarily of signature
$(n+1,n)$. The manifold $(\MM,\f,\xi,\eta,\widetilde{g})$ is also an
almost contact B-metric manifold.

Let us remark that the $2n$-dimensional contact distribution
$\HC=\ker(\eta)$, generated by the contact 1-form $\eta$, can be
considered as the horizontal distribution of the sub-Riemannian
manifold $\MM$. Then $\HC$ is endowed with an almost complex structure
determined as $\f|_\HC$ -- the restriction of $\f$ on $\HC$, as well
as a Norden metric $g|_\HC$, i.e.
\[g|_\HC(\f|_\HC\cdot,\f|_\HC\cdot)=-g|_\HC(\cdot,\cdot).\] Moreover, $\HC$
can be considered as an $n$-dimensional complex Riemannian
manifold with a complex Riemannian metric
$g^{\C}=g|_\HC+i\widetilde{g}|_\HC$~ \cite{GaIv92}.
By this reason we can refer to these manifolds as \emph{almost contact complex Riemannian manifolds}.

Using the Reeb vector field $\xi$ and its dual contact 1-form $\eta$ on an arbitrary almost contact B-metric manifold $\M$, we consider two distributions
in the tangent bundle $T\MM$ of $\MM$ as follows
\begin{equation}\label{HV}
  \HH=\ker(\eta),\qquad \VV=\Span(\xi).
\end{equation}
Then the horizontal distribution $\HH$ and the vertical distribution $\VV$ form a pair of mutually complementary distributions in $T\MM$ which are orthogonal with respect to both of the metrics $g$ and $\tg$, i.e.
\begin{equation}\label{HHVV}
\HH\oplus\VV =T\MM,\qquad \HH\;\bot\;\VV,\qquad \HH\cap\VV=\{o\},
\end{equation}
where $o$ is the zero vector field on $\MM$.
%The distribution $\HH$ is known also as the contact distribution.
%
Thus, there are determined the corresponding horizontal and vertical projectors
\begin{equation}\label{TMhv}
\mathrm{h}:T\MM\mapsto\HH,\qquad \mathrm{v}:T\MM\mapsto\VV
\end{equation}
having the properties $\mathrm{h}\circ \mathrm{h} =\mathrm{h}$, $\mathrm{v}\circ
\mathrm{v}=\mathrm{v}$, $\mathrm{h}\circ \mathrm{v}=\allowbreak{}\mathrm{v}\circ \mathrm{h}=0$.
An arbitrary vector field $x$ in $T\MM$ has
respective projections $x^{\mathrm{h}}$ and $x^{\mathrm{v}}$
so that
\begin{equation}\label{hv}
x=x^{\mathrm{h}}+x^{\mathrm{v}},
\end{equation}
where
\begin{equation}\label{Xhv}
x^{\mathrm{h}}=-\f^2x=x-\eta(x)\xi, \qquad x^{\mathrm{v}}=\eta(x)\xi
\end{equation}
are the so-called horizontal and vertical components, respectively.

The structure group of $\M$ is $\GG\times\II$, where $\GG$ is the group determined in \eqref{GG} and $\II$ is the
identity on $\VV$. Consequently, the structure group consists of the real square matrices of
order $2n+1$ of the following type
\begin{equation}\label{GxI}
\left(%
\begin{array}{r|c|c}
  A & B & \mathrm{o}^\top\\ \hline
  -B & A & \mathrm{o}^\top\\ \hline
 \mathrm{o} & \mathrm{o} & 1 \\
\end{array}%
\right),
\end{equation}
where $\mathrm{o}$ and its transpose $\mathrm{o}^\top$ are the zero row
$n$-vector and the zero column $n$-vector; $A$ and $B$ are real invertible matrices of size $n$ satisfying the conditions in \eqref{AB}.

%%%%%%%%%%%%%%%%%%%%%%%%%%%%%%%%%%%%%%%%%%%%%%%%%%%%%%%%%%%%%%%%%%%%%%%%%%%%%%%%%1
\vskip 0.2in \addtocounter{subsection}{1} \setcounter{subsubsection}{0}

\noindent  {\Large\bf \thesubsection. Fundamental tensors $F$ and $\tF$}

\vskip 0.15in

The covariant derivatives of $\f$, $\xi$, $\eta$ with respect to
the Levi-Civita connection $\n$ play a fundamental role in
differential geometry on the almost contact manifolds.  The
fundamental tensor $F$ of type (0,3) on $\M$ is defined by
\begin{equation}\label{F=nfi}
F(x,y,z)=g\bigl( \left( \nabla_x \f \right)y,z\bigr).
\end{equation}
It has the following basic properties:
\begin{equation}\label{F-prop}
\begin{split}
F(x,y,z)&=F(x,z,y)\\[6pt]
&=F(x,\f y,\f z)+\eta(y)F(x,\xi,z)
+\eta(z)F(x,y,\xi).
\end{split}
\end{equation}
The relations of $\n\xi$ and $\n\eta$ with $F$ are:
\begin{equation}\label{Fxieta}
    \left(\n_x\eta\right)(y)=g\left(\n_x\xi,y\right)=F(x,\f y,\xi).
\end{equation}

The 1-forms $\ta$, $\ta^*$ and $\om$, called \emph{Lee forms} on $\M$, are associated with $F$ by the following way:
\begin{equation}\label{titi}
\begin{array}{l}
\ta(z)=g^{ij}F(e_i,e_j,z),\quad \\[6pt]
\ta^*(z)=g^{ij}F(e_i,\f e_j,z),\\[6pt]
\om(z)=F(\xi,\xi,z),
\end{array}
\end{equation}
where $g^{ij}$ are the components of the inverse matrix of $g$
with respect to a basis $\left\{e_i;\xi\right\}$
$(i=1,2,\dots,2n)$ of the tangent space $T_p\MM$ of $\MM$ at an
arbitrary point $p\in \MM$.

Obviously, the equality $\om(\xi)=0$ and
the following relation are always valid:
\begin{equation}\label{tata*}
\ta^*\circ\f=-\ta\circ\f^2.
\end{equation}
For the corresponding traces $\widetilde\ta$
and $\widetilde\ta^*$ with respect to $\widetilde g$ we have
\[
\widetilde\ta=-\ta^*,\qquad \widetilde\ta^*=\ta.
\]

A classification of the almost contact B-metric manifolds with
respect to the properties of $F$ is given by G. Ganchev, V. Mihova and K. Gribachev in \cite{GaMiGr}. This classification
includes the basic classes $\F_1$, $\F_2$, $\dots$, $\F_{11}$.
Their intersection is the special class $\F_0$ determined by the
condition $F(x,y,z)=0$. Hence $\F_0$ is the class of almost
contact B-metric manifolds with $\n$-parallel structures, i.e.
\[
\n\f=\n\xi=\n\eta=\n g=\n \widetilde{g}=0.
\]
The $\F_0$-manifolds are also known as \emph{cosymplectic B-metric manifolds}.
Further, we use the following characteristic conditions of the
basic classes:  \cite{GaMiGr,Man8}
%\begin{subequations}
\begin{equation}\label{Fi}
\begin{split}
\F_{1}: \quad &F(x,y,z)=\dfrac{1}{2n}\bigl\{g(x,\f y)\ta(\f z)+g(\f x,\f
y)\ta(\f^2 z)\\
&\phantom{F(x,y,z)=\dfrac{1}{2n}\bigl\{}
+g(x,\f z)\ta(\f y)+g(\f x,\f
z)\ta(\f^2 y)
\bigr\};\\[6pt]
\F_{2}: \quad &F(\xi,y,z)=F(x,\xi,z)=0,\quad
              \sx F(x,y,\f z)=\ta=0;\\[6pt]
\F_{3}: \quad &F(\xi,y,z)=F(x,\xi,z)=0,\quad
              \sx F(x,y,z)=0;\\[6pt]
\F_{4}: \quad &F(x,y,z)=-\dfrac{\ta(\xi)}{2n}\bigl\{g(\f x,\f y)\eta(z)+g(\f x,\f z)\eta(y)\bigr\};\\[6pt]
\F_{5}: \quad &F(x,y,z)=-\dfrac{\ta^*(\xi)}{2n}\bigl\{g( x,\f y)\eta(z)+g(x,\f z)\eta(y)\bigr\};\\[6pt]
\F_{6}: \quad &F(x,y,z)=F(x,y,\xi)\eta(z)+F(x,z,\xi)\eta(y),\quad \\[6pt]
                \quad &F(x,y,\xi)= F(y,x,\xi)=-F(\f x,\f y,\xi),\quad \\[6pt]
                &\ta=\ta^*=0; \\[6pt]
%
%\end{split}
%\end{equation}
%\begin{equation}
%\begin{split}
\F_{7}: \quad &F(x,y,z)=F(x,y,\xi)\eta(z)+F(x,z,\xi)\eta(y),\quad \\[6pt]
                \quad &F(x,y,\xi)=- F(y,x,\xi)=-F(\f x,\f y,\xi); \\[6pt]
\F_{8}: \quad &F(x,y,z)=F(x,y,\xi)\eta(z)+F(x,z,\xi)\eta(y),\\[6pt]
                \quad &F(x,y,\xi)=F(y,x,\xi)=F(\f x,\f y,\xi); \\[6pt]
\F_{9}: \quad &F(x,y,z)=F(x,y,\xi)\eta(z)+F(x,z,\xi)\eta(y),\\[6pt]
                \quad &F(x,y,\xi)=- F(y,x,\xi)=F(\f x,\f y,\xi); \\[6pt]
\F_{10}: \quad &F(x,y,z)=F(\xi,\f y,\f z)\eta(x); \\[6pt]
\F_{11}:
\quad &F(x,y,z)=\eta(x)\left\{\eta(y)\om(z)+\eta(z)\om(y)\right\}.
\end{split}
\end{equation}
%\end{subequations}

In \cite{HMan}, it is proved that
$\M$ belongs to $\F_i$ $(i=1,\dots,11)$ if and only if
$F$ satisfies the condition $F=F^i$, where the components $F^i$ of
$F$ are the following
\begin{subequations}\label{Fi-Ico}
\begin{equation}
\begin{split}
&F^{1}(x,y,z)=\frac{1}{2n}\bigl\{g(\f x,\f y)\ta(\f^2 z) +g(x,\f
y)\ta(\f z)\\
&\phantom{F^{1}(x,y,z)=\frac{1}{2n}\bigl\{}
+g(\f x,\f z)\ta(\f^2 y)+g(x,\f z)\ta(\f y)\bigr\}, \\[6pt]
&F^{2}(x,y,z)=-\frac{1}{4}\bigl\{ %
F(\f^2 x,\f^2 y,\f^2 z)+F(\f^2 y,\f^2 z,\f^2 x)\\
&\phantom{F^{2}(x,y,z)=-\frac{1}{4}\bigl\{}
-F(\f y,\f^2 z,\f x)+F(\f^2 x,\f^2 z,\f^2 y)  \\
&\phantom{F^{2}(x,y,z)=-\frac{1}{4}\bigl\{}%
+F(\f^2 z,\f^2 y,\f^2 x)-F(\f z,\f^2 y,\f
x)\bigr\}\\
&\phantom{F^{2}(x,y,z)=}%
-\frac{1}{2n}\bigl\{g(\f x,\f y)\ta(\f^2 z) +g(x,\f
y)\ta(\f z) \\
&\phantom{F^{2}(x,y,z)=-\frac{1}{2n}\bigl\{}%
+g(\f x,\f z)\ta(\f^2 y)+g(x,\f z)\ta(\f
y)\bigr\},  \\[6pt]
&F^{3}(x,y,z)=-\frac{1}{4}\bigl\{%
F(\f^2 x,\f^2 y,\f^2 z)-F(\f^2 y,\f^2 z,\f^2 x) \\
&\phantom{F^{3}(x,y,z)=-\frac{1}{4}\bigl\{}%
+F(\f y,\f^2 z,\f
x)+F(\f^2 x,\f^2 z,\f^2 y)\\
&\phantom{F^{3}(x,y,z)=-\frac{1}{4}\bigl\{}%
-F(\f^2 z,\f^2 y,\f^2 x) +F(\f z,\f^2 y,\f x)\bigr\}, \\[6pt]
&F^{4}(x,y,z)=-\frac{\ta(\xi)}{2n}\bigl\{g(\f x,\f y)\eta(z)+g(\f x,\f z)\eta(y)\bigr\},  \\[6pt]
&F^{5}(x,y,z)=-\frac{\ta^*(\xi)}{2n}\bigl\{g(x,\f y)\eta(z)+g(x,\f z)\eta(y)\bigr\},\\[6pt]
&F^{6}(x,y,z)=\frac{\ta(\xi)}{2n}\bigl\{g(\f x,\f y)\eta(z)+g(\f x,\f z)\eta(y)\bigr\}\\
&\phantom{F^{6}(x,y,z)=}%
+\frac{\ta^*(\xi)}{2n}\bigl\{g(x,\f y)\eta(z)+g(x,\f z)\eta(y)\bigr\}\\
&\phantom{F^{6}(x,y,z)=}%
+\frac{1}{4}%
   \bigl\{F(\f^2 x,\f^2 y,\xi)+F(\f^2 y,\f^2 x,\xi)\\
&\phantom{F^{6}(x,y,z)=+\frac{1}{4}\bigl\{}%
   -F(\f x,\f y,\xi)-F(\f y,\f x,\xi)\bigr\}\eta(z)\\
&\phantom{F^{6}(x,y,z)=}%
+\frac{1}{4}\bigl\{F(\f^2 x,\f^2 z,\xi)+F(\f^2 z,\f^2 x,\xi)\\
&\phantom{F^{6}(x,y,z)=+\frac{1}{4}\bigl\{}%
-F(\f x,\f z,\xi)-F(\f z,\f x,\xi)\bigr\}\eta(y),\\[6pt]
&F^{7}(x,y,z)=\frac{1}{4}
    \bigl\{F(\f^2 x,\f^2 y,\xi)-F(\f^2 y,\f^2 x,\xi)\\
&\phantom{F^{7}(x,y,z)=\frac{1}{4}\bigl\{}%
    -F(\f x,\f y,\xi)+F(\f y,\f x,\xi)\bigr\}\eta(z)
\end{split}
\end{equation}
\begin{equation}
\begin{split}
&\phantom{F^{7}(x,y,z)=} %
+\frac{1}{4}
    \bigl\{F(\f^2 x,\f^2 z,\xi)-F(\f^2 z,\f^2 x,\xi)\\
&\phantom{F^{7}(x,y,z)=+\frac{1}{4}\bigl\{}%
    -F(\f x,\f z,\xi)+F(\f z,\f x,\xi)\bigr\}\eta(y),\\[6pt]
&F^{8}(x,y,z)=\frac{1}{4}
    \bigl\{F(\f^2 x,\f^2 y,\xi)+ F(\f^2 y,\f^2 x,\xi)\\
&\phantom{F^{8}(x,y,z)=\frac{1}{4}\bigl\{}%
    +F(\f x,\f y,\xi)+ F(\f y,\f x,\xi)\bigr\}\eta(z)\\
&\phantom{F^{8}(x,y,z)=}
    +\frac{1}{4}
    \bigl\{F(\f^2 x,\f^2 z,\xi)+ F(\f^2 z,\f^2 x,\xi)\\
&\phantom{F^{8}(x,y,z)=+\frac{1}{4}\bigl\{}%
    +F(\f x,\f z,\xi)+ F(\f z,\f x,\xi)\bigr\}\eta(y),\\[6pt]
&F^{9}(x,y,z)=\frac{1}{4}
    \bigl\{F(\f^2 x,\f^2 y,\xi)-F(\f^2 y,\f^2 x,\xi)\\
&\phantom{F^{9}(x,y,z)=\frac{1}{4}\bigl\{}%
    +F(\f x,\f y,\xi)-F(\f y,\f x,\xi)\bigr\}\eta(z)\\
&\phantom{F^{9}(x,y,z)=} +\frac{1}{4}
    \bigl\{F(\f^2 x,\f^2 z,\xi)-F(\f^2 z,\f^2 x,\xi)\\
&\phantom{F^{9}(x,y,z)=+\frac{1}{4}\bigl\{}%
    +F(\f x,\f z,\xi)-F(\f z,\f x,\xi)\bigr\}\eta(y), \\[6pt]
&F^{10}(x,y,z)=\eta (x) F(\xi, \f^2 y, \f^2 z),\\[6pt]
&F^{11}(x,y,z)= -\eta (x) \left\{\eta(y) F(\xi, \xi,\f^2 z) + \eta(z)
F(\xi, \f^2 y, \xi)\right\}.
\end{split}
\end{equation}
\end{subequations}

It is said that an almost contact B-metric manifold
belongs to a direct sum of two or more basic classes, i.e.
$\M\in\F_i\oplus\F_j\oplus\cdots$, if and only if the fundamental
tensor $F$ on $\M$ is the sum of the corresponding components
$F^i$, $F^j$, $\ldots$ of $F$, i.e. the following condition is
satisfied $F=F^i+F^j+\cdots$ for  $i,j\in\{1,2,\dots,11\}$, $i\neq j$.

For the minimal dimension 3 of an almost contact B-metric manifold, it is known that the basic classes are only seven, because $\F_2$, $\F_3$, $\F_6$ and $\F_7$ are restricted to $\F_0$ \cite{GaMiGr,HMan}.
The geometry of the considered manifolds in dimension 3 is recently studied in \cite{HMan,HMan2,HMan4,HMan5,HMan6,HMan7,HManNak}.

Using \eqref{F-prop} and taking the traces with respect to $g$ denoted by $\tr$ and the traces with respect to $\tg$ denoted by $\tr^*$, we obtain the following relations
\begin{equation}\label{divtr}
\ta(\xi)=\Div^*(\eta),\qquad \ta^*(\xi)=\Div(\eta),
\end{equation}
where $\Div$ and $\Div^*$ denote the divergence using a trace by $g$ and by $\tg$, respectively.

Since $g(\xi,\xi)=1$ implies $g(\n_x\xi,\xi)=0$, then we obtain $\n_x\xi\in\HH$.
The shape operator $S:\HH\mapsto\HH$ for the metric $g$ is defined by
\begin{equation}\label{Sx}
S(x)=-\n_x\xi.
\end{equation}

As a corollary, the covariant derivative of $\xi$ with respect to $\n$ and the dual covariant derivative of $\eta$ because of \eqref{Fxieta} are determined in each class as follows:
\begin{equation}\label{Fi:nxi}
\begin{split}
&\F_{1}:\; \n\xi=0;\qquad
\F_{2}:\; \n\xi=0;\qquad
\F_{3}:\; \n\xi=0;\qquad\\[6pt]
&\F_{4}:\; \n\xi=\dfrac{1}{2n}\Div^*(\eta)\,\f;\qquad
\F_{5}:\; \n\xi=-\dfrac{1}{2n}\Div(\eta)\,\f^2;\\[6pt]
%
%\end{array}
%\end{equation}
%\begin{equation}
%\begin{array}{rl}
&\F_{6}:\;  g(\n_x\xi,y)=g(\n_y\xi,x)=-g(\n_{\f x}\xi,\f y),\quad \\
&\phantom{\F_{6}:\; \ }
\Div(\eta)=\Div^*(\eta)=0; \\[6pt]
&\F_{7}:\; g(\n_x\xi,y)=-g(\n_y\xi,x)=-g(\n_{\f x}\xi,\f y); \qquad\\[6pt]
&\F_{8}:\; g(\n_x\xi,y)=-g(\n_y\xi,x)=g(\n_{\f x}\xi,\f y); \\[6pt]
&\F_{9}:\; g(\n_x\xi,y)=g(\n_y\xi,x)=g(\n_{\f x}\xi,\f y); \qquad\\[6pt]
&\F_{10}:\; \n\xi=0; \qquad
\F_{11}:\; \n\xi=\eta\otimes\f\om^{\sharp},
\end{split}
\end{equation}
where $\sharp$ denotes the musical isomorphism of $T^*\MM$ in $T\MM$ given by $g$.

The latter characteristics of the basic classes imply the following
\begin{prop}\label{prop-nxi=0}
The class of almost contact B-metric manifolds with vanishing $\n \xi$ is $\F_1\oplus\F_2\oplus\F_3\oplus\F_{10}$.
\end{prop}
\begin{proof}
Let $\M$ be in the general class $\F_1\oplus\F_2\oplus\cdots\oplus\F_{11}$.
Then the equality $F=F^1+F^2+\ldots+F^{11}$ is valid, where $F^i$ are given in \eqref{Fi-Ico}.
Suppose $\n\xi$ vanishes, we obtain $F^4=F^5=\dots=F^9=F^{11}=0$ by \eqref{Fxieta}. Therefore, we have $F=F^1+F^2+F^3+F^{10}$, i.e. the manifold belongs to $\F_1\oplus\F_2\oplus\F_3\oplus\F_{10}$.

Vice versa,
let $\M$ be in the class $\F_1\oplus\F_2\oplus\F_3\oplus\F_{10}$, i.e. the fundamental tensor has the form $F=F^1+F^2+F^3+F^{10}$ for arbitrary $(x,y,z)$.
Bearing in mind $\f\xi=0$ and \eqref{Fi-Ico}, we deduce that $F^i(x,\f y,\xi)$ vanishes for $i\in\{1,2,3,10\}$ and then $F(x,\f y,\xi)=0$ is valid. The latter equality is equivalent to the condition $\n\xi=0$, according to \eqref{Fxieta}.
\end{proof}
%The associated metric $\tg$ of $g$ on $\MM$ is defined by
%\(\tg(x,y)=g(x,\f y)+\eta(x)\eta(y)\).  The manifold
%$(\MM,\f,\xi,\eta,\tg)$ is also an almost contact B-metric manifold.
%The B-metric $\tg$ is also of signature $(n+1,n)$.
%The Levi-Civita connection of $\tg$ is denoted by
%$\nn$.
%Let us denote the potential of $\nn$ regarding $\n$ by $\Phi$, i.e. $\Phi(x,y)=\nn_x y - \n_x y$.

Let us consider the tensor $\Phi$ of type (1,2) defined in \cite{GaMiGr} as the difference of the Levi-Civita connections $\tn$ and $\n$ of the corresponding B-metrics $\tg$ and $g$ as follows
\begin{equation}\label{Phi0}
    \Phi(x,y)=\tn_x y-\n_x y.
\end{equation}
This tensor is known also as the \emph{potential} of $\tn$ regarding $\n$ because of the formula
\begin{equation}\label{pot}
    \tn_x y=\n_x y+\Phi(x,y).
\end{equation}
Since both the connections are torsion-free, then $\Phi$ is symmetric, i.e. $\Phi(x,y)=\Phi(y,x)$ holds.
Let the corresponding tensor of type $(0,3)$ with respect to $g$ be defined by
\begin{equation}\label{Phi03=}
    \Phi(x,y,z)=g(\Phi(x,y),z).
\end{equation}

In \cite{NakGri93}, it is given a characterization of all basic classes in terms of $\Phi$ by means of
the following relations between $F$ and $\Phi$ known from \cite{GaMiGr}
\begin{equation}\label{FPhi}
\begin{array}{l}
F(x,y,z)=\Phi(x,y,\f z)+\Phi(x,z,\f y)\\[6pt]
\phantom{F(x,y,z)=}
+\dfrac12\eta(z)\{\Phi(x,y,\xi)-\Phi(x,\f y,\xi)\\[6pt]
\phantom{F(x,y,z)=+\dfrac12\eta(z)\{}
+\Phi(\xi,x,y)-\Phi(\xi,x,\f y)\}\\[6pt]
\phantom{F(x,y,z)=}
+\dfrac12\eta(y)\{\Phi(x,z,\xi)-\Phi(x,\f z,\xi)\\[6pt]
\phantom{F(x,y,z)=+\dfrac12\eta(y)\{}
+\Phi(\xi,x,z)-\Phi(\xi,x,\f z)\},
\end{array}
\end{equation}
\begin{equation}\label{PhiF}
\begin{array}{l}
2\Phi(x,y,z)=-F(x,y,\f z)-F(y,x,\f z)+F(\f z,x,y)\\[6pt]
\phantom{2\Phi(x,y,z)=}
+\eta(x)\{F(y,z,\xi)+F(\f z,\f y,\xi)\}\\[6pt]
\phantom{2\Phi(x,y,z)=}
+\eta(y)\{F(x,z,\xi)+F(\f z,\f x,\xi)\}\\[6pt]
\phantom{2\Phi(x,y,z)=}
+\eta(z)\{-F(\xi,x,y)+F(x,y,\xi)+F(y,x,\xi)\\[6pt]
\phantom{2\Phi(x,y,z)=+\eta(z)\{}
+F(x,\f y,\xi)+F(y,\f x,\xi)\\[6pt]
\phantom{2\Phi(x,y,z)=+\eta(z)\{}
-\omega(\f x)\eta(y)-\omega(\f y)\eta(x)\}.
\end{array}
\end{equation}

%
%Let the corresponding 1-form of $\Phi_J$ be denote by $f_J$ and be defined as follows
%\[
%f_J(Z)=h^{ij}\Phi_J(E_i,E_j,Z).
%\]
%Using \eqref{PhiJFJ} and \eqref{wtataJ}, we obtain the following relation $f_J=-\widetilde{\ta}_J$.

The corresponding fundamental tensor $\tF$ on $(\MM,\f,\xi,\eta,\widetilde{g})$ is determined by $\tF(x,y,z)=\tg(( \nn_x \f )y,z)$. In \cite{dissMan}, it is given the relation between $F$ and $\tF$ as follows
\begin{subequations}\label{tFF}
\begin{equation}
\begin{split}
  2\tF(x,y,z)=&-F(y,\f z,x)F(\f y,z,x)\\[6pt]
              &-F(z,\f y,x)+F(\f z,y,x)\\[6pt]
  &%\phantom{2\tF(x,y,z)=}
  +\eta(x)\{F(y,z,\xi)+F(\f z,\f y,\xi)\\[6pt]
  &\phantom{+\eta(x)\{}
  +F(z,y,\xi)+F(\f y,\f z,\xi)\}\\[6pt]
  \end{split}
\end{equation}
\begin{equation}
\begin{split}
  &%\phantom{2\tF(x,y,z)=}
  +\eta(y)\{F(x,z,\xi)+F(\f z,\f x,\xi)+F(x,\f z,\xi)\}\\[6pt]
  &%\phantom{2\tF(x,y,z)=}
  +\eta(z)\{F(x,y,\xi)+F(\f y,\f x,\xi)+F(x,\f y,\xi)\}.
\end{split}
\end{equation}
\end{subequations}
%\bigskip\bigskip\bigskip\bigskip\bigskip\bigskip\bigskip\bigskip

Obviously, the special class $\F_0$ is determined by the following equivalent conditions: $F=0$, $\Phi=0$, $\tF=0$ and $\n=\nn$.

The properties of $\nn_x\xi$ when $(\MM,\f,\xi,\eta,\tg)$ is in each of the basic classes are determined in a similar way as in \eqref{Fi:nxi}.

%\newpage

%%%%%%%%%%%%%%%%%%%%%%%%%%%%%%%%%%%%%%%%%%%%%%%%%%%%%%%%%%%%%%%%%%%%%%%%%%%%%%%%%1
\vskip 0.2in \addtocounter{subsection}{1} \setcounter{subsubsection}{0}

\noindent  {\Large\bf \thesubsection. Pair of the Nijenhuis tensors}

\vskip 0.15in
%\subsection{The pair of the Nijenhuis tensors}
%
%\vskip 0.2in \addtocounter{subsubsection}{1}
%
%\noindent  {\Large\bf{\emph{\thesubsubsection. The Nijenhuis tensor $N$}}}%\\[6pt]\vskip2pt}
%
%\vskip 0.15in

%\subsection{The Nijenhuis tensor $N$}

An almost contact structure $(\f,\xi,\eta)$ on $\MM$ is called
\emph{normal} and respectively $(\MM,\f,\xi,\eta)$ is a \emph{normal
almost contact manifold} if the corresponding almost complex
structure $J$ generated on $\MM\times \R$ is integrable (i.e.
a complex manifold) \cite{SaHa}.
An almost contact
structure is normal if and only if the Nijenhuis tensor of
$(\f,\xi,\eta)$ is zero \cite{Blair}.

The Nijenhuis tensor $N$ of the almost contact structure is
defined by
\begin{equation}\label{NN}
N = [\f, \f]+ \D{\eta}\otimes\xi,
\end{equation}
where
\begin{equation}\label{[fifi]}
[\f, \f](x, y)=\left[\f x,\f
y\right]+\f^2\left[x,y\right]-\f\left[\f x,y\right]-\f\left[x,\f
y\right]
\end{equation}
is the Nijenhuis torsion of $\f$ and $\D{\eta}$ is the exterior
derivative of $\eta$.

By analogy with the skew-symmetric Lie brackets $[x,y]=\n_x y-\n_y
x$, let us consider the symmetric braces $\{x,y\}=\n_x y+\n_y x$ given by the same formula as in \eqref{braces}.
%\[
%\begin{split}
%g(\{x,y\},z)&=g(\nabla_xy+\nabla_yx,z)\\
%&=xg(y,z)+yg(x,z)-zg(x,y)-g([y,z],x)+g([z,x],y).
%\end{split}
%\]
Then we introduce the symmetric tensor
\begin{equation}\label{[f,f]}
\{\f, \f\}(x, y)=\{\f x,\f
y\}+\f^2\{x,y\}-\f\{\f x,y\}-\f\{x,\f
 y\}.
\end{equation}
%in correspondence with the skew-symmetric tensor $[\f,\f](x,y)$.
%
Additionally, we use the relation between the Lie derivative of the metric $g$ along
$\xi$ and the covariant derivative of $\eta$
\[
\left(\LLL_{\xi}g\right)(x,y)
 =\left(\n_x\eta\right)(y)+\left(\n_y\eta\right)(x),
\]
as an alternative
 of
 $\D{\eta}(x,y)=\left(\n_x\eta\right)(y)-\left(\n_y\eta\right)(x).$
Then, we give the following
%in%\cite{Mira-Diss,ManIv36}
\begin{dfn}
The (1,2)-tensor $\widehat{N}$ defined by
\begin{equation}\label{S}
\widehat{N} =\{\f,\f\}+\xi\otimes\LLL_{\xi}g,
\end{equation}
is called the \emph{associated Nijenhuis tensor of the almost contact B-metric structure $(\f,\xi,\eta,g)$}.
\end{dfn}

Obviously, $N$ is antisymmetric and $\widehat N$ is symmetric, i.e.
\[
N(x,y)=-N(y,x),\qquad \widehat N(x,y)=\widehat N(y,x).
\]

From \eqref{NN} and \eqref{[fifi]}, using the expressions of the Lie brackets and
$\D{\eta}$, we get the following form of $N$ in terms of the covariant
derivatives with respect to $\n$:
\begin{equation}\label{N}
\begin{split}
N(x,y)=\left(\n_{\f
x}\f\right)y&-\f\left(\n_{x}\f\right)y-\left(\n_{\f y}\f\right)x+\f\left(\n_{y}\f\right)x
\\[6pt]
&+\left(\n_{x}\eta\right)(y)\cdot\xi-\left(\n_{y}\eta\right)(x)\cdot\xi.
\end{split}
\end{equation}

\begin{prop}\label{prop-Nhat=nabla}
The tensor $\widehat{N}$ has the following form in terms of
  $\n\f$ and $\n\eta$:
\begin{equation}\label{Sn}
\begin{split}
\widehat{N}(x,y)=\left(\n_{\f
x}\f\right)y&-\f\left(\n_{x}\f\right)y+\left(\n_{\f y}\f\right)x-\f\left(\n_{y}\f\right)x
\\[6pt]
&+\left(\n_{x}\eta\right)(y)\cdot\xi+\left(\n_{y}\eta\right)(x)\cdot\xi.
\end{split}
\end{equation}
\end{prop}
\begin{proof}
We obtain immediately
\[
\begin{array}{l}
\widehat{N}(x,y)
=\{\f,\f\}(x,y)+\left(\LLL_{\xi}g\right)(x,y)\xi\\[6pt]
\phantom{\widehat{N}(x,y)}
=\{\f x,\f
y\}+\f^2\{x,y\}-\f\{\f x,y\}-\f\{x,\f
y\}\\[6pt]
\phantom{\widehat{N}(x,y)=}
+\left(\n_x\eta\right)(y)\cdot\xi+\left(\n_y\eta\right)(x)\cdot\xi
\\[6pt]
\phantom{\widehat{N}(x,y)} =\n_{\f x}\f y+\n_{\f y}\f x
+\f^2\n_{x}y+\f^2\n_{y}x-\f\n_{\f x}y\\[6pt]
\phantom{\widehat{N}(x,y)=}
-\f\n_{y}\f x-\f\n_{x}\f
y-\f\n_{\f y}x\\[6pt]
\phantom{\widehat{N}(x,y)=}
+\left(\n_{x}\eta\right)(y)\cdot\xi
+\left(\n_{y}\eta\right)(x)\cdot\xi\\[6pt]
\phantom{\widehat{N}(x,y)} %
=\left(\n_{\f x}\f\right)y+\left(\n_{\f
y}\f\right)x-\f\left(\n_{x}\f\right)y-\f\left(\n_{y}\f\right)x
\\[6pt]
\phantom{\widehat{N}(x,y)=}
+\left(\n_{x}\eta\right)(y)\cdot\xi+\left(\n_{y}\eta\right)(x)\cdot\xi,
\end{array}
\]
which completes the proof.
\end{proof}

The corresponding tensors of type (0,3) are denoted by the same
letters by the following way
\[
N(x,y,z)=g(N(x,y),z),\qquad \widehat N(x,y,z)=g(\widehat N(x,y),z).
\]

Then, by virtue of \eqref{N}, \eqref{Sn} and \eqref{F=nfi}, the tensors $N$ and $\widehat N$ are expressed in
terms of $F$ as follows:  %\cite{ManIv36}
\begin{equation}\label{enu}%\label{NF}
\begin{split}
&N(x,y,z) = F(\f x,y,z)-F(x,y,\f z)\\[6pt]
&\phantom{N(x,y,z) =}
-F(\f y,x,z)+F(y,x,\f z)\\[6pt]
&\phantom{N(x,y,z) =}
+\eta(z)\{F(x,\f y,\xi)-F(y,\f x,\xi)\},
\end{split}
\end{equation}
\begin{equation}\label{enhat}
\begin{split}
&\widehat N(x,y,z) = F(\f x,y,z)-F(x,y,\f z)\\[6pt]
&\phantom{\widehat N(x,y,z) =}
+F(\f y,x,z)-F(y,x,\f z)\\[6pt]
&\phantom{\widehat N(x,y,z) =}
+\eta(z)\{F(x,\f y,\xi)+F(y,\f x,\xi)\}.
\end{split}
\end{equation}

Bearing in mind \eqref{str1}, \eqref{str2} and \eqref{F-prop}, from \eqref{enu} and \eqref{enhat}
 we obtain the following properties of the Nijenhuis tensors on
 an arbitrary almost contact B-metric manifold:
\begin{equation}\label{Nfi-prop}%\label{NhatN-prop}
\begin{array}{l}
N(x, \f y,\f z)=N(x, \f^2 y,\f^2 z),\\[6pt]
N(\f x, y,\f z)=N(\f^2 x, y,\f^2 z),\\[6pt]
N(\f x,\f  y, z)=-N(\f^2 x,\f^2 y,z),
\end{array}
\end{equation}
\begin{equation}\label{hatNfi-prop}
\begin{array}{l}
\widehat N(x, \f y,\f z)=\widehat N(x, \f^2 y,\f^2 z),\\[6pt]
\widehat N(\f x, y,\f z)=\widehat N(\f^2 x, y,\f^2 z),\\[6pt]
\widehat N(\f x,\f  y, z)=-\widehat N(\f^2 x,\f^2 y,z),
\end{array}
\end{equation}
\begin{equation}\label{NhatN-prop2}
N(\xi, \f y,\f z)+N(\xi, \f z,\f y)+\widehat N(\xi, \f y,\f z)
+\widehat N(\xi, \f z,\f y)=0.
\end{equation}

%\begin{lem}\label{lem-N}
%The Nijenhuis tensor on an arbitrary almost B-metric manifold has
%the following properties:
%\begin{gather}
%N(\f x, \f y,\f z)=-N(\f^2 x, \f^2 y,\f z)=N(\f x, \f^2 y,\f^2 z),\nonumber\\[6pt]
%N(\f^2 x, \f y,\f z)=N(\f x, \f^2 y,\f z)=-N(\f x, \f y,\f^2
%z).\nonumber
%\end{gather}
%\end{lem}
%\begin{proof}
%Bearing in mind properties \eqref{F-prop} of $F$ and relation
%\eqref{enu}, the equalities in the first line of the statement
%follow. They imply the equalities in the last line by virtue of
%\eqref{str1} and \eqref{str2}.
%\end{proof}

The Nijenhuis tensors $N$ and
$\widehat N$ play a fundamental role in natural connections (i.e. such connections that the tensors of the structure $(\f,\xi,\eta,g)$  are parallel with respect to them) on an almost contact B-metric manifold.
The torsions and the potentials of these connections
are expressed by these two tensors. By this reason
we characterize the classes of the considered manifolds in terms of $N$ and
$\widehat N$.

Taking into account \eqref{enu} and \eqref{Fi},
we compute $N$
for each of the basic classes $\F_i$ $(i=1,2,\dots,11)$ of $\MM=\M$:
\begin{equation}\label{N-1-11:N=0}
\begin{array}{l}
N(x,y)=0,
\qquad\quad\quad\quad\quad\quad\; \MM\in\F_1\oplus\F_2\oplus\F_4\oplus\F_5\oplus\F_6;
\end{array}
\end{equation}
and
\begin{equation}\label{N-1-11:Nne=0}
\begin{array}{ll}
N(x,y)=2\left\{\left(\n_{\f
x}\f\right)y-\f\left(\n_{x}\f\right)y\right\},
\; &\MM\in\F_3;\\[6pt]
N(x,y)=4\left(\n_{x}\eta\right)(y)\cdot \xi ,
\; &\MM\in\F_7;\\[6pt]
N(x,y)=2\left\{\eta(x)\n_{y}\xi-\eta(y)\n_{x}\xi\right\},
\; &\MM\in\F_8\oplus\F_9;\\[6pt]
N(x,y)=-\eta(x)\f\left(\n_{\xi}\f\right)y+\eta(y)\f\left(\n_{\xi}\f\right)x,
 &\MM\in\F_{10};\\[6pt]
N(x,y)=\left\{\eta(x)\om(\f y)-\eta(y)\om(\f x)\right\}, \;
&\MM\in\F_{11}.
\end{array}
\end{equation}

It is known that the class of the normal almost contact B-metric
manifolds, i.e. $N=0$, is
$\F_1\oplus\F_2\oplus\F_4\oplus\F_5\oplus\F_6$.

By virtue of \eqref{Sn} and the form of $F$ in \eqref{Fi}, we establish that
$\widehat{N}$ has the following form on $\MM=\M$ belonging to $\F_i$
$(i=1,2, \dots, 11)$, respectively:
%\begin{subequations}
\begin{equation}\label{hatN-1-11}
\begin{array}{ll}
    \widehat{N}(x,y)=\dfrac{2}{n}\bigl\{g(\f x,\f y)\f \ta^{\sharp}+g(x,\f
y)\ta^{\sharp}\bigr\}, \quad &\MM\in\F_1;\\[6pt]
    \widehat{N}(x,y)=2\left\{\left(\n_{\f
x}\f\right)y-\f\left(\n_{x}\f\right)y\right\}, \quad
&\MM\in\F_2;\\[6pt]
%\end{array}
%\end{equation}
%\begin{equation}
%\begin{array}{ll}
\widehat{N}(x,y)=0,
\quad &\MM\in\F_3\oplus \F_7;\\[6pt]
    \widehat{N}(x,y)=\dfrac{2}{n}\ta(\xi)g(x,\f y) \xi,
\quad &\MM\in\F_4;\\[6pt]
    \widehat{N}(x,y)=-\dfrac{2}{n}\ta^*(\xi)g(\f x,\f y) \xi,
\quad &\MM\in\F_5;\\[6pt]
    \widehat{N}(x,y)=4\left(\n_{x}\eta\right)(y)  \xi,
\quad &\MM\in\F_6;\\[6pt]
    \widehat{N}(x,y)=-2\left\{\eta(x) \n_{y}\xi+\eta(y)
\n_{x}\xi\right\}, \quad &\MM\in\F_8\oplus\F_9;\\[6pt]
    \widehat{N}(x,y)=-\eta(x)\f \left(
\n_{\xi}\f\right)y\\[6pt]
\phantom{\widehat{N}(x,y)=}
-\eta(y)\f \left( \n_{\xi}\f\right)x,
\quad &\MM\in\F_{10};\\[6pt]
    \widehat{N}(x,y)=-2\eta(x)\eta(y)\f \om^{\sharp}\\[6pt]
\phantom{\widehat{N}(x,y)=}
\phantom{}+\left\{\eta(x)\om(\f y)+\eta(y)\om(\f x)\right\}  \xi, \quad
&\MM\in\F_{11},
\end{array}
\end{equation}
%\end{subequations}
where $\ta(z)=g(\ta^{\sharp},z)$ and $\om(z)=g(\om^{\sharp},z)$.

Then, we obtain the truthfulness of the following

\begin{prop}\label{prop-Nhat=0}
The class of the almost contact B-metric manifolds with vanishing
$\widehat{N}$ is $\F_3\oplus\F_7$.
\end{prop}

To characterize almost contact B-metric manifolds we need an expression of $F$ by $N$ and $\N$.

\begin{thm}\label{thm:FNhatN}
Let $\M$ be an almost contact B-metric manifold.
Then the fundamental tensor $F$ is given in terms of the pair of Nijenhuis tensors by the formula
\begin{equation}\label{nabf}%\label{FNhatN}
\begin{split}
F(x,y,z)&=-\frac14\bigl\{N(\f x,y,z)+N(\f x,z,y)\\[6pt]
&\phantom{=-\frac14\bigl\{}
+\widehat N(\f x,y,z)+\widehat N(\f x,z,y)\bigr\}\\[6pt]
&\phantom{=\ }+\frac12\eta(x)\bigl\{N(\xi,y,\f z)+\widehat
N(\xi,y,\f z)\\[6pt]
&\phantom{=\ +\frac12\eta(x)\bigl\{N(\xi,y,\f z)\,}
+\eta(z)\widehat N(\xi,\xi,\f y)\bigr\}.
\end{split}
\end{equation}
\end{thm}
\begin{proof}
Taking the sum of \eqref{enu} and \eqref{enhat}, we obtain
\begin{equation}\label{sum1}
\begin{split}
F(\f x,y,z)-F(x,y,\f z)&=\frac12\bigl\{N(x,y,z)+\widehat
N(x,y,z)\bigr\}\\[6pt]
&\phantom{=}
-\eta(z)F(x,\f y,\xi).
\end{split}
\end{equation}
The  identities  \eqref{F-prop} together with \eqref{str1} imply
\begin{equation}\label{sum2}
F(x,y,\f z)+F(x,z,\f y)=\eta(z)F(x,\f y,\xi)+\eta(y)F(x,\f z,\xi).
\end{equation}
A suitable combination of \eqref{sum1} and \eqref{sum2} yields
\begin{equation}\label{nabff}
\begin{split}
F(\f x,y,z)=\frac14\bigl\{&N(x,y,z)+N(x,z,y)\\[6pt]
&+\widehat
N(x,y,z)+\widehat N(x,z,y)\bigr\}.
\end{split}
\end{equation}
Applying \eqref{str1}, we obtain from \eqref{nabff}
\begin{equation}\label{nabff1}
\begin{split}
F(x,y,z)&=\eta(x)F(\xi,y,z)\\[6pt]
&\phantom{=}-\frac14\bigl\{N(\f x,y,z)+N(\f x,z,y)\\[6pt]
&\phantom{=-\frac14\bigl\{}+\widehat N(\f x,y,z)+\widehat N(\f x,z,y)\bigr\}.
\end{split}
\end{equation}
Set $x=\xi$ and $z\rightarrow \f z$ into \eqref{sum1} and use \eqref{str1} to get
\begin{equation}\label{nabff2}
F(\xi,y,z)=\frac12\bigl\{N(\xi,y,\f z)+\widehat N(\xi,y,\f
z)\bigr\}+\eta(z)F(\xi,\xi,y).
\end{equation}
Finally, set $y=\xi$ into \eqref{nabff2} and use the general
identities $N(\xi,\xi)=F(\xi,\xi,\xi)=0$ to obtain
\begin{equation}\label{fff}F(\xi,\xi,z)=\frac12\widehat N(\xi,\xi,\f z).
\end{equation}
Substitute \eqref{fff} into \eqref{nabff2} and the obtained identity insert
into \eqref{nabff1} to get \eqref{nabf}.
\end{proof}

%In \cite{IvMaMa14}, the tensor $F$ is expressed by the Nijenhuis tensors
%on an arbitrary $\M$ as follows:
%\begin{equation}\label{nabf}
%\begin{split}
%F(x,y,z)&=-\dfrac14\bigl\{N(\f x,y,z)+N(\f x,z,y)\\[6pt]
%&\phantom{=-\dfrac14\bigl\{}
%+\widehat N(\f x,y,z)+\widehat N(\f x,z,y)\bigr\}\\[6pt]
%&\phantom{=\ }+\dfrac12\eta(x)\bigl\{N(\xi,y,\f z)+\widehat
%N(\xi,y,\f z)\\[6pt]
%&
%\phantom{=\ +\dfrac12\eta(x)\bigl\{N(\xi,y,\f z)\,}
%+\eta(z)\widehat N(\xi,\xi,\f y)\bigr\}.
%\end{split}
%\end{equation}

As corollaries, in the cases when $N=0$ or $\widehat N=0$,
the  relation \eqref{nabf} takes the following form, respectively:
\begin{equation*}
\begin{split}
F(x,y,z)=&-\dfrac14\bigl\{\widehat N(\f x,y,z)+\widehat N(\f x,z,y)\bigr\}\\
%\phantom{F(x,y,z)=}
&%\phantom{=}
+\dfrac12\eta(x)\bigl\{\widehat
N(\xi,y,\f z)+\eta(z)\widehat N(\xi,\xi,\f y)\bigr\},\\[6pt]
F(x,y,z)=&-\dfrac14\bigl\{N(\f x,y,z)+N(\f x,z,y)\bigr\}%\\[6pt]
%&%\phantom{=}
+\dfrac12\eta(x)N(\xi,y,\f z).
\end{split}
\end{equation*}

\vspace{20pt}

\begin{center}
$\divideontimes\divideontimes\divideontimes$
\end{center} 

\newpage

\addtocounter{section}{1}\setcounter{subsection}{0}\setcounter{subsubsection}{0}

\setcounter{thm}{0}\setcounter{dfn}{0}\setcounter{equation}{0}

\label{par:can.conn}

 \Large{

\
\\[6pt]
\bigskip

\
\\[6pt]
\bigskip

\lhead{\emph{Chapter I $|$ \S\thesection. Canonical-type connections on almost contact manifolds with
B-metric
}}
%\thispagestyle{empty}

%\noindent  {\Huge\bf \S\thesection. Canonical-type connections \\[12pt]
%\phantom{\S\thesection. }on almost contact manifolds\\[14pt]
%\phantom{\S\thesection. }with B-metric}%\\[6pt]\vskip2pt}

\noindent
\begin{tabular}{r"l}
  %\hline
\hspace{-6pt}{\Huge\bf \S\thesection.}  & {\Huge\bf Canonical-type connections} \\[12pt]
                             & {\Huge\bf on almost contact manifolds} \\[12pt]
                             & {\Huge\bf with B-metric}
  %\hline
\end{tabular}

\vskip 1cm

\begin{quote}
\begin{large}
In the present section, a canonical-type connection on the almost contact manifolds with B-metric is constructed. It is proved that its torsion is
invariant with respect to a subgroup of the general conformal
transformations of the almost contact B-metric structure. The
basic classes of the considered manifolds are characterized in
terms of the torsion of the canonical-type connection.

The main results of this section are published in \cite{ManIv38}.
\end{large}
\end{quote}

%\vskip 0.2in \addtocounter{subsection}{1}
%
%\noindent  {\Large\bf \thesubsection. Introduction}

\vskip 0.15in

%
%In the odd-dimensional case, the additional direction is spanned
%by a vector field $\xi$. Then its dual 1-form $\eta$ determines a
%codimension one distribution $\HC = \ker(\eta)$ endowed with an
%almost complex structure $\f$. Then we have an almost contact
%structure $(\f,\xi,\eta)$. If the almost complex structure is
%equipped with a Hermitian metric then the almost contact
%manifold is called an \emph{almost contact metric manifold}.
%In the case when the restriction
%of the metric on $\HC$ is a Norden metric
%then we deal with an \emph{almost contact manifold with B-metric}
%(or an \emph{almost contact complex Riemannian manifold}). Any
%B-metric as an odd-dimensional counterpart of a Norden metric is a
%pseudo-Riemannian metric of signature $(n+1,n)$.
%

In differential geometry of manifolds with additional tensor structures
there are studied those affine connections which preserve the structure
tensors and the metric, known also as natural connections on the considered manifolds.

Natural connections of canonical type are considered on the almost complex
manifolds with Norden metric in \S3 and \cite{GaMi87,GaGrMi85,Mek09}.
%The connection in \cite{GaGrMi85} is the so-called B-connection,
%which is studied in the class of the locally conformal K\"ahlerian
%manifolds with Norden metric.

Here, we are interested in almost contact B-metric manifolds.
These manifolds are the odd-dimensional
extension of the almost complex manifolds with Nor\-den metric and
the case with indefinite metrics corresponding to the almost
contact metric manifolds.
The geometry of some natural connections on almost contact B-metric manifolds are studied in \cite{ManGri2,Man4,Man31,ManIv39}.

In the present section we consider natural connections of canonical type on the almost contact
manifolds with B-metric.
The section is organized as follows.
In Subsection~\thesection.1 we define a natural connection on an
almost contact manifold with B-metric and we give a necessary and sufficient condition an affine connection to be natural.
In Subsection~\thesection.2 we consider a known natural connection (the $\f$B-connection) on these manifolds and give expressions of its torsion with respect to the pair of Nijenhuis tensors.
In Subsection~\thesection.3 we define a natural connection of canonical type (the $\f$-canonical connection) on an
almost contact manifold with B-metric. We determine the class of
the considered manifolds where the $\f$-canonical connection and the $\f$B-connection coincide.
Then, we consider the group $G$ of the general conformal
transformations of the almost contact B-metric structure and
determine the invariant class of the considered manifolds and a
tensor invariant of the group $G$. Also, we establish that the torsion of the canonical-type
connection is invariant only regarding the subgroup $G_0$ of $G$. Thus, we
characterize the basic classes of the considered manifolds by the
torsion of the canonical-type connection. In the end of this subsection we supply a relevant example.
In Subsection~\thesection.4 we consider a natural connection with totally skew-symmetric torsion (the $\f$KT-connection) and give a necessary and sufficient condition for existence of this connection.
Finally, we establish a linear relation between the three considered connections.

\vskip 0.2in \addtocounter{subsection}{1}

\noindent  {\Large\bf \thesubsection. Natural connection on almost contact B-metric manifold} %\label{sec:fB}

\vskip 0.15in

Let us consider an arbitrary almost contact B-metric manifold $\M$.
%The tangent space $T_pM$ at an arbitrary point $p$ in $\MM$ is a vector space equipped
%with an almost contact B-metric structure.

\begin{dfn}%[\cite{Man31}]\label{defn-natural}
An affine connection $\n^*$ is called a \emph{natural connection} on
$(\MM,\f,\allowbreak\xi,\eta,g)$ if the almost contact
structure $(\f,\xi,\eta)$ and the B-metric $g$ are parallel with
respect to $\n^*$,  i.e.  \[\n^*\f=\n^*\xi=\n^*\eta=\n^*g=0.\]
\end{dfn}
As a corollary, the associated metric $\tg$ is also parallel
with respect to the natural connection $\n^*$ on $\M$, i.e. $\n^*\tg=0$.

%It is easy to establish  (see, e.g. \cite{Man31}) that
%%\begin{prop}[\cite{Man31}]\label{prop-nat-con}
%%
%an affine connection $D$ is a natural connection on an almost contact B-metric manifold if and only if %
%%\begin{equation}\label{D-can}
%% Q(x,y,\f z)-Q(x,\f y,z)=F(x,y,z),\;
%% Q(x,y,z)=-Q(x,z,y).
%%\end{equation}
%%\end{prop}
%%
%\begin{equation}\label{1ab}
%\begin{array}{l}
% Q(x,y,\f z)-Q(x,\f y,z)=F(x,y,z),%\label{1a}
%\qquad\\[6pt]
% Q(x,y,z)=-Q(x,z,y).%\label{1b}
%\end{array}
%\end{equation}

Therefore, an arbitrary
natural connection $\n^*$ on $\M\notin\F_0$ plays the same role as
$\n$ on $\M\in\F_0$. Obviously, $\n^*$ and $\n$ coincide when
$\M\in\F_0$. Because of that, we are interested in natural
connections on $\M\notin\F_0$.

\begin{thm}\label{thm-nat}
An affine connection $\n^*$ is natural on $\M$ if and only if
$\n^*\f=\n^*g=0$.
\end{thm}
\begin{proof}
It is known, that an affine connection $\n^*$ is a natural connection
on $(\MM,\f,\xi,\allowbreak\eta,g)$ if and only if
the following properties for the potential $Q$ of $\n^*$ with respect to $\n$ are valid \cite{Man31}:%
\begin{equation}\label{1ab}
\begin{array}{l}
 Q(x,y,\f z)-Q(x,\f y,z)=F(x,y,z),%\label{1a}
\qquad\\[6pt]
 Q(x,y,z)=-Q(x,z,y).%\label{1b}
\end{array}
\end{equation}

These conditions are equivalent to $\n^*\f=0$ and $\n^*g=0$,
respectively. Moreover, $\n^*\xi=0$ is equivalent to the relation
\[
Q(x,\xi,z)=-F(x,\xi,\f z),
\]
which is  a consequence of the first
equality of \eqref{1ab}. Finally, since
$\eta=\xi\,\lrcorner\,g$, then supposing $\n^*g=0$ we have $\n^*\xi=0$
if and only if $\n^*\eta=0$. Thus, the statement is truthful.
\end{proof}

\vskip 0.2in \addtocounter{subsection}{1}

\noindent  {\Large\bf \thesubsection. $\f$B-connection } %\label{sec:fB}

\vskip 0.15in

In  \cite{ManGri2}, it is introduced a natural connection on $\M$.
%by \eqref{fB0}.
%\begin{equation}\label{fB}
%\begin{split}
%    &\n'_xy=\n_xy+Q'(x,y),\quad \\[6pt]
%    &Q'(x,y)=\dfrac{1}{2}\bigl\{\left(\n_x\f\right)\f
%y+\left(\n_x\eta\right)y\ \xi\bigr\}-\eta(y)\n_x\xi.
%\end{split}
%\end{equation}
In \cite{ManIv37}, this connection %determined by \eqref{fB0}
is called
a \emph{$\f$B-connec\-tion}.
It is studied for all main classes $\F_1$, $\F_4$, $\F_5$, $\F_{11}$ of almost contact B-metric manifolds in  \cite{ManGri1,ManGri2,Man3,Man4,ManIv37} with respect to  properties
of the torsion and the curvature as well as the conformal geometry.
A basic class is called a main class if the fundamental tensor $F$ is expressed explicitly by the metric $g$.
Main classes contain the conformally equivalent manifolds of cosymplectic B-metric manifolds by transformations of $G$.
The restriction of the $\f$B-connection $\n'$ on $\HC$ coincides with the B-connection
$\DDD'$ on the corresponding almost Norden manifold,
given in \eqref{Tb} and studied for the class $\W_1$ in  \cite{GaGrMi87}.

%The following relations for the torsion forms and the Lee forms are valid:
%\[
%\dot{t}^*\circ\f=-\dot{t}\circ\f^2,\qquad
%\dot{t}=\dfrac{1}{2}\bigl\{\ta^*+\ta^*(\xi)\eta\bigr\},\quad
%\dot{t}^*=-\dfrac{1}{2}\bigl\{\ta+\ta(\xi)\eta\bigr\},\quad
%\hat{\dot{t}}=-\om\circ\f.
%\]

%\begin{equation}\label{T0N} %  съмнителна е
%\begin{split}
%T'(x,y,z) &= \dfrac{1}{2}\bigl\{N(x,y,z)+\D\eta(x,y)\eta(z)\bigr\}\\[6pt]
%&-\dfrac{1}{8}\Bigl\{\eta(x)\left[N(\f y,\f z,\xi)-N(\xi,\f y,\f z)
%+\widehat N(\f y,\f z,\xi)+\widehat N(\xi,\f y,\f z)\right]\Bigr\}_{[x\leftrightarrow y]}.
%\end{split}
%\end{equation}

Further, we use \eqref{nabf} and  the %operators $h$ and $v$ on $T_pM$, $p\in M$, giving the
orthonormal decomposition given in \eqref{HHVV}, \eqref{hv} and \eqref{Xhv}.
Then we give the expression of the torsion of the $\f$B-con\-nec\-tion in terms of the pair of Nijenhuis tensors
regarding the horizontal and vertical components of vector fields:
%\begin{subequations}
\begin{equation}\label{T0N}
\begin{split}
T'(x,y,z) =\dfrac{1}{8}\bigl\{&N(x^{\mathrm{h}},y^{\mathrm{h}},z^{\mathrm{h}})+\sx N(x^{\mathrm{h}},y^{\mathrm{h}},z^{\mathrm{h}})\\[6pt]
&+\N(z^{\mathrm{h}},y^{\mathrm{h}},x^{\mathrm{h}})-\N(z^{\mathrm{h}},x^{\mathrm{h}},y^{\mathrm{h}})\bigr\}\\[6pt]
+\dfrac{1}{2}\bigl\{&N(x^{\mathrm{h}},y^{\mathrm{h}},z^{\mathrm{v}})+ N(x^{\mathrm{v}},y^{\mathrm{h}},z^{\mathrm{h}})- N(y^{\mathrm{v}},x^{\mathrm{h}},z^{\mathrm{h}})\\[6pt]
&-\N(z^{\mathrm{v}},x^{\mathrm{v}},y^{\mathrm{h}})+\N(z^{\mathrm{v}},y^{\mathrm{v}},x^{\mathrm{h}})\bigr\}\\[6pt]
+\dfrac{1}{4}\bigl\{
&  N(z^{\mathrm{h}},x^{\mathrm{h}},y^{\mathrm{v}})- N(z^{\mathrm{h}},y^{\mathrm{h}},x^{\mathrm{v}})
+ N(z^{\mathrm{v}},x^{\mathrm{h}},y^{\mathrm{h}})\\[6pt]
&- N(z^{\mathrm{v}},y^{\mathrm{h}},x^{\mathrm{h}})
-\N(z^{\mathrm{h}},x^{\mathrm{h}},y^{\mathrm{v}})+\N(z^{\mathrm{h}},y^{\mathrm{h}},x^{\mathrm{v}})\\[6pt]
%\end{split}
%\end{equation}
%\begin{equation}
%\begin{split}
&-\N(z^{\mathrm{v}},x^{\mathrm{h}},y^{\mathrm{h}})+\N(z^{\mathrm{v}},y^{\mathrm{h}},x^{\mathrm{h}})\bigr\}.
\end{split}
\end{equation}
%\end{subequations}

Taking into account \eqref{fBT=F}, \eqref{T0N} and \eqref{N-1-11:Nne=0},
we obtain for the manifolds from $\F_3\oplus\F_7$ the following
\begin{equation}\label{T0N-F37}
\begin{split}
T'(x,y,z) &=\dfrac{1}{8}\bigl\{N(x^{\mathrm{h}},y^{\mathrm{h}},z^{\mathrm{h}})+\sx N(x^{\mathrm{h}},y^{\mathrm{h}},z^{\mathrm{h}})\bigr\}
\\[6pt]
&+\dfrac{1}{4}\bigl\{N(x^{\mathrm{h}},y^{\mathrm{h}},z^{\mathrm{v}})+\sx N(x^{\mathrm{h}},y^{\mathrm{h}},z^{\mathrm{v}})\bigr\}.
\end{split}
\end{equation}
Therefore, using the notation $N^{\mathrm{h}}(x,y,z)=N(x^{\mathrm{h}},y^{\mathrm{h}},z^{\mathrm{h}})$,
for the basic classes with vanishing $\N$ we have:
\begin{equation}\label{T0N-F3F7}
\begin{split}
&\F_3:\quad T'=\dfrac{1}{8}\bigl\{N^{\mathrm{h}}+\s N^{\mathrm{h}}\bigr\},\qquad\\[6pt]
&\F_7:\quad
T'=\dfrac{1}{2}\bigl\{\D\eta\otimes\eta+\eta\wedge\D\eta\bigr\}.
\end{split}
\end{equation}

On any almost contact manifold with B-metric $\M$, it is introduced in \cite{ManGri2} a natural connection $\n'$,
defined by
\begin{equation}\label{fB0}
    \n'_xy=\n_xy+Q'(x,y),
\end{equation}
where its potential with respect to the Levi-Civita connection $\n$ has the following form
\[
Q'(x,y)=\dfrac{1}{2}\bigl\{\left(\n_x\f\right)\f
y+\left(\n_x\eta\right)y\cdot\xi\bigr\}-\eta(y)\n_x\xi.
\]
Therefore, for the corresponding (0,3)-tensor $Q'(x,y,z)=g(Q'(x,y),z)$ we have
\begin{equation}\label{Q0}
\begin{split}
Q'(x,y,z)=\dfrac{1}{2}\bigl\{F(x,\f y,z)&+\eta(z)F(x,\f
y,\xi)\\[6pt]
&-2\eta(y)F(x,\f z,\xi)\bigr\}.
\end{split}
\end{equation}
The torsion of $\n'$ is expressed by the fundamental tensor $F$ by the following way
\begin{equation}\label{fBT=F}
\begin{array}{l}
T'(x,y,z)=-\dfrac{1}{2}F(x,\f y,\f^2 z)+\dfrac{1}{2}F(y,\f x,\f^2 z)\\[6pt]
\phantom{T'(x,y,z)=}
+\eta(x)F(y,\f z,\xi)-\eta(y)F(x,\f z,\xi)
\\[6pt]
\phantom{T'(x,y,z)=}
+\eta(z)\bigl\{F(x,\f y,\xi)-F(y,\f x,\xi)\bigr\}.
\end{array}
\end{equation}

The torsion forms associated with the torsion $T$ of an arbitrary affine connection are defined as follows:
\begin{equation}\label{ttt}
\begin{array}{l}
t(x)=g^{ij}T(x,e_i,e_j),\qquad \\[6pt]
t^*(x)=g^{ij}T(x,e_i,\f e_j),\qquad \\[6pt]
\hat{t}(x)=T(x,\xi,\xi)
\end{array}
\end{equation}
regarding an arbitrary basis $\left\{e_i;\xi\right\}$ $(i=1,2,\dots,2n)$ of $T_p\MM$. Obviously, $\hat{t}(\xi)=0$ is always valid.

Applying \eqref{titi} and \eqref{ttt}, we have the following relations for the torsion forms of $T'$ and the Lee forms:
\begin{equation}\label{tB}
\begin{array}{l}
    t'=\dfrac{1}{2}\left\{\ta^*+\ta^*(\xi)\eta\right\},\qquad \\[6pt]
    t'^{*}=-\dfrac{1}{2}\left\{\ta+\ta(\xi)\eta\right\},\qquad\\[6pt]
    \hat{t}'=-\om\circ\f.
\end{array}
\end{equation}

The equality \eqref{tata*} and \eqref{tB} imply the
following relation:
\begin{equation}\label{t't'*}
t'^*\circ\f=-t'\circ\f^2.
\end{equation}

%\begin{equation}\label{T0}
%\begin{split}
%T'(x,y,z)=\dfrac{1}{2}\bigl\{F(x,\f y,z)&+\eta(z)F(x,\f y,\xi)\\[6pt]
%&+2\eta(x)F(y,\f z,\xi)\bigr\}_{[x\leftrightarrow y]}.
%\end{split}
%\end{equation}

\vskip 0.2in \addtocounter{subsection}{1}

\noindent  {\Large\bf \thesubsection. $\f$-Canonical connection }

\vskip 0.15in
%\section{}

%If $T$ is the torsion of $D$, i.e. $T(x,y)=D_x y-D_y x-[x, y]$,
%then the corresponding tensor of type (0,3) is determined by
%$T(x,y,z)=g(T(x,y),z)$.
%
%
%Let us denote the difference between the natural connection $D$
%and the Levi-Civita connection $\n$ of $g$ by $Q(x,y)=D_xy-\n_xy$
%and the corresponding tensor of type (0,3) --- by
%$Q(x,y,z)=g\left(Q(x,y),z\right)$.
%

\begin{dfn}\label{defn-canonical}
A natural connection $\n''$ is called a \emph{$\f$-canonical
connection} on the manifold $(\MM,\f,\xi,\allowbreak\eta,g)$ if the
torsion tensor $T''$ of $\n''$ satisfies the following identity
\begin{equation}\label{T-can}
\begin{array}{l}
    T''(x,y,z)-T''(x,\f y,\f z)
    -T''(x,z,y)+T''(x,\f z,\f y)\\[6pt]
    -\eta(x)\bigl\{T''(\xi,y,z)-T''(\xi, \f y,\f z)\\[6pt]
    \phantom{-\eta(x)\bigl\{T''(\xi,y,z)}
                    -T''(\xi,z,y)+T''(\xi, \f z,\f y)\bigr\}\\[6pt]
    -\eta(y)\bigl\{T''(x,\xi,z)-T''(x,z,\xi)-\eta(x)T''(z,\xi,\xi)\bigr\}\\[6pt]
    +\eta(z)\bigl\{T''(x,\xi,y)-T''(x,y,\xi)-\eta(x)T''(y,\xi,\xi)\bigr\}
    =0.
\end{array}
\end{equation}
\end{dfn}

Let us remark that the restriction of the $\f$-canonical
connection $\n''$ on $\M$ to the contact distribution $\HH$ is the
unique canonical connection $\DDD''$ on the corresponding almost complex
manifold with Norden metric, studied in \cite{GaMi87} and \S3.

We construct an affine connection $\n''$ as follows:
\begin{equation}\label{D=nQ}
g(\n''_xy,z)=g(\n_xy,z)+Q''(x,y,z),
\end{equation}
where
\begin{equation}\label{Q-can-F}
\begin{split}
Q''(x,y,z)%=&\\[6pt]
&= Q'(x,y,z)\\[6pt]
&\phantom{=}-\dfrac{1}{8}\left\{N(\f^2 z,\f^2 y,\f^2 x)+2N(\f z,\f
y,\xi)\eta(x)\right\}.
\end{split}
\end{equation}

By direct computations, we check that $\n''$ satisfies conditions
\eqref{1ab} and therefore it is a natural connection on $\M$.
Its torsion is given in the following equality
%The torsion tensor of the the $\f$-canonical connection is
\label{Tcan}
\begin{equation}\label{T-can-F}
\begin{array}{l}
T''(x,y,z)=T'(x,y,z)\\[6pt]
\phantom{T''(x,y,z)=}
-\dfrac{1}{8}\bigl\{
N(\f^2 z,\f^2 y,\f^2 x)-N(\f^2 z,\f^2 x,\f^2 y)\bigr\}\\[6pt]
\phantom{T''(x,y,z)=}
-\dfrac{1}{4}\bigl\{N(\f z,\f y,\xi)\eta(x)-N(\f z,\f x,\xi)\eta(y)\bigr\},
\end{array}
\end{equation}
where $T'$ is the torsion tensor of the $\f$B-connection from
\eqref{fBT=F}.

The relation \eqref{T-can-F} is equivalent to
\begin{equation}\label{T-can-N}
\begin{split}
T''(x,y,z)=\ &T'(x,y,z)\\[6pt]
&+\dfrac{1}{8}\bigl\{N(x^{\mathrm{h}},y^{\mathrm{h}},z^{\mathrm{h}})-\sx N(x^{\mathrm{h}},y^{\mathrm{h}},z^{\mathrm{h}})\bigr\}\\[6pt]
&+\dfrac{1}{4}\bigl\{N(x^{\mathrm{h}},y^{\mathrm{h}},z^{\mathrm{v}})-\sx N(x^{\mathrm{h}},y^{\mathrm{h}},z^{\mathrm{v}})\bigr\}.
\end{split}
\end{equation}

%\begin{equation}\label{T-can-F}
%\begin{split}
%%T''(x,y,z)%=&\\[6pt]
%%&=\dfrac{1}{2}\bigl\{F(x,\f y,z)-F(y,\f x,z)\\[6pt]
%%&+\eta(z)\left[F(x,\f y,\xi)-F(y,\f
%%x,\xi)\right]\bigr\}\\[6pt]
%%&+\eta(x)F(y,\f z,\xi)-\eta(y)F(x,\f z,\xi)\\[6pt]
%%&-\dfrac{1}{8}\left[N(\f^2 z,\f y,\f x)-N(\f^2 z,\f x,\f
%%y)\right]\\[6pt]
%%&-\dfrac{1}{4}\left[\eta(x)N(\f z,\f y,\xi)-\eta(y)N(\f z,\f
%%x,\xi)\right].\\[6pt]
%T''(x,y,z)&=T'(x,y,z)\\[6pt]
%&\phantom{=}
%-\dfrac{1}{8}\left\{N(\f^2 z,\f^2 y,\f^2
%x)+2N(\f z,\f y,\xi)\eta(x)\right\}_{[x\leftrightarrow y]}.
%\end{split}
%\end{equation}

We verify immediately that $T''$ satisfies \eqref{T-can} and thus
$\n''$, determined by \eqref{D=nQ} and \eqref{Q-can-F}, is a
$\f$-canonical connection on $\M$.

The explicit expression \eqref{D=nQ}, supported by \eqref{Q0} and
\eqref{enu}, of the $\f$-canonical connection by the tensor $F$
implies that the $\f$-canonical connection is unique.

Moreover, the torsion forms of the $\f$-canonical connection coincide with those
of the $\f$B-connection $\n'$ given in \eqref{tB}.

Immediately we get the following
\begin{prop}\label{prop-Q=Q0}
A necessary and sufficient condition for the $\f$-ca\-no\-nic\-al
connection to
coincide with the $\f$B-connection %(,  i.e.  $Q=Q_0$,
is $N(\f\cdot,\f\cdot)=0$.
\end{prop}

\begin{lem}\label{lem-Nfifi}
The class $\U_0=
\F_1\oplus\F_2\oplus\F_4\oplus\F_5\oplus\F_6\oplus\F_8\oplus\F_9\oplus\F_{10}\oplus\F_{11}$
of the almost contact B-metric manifolds is determined by the
condition $N(\f\cdot,\f\cdot)=0$.
\end{lem}
\begin{proof}
It follows directly from \eqref{N-1-11:N=0} and \eqref{N-1-11:Nne=0}.
\end{proof}

Thus, \propref{prop-Q=Q0} and \lemref{lem-Nfifi} imply
\begin{cor}\label{cor-Q=Q0-class}
The $\f$-canonical connection and the $\f$B-connection coincide on
an almost con\-tact B-metric manifold $\M$ if and only if $\M$ is
in the class $\U_0$.
\end{cor}

Then, bearing in mind \eqref{T0N}, we obtain that
the torsions of the $\f$-canonical connection and the $\f$B-connection on a
manifold from $\U_0$ have the form
\begin{equation*}\label{T0N=}
\begin{split}
T''(x,y,z) &=T'(x,y,z)\\[6pt]
 &=\dfrac{1}{8}\bigl\{\N(z^{\mathrm{h}},y^{\mathrm{h}},x^{\mathrm{h}})-\N(z^{\mathrm{h}},x^{\mathrm{h}},y^{\mathrm{h}})\bigr\}\\[6pt]
&\phantom{+}
+\dfrac{1}{2}\bigl\{N(x^{\mathrm{v}},y^{\mathrm{h}},z^{\mathrm{h}})-N(y^{\mathrm{v}},x^{\mathrm{h}},z^{\mathrm{h}})\\[6pt]
&\phantom{++\dfrac{1}{2}\bigl\{}
+\N(x^{\mathrm{v}},y^{\mathrm{h}},z^{\mathrm{h}})-\N(y^{\mathrm{v}},x^{\mathrm{h}},z^{\mathrm{h}})\\[6pt]
&\phantom{++\dfrac{1}{2}\bigl\{}
-\N(z^{\mathrm{v}},x^{\mathrm{v}},y^{\mathrm{h}})+\N(z^{\mathrm{v}},y^{\mathrm{v}},x^{\mathrm{h}})\bigr\}\\[6pt]
&\phantom{+}
+\dfrac{1}{4}\bigl\{N(z^{\mathrm{v}},x^{\mathrm{h}},y^{\mathrm{h}})-N(z^{\mathrm{v}},y^{\mathrm{h}},x^{\mathrm{h}})\\[6pt]
&\phantom{++\dfrac{1}{4}\bigl\{}
-\N(z^{\mathrm{v}},x^{\mathrm{h}},y^{\mathrm{h}})+\N(z^{\mathrm{v}},y^{\mathrm{h}},x^{\mathrm{h}})\bigr\}.
\end{split}
\end{equation*}

The torsions $T'$ and $T''$ are different to each other on a
manifold that belongs to the basic classes $\F_3$ and $\F_7$
as well as to their direct sums with other classes.
For $\F_3\oplus\F_7$, using \eqref{T0N-F37} and \eqref{T-can-N},
we obtain the form of the torsion of the $\f$-canonical connection as follows
\begin{equation*}%\label{Tcan-F37}
T''(x,y,z)=\dfrac{1}{4} N(x^{\mathrm{h}},y^{\mathrm{h}},z^{\mathrm{h}})+\dfrac{1}{2} N(x^{\mathrm{h}},y^{\mathrm{h}},z^{\mathrm{v}}).
\end{equation*}
Therefore,
using \eqref{N-1-11:Nne=0}, the torsion of the $\f$-canonical connection for $\F_3$ and $\F_7$ is expressed by
\begin{equation}\label{Tcan-F3F7}
\F_3:\quad T''=\dfrac{1}{4} N^{\mathrm{h}},\qquad \F_7:\quad T''=\D\eta\otimes\eta.
\end{equation}

\vskip 0.2in \addtocounter{subsubsection}{1}

\noindent  {\Large\bf{\emph{\thesubsubsection. $\f$-Canonical connection and general contact conformal group $G$}}}%\\[6pt]\vskip2pt}

\vskip 0.15in

%\section{}

Now we consider the group of transformations of the
$\f$-canonical connection  generated by the general contact conformal transformations of the almost contact B-metric
structure.

%Let $\M$ be an almost contact B-metric manifold.
According to  \cite{Man4}, the general
contact conformal transformations of the almost contact B-metric
structure are defined by
\begin{equation}\label{Transf}
\begin{array}{c}
    \overline{\xi}=e^{-w}\xi,\quad \overline{\eta}=e^{w}\eta,\quad \\[6pt]
    \overline{g}(x,y)=\al g(x,y)+\bt g(x,\f y)+(\gm-\al)\eta(x)\eta(y),
\end{array}
\end{equation}
where $\al=e^{2u}\cos{2v}$, $\bt=e^{2u}\sin{2v}$, $\gm=e^{2w}$ for
differentiable functions $u$, $v$, $w$ on $\MM$. These
transformations form a group denoted by $G$.

If $w=0$, we obtain the contact conformal transformations of the
B-metric, introduced in \cite{ManGri1}. By $v=w=0$, the
transformations \eqref{Transf} are reduced to the usual conformal
transformations of $g$.

Let us remark that $G$ can be considered as a contact complex
conformal gauge group, i.e. the composition of an almost contact
group preserving $\HH$ and a complex conformal transformation of the
complex Riemannian metric $\overline{g^{\C}}=e^{2(u+iv)}g^{\C}$ on
$\HH$.

Note that the normality condition $N=0$ is not preserved under the
action of $G$. We have

\begin{prop}\label{prop:U0invar}
The tensor $N(\f\cdot,\f\cdot)$ is an invariant of the group $G$
on any almost contact B-metric manifold.
\end{prop}
\begin{proof}
Taking into account \eqref{NN} and \eqref{Transf}, we obtain
\[
\overline{N}=N+ (\D w\wedge\eta)\otimes\xi
\]
and clearly we have
$\overline{N}(\f x,\f y)=N(\f x,\f y)$.
\end{proof}

According to \lemref{lem-Nfifi}, we establish  the following
\begin{cor}
The class $\U_0$ is closed by the action of the group $G$.
\end{cor}

Let $\M$ and $(\MM,\f,\overline{\xi},\overline{\eta},\overline{g})$ be contactly
conformally equivalent with respect to a transformation from $G$.
The Levi-Civita connection of $\overline{g}$ is denoted by
 $\overline{\n}$.
Using the general Koszul formula for the metric $g$ and the corresponding Levi-Civita connection $\n$
\begin{equation}\label{koszul}
\begin{split}
2g(\nabla_xy,z)&=x\left(g(y,z)\right)+y\left(g(z,x)\right)-z\left(g(x,y)\right)\\%[4pt]
%\phantom{g(\nabla_xy,z)=\frac12\bigl[}
&+g([x,y],z)-g([y,z],x)+g([z,x],y),
\end{split}
\end{equation}
by straightforward computations we get the following relation
%\begin{lem}\label{lem-barNabla}
between  $\n$ and $\overline{\n}$%, known from \cite{Man-diss}
:
\begin{subequations}\label{bar-n-n}
\begin{equation}
\begin{split}
    &2\left(\al^2+\bt^2\right)g\left(\overline{\n}_xy-\n_xy,z\right)=\\[6pt]
    &=\dfrac{1}{2}\Bigl\{-\al\bt \left[2F(x,y,\f^2z)-F(\f^2z,x,y)\right]\\[6pt]
    &\phantom{=\dfrac{1}{2}\Bigl\{}-\bt^2 \left[2F(x,y,\f z)-F(\f z,x,y)\right]\\[6pt]
    &\phantom{=\dfrac{1}{2}\Bigl\{}+\dfrac{\bt}{\gm}\left(\al^2+\bt^2\right)
    \left[2F(x,y,\xi)-F(\xi,x,y)\right]\eta(z)\\[6pt]
   &\phantom{=\dfrac{1}{2}\Bigl\{}+2\left(\dfrac{\al}{\gm}-1\right)\left(\al^2+\bt^2\right)
    F(\f^2x,\f y,\xi)\eta(z)%\\[6pt]
\end{split}
\end{equation}
\begin{equation}
\begin{split}
    &\phantom{=\dfrac{1}{2}\Bigl\{}+2\al(\gm-\al)
    \left[F(x,\f z,\xi)+F(\f^2z,\f x,\xi)\right]\eta(y)\\[6pt]
    &\phantom{=\dfrac{1}{2}\Bigl\{}-2\bt(\gm-\al)
    \left[F(x,\f^2z,\xi)-F(\f z,\f x,\xi)\right]\eta(y)\\[6pt]
     &\phantom{=\dfrac{1}{2}\Bigl\{}
    -2\left[\al\,\D\al(x)+\bt\,\D\bt(x)\right]g(\f y,\f z)\\[6pt]
    &\phantom{=\dfrac{1}{2}\Bigl\{}
    +2\left[\al\,\D\bt(x)-\bt\,\D\al(x)\right]g(y,\f
    z)\\[6pt]
    &\phantom{=\dfrac{1}{2}\Bigl\{}
    -\left[\al\,\D\al(\f^2z)+\bt\,\D\al(\f z)\right]g(\f x,\f y)\\[6pt]
    &\phantom{=\dfrac{1}{2}\Bigl\{}
    +\left[\al\,\D\bt(\f^2z)+\bt\,\D\bt(\f z)\right]g(x,\f y)\\[6pt]
    &\phantom{=\dfrac{1}{2}\Bigl\{}+\left[\al\,\D\gm(\f^2z)+\bt\,\D\gm(\f z)\right]\eta(x)\eta(y)\\[6pt]
    &\phantom{=\dfrac{1}{2}\Bigl\{}+\dfrac{1}{\gm}(\al^2+\bt^2)\bigl\{\D\al(\xi)g(\f x,\f y)-\D\bt(\xi)g(x,\f y)\bigr\}\eta(z)\\[6pt]
    &\phantom{=\dfrac{1}{2}\Bigl\{}+\dfrac{1}{\gm}(\al^2+\bt^2)\bigl\{2\D\gm(x)\eta(y)-\D\gm(\xi)\eta(x)\eta(y)\bigr\}\eta(z)\Bigr\}_{(x\leftrightarrow
    y)},
\end{split}
\end{equation}
\end{subequations}
where (for the sake of brevity) we use the notation
$\{A(x,y,z)\}_{(x\leftrightarrow y)}$ instead of the sum
$\{A(x,y,z)+A(y,x,z)\}$ for any tensor $A(x,y,z)$.

Using \eqref{F=nfi} and \eqref{bar-n-n}, we obtain the following
formula for the  transformation by $G$ of the tensor $F$:
\begin{equation}\label{barF-F}
\begin{split}
    2\overline{F}(x,y,z)&=2\al F(x,y,z)\\[6pt]
    &%\phantom{=}
    +\Bigl\{\bt \left\{F(\f y,z,x)-F(y,\f z,x)+F(x,\f y,\xi)\eta(z)\right\}\\[6pt]
    &\phantom{+\Bigl\{}+(\gm-\al)\bigl\{\left[F(x,y,\xi)+F(\f y,\f x,\xi)\right]\eta(z)\\[6pt]
    &\phantom{+\Bigl\{+(\gm-\al)\bigl\{}
    +\left[F(y,z,\xi)+F(\f z,\f y,\xi)\right]\eta(x)\bigr\}\\[6pt]
    &\phantom{+\Bigl\{}
    -\left[\D\al(\f y)+\D\bt(y)\right]g(\f x,\f z)\\[6pt]
    &\phantom{+\Bigl\{}
    -\left[\D\al(y)-\D\bt(\f y)\right]g(x,\f z)\\[6pt]
    &\phantom{+\Bigl\{}+\eta(x)\eta(y)\D\gm(\f z)\Bigr\}_{(y\leftrightarrow
    z)}.
\end{split}
\end{equation}

\begin{prop}\label{prop:bar-n'-n'}
Let  the almost contact B-metric manifolds $\M$ and
$(\MM,\f,\overline{\xi},\overline{\eta},\overline{g})$ be contactly conformally
equivalent with respect to a transformation from $G$. Then the
corresponding $\f$-canonical connections $\overline{\n}''$ and $\n''$ as
well as their torsions $\overline{T}''$ and  $T''$ are related as
follows:
\begin{equation}\label{bar-n'-n'}
\begin{split}
    \overline{\n}''_xy&=\n''_xy
    -\D u(x)\f^2 y+\D v(x)\f y+\D w(x)\eta(y)\xi\\[6pt]
    &\phantom{=}+\dfrac{1}{2}\bigl\{\left[\D u(\f^2y)-\D v(\f y)\right]\f^2 x-\left[\D u(\f y)+\D v(\f^2 y)\right]\f
    x\\[6pt]
    &\phantom{=+\dfrac{1}{2}\bigl\{}-g(\f x,\f y)\left[\f^2p-\f q\right]+g(x,\f y)\left[\f p+\f^2
    q\right]\bigr\},
\end{split}
\end{equation}
where $p=\grad{u}$, $q=\grad{v}$;
\begin{equation}\label{T''TP}
\begin{split}
    \overline{T}''(x,y)&=T''(x,y)+\{\D w(x)\eta(y)-\D w(y)\eta(x)\}\xi\\[6pt]
    &\phantom{=}
    -\dfrac{1}{2}\Bigl\{
    \left[\D u(\f^2y)+\D v(\f y)-2\D u(\xi)\eta(y)\right]\f^2 x\\[6pt]
    &\phantom{=+\dfrac{1}{2}\Bigl\{}
    -\left[\D u(\f^2x)+\D v(\f x)-2\D u(\xi)\eta(x)\right]\f^2 y\\[6pt]
    &\phantom{=+\dfrac{1}{2}\Bigl\{}
    +\left[\D u(\f y)-\D v(\f^2 y)+2\D v(\xi)\eta(y)\right]\f x\\[6pt]
    &\phantom{=+\dfrac{1}{2}\Bigl\{}
    -\left[\D u(\f x)-\D v(\f^2 x)+2\D v(\xi)\eta(x)\right]\f y\Bigr\}.
\end{split}
\end{equation}
\end{prop}

\begin{proof}
Taking into account \eqref{fB0}, we have the following equality on
$\M$:
\begin{equation}\label{5}
\begin{split}
g\left(\n'_xy-{\n}_xy,z\right)%=\\[6pt]
=\dfrac{1}{2}\bigl\{{F}(x,\f y,z)&+{F}(x,\f y,\xi)\eta(z)\\[6pt]
&-2{F}(x,\f
z,\xi)\eta(y)\bigr\}.
\end{split}
\end{equation}
Then we can rewrite the corresponding equality on the manifold
$(\MM,\f,\allowbreak{}\overline{\xi},\overline{\eta},\overline{g})$, which is the image of $\M$
by a transformation belonging to $G$:
\begin{equation}\label{6}
\begin{split}
\overline{g}\left(\overline{\n}'_xy-\overline{\n}_xy,z\right)%=\\[6pt]
=\dfrac{1}{2}\bigl\{\overline{F}(x,\f y,z)&+\overline{F}(x,\f
y,\overline{\xi})\overline{\eta}(z)\\[6pt]
&-2\overline{F}(x,\f
z,\overline{\xi})\overline{\eta}(y)\bigr\}.
\end{split}
\end{equation}

By virtue of \eqref{5}, \eqref{6}, \eqref{barF-F} and
\eqref{bar-n-n}, we get the following formula of the
transformation by $G$ of the $\f$B-connection:
\begin{subequations}\label{bar-n0-n0}
\begin{equation}
\begin{split}
    g\left(\overline{\n}'_xy-\n'_xy,z\right)&=%\\[6pt]
    \dfrac{1}{8}\sin{4v}\, N(\f z,\f  y,\f x)\\[6pt]
    &-\dfrac{1}{4}\sin^2{2v}\, N(\f^2z,\f^2 y,\f^2 x)%\\[6pt]
\end{split}
\end{equation}
\begin{equation}
\begin{split}
\phantom{g\left(\overline{\n}'_xy-\n'_xy,z\right)}
     &-\dfrac{1}{4}e^{2(w-u)}\sin{2v}\,N(\f^2 z,\f y,\xi)\eta(x)\\[6pt]
    &-\dfrac{1}{4}\left(1-e^{2(w-u)}\cos{2v}\right)N(\f z,\f y,\xi)\eta(x)\\[6pt]
&-\D u(x)g(\f y,\f z)+\D v(x)g(y,\f z)\\[6pt]
    &+\D w(x)\eta(y)\eta(z)\\[6pt]
    &+\dfrac{1}{2}\left[\D u(\f^2y)-\D v(\f y)\right]g(\f x,\f z)\\[6pt]
    &-\dfrac{1}{2}\left[\D u(\f y)+\D v(\f^2 y)\right]g( x,\f z)\\[6pt]
    &-\dfrac{1}{2}\left[\D u(\f^2z)-\D v(\f z)\right]g(\f x,\f y)\\[6pt]
    &+\dfrac{1}{2}\left[\D u(\f z)+\D v(\f^2 z)\right]g( x,\f y).
\end{split}
\end{equation}
\end{subequations}

From \eqref{enu}, \eqref{barF-F}, \eqref{F-prop} and
\eqref{Transf}, it follows the formula for the transformation by
$G$ of the Nijenhuis tensor:
\begin{equation}\label{barN-N}
\begin{split}
    \overline{N}(\f x,\f y,z)=\al\, N(\f x,\f y,z)&+\bt\, N(\f x,\f y,\f z)\\[6pt]
    &+(\gm-\al) N(\f x,\f
    y,\xi)\eta(z).
\end{split}
\end{equation}

Taking into account \eqref{Q-can-F}, \eqref{barN-N},
\eqref{Transf} and \eqref{bar-n0-n0}, we get \eqref{bar-n'-n'}.
As a consequence of \eqref{bar-n'-n'}, the torsions $T''$ and
$\overline{T}''$ of  $\n''$ and $\overline{\n}''$, respectively, are related as
in \eqref{T''TP}.
\end{proof}

The torsion forms associated with $T''$ of the $\f$-canonical
connection are defined by the same
way as in  \eqref{ttt}.
%:
%\begin{equation}\label{t}
%\begin{array}{c}
%t'(x)=g^{ij}T''(x,e_i,e_j),\quad
%t'^*(x)=g^{ij}T''(x,e_i,\f e_j),\quad \\[6pt]
%\hat{t}'(x)=T''(x,\xi,\xi).
%\end{array}
%\end{equation}
%Obviously, $\hat{t}(\xi)=0$ is always valid.

Using \eqref{ttt}, \eqref{T-can-F}, \eqref{fBT=F}, \eqref{F-prop} and
\eqref{Nfi-prop}, we obtain that the torsion forms of the
$\f$-canonical connection are expressed with respect to the Lee forms
by the same way as in \eqref{tB} for the torsion forms of the $\f$B-connection, namely:
\begin{equation}\label{TD-sledi}
\begin{array}{l}
t''=\dfrac{1}{2}\bigl\{\ta^*+\ta^*(\xi)\eta\bigr\},\\[6pt]
t''^*=-\dfrac{1}{2}\bigl\{\ta+\ta(\xi)\eta\bigr\},\\[6pt]
\hat{t}''=-\om\circ\f.
\end{array}
\end{equation}

\newpage
\vskip 0.15in \addtocounter{subsubsection}{1}

\noindent  {\Large\bf{\emph{\thesubsubsection. $\f$-Canonical connection and general contact conformal subgroup $G_0$}}}%\\[6pt]\vskip2pt}

\vskip 0.1in

%\section{}

Let us consider the subgroup $G_0$ of $G$ defined by the
conditions
\begin{equation}\label{G0}
\begin{split}
    \D u\circ\f^2+\D v\circ\f&=\D u\circ\f -\D v\circ\f^2 \\[6pt]
    &=\D u(\xi)=\D v(\xi)=\D w\circ\f=0.
\end{split}
\end{equation}
By direct computations, from  \eqref{Fi}, \eqref{Transf},
\eqref{barF-F} and \eqref{G0}, we prove the truthfulness of the following
\begin{thm}\label{thm:FiG0invar}
Each of the basic classes $\F_i$ $(i=1,2,\dots,11)$ of the almost contact B-metric
manifolds is closed by the action of the group $G_0$. Moreover, $G_0$ is
the largest subgroup of $G$ preserving the Lee forms $\ta$, $\ta^*$,
$\om$ and the special class $\F_0$.
\end{thm}

 %In \cite{Man4}, it is proved
%that the Bochner-type curvature tensors for $\n$ and for $\n'$ on
%a manifold from $\F_0$ are invariant with respect to the
%transformations of the group $G_0$.

\begin{thm}\label{thm:TcanG0invar}
The torsion of the $\f$-canonical connection is invariant with
respect to the general contact conformal transformations if and
only if these transformations belong to the group $G_0$.
\end{thm}
\begin{proof}
\propref{prop:bar-n'-n'} and \eqref{G0} imply immediately
\begin{equation}\label{bar-n'-n'-G0}
\begin{split}
    \overline{\n}''_xy=\n''_xy
    &-\D u(x)\f^2 y+\D v(x)\f y+\D w(\xi)\eta(x)\eta(y)\xi\\[6pt]
    &-\D u(y)\f^2 x+\D v(y)\f
    x+g(\f x,\f y)p-g(x,\f y)q.
\end{split}
\end{equation}
The statement follows from \eqref{bar-n'-n'-G0} or alternatively
from \eqref{T''TP} and \eqref{G0}.
\end{proof}

Bearing in mind the invariance of $\F_i$ $(i=1,2,\dots,11)$ and
$T''$ with respect to the transformations of $G_0$, we establish
that each of the eleven basic classes of the manifolds $\M$ is
characterized by the torsion of the $\f$-canonical connection.
Then we give this characterization in the following
%\subsection{Characterization of the Classes of Almost Contact
%B-Metric Manifolds by the Torsion of the $\f$-Canonical
%Connection}
\begin{prop}\label{prop:FiT}
The basic classes of the almost contact B-metric manifolds are
characterized by conditions for the torsion of the $\f$-canonical
connection as follows:
\[
\begin{array}{rl}
\F_1:\; &T''(x,y)=\dfrac{1}{2n}\left\{t''(\f^2 x)\f^2 y-t''(\f^2 y)\f^2 x\right.\\[6pt]
            &\phantom{T''(x,y)=\dfrac{1}{2n}\ }\left.
            +t''(\f x)\f y-t''(\f y)\f x\right\}; \\[6pt]
\F_2:\; &T''(\xi,y)=0,\quad \eta\left(T''(x,y)\right)=0,\quad \\[6pt]
        &T''(x,y)=T''(\f x,\f y),\quad t''=0;%\\[6pt]
\end{array}
\]
\[
\begin{array}{rl}
\F_3:\; &T''(\xi,y)=0,\quad \eta\left(T''(x,y)\right)=0,\; \\[6pt]
        &T''(x,y)=\f T''(x,\f y);\\[6pt]
\F_4:\; &T''(x,y)=\dfrac{1}{2n}t''^*(\xi)\left\{\eta(y)\f x-\eta(x)\f y\right\};\\[6pt]
\F_5:\; &T''(x,y)=\dfrac{1}{2n}t''(\xi)\left\{\eta(y)\f^2 x-\eta(x)\f^2 y\right\};\\[6pt]
\F_6:\; &T''(x,y)=\eta(x)T''(\xi,y)-\eta(y)T''(\xi,x),\;\\[6pt]
        &T''(\xi,y,z)=T''(\xi,z,y)=-T''(\xi,\f y,\f z);\\[6pt]
\F_{7}:\; &T''(x,y)=\eta(x)T''(\xi,y)-\eta(y)T''(\xi,x)+\eta(T''(x,y))\xi,\\[6pt]
            &T''(\xi,y,z)=-T''(\xi,z,y)=- T''(\xi,\f y,\f z)\\[6pt]
            &\phantom{T''(\xi,y,z)}
            =\dfrac{1}{2}T''(y,z,\xi)=- \dfrac{1}{2}T''(\f y,\f z,\xi);\\[6pt]
\F_{8}:\; &T''(x,y)=\eta(x)T''(\xi,y)-\eta(y)T''(\xi,x)+\eta(T''(x,y))\xi,\\[6pt]
            &T''(\xi,y,z)=-T''(\xi,z,y)= T''(\xi,\f y,\f z)\\[6pt]
            &\phantom{T''(\xi,y,z)}
            =\dfrac{1}{2}T''(y,z,\xi)= \dfrac{1}{2}T''(\f y,\f z,\xi);\\[6pt]
\F_{9}:\; &T''(x,y)=\eta(x)T''(\xi,y)-\eta(y)T''(\xi,x),\; \\[6pt]
            &T''(\xi,y,z)= T''(\xi,z,y)=T''(\xi,\f y,\f z);\\[6pt]
\F_{10}:\; &T''(x,y)=\eta(x)T''(\xi,y)-\eta(y)T''(\xi,x),\; \\[6pt]
            &T''(\xi,y,z)=- T''(\xi,z,y)=T''(\xi,\f y,\f z);\\[6pt]
\F_{11}:\; &T''(x,y)=\left\{\hat{t''}(x)\eta(y)-\hat{t''}(y)\eta(x)\right\}\xi.\\[6pt]
\end{array}
\]
\end{prop}
\begin{proof}
According to \propref{prop-Q=Q0}, \corref{cor-Q=Q0-class}, equalities
\eqref{fBT=F} and \eqref{T-can-F}, we have the following form of the
torsion of the $\f$-canonical connection when $\M$ belongs to $\F_i$ for $i\in\{1,2,\dots,11\}$; $i\neq 3,7$:
\begin{equation*}\label{TD}
\begin{split}
T''(x,y)=T'(x,y)=&\dfrac{1}{2}\bigl\{
 \left(\n_x\f\right)\f y+\left(\n_x\eta\right)y\cdot\xi+2\eta(x)\n_y\xi\\[6pt]
 &\phantom{\dfrac{1}{2}\bigl\{}
-\left(\n_y\f\right)\f x+\left(\n_y\eta\right)x\cdot\xi+2\eta(y)\n_x\xi\bigr\}.
\end{split}
\end{equation*}
For the classes $\F_3$ and $\F_7$, we use \eqref{T-can-F} and
equalities \eqref{N-1-11:N=0} and \eqref{N-1-11:Nne=0}.

Then, using \eqref{F-prop}, \eqref{TD-sledi}, \eqref{t't'*} and
\eqref{Fi}, we obtain the characteristics in the statement.
\end{proof}

%
%%\begin{definition}[ \cite{ManIv38}]\label{defn-canonical}
%A natural connection $\n''$ is called a \emph{$\f$-canonical
%connection} on $(M,\f,\xi,\allowbreak\eta,g)$ if its torsion
%$T''$ satisfies the following identity: \cite{ManIv38}
%\begin{equation*}%\label{T-can}
%\begin{array}{l}
%    T''(x,y,z)-T''(x,\f y,\f z)-T''(x,z,y)+T''(x,\f z,\f y)\\[6pt]
%    -\eta(x)\bigl\{T''(\xi,y,z)-T''(\xi,z,y)
%    -T''(\xi, \f y,\f z)+T''(\xi, \f z,\f y)\bigr\}\\[6pt]
%    -\eta(y)\bigl\{T''(x,\xi,z)-T''(x,z,\xi)-\eta(x)T''(z,\xi,\xi)\bigr\}\\[6pt]
%    +\eta(z)\bigl\{T''(x,\xi,y)-T''(x,y,\xi)-\eta(x)T''(y,\xi,\xi)\bigr\}
%    =0.
%\end{array}
%\end{equation*}
%%\end{definition}

\newpage
\vskip 0.15in \addtocounter{subsubsection}{1}

\noindent  {\Large\bf\emph{\thesubsubsection. An example of an almost contact B-metric
manifold with coinciding $\f$B-connection and $\f$-canonical connection}}\label{exa:sphera}

\vskip 0.1in

%\section{An example of an almost contact B-metric manifold }

In \cite{GaMiGr}, it is given an example of the considered
manifolds as follows.
Let the vector space
\[
\R^{2n+2}=\left\{\left(u^1,\dots,u^{n+1};v^1,\dots,v^{n+1}\right)\
|\ u^i,v^i\in\R\right\}
\]
 be considered as a complex Riemannian
manifold with the canonical complex structure $J$ and the metric
$g$ defined by
\[
g(x,x)=-\delta_{ij}\lm^i\lm^j+\delta_{ij}\mu^i\mu^j
\]
 for
$x=\lm^i\dfrac{\partial}{\partial
u^i}+\mu^i\dfrac{\partial}{\partial v^i}$. Identifying the point
$p\in\R^{2n+2}$ with its position vector, it is considered the
time-like sphere
\[
\SSS: g(U,U)=-1
\]
of $g$ in $\R^{2n+2}$, where $U$
is the unit normal to the tangent space $T_p\SSS$ at $p\in \SSS$. It is
set
\[
g(U,JU)=\tan\psi,\quad
\psi\in\left(-\dfrac{\pi}{2},\dfrac{\pi}{2}\right).
\]
Then the
almost contact structure is introduced by
\[
\xi=\sin\psi\
U+\cos\psi\ JU,\quad \eta=g(\cdot,\xi),\quad \f=J-\eta\otimes J\xi.
\]
 It
is shown that $(\SSS,\f,\xi,\eta,g)$ is an almost contact B-metric
manifold in the class $\F_4\oplus\F_5$.

Since the $\f$-canonical connection coincides with the
$\f$B-connection on any manifold in $\F_4\oplus\F_5$, according to
\corref{cor-Q=Q0-class}, then by virtue of \eqref{fBT=F} we get the
torsion tensor and the torsion forms of the $\f$-canonical
connection as follows:
\[
\begin{array}{l}
T''(x,y,z)=\cos\psi\ \{\eta(x)g(y,\f z)-\eta(y)g(x,\f z)\}\\[6pt]
\phantom{T''(x,y,z)\,}
-\sin\psi\ \{\eta(x) g(\f y,\f z)-\eta(y)g(\f x,\f z)\},\\[6pt]
t''=2n\sin\psi\ \eta,\qquad t''^*=-2n\cos\psi\ \eta,\qquad
\hat{t}''=0.
\end{array}
\]
These equalities are in accordance with \propref{prop:FiT}.

%%%%%%%%%%%%%%%%%%%%%%%%%%%%%%%%%%%%%%%%%%%%%%%%%%%%%%%%%%%%%%%%%%%%%%%%%%%%%%%%%1
\vskip 0.2in \addtocounter{subsection}{1} \setcounter{subsubsection}{0}

\noindent  {\Large\bf \thesubsection. $\f$KT-connection}

\vskip 0.15in

%\subsection{The $\f$KT-connection }

In  \cite{Man31}, on an almost contact B-metric manifold $\M$, it is introduced a natural connection $\n'''$
called a \emph{$\f$KT-connec\-tion}, which torsion $T'''$ is
totally skew-symmetric, i.e. a 3-form. There, it is proved that the
$\f$KT-connection exists only on
$\M$ belonging to $\F_3\oplus\F_7$, i.e. the considered manifold has a Killing vector field $\xi$ and a vanishing cyclic
sum $\s$ of $F$.

\begin{cor}\label{cor:fKTassN=0}
The $\f$KT-connection exists on an almost contact B-metric
manifold if and only if the tensor $\widehat{N}$ vanishes on this manifold.
\end{cor}
\begin{proof}
According to
\propref{prop-Nhat=0}, the class $\F_3\oplus\F_7$ is characterized
by the condition $\widehat{N}=0$. Bearing in mind the statement above, the proof is completed.
\end{proof}

The $\f$KT-connection is the
odd-dimensional analogue of the KT-con\-nec\-tion $\DDD'''$ discussed in Subsection 3.3 on the corresponding class of quasi-K\"ahler manifolds with Norden metric.

According to \cite{Man31}, the unique $\f$KT-connection $\n'''$ is determined by
    \[
            g(\n'''_xy,z)=g(\n_xy,z)+\dfrac{1}{2}T'''(x,y,z),
    \]
    where the torsion tensor is defined by
\begin{equation}\label{fKT-T37} %
\begin{array}{l}
T'''(x,y,z)=-\dfrac{1}{2} \sx\bigl\{F(x,y,\f z)-3\eta(x)F(y,\f
z,\xi)\bigr\}\\[6pt]
\phantom{T'''(x,y,z)}
=\dfrac{1}{2}\left(\eta\wedge
\D\eta\right)(x,y,z)+\dfrac{1}{4}\sx
N(x,y,z). %
\end{array}
\end{equation} %
Obviously, the torsion forms of the $\f$KT-connection are zero.

%\begin{equation}\label{T37} %
%\begin{split}
%T'''(x,y,z)&=-\dfrac{1}{2} \sx\bigl\{F(x,y,\f z)-3\eta(x)F(y,\f
%z,\xi)\bigr\}\\[6pt]
%&=\dfrac{1}{4}\sx
%N(x,y,z)+\dfrac{1}{2}\left(\eta\wedge
%\D\eta\right)(x,y,z). %
%\end{split}
%\end{equation} %
%Obviously, the torsion forms of the $\f$KT-connection are zero.

%$T\in\T_{3}\oplus\T_{6}\oplus\T_{7}\oplus\T_{12}$, but
%\texttt{The torsion $T'''$ of the $\f$KT-connection belongs to \(
%\T_{3}\oplus\T_{6}\oplus\T_{7}\oplus\T_{12}\), according to \cite{ManIv36}.}
%More precisely, if
%$M\in\F_3$ (resp. $\F_7$) then $T'''\in\T_3\oplus\T_6$ (resp.
%$\T_7\oplus\T_{12}$).

From \eqref{fKT-T37} and \eqref{N-1-11:Nne=0}, for the classes  $\F_3$ and $\F_7$ we obtain
\begin{equation}\label{TKT-F3F7} %
\F_3:\quad T'''=\dfrac{1}{4} \s
N^{\mathrm{h}},\qquad
\F_7:\quad T'''=\eta\wedge
\D\eta. %
\end{equation} %

As it is stated in \corref{cor-Q=Q0-class}, the $\f$B-connection and the
$\f$-canonical con\-nec\-tion  coincide if and
only if $\M$ belongs to $\F_i$,
$i\in\{1,2,\allowbreak{}\dots,11\}\setminus \{3,7\}$, i.e. we have $\n'\equiv\n''$ if and only if the
$\f$KT-connection $\n'''$ does not exist.

For the rest basic classes $\F_3$ and $\F_7$ (where the $\f$KT-connection exists), we obtain
\begin{prop}\label{prop-3D}
Let $\M$ be an arbitrary manifold belonging to $\F_i$, $i\in\{3,7\}$. The
$\f$B-con\-nec\-tion $\n'$ is the average connection of the $\f$-canonical con\-nec\-tion
$\n''$ and the
$\f$KT-connection $\n'''$, i.e. the following relation is valid
\[
\n'=\dfrac12\left\{\n''+\n'''\right\}.
\]
\end{prop}
\begin{proof}
 By virtue of \eqref{pijT-B}, \eqref{pijT-KT} and
\eqref{T-can-F}
we obtain: \\[6pt]
1) for $\F_3$
\begin{align*}
&p_{1,2}(T')(x,y,z)=p_{1,2}(T'')(x,y,z)
=p_{1,2}(T''')(x,y,z)\\[6pt]
&
\phantom{p_{1,2}(T')(x,y,z)}
=-\dfrac{1}{2}\bigl\{F(\f^2x,\f^2y,\f z)+F(\f^2y,\f^2z,\f
x)
\\[6pt]
&
\phantom{p_{1,2}(T')(x,y,z)=-\dfrac{1}{2}\bigl\{}
-F(\f^2z,\f^2x,\f y)\bigr\},\\[6pt]%
&p_{1,4}(T')(x,y,z)=
\dfrac{1}{2}p_{1,4}(T''')(x,y,z)=-\dfrac{1}{2}F(\f^2z,\f^2x,\f y),\\[6pt]%
&p_{1,4}(T'')(x,y,z)=0;
\end{align*}
2) for $\F_7$
\begin{align*}
&p_{2,1}(T')(x,y,z)=p_{2,1}(T'')(x,y,z)=p_{2,1}(T''')(x,y,z)\\[6pt]%
&
\phantom{p_{2,1}(T')(x,y,z)}
= 2\eta(z)F(x,\f y,\xi),\\[6pt]%
&p_{3,2}(T')(x,y,z)=\dfrac{1}{2}p_{3,2}(T''')(x,y,z)\\[6pt]%
&
\phantom{p_{3,2}(T')(x,y,z)}
=\eta(x)F(y,\f z,\xi)-\eta(y)F(x,\f z,\xi),\\[6pt]
&p_{3,2}(T'')(x,y,z)=0.
\end{align*}
Thus, we establish the equality  $2T'=T''+T'''$ for
$\F_3$ and $\F_7$. Then, using \eqref{QT}, we obtain the expression
$2Q'=Q''+Q'''$ for the corresponding potentials with respect to $\n$, defined by
\[
\begin{split}
Q'(x,y,z)&=g(\n'_xy-\n_xy,z), \quad\\[6pt]
Q''(x,y,z)&=g(\n''_xy-\n_xy,z), \quad \\[6pt]
Q'''(x,y,z)&=g(\n'''_xy-\n_xy,z).
\end{split}
\]
Therefore, we have the
statement.
\end{proof}

\vspace{20pt}

\begin{center}
$\divideontimes\divideontimes\divideontimes$
\end{center} 

\newpage

\addtocounter{section}{1}\setcounter{subsection}{0}\setcounter{subsubsection}{0}

\setcounter{thm}{0}\setcounter{dfn}{0}\setcounter{equation}{0}

\label{par:classT}

 \Large{

\
\\[6pt]
\bigskip

\
\\[6pt]
\bigskip

\lhead{\emph{Chapter I $|$ \S\thesection. Classification of affine connections on almost contact
manifolds with B-metric  %\label{sec:36}
}}
%\thispagestyle{empty}

%\noindent  {\Huge\bf \S\thesection. A classification of affine \\[12pt]
%\phantom{\S\thesection. }connections on almost contact\\[12pt]
%\phantom{\S\thesection. }manifolds with B-metric }%\\[6pt]\vskip2pt}

\noindent
\begin{tabular}{r"l}
  %\hline
\hspace{-6pt}{\Huge\bf \S\thesection.}  & {\Huge\bf Classification of affine connections } \\[12pt]
                             & {\Huge\bf on almost contact manifolds} \\[12pt]
                             & {\Huge\bf with B-metric}
  %\hline
\end{tabular}

\vskip 1cm

\begin{quote}
\begin{large}
In the present section the space of the torsion (0,3)-tensors of the affine connections on almost contact manifolds with B-metric is
decomposed in 15 orthogonal and invariant subspaces with respect
to the action of the structure group. This decomposition gives a rise to a classification of the corresponding affine connections. Three known connections,
preserving the structure, are characterized regarding this
classification.

The main results of this section are published in \cite{ManIv36}.
\end{large}
\end{quote}

%
%\vskip 0.2in \addtocounter{subsection}{1}
%
%\noindent  {\Large\bf \thesubsection. Introduction}

\vskip 0.15in

The investigations of affine connections on
manifolds take a central place in the study of the
differential geometry of these manifolds. The affine connections
preserving the metric are completely characterized by their
torsion tensors. In ac\-cord\-ance with our goals, it is important
to describe affine connections regarding the properties of their
torsion tensors with respect to the structures on the manifold.
Such a classification of the space of the torsion tensors is made
in \cite{GaMi87} by G.~Ganchev and V.~Mihova in the case of almost complex manifolds with
Norden metric.
%These manifolds are the even-dimensional counterpart of the odd-dimensional almost contact manifolds.

The idea of decomposition of the space of the basic (0,3)-tensors,
generated by the covariant derivative of the fundamental tensor of
type $(1,1)$, is used by different authors in order to  obtain
classifications of manifolds with additional tensor structures.
For example, let us mention the classification of almost Hermitian
manifolds given in \cite{GrHe}, of almost complex manifolds with
Norden metric -- in \cite{GaBo}, of almost contact metric
manifolds -- in \cite{AlGa}, of almost contact manifolds with
B-metric -- in \cite{GaMiGr}, of Riemannian almost product
manifolds -- in \cite{Nav}, of Riemannian manifolds with traceless
almost product structure -- in \cite{StaGri}, of almost
paracontact metric manifolds -- in \cite{NakZam}, of almost
paracontact Riemannian manifolds of type $(n,n)$ -- in
\cite{ManSta01}.

The affine connections preserving the structure (also known as
natural connections) are particularly interesting in differential
geometry.
Canonical Hermitian connections on almost Hermitian manifolds are discussed in the beginning of \S3. %\ref{par:conn}.
%In
%\cite{Fri-Iv2} and \cite{Fri-Iv} all almost Hermitian and almost
%contact metric structures admitting a connection with totally
%skew-symmetric torsion tensor are described.                          povtarya se

%
Natural connections of canonical type are considered on the
Riemannian almost product manifolds in
\cite{Dobr11-1,Dobr11-2,MekDobr} and on the almost complex
manifolds with Norden metric in \cite{GaMi87,GaGrMi85,Mek09}.
%The connection in \cite{GaGrMi} is the so-called B-connection,
%which is studied in the class of the locally conformal K\"ahlerian
%manifolds with Norden metric.
The Tanaka-Webster connection on a contact metric manifold is
introduced (\cite{Tanno,Tan,Web}) in the context of CR-geometry. A
natural connection with minimal torsion on the quaternionic
contact structures, introduced in \cite{Biq}, is known as the
Biquard connection.
%

%\bigskip

The goal of the present section is to describe the torsion space with
respect to the almost contact B-metric structure, which can be
used to study some natural connections on these manifolds.

This section is organized as follows.
%In Section~\thesection.1, we present some necessary facts about the considered manifolds.
Subsection~\thesection.1 is devoted to the decomposition of the space of
torsion tensors on almost contact manifolds with B-metric. On this basis, in Subsection~\thesection.2,
we classify all affine connections on the considered manifolds.
In
Subsection~\thesection.3, we find the position of three known natural
connections from \S5 in the obtained classification.

%\begin{convention}
%\begin{enumerate}
%    \item[\emph{(a)}] We shall use $X$, $Y$, $Z$ to denote elements
%of the algebra $\X(\MM)$ on the smooth vector fields on $\MM$.
%Moreover, $x$, $y$, $z$ will stand for arbitrary vectors in the
%tangent space $T_p\MM$ of $\MM$ at an arbitrary point $p$ in $\MM$;
%    \item[\emph{(b)}]
%    The notation $\sx$ means
%the cyclic sum by the three arguments $x$, $y$, $z$. For example,
%$\sx F(x,y,z)=F(x,y,z)+F(y,z,x)+F(z,x,y)$;
%\end{enumerate}
%\end{convention}

\vskip 0.2in \addtocounter{subsection}{1} \setcounter{subsubsection}{0}

\noindent  {\Large\bf \thesubsection. A decomposition of the space of torsion tensors}

\vskip 0.15in

%\section{A Decomposition of the Space of Torsion Tensors}\label{sec:2}

The object of our considerations are the affine connections with
torsion. Thus, we have to study the properties of the torsion
tensors with respect to the almost contact structure and the B-metric.

If $T$ is the torsion tensor of an affine connection $\n^*$, i.e.
\[
T(x,y)=\n^*_x y-\n^*_y x-[x,
y],
\]
then the corre\-sponding tensor of type (0,3) is determined as usually
by $T(x,y,z)=g(T(x,y),z)$.

Let us consider $T_p\MM$ at arbitrary $p\in \MM$ as a
$(2n+1)$-dimension\-al vector space with almost contact B-metric
structure $(V,\f,\xi,\eta,g)$. Moreover, let $\T$ be the vector
space of all tensors $T$ of type (0,3) over $V$ having
skew-symmetry by the first two arguments, i.e. \[
\T=\left\{T(x,y,z)\in\R,\; x,y,z\in V \; \vert\;\;
T(x,y,z)=-T(y,x,z)\right\}. \]

The metric $g$ induces an inner product
$\langle\cdot,\cdot\rangle$ on $\T$ defined by
\[
\langle
T_1,T_2\rangle=g^{iq}g^{jr}g^{ks}\allowbreak{}
T_1(e_i,e_j,e_k)\allowbreak{}T_2(e_q,e_r,e_s)
\]
 for any
$T_1, T_2\in\T$ and a basis $\left\{e_i\right\}$
$(i=1,2,\dots,2n+1)$ of $V$.

The structure group $\GG\times\II$ consisting of matrices of the form \eqref{GxI}
has a standard representation in $V$ which induces a natural representation $\lm$ of $\GG\times\II$ in
 $\T$ as follows
\[
 \left((\lm
a)T\right)(x,y,z)=T\left(a^{-1}x,a^{-1}y,a^{-1}z\right)
\]
 for any
$a\in \GG\times\II$ and $T\in\T$, so that
\[
\langle(\lm a)T_1,(\lm
a)T_2\rangle=\langle T_1,T_2\rangle,\qquad T_1,T_2\in \T.
\]

Using the projectors $\mathrm{h}$ and $\mathrm{v}$ on $V$, which are introduced as in \eqref{TMhv} and \eqref{Xhv},
we have an
orthogonal decomposition of $V$ in the form
\[
V=\mathrm{h}(V)\oplus \mathrm{v}(V).
\]
Then we construct a partial decomposition of $\T$ as follows.

At first, we define the operator $p_1:\ \T\rightarrow\T$ by
\[
p_1(T)(x,y,z)=-T(\f^2x,\f^2y,\f^2z),\quad T\in\T.
\]
It is easy to check the following
\begin{lem}\label{lem-p1}
The operator $p_1$ has the following properties:
 \begin{enumerate}%\renewcommand{\theenumi}{\Roman{enumi}}\renewcommand{\labelenumi}{(\theenumi)}
    \item $p_1\circ p_1 = p_1$;
    \item  $\langle p_1(T_1),T_2\rangle=\langle T_1,p_1(T_2)\rangle,\quad T_1, T_2 \in\T$;
    \item  $p_1\circ (\lm a)=(\lm a)\circ p_1$.
\end{enumerate}
\end{lem}
According to \lemref{lem-p1}, we have the following orthogonal
decomposition of $\T$ by the image and the kernel of $p_1$:
\[
\begin{array}{l}
\PP_1=\iim(p_1)=\left\{T\in\T\ \vert\ p_1(T)=T\right\},\\[6pt]
\PP_1^\bot=\ker(p_1)=\left\{T\in\T\ \vert\ p_1(T)=0\right\}.
\end{array}
\]

Further, we consider the operator $p_2:\
\PP_1^\bot\rightarrow\PP_1^\bot$, defined by
\[
p_2(T)(x,y,z)=\eta(z)T(\f^2 x, \f^2 y, \xi),\quad T\in\PP_1^\bot.
\]
We obtain immediately the truthfulness of the following
\begin{lem}\label{lem-p2}
The operator $p_2$ has the following properties:
 \begin{enumerate}
   \item $p_2\circ p_2 = p_2$;
   \item $\langle p_2(T_1),T_2\rangle=\langle
    T_1,p_2(T_2)\rangle,\quad T_1, T_2 \in\PP_1^\bot$;
   \item $p_2\circ (\lm a)=(\lm a)\circ p_2$.
 \end{enumerate}
\end{lem}
Then, bearing in mind  \lemref{lem-p2}, we obtain
\[
\begin{split}
&\PP_2=\iim(p_2)=\left\{T\in\PP_1^\bot\ \vert\ p_2(T)=T\right\},\qquad\\[6pt]
&\PP_2^\bot=\ker(p_2)=\left\{T\in\PP_1^\bot\ \vert\ p_2(T)=0\right\}.
\end{split}
\]

Finally, we consider the operator $p_3:\
\PP_2^\bot\rightarrow\PP_2^\bot$ defined by
\[
p_3(T)(x,y,z)=\eta(x)T(\xi,\f^2 y, \f^2 z)+\eta(y)T(\f^2
x,\xi,\f^2 z),\quad T\in\PP_2^\bot
\]
and we get the following
\begin{lem}\label{lem-p3}
The operator $p_3$ has the following properties:
\begin{enumerate}
  \item $p_3\circ p_3 = p_3$;
  \item $\langle p_3(T_1),T_2\rangle=\langle
    T_1,p_3(T_2)\rangle,\quad T_1, T_2 \in\PP_2^\bot$;
  \item $p_3\circ (\lm a)=(\lm a)\circ p_3$.
\end{enumerate}
\end{lem}
By virtue of \lemref{lem-p3}, we have
\[
\begin{split}
&\PP_3=\iim(p_3)=\left\{T\in\PP_2^\bot\ \vert\ p_3(T)=T\right\},\qquad\\[6pt]
&\PP_4=\ker(p_3)=\left\{T\in\PP_2^\bot\ \vert\ p_3(T)=0\right\}.
\end{split}
\]

From \lemref{lem-p1}, \lemref{lem-p2} and \lemref{lem-p3} we have
immediately
\begin{thm}\label{thm-W1234}
The decomposition $ \T=\PP_1\oplus\PP_2\oplus\PP_3\oplus\PP_4 $ is
orthogonal and invariant under the action of $\GG\times\II$. The
subspaces $\PP_i$ $(i=1,2,3,4)$ are determined by
\begin{equation}\label{W1234}
\begin{split}
\PP_1:\quad &T(x,y,z)=-T(\f^2 x, \f^2 y, \f^2 z),\\[6pt]
\PP_2:\quad &T(x,y,z)=\eta(z)T(\f^2 x, \f^2 y, \xi),\\[6pt]
\PP_3:\quad &T(x,y,z)=\eta(x)T(\xi,\f^2 y, \f^2 z)+\eta(y)T(\f^2 x,\xi,\f^2 z),\\[6pt]
\PP_4:\quad &T(x,y,z)=-\eta(z)\left\{\eta(y)T(\f^2 x, \xi,\xi)+\eta(x)T(\xi,\f^2 y,\xi)\right\}\\[6pt]
\end{split}
\end{equation}
for arbitrary vectors $x, y, z \in V$.
\end{thm}

\begin{cor}\label{cor-W1234}
The subspaces $\PP_i$ $(i=1,2,3,4)$ are characterized as follows:
\begin{align*}
\PP_1&=\left\{T\in\T\ \vert\ T(x^{\mathrm{v}},y,z)=T(x,y,z^{\mathrm{v}})=0\right\},\\[6pt]
\PP_2&=\left\{T\in\T\ \vert\ T(x^{\mathrm{v}},y,z)=T(x,y,z^{\mathrm{h}})=0\right\},%\\[6pt]
\end{align*}
\begin{align*}
\PP_3&=\left\{T\in\T\ \vert\ T(x,y,z^{\mathrm{v}})=T(x^{\mathrm{h}},y^{\mathrm{h}},z)=0\right\},\\[6pt]
\PP_4&=\left\{T\in\T\ \vert\ T(x,y,z^{\mathrm{h}})=T(x^{\mathrm{h}},y^{\mathrm{h}},z)=0\right\},
\end{align*}
where $x, y, z \in V$.
\end{cor}

According to \corref{cor-W1234}, \eqref{W1234} and \eqref{ttt}, we
obtain the following
\begin{cor}\label{cor-t-W1}
The torsion forms of $T$ have the following properties in each of
the
sub\-spaces $\PP_{i}$ $(i=1,2,3,4)$:
\begin{enumerate}
  \item If $T\in\PP_{1}$, then $t\circ \mathrm{v}=t^*\circ
    \mathrm{v}=\hat{t}=0$;
  \item If $T\in\PP_{2}$, then $t=t^*=\hat{t}=0$;
  \item If $T\in\PP_{3}$, then $t\circ \mathrm{h}=t^*\circ
    \mathrm{h}=\hat{t}=0$;
  \item If $T\in\PP_{4}$, then $t=t^*=0$.
\end{enumerate}
\end{cor}

Further we continue the decomposition of the subspaces $\PP_i$
$(i=1,2,\allowbreak{}3,4)$ of $\T$.

\vskip 0.2in \addtocounter{subsubsection}{1}

\noindent  {\Large\bf \emph{\thesubsubsection. The subspace $\PP_1$}}

\vskip 0.15in

%\subsection{The subspace $\PP_1$}

Since the endomorphism $\f$ induces an almost complex structure on
$\HC$ (which is the orthogonal complement $\{\xi\}^\bot$
of the subspace $\VV$) and the restriction of $g$ on $\HC$ is
a Norden metric (because the almost complex structure causes an
anti-isometry on $\HC$), then the decomposition of $\PP_1$ is made as
the decomposition of the space of the torsion tensors on an almost
complex manifold with Norden metric known from \cite{GaMi87}.

Let us consider the linear operator $L_{1,0}:\ \PP_1\rightarrow
\PP_1$ defined by
\[
L_{1,0}(T)(x,y,z)=-T(\f x,\f y,\f^2 z).
\]
Then, it follows immediately
\begin{lem}\label{lem-L10}
The operator $L_{1,0}$ is an involutive isometry on $\PP_1$ and it
is invariant with respect to the group $\GG\times\II$, i.e.
\begin{enumerate}
  \item $L_{1,0}\circ L_{1,0}=\Id_{\PP_1}$;
  \item $\langle L_{1,0}(T_1),L_{1,0}(T_2)\rangle=\langle T_1,T_2 \rangle$;
  \item $L_{1,0}((\lm a)T)=(\lm a)(L_{1,0}(T))$,
\end{enumerate}
where $T_1,T_2\in\PP_1$, $a\in \GG\times\II$.
\end{lem}

Therefore, $L_{1,0}$ has two eigenvalues $+1$ and $-1$, and the
corresponding eigen\-spaces
\[
\begin{array}{l}
\PP_1^+=\left\{T\in \PP_1\ \vert\ L_{1,0}(T)= T\right\},\\[6pt]
\PP_1^-=\left\{T\in \PP_1\ \vert\ L_{1,0}(T)=- T\right\}
\end{array}
\]
are invariant orthogonal subspaces of $\PP_1$.

In order to decompose $\PP_1^-$, we consider the linear operator
$L_{1,1}:\ \PP_1^-\rightarrow \PP_1^-$ defin\-ed by
\[
L_{1,1}(T)(x,y,z)=-T(\f x,\f^2 y,\f z).
\]
Let us denote the eigenspaces
\[
\begin{array}{l}
\PP_{1,1}=\left\{T\in \PP_1^-\
\vert\ L_{1,1}(T)=- T\right\},\\[6pt]
\PP_{1,2}=\left\{T\in \PP_1^-\
\vert\ L_{1,1}(T)= T\right\}.
\end{array}
\]
We have
\begin{lem}\label{lem-L11}
The operator $L_{1,1}$ is an involutive isometry on $\PP_1^-$ and it
is invariant with respect to $\GG\times\II$.
\end{lem}

According to the latter lemma, the eigenspaces $\PP_{1,1}$ and
$\PP_{1,2}$ are invariant and orthogonal.

To decompose $\PP_1^+$, we define the linear operator $L_{1,2}:\
\PP_1^+\rightarrow \PP_1^+$ as follows:
\[
\begin{split}
L_{1,2}(T)(x,y,z)=-\dfrac{1}{2}&\left\{T(\f^2 z,\f^2 x,\f^2 y)
+T(\f^2 z,\f x,\f y)\right.\\[6pt]
&\left.\ -T(\f^2 z,\f^2 y,\f^2 x)
-T(\f^2 z,\f y,\f x)\right\}.
\end{split}
\]

\begin{lem}\label{lem-L12}
The operator $L_{1,2}$ is an involutive isometry on $\PP_1^+$ and
it is invariant with respect to $\GG\times\II$.
\end{lem}

Thus, the eigenspaces
\[
\begin{array}{l}
\PP_{1,3}=\left\{T\in \PP_1^+\ \vert\
L_{1,2}(T)= T\right\},\\[6pt]
\PP_{1,4}=\left\{T\in \PP_1^+\ \vert\
L_{1,2}(T)=- T\right\}.
\end{array}
\]
are invariant and orthogonal.

Using \lemref{lem-L10}, \lemref{lem-L11} and \lemref{lem-L12}, we
get the following
\begin{thm}\label{thm-T1k}
The decomposition $
\PP_1=\PP_{1,1}\oplus\PP_{1,2}\oplus\PP_{1,3}\oplus\PP_{1,4} $ is
orthogonal and invariant with respect to the structure group.
\end{thm}

Bearing in mind the definition of the subspaces $\PP_{1,i}$
$(i=1,2,3,4)$, we obtain
\begin{prop}\label{prop-T1k}
The subspaces $\PP_{1,i}$ $(i=1,2,3,4)$ of $\PP_1$ are determined
by:
\begin{align*}
\PP_{1,1}:\quad
    &T(\xi,y,z)=T(x,y,\xi)=0,\quad \\[6pt]
    &T(x,y,z)=-T(\f x,\f y,z)=-T(x,\f y,\f z);\\[6pt]
%\Leftrightarrow\quad
    %&T(\xi,y,z)=T(x,y,\xi)=0,\quad %\\[6pt]
    %T(\f x,y,z)=T(x,\f y,z)=T(x,y,\f z); \\[6pt]
%
\PP_{1,2}:\quad &T(\xi,y,z)=T(x,y,\xi)=0,\quad \\[6pt]
            &T(x,y,z)=-T(\f x,\f y,z)=T(\f x,y,\f z);\\[6pt]
%\end{align*}
%\begin{align*}
%
\PP_{1,3}:\quad &T(\xi,y,z)=T(x,y,\xi)=0,\quad \\[6pt]
            &T(x,y,z)-T(\f x,\f y,z)=\sx  T(x,y,z)=0;\\[6pt]
\PP_{1,4}:\quad &T(\xi,y,z)=T(x,y,\xi)=0,\quad \\[6pt]
            &T(x,y,z)-T(\f x,\f y,z)=\sx T(\f x,y,z)=0.
\end{align*}
\end{prop}

Using \corref{cor-t-W1} (i), \propref{prop-T1k} and \eqref{ttt}, we
obtain
\begin{cor}\label{cor-t-T1i}
The torsion forms $t$ and $t^*$ of $T$ have the following
properties in the subspaces $\PP_{1,i}$ $(i=1,2,3,4)$:
\begin{enumerate}
  \item If $T\in\PP_{1,1}$, then $t=-t^*\circ\f$, $t\circ\f=t^*$;
  \item If $T\in\PP_{1,2}$, then $t=t^*=0$;
  \item If $T\in\PP_{1,3}$, then $t=t^*\circ\f$, $t\circ\f=-t^*$;
  \item If $T\in\PP_{1,4}$, then $t=t^*=0$.
\end{enumerate}
\end{cor}

Let us remark that each of the subspaces $\PP_{1,1}$ and $\PP_{1,3}$
can be additionally decomposed to a couple of subspaces --- one of
zero traces $(t,\ t^*)$ and one of non-zero traces $(t,\ t^*)$,
i.e.
\begin{align}\label{W11W13}
\PP_{1,1}=\PP_{1,1,1}\oplus \PP_{1,1,2}, \qquad
\PP_{1,3}=\PP_{1,3,1}\oplus \PP_{1,3,2},
\end{align}
where
\begin{align*}\label{W111W112}
&\PP_{1,1,1}=\left\{T\in \PP_{1,1}\ \vert\ t\neq 0\right\}, \qquad%\\[6pt]
\PP_{1,3,1}=\left\{T\in \PP_{1,3}\ \vert\ t\neq 0\right\},\\[6pt]
&\PP_{1,1,2}=\left\{T\in \PP_{1,1}\ \vert\ t=0\right\},\qquad%\\[6pt]
\PP_{1,3,2}=\left\{T\in \PP_{1,3}\ \vert\ t=0\right\}.
\end{align*}

\begin{prop}\label{prop-p_1i}
Let $T\in\T$ and $p_{1,i}$ $(i=1,2,3,4)$ be the projection
operators of $\T$ in $\PP_{1,i}$, generated by the decomposition
above. Then we have
\begin{align*}
p_{1,1}(T)(x,y,z)&=-\dfrac{1}{4}\bigl\{T(\f^2 x,\f^2 y,\f^2 z)-T(\f x,\f y,\f^2 z) \\[6pt]
&\phantom{=-\dfrac{1}{4}\ }
- T(\f x,\f^2 y,\f z) - T(\f^2 x,\f y,\f z)\bigr\};       \\[6pt]
p_{1,2}(T)(x,y,z)&=-\dfrac{1}{4}\bigl\{T(\f^2 x,\f^2 y,\f^2 z)-T(\f x,\f y,\f^2 z) \\[6pt]
&\phantom{=-\dfrac{1}{4}\ }
+ T(\f x,\f^2 y,\f z) + T(\f^2 x,\f y,\f z)\bigr\};       \\[6pt]
p_{1,3}(T)(x,y,z)&=-\dfrac{1}{4} \left\{T(\f^2 x,\f^2 y,\f^2
z)+T(\f x,\f y,\f^2 z) \right\}\\[6pt]
&\phantom{=\,}
+\dfrac{1}{8} \bigl\{T(\f^2 z,\f^2 x,\f^2 y) + T(\f^2 z,\f x,\f y)
 \\[6pt]
&\phantom{=-\dfrac{1}{4}\ }
+T(\f z,\f x,\f^2 y)- T(\f z,\f^2 x,\f y)
 \\[6pt]
&\phantom{=-\dfrac{1}{4}\ }
-T(\f^2 z,\f^2 y,\f^2 x) - T(\f^2 z,\f y,\f x)
 \\[6pt]
&\phantom{=-\dfrac{1}{4}\ }
-T(\f z,\f y,\f^2 x)+ T(\f z,\f^2 y,\f x)\bigr\},
\\[6pt]
%
%\end{align*}
%\begin{align*}
p_{1,4}(T)(x,y,z)&=-\dfrac{1}{4} \left\{T(\f^2 x,\f^2 y,\f^2
z)+T(\f x,\f y,\f^2 z) \right\}\\[6pt]
&\phantom{=\,}
-\dfrac{1}{8} \bigl\{T(\f^2 z,\f^2 x,\f^2 y) + T(\f^2 z,\f x,\f y)
 \\[6pt]
&\phantom{=-\dfrac{1}{4}\ }
+T(\f z,\f x,\f^2 y)- T(\f z,\f^2 x,\f y)
 \\[6pt]
&\phantom{=-\dfrac{1}{4}\ }
-T(\f^2 z,\f^2 y,\f^2 x) - T(\f^2 z,\f y,\f x)
 \\[6pt]
&\phantom{=-\dfrac{1}{4}\ }
-T(\f z,\f y,\f^2 x)+ T(\f z,\f^2 y,\f x)\bigr\}.
\end{align*}
\end{prop}
\begin{proof}
Let us show the calculations about $p_{1,1}$ for example, using
\cite{GaMi87}. \lemref{lem-L10} implies that the tensor
$\dfrac{1}{2}\left\{T-L_{1,0}(T)\right\}$ is the projection of
$T\in\PP_1$ in $\PP_1^-=\PP_{1,1}\oplus\PP_{1,2}$. Using
\lemref{lem-L11}, we find the expression of $p_{1,1}$ in terms of
the operators $L_{1,0}$ and $L_{1,1}$ for $T\in\PP_1$, namely
\[
p_{1,1}(T)=\dfrac{1}{4}\left\{T-L_{1,0}(T)-L_{1,1}(T)+L_{1,1}\circ
L_{1,0}(T)\right\},
\]
which  implies the stated expression of $p_{1,1}$, taking into
account that $T\in \PP_1$ is the image of $T\in\T$ by $p_1$. In a
similar way we prove the expressions for the other projectors
under consideration.

We verify the following equalities for $i=1,2,3,4$
\[
p_{1,i}\circ p_{1,i}=p_{1,i},\qquad
\sum_i p_{1,i}=\Id_{\PP_{1}}.
\]\vskip-2em
\end{proof}

\vskip 0.2in \addtocounter{subsubsection}{1}

\noindent  {\Large\bf \emph{\thesubsubsection. The subspace $\PP_2$}}

\vskip 0.15in
%\subsection{The subspace $\PP_2$}

Following the demonstrated
procedure for $\PP_1$, we continue the decomposition of the other
main subspaces of $\T$ with respect to the almost contact B-metric
structure.

\begin{lem}
The operator $L_{2,0}$, defined by \[
L_{2,0}(T)(x,y,z)=\eta(z)T(\f x,\f y,\xi), \] is an involutive
isometry on $\PP_2$ and invariant with respect to $\GG\times\II$.
\end{lem}

Hence, the corresponding eigenspaces
\[
\begin{array}{l}
\PP_{2,1}=\left\{T\in
\PP_2\ \vert\ L_{2,0}(T)=- T\right\},\\[6pt]
\PP_{2,2}=\left\{T\in
\PP_2\ \vert\ L_{2,0}(T)= T\right\}.
\end{array}
\]
 are
invariant and orthogonal. Therefore, we have
\begin{thm}\label{thm-T2k}
The decomposition $ \PP_2=\PP_{2,1}\oplus\PP_{2,2} $ is orthogonal
and invariant with respect to the structure group.
\end{thm}

\begin{prop}\label{prop-T2k}
The subspaces of $\PP_2$ are determined by:
\[
\begin{split}
\PP_{2,1}:\quad &T(x,y,z)=\eta(z)T(\f^2 x,\f^2 y,\xi),\quad\\[6pt]
                &T(x,y,\xi)=- T(\f x,\f y,\xi),\\[6pt]
\PP_{2,2}:\quad &T(x,y,z)=\eta(z)T(\f^2 x,\f^2 y,\xi),\quad\\[6pt]
                &T(x,y,\xi)= T(\f x,\f y,\xi).
\end{split}
\]
\end{prop}

Then the tensors $\dfrac{1}{2}\left\{T-L_{2,0}(T)\right\}$ and
$\dfrac{1}{2}\left\{T+L_{2,0}(T)\right\}$ are the projections of
$\PP_2$ in $\PP_{2,1}$ and $\PP_{2,2}$, respectively. Moreover, we
have the truthfulness of the properties
\[
p_{2,1}\circ p_{2,1}=p_{2,1},\qquad
p_{2,2}\circ p_{2,2}=p_{2,2},\qquad
p_{2,1}+p_{2,2}=\Id_{\PP_{2}}.
\]
Therefore, taking into account
$p_2$, we obtain
\begin{prop}\label{prop-p_2j}
Let $p_{2,1}$ and $p_{2,2}$  be the projection operators
of $\T$ in $\PP_{2,1}$ and $\PP_{2,2}$, respectively, generated by the decomposition above. Then
we have for $T\in\T$ the following
\[
\begin{split}
p_{2,1}(T)(x,y,z)=\dfrac{1}{2}\eta(z)\left\{T(\f^2 x,\f^2
y,\xi)- T(\f x,\f y,\xi)\right\},\\[6pt]
p_{2,2}(T)(x,y,z)=\dfrac{1}{2}\eta(z)\left\{T(\f^2 x,\f^2
y,\xi)+ T(\f x,\f y,\xi)\right\}.
\end{split}
\]
\end{prop}

According to \corref{cor-t-W1} (ii), \propref{prop-T2k} and
\eqref{ttt}, we obtain
\begin{cor}\label{cor-t-W2}
The torsion forms of $\hspace{3pt}T$ are zero in each of the subspaces
$\PP_{2,1}$ and  $\PP_{2,2}$, i.e.
    if $T\in\PP_{2,1}\oplus\PP_{2,2}$, then $t=t^*=\hat{t}=0$.
\end{cor}

\vskip 0.2in \addtocounter{subsubsection}{1}

\noindent  {\Large\bf \emph{\thesubsubsection. The subspace $\PP_3$}}

\vskip 0.15in
%\subsection{The subspace $\PP_3$}

\begin{lem}
The following operators $L_{3,k}$ $(k=0,1)$ are involutive
isometries on $\PP_3$ and invariant with respect to $\GG\times\II$:
\[
\begin{array}{l}
L_{3,0}(T)(x,y,z)=\eta(x)T(\xi,\f y,\f z)-\eta(y)T(\xi,\f x,\f z),\qquad\\[6pt]
L_{3,1}(T)(x,y,z)=\eta(x)T(\xi,\f^2 z,\f^2 y)-\eta(y)T(\xi,\f^2 z,\f^2 x).
\end{array}
\]
\end{lem}

By virtue of their action, we obtain consecutively the
corresponding invariant and orthogonal eigenspaces:
\begin{align*}
\PP_3^-&=\left\{T\in \PP_3\ \vert\ L_{3,0}(T)=-T\right\},\qquad \\[6pt]
\PP_3^+&=\left\{T\in \PP_3\ \vert\ L_{3,0}(T)=T\right\},\\[6pt]
\PP_{3,1}&=\left\{T\in \PP_3^-\ \vert\ L_{3,1}(T)=
T\right\},\qquad \\[6pt]
\PP_{3,2}&=\left\{T\in \PP_3^-\ \vert\ L_{3,1}(T)=-
T\right\},\qquad \\[6pt]
\PP_{3,3}&=\left\{T\in \PP_3^+\ \vert\
L_{3,1}(T)= T\right\},\\[6pt]
\PP_{3,4}&=\left\{T\in \PP_3^+\ \vert\
L_{3,1}(T)=- T\right\}.
\end{align*}

In such a way, we get
\begin{thm}\label{thm-T3k}
The decomposition $
\PP_3=\PP_{3,1}\oplus\PP_{3,2}\oplus\PP_{3,3}\oplus\PP_{3,4} $ is
orthogonal and in\-va\-riant with respect to the structure group.
\end{thm}

\begin{prop}\label{prop-T3k}
The subspaces of $\PP_3$ are determined by:
\begin{align*}
\PP_{3,1}:\quad &T(x,y,z)=\eta(x)T(\xi,\f^2 y,\f^2 z)-\eta(y)T(\xi,\f^2 x,\f^2 z),\quad\\[6pt]
                &T(\xi,y,z)= T(\xi,z,y)=-T(\xi,\f y,\f z); %\\[6pt]
\end{align*}%
\begin{align*}
\PP_{3,2}:\quad &T(x,y,z)=\eta(x)T(\xi,\f^2 y,\f^2 z)-\eta(y)T(\xi,\f^2 x,\f^2 z),\quad\\[6pt]
                &T(\xi,y,z)=- T(\xi,z,y)=-T(\xi,\f y,\f z); \\[6pt]
\PP_{3,3}:\quad &T(x,y,z)=\eta(x)T(\xi,\f^2 y,\f^2 z)-\eta(y)T(\xi,\f^2 x,\f^2 z),\quad\\[6pt]
                &T(\xi,y,z)= T(\xi,z,y)=T(\xi,\f y,\f z); \\[6pt]
\PP_{3,4}:\quad &T(x,y,z)=\eta(x)T(\xi,\f^2 y,\f^2 z)-\eta(y)T(\xi,\f^2 x,\f^2 z),\quad\\[6pt]
                &T(\xi,y,z)=- T(\xi,z,y)=T(\xi,\f y,\f z).
\end{align*}
\end{prop}

By virtue of \corref{cor-t-W1} (iii), \propref{prop-T3k} and equalities
\eqref{ttt}, we obtain
\begin{cor}\label{cor-t-T3k}
The torsion forms $t$ and $t^*$ of $T$ %:\\[6pt]
%    $($i$)$
are zero in $\PP_{3,k}\subset\PP_{3}$ $(k=2,3,4)$.%; \quad$
%    $($ii$)$ have no extra properties
%in $\PP_{3,1}\subset\PP_{3}$.
\end{cor}

Let us remark that $\PP_{3,1}$ can be additionally decomposed to
three subspaces determined by conditions $t=0$, $t^*=0$ and
$t=t^*=0$, respectively, i.e.
\begin{equation*}\label{W31}
%\begin{array}{c}
\PP_{3,1}=\PP_{3,1,1}\oplus \PP_{3,1,2}\oplus \PP_{3,1,3},
%\end{array}
\end{equation*}
where
\[
\begin{array}{l}\label{W311W312W313}
\PP_{3,1,1}=\left\{T\in \PP_{3,1}\ \vert\ t\neq 0,\
t^*=0\right\},\quad\\[6pt]
\PP_{3,1,2}=\left\{T\in \PP_{3,1}\ \vert\ t=0,\ t^*\neq 0\right\}, \\[6pt]
\PP_{3,1,3}=\left\{T\in \PP_{3,1}\ \vert\ t=0,\ t^*=0\right\}.
\end{array}
\]

\begin{prop}\label{prop-p_3k}
Let $T\in\T$ and $p_{3,k}$ $(k=1,2,3,4)$ be the projection
operators of $\T$ in $\PP_{3,k}$, generated by the decomposition
above. Then we have
\[
p_{3,k}(T)(x,y,z)=
                \dfrac{1}{4}\left\{\eta(x)A_{3,k}(y,z)
-\eta(y)A_{3,k}(x,z)\right\},
\]
where
\begin{align*}
A_{3,1}(y,z)=T(\xi,\f^2 y,\f^2 z)
                + T(\xi,\f^2 z,\f^2 y)\\[6pt]
                -T(\xi,\f y,\f z)
                - T(\xi,\f z,\f y),
\\[6pt]
A_{3,2}(y,z)=T(\xi,\f^2 y,\f^2 z)
                - T(\xi,\f^2 z,\f^2 y)\\[6pt]
                -T(\xi,\f y,\f z)
                + T(\xi,\f z,\f y),
\\[6pt]
A_{3,3}(y,z)=T(\xi,\f^2 y,\f^2 z)
                + T(\xi,\f^2 z,\f^2 y)\\[6pt]
                +T(\xi,\f y,\f z)
                + T(\xi,\f z,\f y),
%\\[6pt]
\end{align*}
\begin{align*}
A_{3,4}(y,z)=T(\xi,\f^2 y,\f^2 z)
                - T(\xi,\f^2 z,\f^2 y)\\[6pt]
                +T(\xi,\f y,\f z)
                - T(\xi,\f z,\f y).
\end{align*}
\end{prop}

\vskip 0.2in \addtocounter{subsubsection}{1}

\noindent  {\Large\bf \emph{\thesubsubsection. The subspace $\PP_4$}}

\vskip 0.15in
%\subsection{The subspace $\PP_4$}

Finally, we only denote $\PP_4$ as $\PP_{4,1}$ and it is determined
as follows
\begin{equation*}%\label{T41}%
\PP_{4,1}:\quad
T(x,y,z)=\eta(z)\left\{\eta(y)\hat{t}(x)-\eta(x)\hat{t}(y)\right\}.
\end{equation*}
Obviously, the projection operator $p_{4,1}: \T
\rightarrow\PP_{4,1}$ has the form
\begin{equation}\label{p_41}
p_{4,1}(T)(x,y,z)=\eta(z)\left\{\eta(y)\hat{t}(x)-\eta(x)\hat{t}(y)\right\}.
\end{equation}

\vskip 0.2in \addtocounter{subsubsection}{1}

\noindent  {\Large\bf \emph{\thesubsubsection. The fifteen subspaces of $\T$}}

\vskip 0.15in%\subsection{The fifteen subspaces of $\T$}

In conclusion of the decomposition explained above, we combine
Theorems \ref{thm-W1234}, \ref{thm-T1k}, \ref{thm-T2k} and
\ref{thm-T3k}. We denote the subspaces $\PP_{i,j}$ and $\PP_{i,j,k}$
by $\T_{s}, s\in\{1,2,\dots,15\}$, as follows:
\begin{equation}\label{Ti15}
\begin{array}{lll}
\T_{1}=\PP_{1,1,1},\qquad\quad &\T_{2}=\PP_{1,1,2},\qquad\quad
&\T_{3}=\PP_{1,2},\qquad\quad \\[6pt]
\T_{4}=\PP_{1,3,1},\qquad\quad
&\T_{5}=\PP_{1,3,2},\qquad\quad
&\T_{6}=\PP_{1,4},\qquad\quad \\[6pt]
\T_{7}=\PP_{2,1},\qquad\quad
&\T_{8}=\PP_{2,2},\qquad\quad &\T_{9}=\PP_{3,1,1},\qquad\quad \\[6pt]
\T_{10}=\PP_{3,1,2},\qquad\quad &\T_{11}=\PP_{3,1,3},\qquad\quad &\T_{12}=\PP_{3,2},\qquad\quad
\\[6pt]
\T_{13}=\PP_{3,3},\qquad\quad &\T_{14}=\PP_{3,4},\qquad\quad &\T_{15}=\PP_{4,1}.
\end{array}
\end{equation}

Then we obtain the following main statement in the present section
\begin{thm}\label{thm-15podpr}
Let $\T$ be the vector space of the torsion tensors of type $(0,3)$
over the vector space $V$ with almost contact $B$-metric structure
$(\f,\xi,\eta,g)$. The decomposition
\begin{equation}\label{TT-15podpr}
\T=\T_{1}\oplus\T_{2}\oplus\cdots\oplus\T_{15}
\end{equation}
is orthogonal and invariant with respect to the structure group
$\GG\times\II$.
\end{thm}

%%%%%%%%%%%%%%%%%%%%%%%%%%%%%%%%%%%%%%%%%%%%%%%%%%%%%%

\vskip 0.2in \addtocounter{subsection}{1} \setcounter{subsubsection}{0}

\noindent  {\Large\bf \thesubsection. The fifteen classes of affine connections }

\vskip 0.15in

It is well known that any metric connection $\n^*$ (i.e. $\n^*g=0$) with a potential $Q$
regarding $\n$, defined by $Q(x,y,z)=g\left(\n^*_xy-\n_xy,z\right)$,
is
completely determined by its torsion tensor $T$ by means of \eqref{QT},
according to the Hayden theorem \cite{Hay}.

For an almost contact B-metric manifold $\M$, there exist  infinitely many affine connections on the tangent space $T_p\MM$, $p\in \MM$. Then the subspace $\T_{s}$ $(s\in\{1,2,\dots,15\})$, where $T$ belongs,
is an important characteristic of $\n^*$. In such a way the
conditions for $T$ described as the subspace $\T_{s}$ give rise to
the corresponding class of the connection with respect to its
torsion tensor.
Therefore, the conditions of torsion tensors determine corresponding classes of the connections on the tangent bundle derived by the almost contact B-metric structure.

\begin{dfn}\label{dfn-Cs}
It is said that an affine connection $\n^*$ on an almost contact B-metric manifold belongs to a class $\CC_s$, $s\in\{1,2,\allowbreak{}\dots,15\}$, if the torsion tensor
$T^*$ of $\n^*$ belongs to the subspace $\TT_{s}$ in the decomposition
\eqref{TT-15podpr} of $\TT$.
\end{dfn}

Bearing in mind \thmref{thm-15podpr}, we obtain the following classifying
\begin{thm}\label{thm:CC}
The set of affine connections $\CC$ on an almost contact B-metric manifold
is divided into 15 basic classes $\CC_s$, $s\in\{1,2,\allowbreak{}\dots,15\}$, by the decomposition
\begin{equation}\label{CC}
    \CC=\CC_{1}\oplus\CC_{2}\oplus\cdots\oplus\CC_{15}.
\end{equation}
\end{thm}

The special class $\CC_{0}$ contains all symmetric connections and it corresponds to
the zero vector subspace $\TT_{0}$ of
$\TT$ determined by the condition $T = 0$. This class belongs to any other
class $\CC_{s}$. The Levi-Civita connections
$\n$ and $\widetilde{\n}$ are symmetric and therefore they belong to the class $\CC_{0}$.

The classes $\CC_{i}\oplus\CC_{j}\oplus\dots$, which are direct sums
of basic classes, are defined in a natural way by
the corresponding subspaces $\TT_{i}\oplus\TT_{j}\oplus\dots$,
following \dfnref{dfn-Cs}. According to \eqref{CC}, the number of all
classes of affine connections on $\M$ is $2^{15}$ and their
definition conditions are readily obtained from those of the basic 15
classes.

\vskip 0.2in \addtocounter{subsection}{1}

\noindent  {\Large\bf\thesubsection. Some natural connections in the introduced classification}

\vskip 0.15in%\subsection{The fifteen }

%\noindent  {\Large\bf 2.6. Петнадесетте класа линейни свързаности върху изучаваните многообразия}%\\\vskip2pt}

Further in the present section we discuss the three mentioned natural
connections with torsion on $\M$. Natural connections are a
generalization of the Levi-Civita connection.

%We get the following
\begin{prop}\label{prop-nat}
Let $\n^*$ be a natural connection with torsion $T$ on an almost
contact B-met\-ric manifold $\M$. Then the following implications
hold true:
\[
\begin{array}{rllll}
T\in\T_{1}\oplus\T_{2}\oplus\T_{6}\oplus\T_{12}  \;&\Rightarrow\;  \MM\in\F_0;\; & %
T\in\T_{3} \;&\Rightarrow\;   \MM\in\F_3; \\[6pt] %
T\in\T_{4} \;&\Rightarrow\;   \MM\in\F_1; \; &
T\in\T_{5} \;&\Rightarrow\;   \MM\in\F_2; \\[6pt] %
T\in\T_{7} \;&\Rightarrow\;   \MM\in\F_7; \; &%
T\in\T_{8}  \;&\Rightarrow\;  \MM\in\F_8\oplus\F_{10};\\[6pt]%
T\in\T_{9} \;&\Rightarrow\;   \MM\in\F_5; \; &%
T\in\T_{10} \;&\Rightarrow\;   \MM\in\F_4; \\[6pt]%
T\in\T_{11} \;&\Rightarrow\;   \MM\in\F_6; \; &%
T\in\T_{13} \;&\Rightarrow\;   \MM\in\F_9; \\[6pt]%
T\in\T_{14} \;&\Rightarrow\;   \MM\in\F_{10}; \; & %
T\in\T_{15} \;&\Rightarrow\;   \MM\in\F_{11}.
\end{array}
\]
\end{prop}
\begin{proof}
The implications follow from \eqref{QT}, \eqref{1ab}, \eqref{Fi},
\eqref{Ti15} and the corresponding characteristic conditions of
$\PP_{i,j}$ and $\PP_{i,j,k}$ as well as the projection operators
$p_{i,j}$. We show the proof in detail for some classes and the
rest follow in a similar way.

By virtue of \eqref{QT} and \eqref{1ab} we have
\begin{equation}\label{FT}
\begin{array}{l}
2F(x,y,z)=T(x,y,\f z)-T(y,\f z,x)+T(\f z,x,y)\\[6pt]
\phantom{2F(x,y,z)=}
-T(x,\f y,z)+T(\f y,z,x)-T(z,x,\f y).
\end{array}
\end{equation}

Let us consider $T\in\PP_{1,1}=\T_1\oplus\T_2$, which is equivalent
to $T=p_{1,1}(T)$. Then, ac\-cording to \propref{prop-p_1i}, we
have
\[
\begin{array}{l}
T(x,y,z)=-\dfrac{1}{4}\bigl\{T(\f^2 x,\f^2 y,\f^2 z)-T(\f x,\f
y,\f^2 z)\\[6pt]
\phantom{T(x,y,z)=-\dfrac{1}{4}\bigl\{}
-T(\f x,\f^2 y,\f z)-T(\f^2
x,\f y,\f z)\bigr\},
\end{array}
\]
which together with \eqref{FT} imply $F(x,y,z)=0$. Therefore, we
obtain $\MM\in\F_0$.

Now, let us suppose $T\in\PP_{1,2}=\T_3$ and hence $T=p_{1,2}(T)$,
which has the following form, taking into account
\propref{prop-p_1i}:
\[
\begin{array}{l}
T(x,y,z)=-\dfrac{1}{4}\bigl\{T(\f^2 x,\f^2 y,\f^2 z)-T(\f x,\f
y,\f^2 z)\\[6pt]
\phantom{T(x,y,z)=-\dfrac{1}{4}\bigl\{}
+T(\f x,\f^2 y,\f z)+T(\f^2
x,\f y,\f z)\bigr\}.
\end{array}
\]
Then, according to the latter equality and \eqref{FT}, we obtain
\begin{subequations}\label{FTT}
\begin{equation}
\begin{split}
F(x,y,z)&=-\dfrac{1}{4}\bigl\{-T(\f^2 x,\f^2 y,\f z)+T(\f x,\f
y,\f z)\\[6pt]
&\phantom{=-\dfrac{1}{4}\bigl\{}
+T(\f^2 x,\f y,\f^2 z)+T(\f x,\f^2 y,\f^2 z)\\[6pt]
&\phantom{=-\dfrac{1}{4}\bigl\{}
-T(\f z,\f^2 x,\f^2 y)-T(\f^2 z,\f x,\f^2 y)%\\[6pt]
\end{split}
\end{equation}
\begin{equation}
\begin{split}
&\phantom{=-\dfrac{1}{4}\bigl\{}
 +T(\f^2 z,\f^2 x,\f y)-T(\f z,\f x,\f y)\bigr\}
\end{split}
\end{equation}
\end{subequations}
and consequently $F(\xi,y,z)=F(x,y,\xi)=0$. Next, we take the
cyclic sum of \eqref{FTT} by the arguments $x,y,z$ and the result
is $\sx F(x,y,z)=0$. Therefore, $\M$ belongs to $\F_3$.
\end{proof}

%
%\vskip 0.2in \addtocounter{subsection}{1} \setcounter{subsubsection}{0}
%
%\noindent  {\Large\bf \thesubsection. Some Special Connections in the Introduced Classification}
%
%\vskip 0.15in
%%\section{Known Natural Connections in the Introduced Classification}\label{sec:3}

Bearing in mind the class of the almost contact B-metric manifolds
with $N=0$ and \propref{prop-nat}, we obtain immediately
\begin{cor}
An almost contact B-metric manifold $\M$ % belonging to $\F_i\setminus\F_0$, $i\in\{1,2,\dots,11\}$,
is normal, i.e. it has vanishing $N$, if any natural
connection on $\M$ belongs to the class
$\CC_{4}\oplus\CC_{5}\oplus\CC_{9}\oplus\CC_{10}\oplus\CC_{11}$.
\end{cor}

Similarly, \propref{prop-Nhat=0} and \propref{prop-nat} imply
\begin{cor}
An almost contact B-metric manifold $\M$ %belonging to $\F_i\setminus\F_0$, $i\in\{1,2,\dots,11\}$,
has vanishing $\widehat{N}$, if any natural connection
on $\M$ belongs to the class $\CC_{3}\oplus\CC_{7}$.
\end{cor}

\vskip 0.2in \addtocounter{subsubsection}{1}

\noindent  {\Large\bf \emph{\thesubsubsection. The $\f$B-connection in the classification}}

\vskip 0.15in

%\subsection{The $\f$B-connection in the classification}

%In \cite{ManGri2}, it is introduced a natural con\-nect\-ion $\n'$, called a $\f$B-connection, on $\M$ in any basic class %,
%%i.e. $D\f=D\xi=D\eta=Dg=D\tg=0$,
%by \eqref{fB0}.
%%\begin{equation}\label{fB}
%%\n'_xy=\n_xy+\dfrac{1}{2}\bigl\{\left(\n_x\f\right)\f
%%y+\left(\n_x\eta\right)y\cdot\xi\bigr\}-\eta(y)\n_x\xi.
%%\end{equation}
%%
%In \cite{ManIv37}, this connection is called a
%\emph{$\f$B-connection}. It is studied for some classes of the
%manifolds $\M$ in \cite{ManGri2,Man3,Man4,ManIv37}. The
%$\f$B-connec\-tion is the odd-dimensional analogue of the
%B-connection on the corresponding almost complex manifold with
%Norden metric, studied  in \cite{GaGrMi85} for the class of the
%conformal K\"ahler manifold with Norden metric.

The $\f$B-connection $\n'$ is discussed in \S5 and it has a torsion tensor $T'$ and corresponding torsion 1-forms, given in \eqref{fBT=F} and \eqref{tB}, respectively.

Applying Propositions \ref{prop-p_1i}, \ref{prop-p_2j},
\ref{prop-p_3k} and equation  \eqref{p_41} for the torsion tensor
$T'$ from \eqref{fBT=F}, we obtain the components of $T'$
in each of the subspaces $\PP_{i,j}$:
\begin{subequations}\label{pijT-B}
\begin{equation}
\begin{split}
&p_{1,1}(T')(x,y,z)=0,\quad\\[6pt]
&p_{1,2}(T')(x,y,z)=\dfrac{1}{4}\left\{F(\f x,\f y,\f z)-F(\f^2 x,\f^2 y,\f z)\right.\\[6pt]
&\phantom{p_{1,2}(T')(x,y,z)=-\dfrac{1}{4}}\left.
-F(\f y,\f x,\f z)+F(\f^2 y,\f^2 x,\f z)
\right\}
,\\[6pt]
&p_{1,3}(T')(x,y,z)=-\dfrac{1}{8}\bigl\{
F(\f^2 z,\f^2 y,\f x) + F(\f^2 x,\f^2 y,\f z)\\[6pt]
&
\phantom{p_{1,3}(T')(x,y,z)=-\dfrac{1}{8}\bigl\{}
+ F(\f x,\f y,\f z) - F(\f^2 z,\f^2 x,\f y)\\[6pt]
&
\phantom{p_{1,3}(T')(x,y,z)=-\dfrac{1}{8}\bigl\{}
- F(\f^2 y,\f^2 x,\f z) - F(\f y,\f x,\f z)
\bigr\},\\[6pt]%
\end{split}
\end{equation}
\begin{equation}
\begin{split}
&p_{1,4}(T')(x,y,z)=-\dfrac{1}{8}\bigl\{
F(\f^2 z,\f^2 y,\f x) - F(\f^2 x,\f^2 y,\f z)\\[6pt]
&
\phantom{p_{1,4}(T')(x,y,z)=-\dfrac{1}{8}\bigl\{}
- F(\f x,\f y,\f z) - F(\f^2 z,\f^2 x,\f y)\\[6pt]
&
\phantom{p_{1,4}(T')(x,y,z)=-\dfrac{1}{8}\bigl\{}
+ F(\f^2 y,\f^2 x,\f z) + F(\f y,\f x,\f z)
\bigr\},\\[6pt]%
&p_{2,1}(T')(x,y,z)=-\dfrac{1}{2}\eta(z)\bigl\{
F(\f^2 x,\f y,\xi)- F(\f x, y,\xi)\\[6pt]
&
\phantom{p_{2,1}(T')(x,y,z)=-\dfrac{1}{2}\eta(z)\bigl\{}
-F(\f^2 y,\f x,\xi)+ F(\f y, x,\xi)
\bigr\},\\[6pt]
&p_{2,2}(T')(x,y,z)=-\dfrac{1}{2}\eta(z)\bigl\{
F(\f^2 x,\f y,\xi)+ F(\f x, y,\xi)\\[6pt]
&
\phantom{p_{2,2}(T')(x,y,z)=-\dfrac{1}{2}\eta(z)\bigl\{}
-F(\f^2 y,\f x,\xi)- F(\f y, x,\xi)
\bigr\},\\[6pt]
&p_{3,1}(T')(x,y,z)=
                \dfrac{1}{4}\bigl\{
                \eta(y)\bigl[F(\f^2 x,\f z,\xi)+ F(\f^2 z,\f x,\xi)\\[6pt]
 &\phantom{p_{3,1}(T')(x,y,z)=\dfrac{1}{4}\bigl\{\eta(y)\bigl[}
                -F(\f x, z,\xi)- F(\f z, x,\xi)\bigr]\\[6pt]
 &\phantom{p_{3,1}(T')(x,y,z)=\dfrac{1}{4}\bigl\{}
                -\eta(x)\bigl[F(\f^2 y,\f z,\xi)+ F(\f^2 z,\f y,\xi)\\[6pt]
 &\phantom{p_{3,1}(T')(x,y,z)=\dfrac{1}{4}\bigl\{-\eta(x)\bigl[}
                -F(\f y, z,\xi)- F(\f z, y,\xi)\bigr]\bigr\},\\[6pt]
&p_{3,2}(T')(x,y,z)=
                \dfrac{1}{4}\bigl\{\eta(y)\bigl[F(\f^2 x,\f z,\xi)- F(\f^2 z,\f x,\xi)\\[6pt]
 &\phantom{p_{3,2}(T')(x,y,z)=\dfrac{1}{4}\bigl\{\eta(y)\bigl[}
                -F(\f x, z,\xi)+ F(\f z, x,\xi)\bigr]\\[6pt]
 &\phantom{p_{3,2}(T')(x,y,z)=\dfrac{1}{4}\bigl\{}
                -\eta(x)\bigl[F(\f^2 y,\f z,\xi)- F(\f^2 z,\f y,\xi)\\[6pt]
 &\phantom{p_{3,2}(T')(x,y,z)=\dfrac{1}{4}\bigl\{-\eta(x)\bigl[}
                -F(\f y, z,\xi)+ F(\f z, y,\xi)\bigr]
                \bigr\},\\[6pt]
&p_{3,3}(T')(x,y,z)=
                \dfrac{1}{4}\bigl\{\eta(y)\bigl[F(\f^2 x,\f z,\xi)+F(\f^2 z,\f x,\xi)\\[6pt]
 &\phantom{p_{3,3}(T')(x,y,z)=\dfrac{1}{4}\bigl\{\eta(y)\bigl[}
                +F(\f x, z,\xi)+F(\f z, x,\xi)\bigr]\\[6pt]
 \end{split}
\end{equation}
\begin{equation}
\begin{split}
 &\phantom{p_{3,3}(T')(x,y,z)=\dfrac{1}{4}\bigl\{}
                -\eta(x)\bigl[F(\f^2 y,\f z,\xi)+F(\f^2 z,\f y,\xi)\\[6pt]
 &\phantom{p_{3,3}(T')(x,y,z)=\dfrac{1}{4}\bigl\{-\eta(x)\bigl[}
                +F(\f y, z,\xi)+F(\f z, y,\xi)\bigr]
                \bigr\},\\[6pt]
&p_{3,4}(T')(x,y,z)=\dfrac{1}{4}\bigl\{\eta(y)\bigl[F(\f^2 x,\f z,\xi)-F(\f^2 z,\f x,\xi)\\[6pt]
 &\phantom{p_{3,4}(T')(x,y,z)=\dfrac{1}{4}\bigl\{\eta(y)\bigl[}
                +F(\f x, z,\xi)-F(\f z, x,\xi)\\[6pt]
 &\phantom{p_{3,4}(T')(x,y,z)=\dfrac{1}{4}\bigl\{\eta(y)\bigl[}
                +2F(\xi,\f x,\f^2 z)\bigr]\\[6pt]
&\phantom{p_{3,4}(T')(x,y,z)=\dfrac{1}{4}\bigl\{}
                -\eta(x)\bigl[F(\f^2 y,\f z,\xi)-F(\f^2 z,\f y,\xi)\\[6pt]
&\phantom{p_{3,4}(T')(x,y,z)=\dfrac{1}{4}\bigl\{-\eta(x)\bigl[}
                +F(\f y, z,\xi)-F(\f z, y,\xi)\\[6pt]
&\phantom{p_{3,4}(T')(x,y,z)=\dfrac{1}{4}\bigl\{-\eta(x)\bigl[}
                +2F(\xi,\f y,\f^2 z)\bigr]
                \bigr\},\\[6pt]
&p_{4,1}(T')(x,y,z)=\,\eta(z)\left\{\eta(x)\omega(\f
y)-\eta(y)\omega(\f x)\right\}.
\end{split}
\end{equation}
\end{subequations}

Such a way we establish that the torsion $T'$ of the $\f$B-connection $\n'$ on $\M$ belongs to
\(
\T_{3}\oplus\T_{4}\oplus\cdots\oplus\T_{15}\). Thus, the position of $\n'$
in the classification \eqref{CC} is determined as follows %in the
%following %as in \propref{prop-nat}.
%
\begin{prop}\label{prop-fB}
The $\f$B-connection on $\M$ belongs to the class
\(
\CC_{3}\oplus\CC_{4}\oplus\cdots\oplus\CC_{15}\).
\end{prop}

\vskip 0.2in \addtocounter{subsubsection}{1}

\noindent  {\Large\bf \emph{\thesubsubsection. The $\f$-canonical connection in the classification}}

\vskip 0.15in

According to the classification of the torsion tensors given above in the present section
%in \cite{ManIv36}
and \propref{prop:FiT}, we get
%that the implications in \propref{prop-nat} become equivalences for the
%$\f$-canonical connection on $\M$ belonging to $\F_i$ for
%$i\in\{1,2,\dots,11\}\setminus \{3,7\}$ %, according to \propref{prop:FiT''jk}.
%%\cite{ManIv38}.
%and we obtain
the following
\begin{prop}\label{prop:FiT''jk}
%Let $T''$ be the torsion of the $\f$-canonical connection on an
%almost contact B-metric manifold $\M$.
The correspondence between the classes $\F_i$ $(i\in\{1,2,\allowbreak{}\dots,11\})$ of the manifolds $\M$ and the classes $\CC_{s}$ $(s\in\{1,2,\allowbreak{}\dots,15\})$ of the $\f$-canonical connection  $\n''$ on $\M$ is given as follows:
\[
\begin{array}{ll}
\M\in\F_0\quad & \Leftrightarrow \quad
\n''\in\CC_{1}\oplus\CC_{2}\oplus\CC_{6}\oplus\CC_{12};
\\[6pt]
\M\in\F_1\quad & \Leftrightarrow \quad \n''\in\CC_{4};
\\[6pt]
\M\in\F_2\quad & \Leftrightarrow \quad \n''\in\CC_{5};
\\[6pt]
\M\in\F_3\quad & \Leftrightarrow \quad \n''\in\CC_{3};
%\\[6pt]
\end{array}
\]
\[
\begin{array}{ll}
\M\in\F_4\quad & \Leftrightarrow \quad \n''\in\CC_{10};
\\[6pt]
\M\in\F_5\quad & \Leftrightarrow \quad \n''\in\CC_{9};
\\[6pt]
\M\in\F_6\quad & \Leftrightarrow \quad \n''\in\CC_{11};
\\[6pt]
\M\in\F_7\quad & \Leftrightarrow \quad \n''\in\CC_{7}\oplus\CC_{12};
\\[6pt]
\M\in\F_8\quad & \Leftrightarrow \quad \n''\in\CC_{8}\oplus\CC_{14};
\\[6pt]
\M\in\F_9\quad & \Leftrightarrow \quad \n''\in\CC_{13};
\\[6pt]
\M\in\F_{10}\quad & \Leftrightarrow \quad \n''\in\CC_{14};
\\[6pt]
\M\in\F_{11}\quad & \Leftrightarrow \quad \n''\in\CC_{15}.
\end{array}
\]
\end{prop}

For the example in Subsection 5.3.3 on page~\pageref{exa:sphera}, it follows that the statement $T''\in\PP_{3,1,1}\oplus\PP_{3,1,2}$ is valid,
which supports  \propref{prop:FiT''jk}, bearing in mind \eqref{Ti15}.

\vskip 0.2in \addtocounter{subsubsection}{1}

\noindent  {\Large\bf \emph{\thesubsubsection. The $\f$KT-connection in the classification}}

\vskip 0.15in
%\subsection{The $\f$KT-connection in the classification}

In a similar way as for \eqref{pijT-B}, from \eqref{fKT-T37} we get the
following non-zero components of $T'''$:
\begin{equation}\label{pijT-KT}
\begin{split}
%\begin{aligned}
%
p_{1,2}(T''')(x,y,z)&=-\dfrac{1}{2}\bigl\{F(x,y,\f z)+F(y,z,\f
x)-F(z,x,\f y)\\[6pt]
&\phantom{=-\dfrac{1}{2}\bigl\{}
-\eta(x)F(y,\f z,\xi)+\eta(y)F(z,\f x,\xi)\\[6pt]
&\phantom{=-\dfrac{1}{2}\bigl\{}
+\eta(z)F(x,\f y,\xi)\bigr\},\\[6pt]
p_{1,4}(T''')(x,y,z)&=-F(z,x,\f y)-\eta(x)F(y,\f z,\xi),\\[6pt]
p_{2,1}(T''')(x,y,z)&=2\eta(z)F(x,\f y,\xi),\\[6pt]
p_{3,2}(T''')(x,y,z)&=2\eta(x)F(y,\f z,\xi)+2\eta(y)F(z,\f x,\xi).%\\[6pt]
%
%\end{aligned}
\end{split}
\end{equation}

Therefore, we have that $T'''$ belongs to $\T_{3}\oplus\T_{6}\oplus\T_{7}\oplus\T_{12}$ and the following
\begin{prop}\label{prop-fKT}
The $\f$KT-connection $\n'''$ on $\M$ in the class
$\F_3\oplus\F_7$ belongs to $\CC_{3}\oplus\CC_{6}\oplus\CC_{7}\oplus\CC_{12}$.
Moreover, if
$\M\in\F_3$ (resp. $\F_7$) then $\n'''\in\CC_3\oplus\CC_6$ (resp.
$\CC_7\oplus\CC_{12}$).
\end{prop}

\vspace{20pt}

\begin{center}
$\divideontimes\divideontimes\divideontimes$
\end{center} 
\newpage

\addtocounter{section}{1}\setcounter{subsection}{0}\setcounter{subsubsection}{0}

\setcounter{thm}{0}\setcounter{equation}{0}

\label{par:pair}

 \Large{

\
\\[6pt]
\bigskip

\
\\[6pt]
\bigskip

\lhead{\emph{Chapter I $|$ \S\thesection. Pair of associated Schouten-van Kampen connections
adapted to \ldots %an almost contact \ldots %B-metric structure
}}
%\thispagestyle{empty}

%\noindent  {\Huge\bf \S\thesection. Pair of associated Schouten-\\[12pt]
%\phantom{\S\thesection. }van Kampen connections \\[14pt]
%\phantom{\S\thesection. }adapted to an almost contact \\[14pt]
%\phantom{\S\thesection. }B-metric structure
%}%\\[6pt]\vskip2pt}

\noindent
\begin{tabular}{r"l}
  %\hline
\hspace{-6pt}{\Huge\bf \S\thesection.}  & {\Huge\bf Pair of associated Schouten-} \\[12pt]
                             & {\Huge\bf van Kampen connections} \\[12pt]
                             & {\Huge\bf adapted to an almost contact}\\[12pt]
                             & {\Huge\bf B-metric structure}
  %\hline
\end{tabular}

\vskip 1cm

\begin{quote}
\begin{large}
In the present section there are introduced and studied a pair of associated Schouten-van Kampen affine connections
adapted to the contact distribution and an almost contact B-metric structure generated by
the pair of associated B-metrics and their Levi-Civita connections.
By means of the constructed non-symmetric connections,
the basic classes of almost contact B-metric manifolds are characterized.
Curvature properties of the considered connections are obtained.

The main results of this section are published in \cite{Man50}.
\end{large}
\end{quote}

%
%\vskip 0.2in \addtocounter{subsection}{1}
%
%\noindent  {\Large\bf \thesubsection. Introduction}

\vskip 0.15in
%\section{Introduction}

The Schouten-van Kampen connection has been introduced for a studying of non-holonomic manifolds.
It preserves by parallelism a pair of complementary dis\-tri\-bu\-tions on a differentiable manifold
endowed with an affine connection \cite{SvK,Ia,BejFarr}. This connection is also used for investigations of
hyper\-distributions in Riemannian manifolds (e.g., \cite{Sol}).

On the other hand, any almost contact manifold admits a hyper\-dis\-tribution, the known contact distribution.
In \cite{Olsz}, it is studied the Schouten-van
Kampen connection adapted to an almost (para)contact metric structure. On these manifolds, the studied connection is not natural in general because it preserves the structure tensors except the structure endomorphism.

We consider almost contact B-metric structures. As it is mentioned above, an important characteristic of the almost contact B-metric structure, which differs from the metric one, is that the associated (0,2)-tensor of the B-metric is also a B-metric. Consequently, this pair of B-metrics generates a pair of Levi-Civita connections.

%A counterpart of the almost contact metric structure is the. The B-metric (unlike the compatible metric) restricted on the contact distribution is a Norden metric, i.e. the structure endomorphism acts as an antiisometry (cf. an isometry for the compatible metric) on the contact distribution.
%Other important characteristic of almost contact B-metric structure which differs it from the metric one is that the associated (0,2)-tensor of the B-metric is also a B-metric. This pair of B-metrics generates a pair of Levi-Civita connections.

The present section is organized as follows.
In Subsection~\thesection.1 we introduce a pair of Schou\-ten-van Kampen connections
which is associated to the pair of Levi-Civita connections and adapted to the contact distribution of an almost contact B-metric manifold. Then we determine conditions these connections to coincide and to be natural for the corresponding structures.
In Subsection~\thesection.2 we study basic properties of the potentials and the torsions of the pair of the constructed connections.
Finally, in Subsection~\thesection.3 we give some curvature properties of the studied connections.

%

%
%Then, we characterize the classes of considered manifolds using these connections and obtain some corresponding curvature properties.

\vskip 0.2in \addtocounter{subsection}{1} \setcounter{subsubsection}{0}

\noindent  {\Large\bf \thesubsection. Remarkable metric connections regarding the contact distribution on the considered manifolds}%\\[6pt]\vskip2pt}

\vskip 0.15in
%\section{Remarkable metric connections regarding the contact distribution on the considered manifolds}

Let us consider the horizontal and the vertical distributions $\HH$ and $\VV$
in the tangent bundle $T\MM$ on an arbitrary almost contact B-metric manifold $\M$ given in \eqref{HV} and \eqref{HHVV}.
Further, we use the corresponding projections $x^{\mathrm{h}}$ and $x^{\mathrm{v}}$ of an arbitrary vector field $x$ in $T\MM$ bearing in mind \eqref{hv} and \eqref{Xhv}.

\vskip 0.2in \addtocounter{subsubsection}{1}

\noindent  {\Large\bf{\emph{\thesubsubsection. Schouten-van Kampen connection $\nSvK$ associated to $\n$}}}

\vskip 0.15in
%\subsection{The Schouten-van Kampen connection $\DDD$ associated to $\n$}

Let us consider the Schouten-van Kampen connection $\nSvK$ associated to the Levi-Civita connection $\n$
and adapt\-ed to the pair $(\HH, \VV)$. This connection is defined (locally in \cite{SvK}, see also \cite{Ia}) by
\begin{equation}\label{SvK}
  \nSvK_x y = (\n_x y^{\mathrm{h}})^{\mathrm{h}} + (\n_x y^{\mathrm{v}})^{\mathrm{v}}.
\end{equation}
The latter equality implies the parallelism of $\HH$ and $\VV$ with
respect to $\nSvK$.
From \eqref{Xhv} we obtain
\[
\begin{split}
&(\n_x y^{\mathrm{h}})^{\mathrm{h}}=\n_x y -\eta(y)\n_x \xi-\eta(\n_x y)\xi, \qquad\\[6pt]
&(\n_x y^{\mathrm{v}})^{\mathrm{v}}=\eta(\n_x y)\xi+(\n_x \eta)(y)\xi.
\end{split}
\]
Then we get the expression of the Schouten-van Kampen connection in terms of $\n$ as follows (cf. \cite{Sol})
\begin{equation}\label{SvK=n}
  \nSvK_x y = \n_x y -\eta(y)\n_x \xi+(\n_x \eta)(y)\xi.
\end{equation}

According to \eqref{SvK=n}, the potential $Q^{\circ}$ of $\nSvK$ with respect to $\n$ and the torsion $T^{\circ}$ of $\nSvK$,  defined by $Q^{\circ}(x,y)=\nSvK_xy-\n_xy$ and $T^{\circ}(x,y)= \nSvK_xy-\nSvK_yx-[x,y]$, respectively, have the following form
\begin{equation}\label{Q}
Q^{\circ}(x,y)=-\eta(y)\n_x \xi+(\n_x \eta)(y)\xi,
\end{equation}
\begin{equation}\label{T}
T^{\circ}(x,y)=\eta(x)\n_y \xi-\eta(y)\n_x \xi+\D\eta(x,y)\xi.
\end{equation}

\begin{thm}\label{thm:D-T}
The Schouten-van Kampen connection $\nSvK$ %associated to $\n$ and adapt\-ed to the pair of distributions  $(\HH,\VV)$
 is the unique affine connection having a torsion of the form \eqref{T} and preserving the structures $\xi$, $\eta$ and the metric $g$.
\end{thm}
\begin{proof}
Taking into account \eqref{SvK=n}, we compute directly that the structures $\xi$, $\eta$ and $g$ are parallel with respect to $\nSvK$, i.e. $\nSvK\xi=\nSvK\eta=\nSvK g=0$.
The connection $\nSvK$ preserves the metric and therefore is
completely determined by its torsion $T^{\circ}$.
According to  \cite{Car25}, the two spaces of all torsions and of all potentials
are isomorphic and the bijection is given by \eqref{TQ} and \eqref{QT}.
%\begin{gather}
%T (x,y,z) = Q(x,y,z) - Q(y,x,z) ,\label{TQ}\\[6pt]
%2Q(x,y,z) = T (x,y,z) - T (y,z,x) + T(z,x,y).\label{QT}
%\end{gather}

Then, %by virtue of the latter two relations,
the connection $\nSvK$ determined by \eqref{SvK=n} and its potential $Q^{\circ}$ given in \eqref{Q} are replaced in \eqref{TQ} to determine its torsion $T^{\circ}$ and the result is \eqref{T}. Vice versa, the form of $T^{\circ}$ in \eqref{T} yields by \eqref{QT} the equality for $\nSvK$ in \eqref{SvK=n}.
\end{proof}

Obviously, the connection $\nSvK$ exists on $(\MM,\f,\xi,\eta,g)$ in each class, but $\nSvK$ coincides with $\n$ if and only if the condition \[\eta(y)\n_x \xi-(\n_x \eta)(y)\xi=0\] holds. The latter equality is equivalent to vanishing of $\n_x \xi$ for each $x$. This condition is satisfied only for the manifolds belonging to the class $\F_1\oplus\F_2\oplus\F_3\oplus\F_{10}$, according to \propref{prop-nxi=0}. Let us denote this class briefly by $\U_1$, i.e.
\[
\U_1=\F_1\oplus\F_2\oplus\F_3\oplus\F_{10}.
\]
Thus, we prove the following
\begin{thm}\label{thm:D=n}
The Schouten-van Kampen connection $\nSvK$ %associated to $\n$ and adapt\-ed to the pair $(\HH,\VV)$ on $(\MM,\f,\xi,\eta,g)$
coincides with $\n$ if and only if $(\MM,\f,\xi,\allowbreak{}\eta,g)$ belongs to the class $\U_1$.
\end{thm}

\vskip 0.2in \addtocounter{subsubsection}{1}

\noindent  {\Large\bf{\emph{\thesubsubsection. The conditions $\nSvK$ to be natural for $(\f,\xi,\eta,g)$}}}

\vskip 0.15in
%\subsection{The conditions $\nSvK$ to be natural for $(\f,\xi,\eta,g)$}

Using \eqref{SvK=n}, we express the covariant derivative of $\f$ as follows
\begin{equation}\label{Df}
  (\nSvK_x\f)y=(\n_x\f)y+\eta(y)\f\n_x\xi-\eta(\n_x\f y)\xi.
\end{equation}
Therefore, $\nSvK\f=0$ if and only if \[(\n_x\f)y=-\eta(y)\f\n_x\xi+\eta(\n_x\f y)\xi,\] which yields the following equality by \eqref{F-prop}
\begin{equation}\label{F_D=0}
  F(x,y,z)=F(x,y,\xi)\eta(z)+F(x,z,\xi)\eta(y).
\end{equation}
According to \eqref{Fi} and \cite{Man8}, the latter condition determines the direct sum $\F_4\oplus\cdots\oplus\F_9\oplus\F_{11}$,  which we denote by $\U_2$ for the sake of brevity, i.e.
\[
\U_2=\F_4\oplus\cdots\oplus\F_9\oplus\F_{11}.
\]
Thus, we find the kind of the considered manifolds where $\nSvK$ is a natural connection, i.e. the tensors of the structure $(\f,\xi,\eta,g)$ are covariantly constant regarding $\nSvK$.
In this case it follows that $(\n_x\f)\f y=\left(\n_x\eta\right)(y)\cdot\xi$ holds. Then the Schouten-van Kampen connection $\nSvK$ coincides with the $\f$B-con\-nect\-ion defined by \eqref{fB0}.
Such a way, we establish the truthfulness of the following
\begin{thm}\label{thm:D-nat}
The Schouten-van Kampen connection $\nSvK$ %associated to $\n$ and adapt\-ed to the pair $(\HH,\VV)$
is a natural connection for the structure $(\f,\xi,\eta,g)$ if and only if $(\MM,\f,\xi,\eta,g)$ belongs to the class $\U_2$. In this case $\nSvK$ coincides with the $\f$B-connection.
\end{thm}

Let us remark that in the case when $(\MM,\f,\xi,\eta,g)$ belongs to a class which has a nonzero component in  both of the direct sums
$\U_1$ and $\U_2$, then
the connection $\nSvK$ is not a natural connection and it does not coincide with $\n$.
Then the class of all almost contact B-metric manifolds can be decomposed orthogonally to $\U_1\oplus\U_2$.

\vskip 0.2in \addtocounter{subsubsection}{1}

\noindent  {\Large\bf{\emph{\thesubsubsection. Schouten-van Kampen connection $\tnSvK$ associated to $\nn$}}}

\vskip 0.15in
%\subsection{The Schouten-van Kampen connection $\tnSvK$ associated to $\nn$}

%In this subsection we discuss on the condition the connection $\nSvK$ to preserve

In a similar way as for $\nSvK$, let us consider the Schouten-van Kampen connection $\tnSvK$ associated to the Levi-Civita connection $\nn$ for $\tg$
and adapt\-ed to the pair $(\HH, \VV)$. We define this connection as follows
\[
  \tnSvK_x y = (\nn_x y^{\mathrm{h}})^{\mathrm{h}} + (\nn_x y^{\mathrm{v}})^{\mathrm{v}}.
\]
Then the hyperdistribution $(\HH, \VV)$ is parallel with
respect to $\tnSvK$, too.
Analogously, we express the Schouten-van Kampen connection $\tnSvK$ in terms of $\nn$ by
\begin{equation}\label{tSvK=n}
  \tnSvK_x y = \nn_x y -\eta(y)\nn_x \xi+(\nn_x \eta)(y)\xi.
\end{equation}

By virtue of \eqref{tSvK=n}, the potential $\widetilde{Q}^{\circ}$ of $\tnSvK$ with respect to $\nn$ and the torsion $\widetilde{T}^{\circ}$ of $\tnSvK$ have the following form
\begin{gather}\label{tQ}
\widetilde{Q}^{\circ}(x,y)=-\eta(y)\nn_x \xi+(\nn_x \eta)(y)\xi,
\\[6pt]
\label{tT}
\widetilde{T}^{\circ}(x,y)=\eta(x)\nn_y \xi-\eta(y)\nn_x \xi+\D\eta(x,y)\xi.
\end{gather}

Similarly to \thmref{thm:D-T} we have the following
\begin{thm}\label{thm:tD-tT}
The Schouten-van Kampen connection $\tnSvK$ %associated to $\nn$ and adapt\-ed to the pair of distributions  $(\HH,\VV)$
is the unique affine connection having a torsion of the form \eqref{tT} and preserving the structures $\xi$, $\eta$ and the associated metric $\tg$.
\end{thm}

It is clear that the connection $\tnSvK$ exists on $(\MM,\f,\xi,\eta,\tg)$ in each class, but $\tnSvK$ coincides with $\nn$ if and only if the condition
\[
\eta(y)\nn_x \xi-(\nn_x \eta)(y)\xi=0
\]
is valid or equivalently $\nn \xi$ vanishes.
Similarly to \propref{prop-nxi=0}, we prove that the condition $\nn \xi=0$ holds if and only if $\widetilde{F}$ satisfies the condition \eqref{Fi} of $F$ for %
$\F_1\oplus\F_2\oplus\F_3\oplus\F_{9}$, which we denote by $\widetilde{\U}_1$. Thus, we prove the following
\begin{thm}\label{thm:tD=nn}
The Schouten-van Kampen connection $\tnSvK$ %associated to $\nn$ and adapt\-ed to the pair $(\HH,\VV)$ on $(\MM,\f,\xi,\eta,\tg)$
coincides with $\nn$ if and only if $(\MM,\f,\xi,\allowbreak{}\eta,\tg)$ belongs to the class $\widetilde{\U}_1$.
\end{thm}

Taking into account \eqref{tFF}, we establish immediately %the truthfulness of
\begin{lem}\label{lem:U1}
The manifold $(\MM,\f,\xi,\eta,g)$ belongs to $\U_1$ if and only if the manifold $(\MM,\f,\xi,\eta,\tg)$ belongs to $\widetilde{\U}_1$.
\end{lem}

Then, \thmref{thm:D=n}, \thmref{thm:tD=nn} and \lemref{lem:U1} imply the following
\begin{thm}\label{thm:D=ntD=nn}
    Let $\nSvK$ and $\tnSvK$ be the Schouten-van Kampen connections associated to $\n$ and $\nn$ and adapted to the pair $(\HH,\VV)$ on $(\MM,\f,\xi,\eta,g)$ and $(\MM,\f,\xi,\eta,\tg)$, respectively. Then the following assertions are equivalent:
\begin{enumerate}
  \item $\nSvK$ coincides with $\n$;
  \item $\tnSvK$ coincides with $\nn$;
  \item $(\MM,\f,\xi,\eta,g)$ belongs to $\U_1$;
  \item $(\MM,\f,\xi,\eta,\tg)$ belongs to $\widetilde{\U}_1$.
\end{enumerate}
\end{thm}
\begin{cor}\label{cor:D=ntD=nn}
Let $\nSvK$ and $\tnSvK$ be the Schouten-van Kampen connections associated to $\n$ and $\nn$ and adapted to the pair $(\HH,\VV)$ on $(\MM,\f,\xi,\eta,g)$ and $(\MM,\f,\xi,\eta,\tg)$, respectively. If $\tnSvK\equiv\n$ or $\nSvK\equiv\nn$ then the four connections $\nSvK$, $\tnSvK$, $\n$ and $\nn$ coincide. The coinciding of $\nSvK$, $\tnSvK$, $\n$ and $\nn$ is equivalent to the condition $(\MM,\f,\xi,\eta,g)$ and $(\MM,\f,\xi,\eta,\tg)$ to be cosymplectic B-metric manifolds.
\end{cor}

We obtain from \eqref{tSvK=n} the following relation between $\nSvK$ and $\tnSvK$
\begin{equation}\label{tD=D}
  \tnSvK_x y = \nSvK_x y + \Phi(x,y) -\eta(\Phi(x,y))\xi -\eta(y)\Phi(x,\xi).
\end{equation}

It is clear that $\tnSvK= \nSvK$ if and only if \[\Phi(x,y)=\eta(\Phi(x,y))\xi +\eta(y)\Phi(x,\xi)\] which is equivalent to \[\Phi(x,y)=\eta(\Phi(x,y))\xi +\eta(x)\eta(y)\Phi(\xi,\xi)\] because $\Phi$ is symmetric.
Using relation \eqref{FPhi},
we obtain  condition \eqref{F_D=0} which determines the class $\U_2$.
Then, the following assertion is valid.
\begin{thm}\label{thm:tD=D}
The Schouten-van Kampen connections $\tnSvK$ and $\nSvK$ associated to $\nn$ and $\n$, respectively, and adapt\-ed to the pair  $(\HH,\VV)$
coincide with each other if and only if the manifold belongs to the class $\U_2$.
\end{thm}

\vskip 0.2in \addtocounter{subsubsection}{1}

\noindent  {\Large\bf{\emph{\thesubsubsection. The conditions $\tnSvK$ to be natural for $(\f,\xi,\eta,\tg)$}}}

\vskip 0.15in
%\subsection{The connection $\tnSvK$ to be natural for $(\f,\xi,\eta,\tg)$}

Using \eqref{tD=D}, we have the following relation between the covariant derivatives of $\f$ regarding $\tnSvK$ and $\nSvK$
\begin{equation}\label{tDfiDfi}
\begin{split}
(\tnSvK_x\f)y=(\nSvK_x\f)y &+\Phi(x,\f y)-\f\Phi(x,y)\\[6pt]
                    &+\eta(y)\f\Phi(x,\xi)-\eta(\Phi(x,\f y))\xi.
\end{split}
\end{equation}
By virtue of the latter equality, we establish that $\tnSvK\f$ and $\nSvK{\f}$ coincide if and only if the condition
\[\Phi(x,\f^2 y,\f^2 z)=-\Phi(x,\f y,\f z)\] holds.
Then, using the classification of almost contact B-metric manifolds regarding $\Phi$, given in \cite{NakGri93}, the latter condition is satisfied only when $(\MM,\f,\xi,\eta,g)$ is in the class $\F_3\oplus\U_3$, where $\U_3$ denotes the direct sum $\F_4\oplus\F_5\oplus\F_6\oplus\F_7\oplus\F_{11}$. By direct computations we establish that $(\MM,\f,\xi,\eta,\tg)$ belongs to the same class. Therefore, we obtain
\begin{thm}\label{thm:tDfi=Dfi}
The covariant derivatives of $\f$ with respect to the Schou\-ten-van Kampen connections $\nSvK$ and $\tnSvK$ % associated to $\n$ and $\nn$, respectively, and adapted to the pair $(\HH,\VV)$
coincide if and only if both the manifolds $(\MM,\f,\xi,\allowbreak{}\eta,\allowbreak{}g)$ and $(\MM,\f,\xi,\eta,\tg)$  belong to the class $\F_3\oplus\U_3$. %Then $\nSvK$ coincides with the $\f$B-connection.
%on an almost contact B-metric manifold $(\MM,\f,\xi,\eta,g,\tg)$ ???\\[6pt]
\end{thm}

Using \eqref{PhiF}, \eqref{Df} and \eqref{tDfiDfi}, we obtain that $\tnSvK\f=0$ is equivalent to the condition
\[
F(\f y,\f z,x)+F(\f^2 y,\f^2 z,x)-F(\f z,\f y,x)-F(\f^2 z,\f^2 y,x)=0.
\]
Then, by virtue of \eqref{Fi} we get the following
\begin{thm}\label{thm:tD-nat}
The Schouten-van Kampen connection $\tnSvK$ %associated to $\nn$ and adapt\-ed to $(\HH,\VV)$
is a natural connection for the structure $(\f,\xi,\eta,\tg)$ if and only if $(\MM,\f,\xi,\eta,\tg)$ belongs to the class $\F_1\oplus\F_2\oplus\U_3$. %Then $\nSvK$ coincides with the $\f$B-connection.
%on an almost contact B-metric manifold $(\MM,\f,\xi,\eta,g,\tg)$ ???\\[6pt]
\end{thm}

Consequently, bearing in mind \thmref{thm:D-nat}, \thmref{thm:tDfi=Dfi}, \thmref{thm:tD-nat}, we have the validity of the following
\begin{thm}\label{thm:DtD-nat}
The Schouten-van Kampen connections $\nSvK$ and $\tnSvK$ %associated to $\n$ and $\nn$, respectively, and adapted to  $(\HH,\VV)$
are natural connections on $(\MM,\f,\xi,\eta,g,\tg)$ if and only if
$(\MM,\f,\xi,\eta,g)$ and $(\MM,\f,\xi,\eta,\tg)$ belong to the class $\U_3$.
\end{thm}

\vskip 0.2in \addtocounter{subsection}{1} \setcounter{subsubsection}{0}

\noindent  {\Large\bf \thesubsection. Properties of the potentials and the torsions of the pair of connections $\nSvK$ and $\tnSvK$}%\\[6pt]\vskip2pt}

\vskip 0.15in
%\section{Torsion properties of the potentials and the torsions of the pair of connections $\nSvK$ and $\tnSvK$}

Bearing in mind %the relations between $T^{\circ}$, $Q^{\circ}$ and $S$ given in
\eqref{Q}, \eqref{T}, \eqref{TQ}, \eqref{QT} and \eqref{Sx}, we establish that the properties of the torsion, the potential and the shape operator for $\nSvK$ are related.
Analogously, similar linear relations between the respective torsion, potential and shape operator for $\tnSvK$ are valid.

According to the expressions \eqref{Q} and \eqref{T} of $Q^{\circ}$ and $T^{\circ}$, respectively, their horizontal and vertical components have the following form
\begin{equation}\label{QThv}
\begin{array}{ll}
Q^{\circ\mathrm{h}}=-(\n\xi)\otimes\eta,\qquad & Q^{\circ\mathrm{v}}=(\n\eta)\otimes\xi,
\qquad\\[6pt]
T^{\circ\mathrm{h}}=\eta\wedge(\n\xi),\qquad & T^{\circ\mathrm{v}}=\D\eta\otimes\xi.
\end{array}
\end{equation}
Then, the corresponding (0,3)-tensors
\[
Q^{\circ}(x,y,z)=g(Q^{\circ}(x,y),z),\qquad T^{\circ}(x,y,z)=g(T^{\circ}(x,y),z)
\]
are expressed in terms of $S$ as follows
\[
\begin{split}
&Q^{\circ}(x,y,z)=%g(S(x),z)\eta(y)-g(S(x),y)\eta(z)=
-\pi_1(\xi,S(x),y,z),
\qquad\\[6pt]
&T^{\circ}(x,y,z)=%-g(S(y),z)\eta(x)+g(S(x),z)\eta(y)-\{g(S(x),y)-g(S(y),x)\}\eta(z),
%\\[6pt]
%T^{\circ}(x,y,z)=Q^{\circ}(x,y,z)-Q^{\circ}(y,x,z)=
-\pi_1(\xi,S(x),y,z)+\pi_1(\xi,S(y),x,z),
\end{split}
\]
where it is denoted
\begin{equation}\label{pi1}
\pi_1(x,y,z,w)=g(y,z)g(x,w)-g(x,z)g(y,w).
\end{equation}
Moreover, their horizontal and vertical components have the form
\begin{equation}\label{QTShv}
\begin{array}{ll}
Q^{\circ\mathrm{h}}=S\otimes\eta,\qquad\quad & Q^{\circ\mathrm{v}}=-S^{\flat}\otimes\xi,
\qquad\\[6pt]
T^{\circ\mathrm{h}}=-\eta\wedge S,\qquad\quad & T^{\circ\mathrm{v}}=-2\Alt(S^{\flat})\otimes\xi,
\end{array}
\end{equation}
where we use $S^{\flat}(x,y)=g(S(x),y)$ and $\Alt$ for the alternation.

By virtue of the equalities for the vertical components of $Q^{\circ}$ and $T^{\circ}$ in \eqref{QThv} and \eqref{QTShv}, we obtain immediately
\begin{thm}\label{thm:equiv}
The following equivalences are valid:
\begin{enumerate}
  \item
        $\n\eta$ is symmetric                       $\Leftrightarrow$
        $\eta$ is closed %, i.e. $\D\eta=0$
                                    $\Leftrightarrow$
        $Q^{\circ\mathrm{v}}$ is symmetric               $\Leftrightarrow$
        $T^{\circ\mathrm{v}}$ vanishes                              $\Leftrightarrow$
        $S$ is self-adjoint regarding $g$       $\Leftrightarrow$
        $S^{\flat}$ is symmetric                 $\Leftrightarrow$
        $\MM\in\U_1\oplus\F_4\oplus\F_5\oplus\F_6\oplus\F_9$;
  \item
        $\n\eta$ is skew-symmetric                  $\Leftrightarrow$
        $\xi$ is Killing with respect to $g$ %, i.e. $\LL_{\xi}g=0$
               $\Leftrightarrow$
        $Q^{\circ\mathrm{v}}$ is skew-symmetric                     $\Leftrightarrow$
        $S$ is anti-self-adjoint regarding $g$     $\Leftrightarrow$
        $S^{\flat}$ is skew-symmetric           $\Leftrightarrow$
        $\MM\in\U_1\oplus\F_7\oplus\F_8$;
  \item
        $\n\eta=0$                                  $\Leftrightarrow$
        $\D\eta=\LL_{\xi}g=0$                      $\Leftrightarrow$
        $\n\xi=0$                                   $\Leftrightarrow$
        $S=0$                                       $\Leftrightarrow$
        $S^{\flat}=0$                       \par                $\Leftrightarrow$
        $\nSvK=\n$                                      $\Leftrightarrow$
        $\MM\in\U_1$.
\end{enumerate}
\end{thm}

The horizontal and vertical components of $\widetilde{Q}^{\circ}$ and $\widetilde{T}^{\circ}$ of $\tnSvK$ are respectively
\begin{equation}\label{tQThv}
\begin{array}{ll}
\widetilde{Q}^{\circ\mathrm{h}}=-(\nn\xi)\otimes\eta,\qquad\quad & \widetilde{Q}^{\circ\mathrm{v}}=(\nn\eta)\otimes\xi,
\qquad\\[6pt]
\widetilde{T}^{\circ\mathrm{h}}=\eta\wedge(\nn\xi),\qquad\quad & \widetilde{T}^{\circ\mathrm{v}}=\D\eta\otimes\xi.
\end{array}
\end{equation}

From $\tg(\xi,\xi)=1$ we have $\tg(\nn_x\xi,\xi)=0$ and therefore $\nn\xi\in\HH$.
The shape operator $\tS:\HH\mapsto\HH$ for the metric $\tg$ is defined by $\tS(x)=-\nn_x\xi$.

Since we have
\[
(\nn_x \eta)(y)=(\n_x \eta)(y)-\eta(\Phi(x,y)),\qquad \nn_x \xi=\n_x \xi +\Phi(x,\xi),
\]
then we obtain
\begin{equation}\label{tSSPhi}
\tS(x)=S(x)-\Phi(x,\xi),\qquad \tS^{\flat}(x,y)=S^{\flat}(x,\f y)-\Phi(\xi,x,\f y),
\end{equation}
where $\tS^{\flat}(x,y)=\tg(\tS(x),y)$. % and $S^{\flat}(x,y)=g(S(x),y)$ is denoted.

Moreover,
\eqref{Q}, \eqref{T}, \eqref{QThv}, \eqref{tQ}, \eqref{tT} and \eqref{tQThv} imply the following relations
\begin{gather*}
\begin{split}
&\widetilde{Q}^{\circ}(x,y)=Q^{\circ}(x,y)-\eta(y)\Phi(x,\xi)-\eta(\Phi(x,y))\xi,
\qquad\\[6pt]
&\widetilde{T}^{\circ}(x,y)=T^{\circ}(x,y)+\eta(x)\Phi(y,\xi)-\eta(y)\Phi(x,\xi);
\\[6pt]
&\widetilde{Q}^{\circ\mathrm{h}}=Q^{\circ\mathrm{h}}-(\xi\,\lrcorner\,\Phi)\otimes\eta,\qquad  \widetilde{Q}^{\circ\mathrm{v}}=Q^{\circ\mathrm{v}}-(\eta\circ\Phi)\otimes\xi,
\qquad\\[6pt]
&\widetilde{T}^{\circ\mathrm{h}}=T^{\circ\mathrm{h}}+\eta\wedge(\xi\,\lrcorner\,\Phi),\qquad  \widetilde{T}^{\circ\mathrm{v}}=T^{\circ\mathrm{v}}.
\end{split}
\end{gather*}

Using the latter equalities and \eqref{tSSPhi}, we obtain the following for\-mu\-lae
\begin{gather*}
\begin{split}
&\widetilde{Q}^{\circ}=Q^{\circ}+(\tS-S)\otimes\eta-(\tS^{\flat}-S^{\flat})\otimes\xi,
\qquad\\[6pt]
&\widetilde{T}^{\circ}=T^{\circ}+(\tS-S)\wedge\eta;
%\\[6pt]
\end{split}
\end{gather*}
\begin{gather*}
\begin{split}
&\widetilde{Q}^{\circ\mathrm{h}}=Q^{\circ\mathrm{h}}+(\tS-S)\otimes\eta,\qquad \widetilde{Q}^{\circ\mathrm{v}}=Q^{\circ\mathrm{v}}-(\tS^{\flat}-S^{\flat})\otimes\xi,
\qquad\\[6pt]
&\widetilde{T}^{\circ\mathrm{h}}=T^{\circ\mathrm{h}}+(\tS-S)\wedge\eta,\qquad \widetilde{T}^{\circ\mathrm{v}}=T^{\circ\mathrm{v}}.
\end{split}
\end{gather*}

\begin{thm}\label{thm:equiv-t}
The following equivalences are valid:
\begin{enumerate}
  \item
        $\nn\eta$ is symmetric                       $\Leftrightarrow$
        $\eta$ is closed                              $\Leftrightarrow$
        $\widetilde{Q}^{\circ\mathrm{v}}$ is symmetric                          $\Leftrightarrow$
        $\widetilde{T}^{\circ\mathrm{v}}$ vanishes
        \par $
        \Leftrightarrow$
        $\tS$ is self-adjoint regarding $\tg$           $\Leftrightarrow$
        $\tS^{\flat}$ is symmetric             \par
        $\Leftrightarrow$
        $(\MM,\f,\xi,\eta,\tg)\in\widetilde\U_1\oplus\F_4\oplus\F_5\oplus\F_6\oplus\F_{10}$;
  \item
        $\nn\eta$ is skew-symmetric                  $\Leftrightarrow$
        $\xi$ is Killing with respect to $\tg$ %, i.e. $\LL_{\xi}\tg=0$
        \par
        $\Leftrightarrow$
        $\widetilde{Q}^{\circ\mathrm{v}}$ is skew-symmetric      $\Leftrightarrow$
        $\tS$ is anti-self-adjoint regarding $\tg$   \par
        $\Leftrightarrow$
        $\tS^{\flat}$ is skew-symmetric           $\Leftrightarrow$
        $(\MM,\f,\xi,\eta,\tg)\in\widetilde\U_1\oplus\F_7$;
  \item
        $\nn\eta=0$                                  $\Leftrightarrow$
        $\D\eta=\LL_{\xi}\tg=0$                     $\Leftrightarrow$
        $\nn\xi=0$                                   $\Leftrightarrow$
        $\tS=0$                                       $\Leftrightarrow$
        $\tS^{\flat}=0$               \par                     $\Leftrightarrow$
        $\tnSvK=\nn$                                      $\Leftrightarrow$
        $(\MM,\f,\xi,\eta,\tg)\in\widetilde\U_1$.
\end{enumerate}
\end{thm}

\vskip 0.2in \addtocounter{subsection}{1} \setcounter{subsubsection}{0}

\noindent  {\Large\bf \thesubsection. Curvature properties of the pair of connections $\nSvK$ and $\tnSvK$}%\\[6pt]\vskip2pt}

\vskip 0.15in
%\section{Curvature properties of the pair of connections $\nSvK$ and $\tnSvK$}

\label{curvR}
On an almost contact B-metric manifold $\M$, let $R$ be the curvature tensor of $\n$, i.e. $R=[\n, \n]
- \n_{[\ ,\ ]}$, and the corresponding $(0,4)$-tensor is
determined by $R(x,y,z,w)=g(R(x,y)z,w)$. The Ricci tensor
$\rho$ and the scalar curvature $\tau$  are defined as usual by
$\rho(y,z)=g^{ij}R(e_i,y,z,e_j)$ and
$\tau=g^{ij}\rho(e_i,e_j)$, where $g^{ij}$ are the components of the inverse matrix of $g$ regarding an arbitrary basis $\{e_i\}$ ($i=1,\dots, 2n+1$)
of $T_p\MM$, $p\in \MM$.

Each non-degenerate 2-plane $\al$ in
$T_p\MM$ with respect to $g$ and $R$ has the following sectional curvature
%\begin{equation}\label{sect}
\[
k(\al;p)=\dfrac{R(x,y,y,x)}{\pi_1(x,y,y,x)},
\] %\end{equation}
where $\{x,y\}$ is an arbitrary basis of $\al$.

A 2-plane $\al$ is said to be a \emph{$\xi$-section}, a \emph{$\f$-holomorphic section} or a \emph{$\f$-totally real section}
if $\xi \in \al$, $\al= \f\al$ or $\al\bot \f\al$ regarding $g$, respectively. The latter type of sections exist only for $\dim \MM\geq 5$.

In \cite{Man33}, some curvature properties with respect to $\n$ are studied in several subclasses of $\U_2$.

Let us denote the curvature tensor, the Ricci tensor, the scalar curvature and the sectional curvature of the connection $\nSvK$ by $R^{\circ}$, $\rho^{\circ}$, $\tau^{\circ}$ and $k^{\circ}$, respectively. The corresponding $(0,4)$-tensor of $R^{\circ}$, denoted by the same letter, as well as $\rho^{\circ}$, $\tau^{\circ}$ and $k^{\circ}$ are determined by $g$.

Analogously,
let the corresponding quantities for the connections $\tn$ and $\tnSvK$ be denoted by $\tR$, $\widetilde\rho$, $\widetilde\tau$, $\widetilde{k}$ and $\tR^{\circ}$, $\widetilde{\rho}^{\circ}$, $\widetilde{\tau}^{\circ}$, $\widetilde{k}^{\circ}$, respectively.
The corresponding $(0,4)$-tensors of $\tR$ and $\tR^{\circ}$, denoted by the same letters, as well as $\widetilde\rho$, $\widetilde\tau$, $\widetilde{k}$ and $\widetilde{\rho}^{\circ}$, $\widetilde{\tau}^{\circ}$, $\widetilde{k}^{\circ}$
are obtained by $\tg$.

\begin{thm}\label{thm:KR}
The curvature tensors of %the Schouten-van Kampen connection
$\nSvK$ and %of the Levi-Civita connection
$\n$ (respectively, of $\tnSvK$ and $\tn$) are related as follows
\begin{equation}\label{RDR}
\begin{array}{l}
  R^{\circ}(x,y,z,w)=R\left(x,y,\f^2z,\f^2w\right)+\pi_1\bigl(S(x),S(y),z,w\bigr),\\[6pt]
  \tR^{\circ}(x,y,z,w)=\tR\left(x,y,\f^2z,\f^2w\right)+\widetilde\pi_1\bigl(\tS(x),\tS(y),z,w\bigr),
\end{array}
\end{equation}
where $\widetilde\pi_1$ is constructed as in \eqref{pi1} by $\tg$.
\end{thm}
\begin{proof}
Using \eqref{SvK=n}, we compute $R^{\circ}$. Taking into account that $g(\n_x\xi,\xi)$ vanishes for each $x$ and $\nSvK\xi=0$, we obtain the equality
\[
\begin{array}{l}
  R^{\circ}(x,y)z=R(x,y)z-\eta(z)R(x,y)\xi-\eta(R(x,y)z)\xi\\[6pt]
  \phantom{R^{\circ}(x,y)z=R(x,y)z}
  -g\left(\n_x\xi,z\right)\n_y\xi+g\left(\n_y\xi,z\right)\n_x\xi.
\end{array}
\]
The latter equality implies the first relation in \eqref{RDR}.

The second equality in \eqref{RDR} follows as above but in terms of $\tnSvK$, $\tn$ and their corresponding metric $\tg$.
\end{proof}

Then, \thmref{thm:KR} has the following
\begin{cor}\label{cor:Ric}
The Ricci tensors of %the Schouten-van Kampen connection
$\nSvK$ and
%of the Levi-Civita connection
$\n$ (respectively, of $\tnSvK$ and $\tn$) are related as follows
\begin{equation}\label{RicDRic}
\begin{array}{l}
  \rho^{\circ}(y,z)=\rho(y,z)-\eta(z)\rho(y,\xi)-R(\xi,y,z,\xi)\\[6pt]
  \phantom{\rho^{\circ}(y,z)=\rho(y,z)}
%  -g\left(\n_{\n_y\xi}\xi,z\right)+\ta^*(\xi)g(\n_y\xi,z).\\[6pt]
-g(S(S(y)),z)+\tr(S)g(S(y),z),\\[6pt]
  \widetilde{\rho}^{\circ}(y,z)=\widetilde\rho(y,z)-\eta(z)\widetilde\rho(y,\xi)-\tR(\xi,y,z,\xi)
%  -g\left(\n_{\n_y\xi}\xi,z\right)+\ta^*(\xi)g(\n_y\xi,z).
\\[6pt]
\phantom{  \widetilde{\rho}^{\circ}(y,z)=\widetilde\rho(y,z)}
-\tg(\tS(\tS(y)),z)+\widetilde\tr(\tS)\tg(\tS(y),z),
\end{array}
\end{equation}
where $\widetilde\tr$ denotes the trace with respect to $\tg$.
\end{cor}

Let us remark that we have
$
\tr(S)=\widetilde\tr(\tS)=-\Div(\eta),
$
because of \eqref{divtr}, the definitions of $S$ and $\tS$ as well as $g^{ij}\Phi(\xi,e_i,e_j)=0$, using \eqref{PhiF} and \eqref{F-prop}.

From the definition of the shape operator we get
\[
R(x,y)\xi =-\left(\n_x S\right)y+\left(\n_y S\right)x.
\]
Then, the latter formula and $S(\xi)=-\n_{\xi} \xi=-\f\omega^{\sharp}$ lead to the following expression of one of the components in the right-hand side of the first equality in \eqref{RicDRic}
\[
\begin{split}
R(\xi,y,z,\xi) &= g\bigl(\left(\n_{\xi} S\right)y-\left(\n_y S\right)\xi,z\bigr)\\[6pt]
&=g\bigl(\left(\n_{\xi} S\right)y-\n_y S(\xi)- S(S(y)),z\bigr).
\end{split}
\]
Therefore, taking the trace of the latter equalities and using the relations
\[
\Div(\omega\circ\f)=g^{ij}\left(\n_{e_i} \omega\circ\f\right)e_j=g^{ij}g\left(\n_{e_i}\f \omega^{\sharp},e_j\right)=-\Div(S(\xi))
\]
for the divergence of the 1-form $\omega\circ\f$,
we obtain
\begin{equation}\label{ricxixi}
\rho(\xi,\xi)= \tr(\n_{\xi}S)-\Div(S(\xi))-\tr(S^2).
\end{equation}

Similar equalities for the quantities with a tilde are valid with respect to $\tg$, i.e.
\begin{equation}\label{tricxixi}
\widetilde\rho(\xi,\xi)= \widetilde\tr(\nn_{\xi}\tS)-\widetilde\Div(\tS(\xi))-\widetilde\tr(\tS^2).
\end{equation}

Bearing in mind the latter computations, from \corref{cor:Ric} we obtain the following
\begin{cor}\label{cor:tau}
The scalar curvatures of %the Schouten-van Kampen connection
$\nSvK$ and %of the Levi-Civita connection
$\n$ (respectively, of $\tnSvK$ and $\tn$) are related as follows
\begin{equation*}\label{tauDtau}
\begin{array}{l}
  \tau^{\circ}=\tau-2\rho(\xi,\xi)-\tr(S^2) + (\tr(S))^2,\qquad\\[6pt]
  \widetilde{\tau}^{\circ}=\widetilde\tau-2\widetilde\rho(\xi,\xi)-\widetilde\tr(\tS^2) + (\widetilde\tr(\tS))^2,
\end{array}
\end{equation*}
where $\rho(\xi,\xi)$ and $\widetilde\rho(\xi,\xi)$ are expressed by $S$ and $\tS$ in \eqref{ricxixi} and \eqref{tricxixi}, respectively.
\end{cor}

From \thmref{thm:KR} we obtain the following
\begin{cor}\label{cor:kDk}
The sectional curvatures of an arbitrary 2-plane $\alpha$ at $p\in \MM$ regarding %the Schouten-van Kampen connection
$\nSvK$ and %the Levi-Civita connection
$\n$ (respectively, $\tnSvK$ and $\tn$) are related as follows
\begin{equation}\label{kDk}
\begin{split}
  k^{\circ}(\al;p)=k(\al;p)
&+\dfrac{\pi_1(S(x),S(y),y,x)}{\pi_1(x,y,y,x)}\\[6pt]
&-\dfrac{\eta(x)R(x,y,y,\xi)+\eta(y)R(x,y,\xi,x)}{\pi_1(x,y,y,x)},\\[6pt]
  \widetilde{k}^{\circ}(\al;p)=\widetilde k(\al;p)
&+\dfrac{\widetilde\pi_1(\tS(x),\tS(y),y,x)}{\widetilde\pi_1(x,y,y,x)}\\[6pt]
&-\dfrac{\eta(x)\tR(x,y,y,\xi)+\eta(y)\tR(x,y,\xi,x)}{\widetilde\pi_1(x,y,y,x)},
\end{split}
\end{equation}
where $\{x,y\}$ is an arbitrary basis of $\alpha$.
\end{cor}

If $\al$ is a $\xi$-section  at $p\in \MM$ denoted by $\al_{\xi}$ and $\{x,\xi\}$ is its basis, then from \eqref{kDk} and $g(S(x),\xi)=0$ for each $x$ we obtain that the sectional curvature of $\alpha$ regarding
%the Schouten-van Kampen connection
$\nSvK$ is zero, i.e.
$
k^{\circ}(\al_{\xi};p)=0.
$
Analogously, we have $\widetilde{k}^{\circ}(\al_{\xi};p)=0$.

If $\al$ is a $\f$-section  at $p\in \MM$ denoted by $\al_{\f}$ and $\{x,y\}$ is its arbitrary basis, then from \eqref{kDk} and $\eta(x)=\eta(y)=0$ we obtain that the sectional curvatures of $\al_{\f}$ regarding $\nSvK$ and $\n$ are related as follows
\[
k^{\circ}(\al_{\f};p)=k(\al_{\f};p)
+\dfrac{\pi_1(S(x),S(y),y,x)}{\pi_1(x,y,y,x)}.
%-\pi_1(S(\f x),S(\f^2 x),\f^2 x,\f x).
\]
Analogously, we have
\[
\widetilde{k}^{\circ}(\al_{\f};p)=\widetilde k(\al_{\f};p)
+\dfrac{\widetilde\pi_1(\tS(x),\tS(y),y,x)}{\widetilde\pi_1(x,y,y,x)}.
%-\pi_1(S(\f x),S(\f^2 x),\f^2 x,\f x).
\]

If $\al$ is a $\f$-totally real section orthogonal to $\xi$ denoted by $\al_{\bot}$ and $\{x,y\}$ is its arbitrary basis, then from \eqref{kDk} and $\eta(x)=\eta(y)=0$ we obtain that the sectional curvatures of $\al_{\bot}$ regarding $\nSvK$ and $\n$ are related as follows
\[
k^{\circ}(\al_{\bot};p)=k(\al_{\bot};p)
+\dfrac{\pi_1(S(x),S(y),y,x)}{\pi_1(x,y,y,x)}.
\]
Analogously, we have
\[
\widetilde{k}^{\circ}(\al_{\bot};p)=\widetilde k(\al_{\bot};p)
+\dfrac{\widetilde \pi_1(\tS(x),\tS(y),y,x)}{\widetilde \pi_1(x,y,y,x)}.
\]
In the case when $\al$ is a $\f$-totally real section non-orthogonal to $\xi$ regarding $g$ or $\tg$, the relation between the corresponding sectional curvatures regarding $\nSvK$ and $\n$ (respectively, $\tnSvK$ and $\nn$) is just the first (respectively, the second) equality in \eqref{kDk}.

The equalities in the present section are specialised for the considered manifolds in the different classes since $S$ and $\tS$ have a special form in each class, bearing in mind \eqref{Fi:nxi}.

\vspace{20pt}

\begin{center}
$\divideontimes\divideontimes\divideontimes$
\end{center} 

\newpage

\addtocounter{section}{1}\setcounter{subsection}{0}\setcounter{subsubsection}{0}

\setcounter{thm}{0}\setcounter{dfn}{0}\setcounter{equation}{0}

\label{par:sas}

 \Large{

\
\\[6pt]
\bigskip

\
\\[6pt]
\bigskip

\lhead{\emph{Chapter I $|$ \S\thesection. Sasaki-like almost contact complex Riemannian manifolds
}}
%\thispagestyle{empty}

%\noindent  {\Huge\bf \S\thesection. Sasaki-like almost contact \\[12pt]
%\phantom{\S\thesection. }complex Riemannian manifolds
%}%\\[6pt]\vskip2pt}

\noindent
\begin{tabular}{r"l}
  %\hline
\hspace{-6pt}{\Huge\bf \S\thesection.}  & {\Huge\bf Sasaki-like almost contact} \\[12pt]
                             & {\Huge\bf complex Riemannian manifolds}
  %\hline
\end{tabular}

\vskip 1cm

\begin{quote}
\begin{large}
In the present section a Sasaki-like almost contact complex Riemannian manifold is
defined as an almost contact complex Riemannian manifold which
complex cone is a holomorphic complex Riemannian manifold.
Explicit compact and non-compact examples are given. A
canonical construction producing a Sasaki-like almost contact
complex Riemannian manifold from a holomorphic complex
Riemannian manifold is presented and it is called an $\SSS^1$-solvable
extension.

The main results of this section are published in \cite{IvMaMa14}.
\end{large}
\end{quote}

%
%
%\vskip 0.2in \addtocounter{subsection}{1}
%
%\noindent  {\Large\bf \thesubsection. Introduction}

\vskip 0.15in

The almost contact complex Riemannian manifold is an
$(2n+1)$-di\-men\-sional pseudo-Riemannian manifold equipped with a 1-form
$\eta$ and a codimension one distribution $\HC=\ker(\eta)$ endowed
with a complex Riemannian structure. More precisely, the
$2n$-dimensional distribution $\HC$ is pro\-vided with a pair of an almost complex structure and a pseudo-Riemannian
metric of signature $(n,n)$ compatible in the way that  the almost
complex structure acts as an anti-isometry on the metric. Almost
contact complex Riemannian manifolds are investigated and studied
in \cite{GaMiGr,Man4,Man31,ManGri1,ManGri2,ManIv38,ManIv36,NakGri2}.

The main goal of this section is to find  a  class of almost contact
complex  Riemannian manifolds which characteristics resemble to some basic properties of
the well known Sasakian manifolds. We define this class of
Sasaki-like spaces as an almost contact complex Riemannian
manifold which complex cone is a holomorphic complex Riemannian
manifold. We note that a holomorphic complex Riemannian manifold
is a complex manifold endowed with a complex Riemannian metric
whose local components in holomorphic coordinates are holomorphic
functions (see \cite{Manin}). We  determine the Sasaki-like almost
contact complex Riemannian  structure  with an explicit expression
of the covariant derivative of the structure tensors %(cf. Theorem~\ref{ssss})
and construct explicit compact and non-compact
examples. We also present a canonical construction producing a
Sasaki-like almost contact complex Riemannian manifold from any
holomorphic complex Riemannian manifold  which we call an
$\SSS^1$-\emph{solvable extension}. % (cf. Theorem~\ref{ext}).
When we study the curvature of Sasaki-like spaces, we show that it is completely
determined by the curvature of the underlying holomorphic complex
Riemannian manifold. We develop  gauge transformations of
Sasaki-like spaces, i.e. we  find the class of contact conformal
transformations of an almost contact complex Riemannian manifolds
which preserve the Sasaki-like condition.

%\begin{convention}%\hfill\break\vspace{-15pt} %
We introduce the following convention for the present section. Let $\M$ be a $(2n+1)$-dimensional almost contact complex
Riemannian manifold with a pseudo-Riemannian metric $g$ of
signature $(n+1,n)$.
%\begin{enumerate}[a)]
%\item[\emph{(a)}] %We shall use  $A,B,C$ to denote smooth vector fields on
%$M$, $A,B,C\in\X(M)$.
We shall use  $x$, $y$, $z$, $u$ to denote smooth vector fields on
$\MM$, i.e. $x,y,z,u\in\X(\MM)$.
%
%\item[\emph{(b)}]
We shall use $X$, $Y$, $Z$, $U$
to denote smooth horizontal vector fields on $\MM$, i.e. $%
X,Y,Z,U\in \HC=\ker(\eta)$.
%
%\item[\emph{(c)}]
The $2n$-tuple
$\{e_1,\dots,e_{n},e_{n+1}=\f e_1,\dots,e_{2n}=\f e_n\}$ denotes a
local orthonormal basis of the horizontal space $\HC$.
%
%\item[\emph{(d)}]
For an orthonormal basis
$
\{e_0=\xi,e_1,\dots,e_{n},e_{n+1}=\f e_1,\dots,e_{2n}=\f e_n\}
$ we
denote $\varepsilon_i=\mathrm{sign}(g(e_i,e_i))=\pm 1$, where
$\varepsilon_i=1$ for $i=0,1,\dots,n$ and $\varepsilon_i=-1$ for
$i=n+1,\dots,2n$.
%\end{enumerate}
%\end{convention}

%%%%%%%%%%%%%%%%%%%%%%%%%%%%%%%%%%
\vskip 0.2in \addtocounter{subsection}{1} \setcounter{subsubsection}{0}

\noindent  {\Large\bf \thesubsection. Almost contact complex Riemannian manifolds}%\\[6pt]\vskip2pt}

\vskip 0.15in
%\section{Almost contact complex Riemannian manifolds}

Let $(\MM,\f,\xi,\eta)$ be an almost contact manifold. %,  i.e.  $\MM$ is
%a $(2n+1)$-di\-men\-sional differen\-tia\-ble manifold with an almost
%contact structure $(\f,\allowbreak{}\xi,\allowbreak{}\eta)$ consisting of an endomorphism
%$\f$ of the tangent bundle, a vector field $\xi$ and its dual
%1-form $\eta$ such that the following algebraic relations are
%satisfied \eqref{str1}.
%\begin{equation}\label
%\f\xi = 0,\quad \f^2 = -\Id + \eta \otimes \xi,\quad
%\eta\circ\f=0,\quad \eta(\xi)=1.
%\end{equation}
%
An almost contact structure $(\f,\xi,\eta)$ on $\MM$ is called
\emph{normal} and respectively $(\MM,\f,\allowbreak{}\xi,\eta)$ is a
\emph{normal almost contact manifold} if the corresponding almost
complex structure $\check J$  on $\MM'=\MM\times \R$ defined by
\begin{equation}\label{concom}
\check JX=\f X,\qquad \check J\xi=r\ddr,\qquad \check J\ddr=-\frac{1}{r}\xi
\end{equation}
 is integrable (i.e.
$\MM'$ is a complex manifold) \cite{SaHa}.
%The almost contact
%structure is normal if and only if the Nijenhuis tensor of
%$(\f,\xi,\eta)$ is zero \cite{Blair}. The Nijenhuis tensor $N$ of
%the almost contact structure is defined by
%\[
%\begin{array}{l}
%N = [\f, \f]+\D{\eta}\otimes\xi,\quad\\[6pt]
%[\f, \f](x, y)=\left[\f x,\f
%y\right]+\f^2\left[x,y\right]-\f\left[\f x,y\right]-\f\left[x,\f
%y\right],
%\end{array}
%\]
%where $[\f,\f]$ is the Nijenhuis torsion of $\f$.

Let the almost contact manifold $(\MM,\f,\xi,\eta)$ be endowed with a B-metric as in \eqref{str2}.
%a pseudo-Riemann\-ian metric $g$ of signature $(n+1,n)$ which is
%compatible with the almost contact structure in the following way
%\[
%g(\f x, \f y ) = - g(x, y ) + \eta(x)\eta(y).
%\]
%The associated metric $\widetilde{g}$ of $g$ on $\MM$ is defined by
%$\widetilde{g}(x,y)=g(x,\f y)\allowbreak+\eta(x)\eta(y)$. Both
%metrics $g$ and $\widetilde{g}$  are necessarily of signature
%$(n+1,n)$.
The manifold  $(\MM,\f,\xi,\eta,g)$ is known as an almost contact
manifold with B-metric or an \emph{almost contact B-metric
manifold} \cite{GaMiGr}. The manifold
$(\MM,\f,\xi,\eta,\widetilde{g})$ is also an almost contact B-metric
manifold. We will call these manifolds \emph{almost contact
complex Riemannian manifolds}.

\vskip 0.2in \addtocounter{subsubsection}{1}

\noindent  {\Large\bf{\emph{\thesubsubsection. Relation with holomorphic complex Riemannian manifolds}}}

\vskip 0.15in
%\subsection{Relation with holomorphic complex Riemannian manifolds}

Let us remark that the $2n$-dimensional  distribution
$\HC$  is endowed with an almost complex structure
$J=\f|_\HC$ and a metric $h=g|_\HC$, where $\f|_\HC$ and $g|_\HC$ are  the
restrictions of $\f$ and $g$ on $\HC$, respectively, as well as the metric $h$
is compatible with $J$ as follows
\begin{equation}\label{norden}
h(JX,JY)=-h(X,Y) ,\quad \widetilde{h}(X,Y)=h(X,JY).
\end{equation}
The distribution  $\HC$ can be considered as an $n$-dimensional
complex Riemannian distribution with a complex Riemannian metric
\[g^{\C}=h+i\,\widetilde h=g|_\HC+i\,\widetilde{g}|_\HC.\]

We recall that a $2n$-dimensional manifold with almost complex structure
$(J,h)$ endowed with a pseudo-Riemannian metric of signature
$(n,n)$, satisfying \eqref{norden}, was discussed in the first three sections of the present work.
%as an almost complex
%manifold with Norden metric \cite{N1,N2,v,GaGrMi85,KS,KO},  an almost
%complex manifold with B-metric \cite{GaGrMi87,GaMi87}  or an almost complex
%manifold with complex Riemannian metric \cite{LeB83,Manin,GaIv92,BoFeFrVo99,Low}.
When the almost complex structure $J$ is parallel with respect to
the Levi-Civita connection $\DDD$ of the metric $h$, i.e.
$\DDD J=0$, then the manifold is
known also as
%a K\"ahler-Norden manifold, a K\"ahler manifold with B-metric or
a \emph{holomorphic complex Riemannian manifold}. In this case $J$ is integrable and the local components of the complex metric in a holomorphic coordinate system are holomorphic functions.
A 4-dimensional example of a holomorphic complex Riemannian manifold
has been given in \cite{N1}. Another approach to such manifolds
has been used in \cite{N2}. In \cite{v} has been proved that the 4-dimensional
sphere of Kotel'nikov-Study carries a structure of a holomorphic complex Riemannian manifold.

\vskip 0.2in \addtocounter{subsubsection}{1}

\noindent  {\Large\bf{\emph{\thesubsubsection. The  case of parallel structures}}}

\vskip 0.15in
%\subsection{The  case of parallel structures}

The simplest case  of
almost contact complex Riemannian  manifolds  is when the
structures are $\n$-parallel, $\n\f=\n\xi=\n\eta=\n g=\n
\widetilde{g}=0$, and it is determined by the condition
$F(x,y,z)=0$. In this case the distribution $\HC$ is involutive. The
corresponding integral submanifold is a totally geodesic
submanifold which inherits a holomorphic complex Riemannian
structure and the almost contact complex Riemannian manifold is
locally  a pseudo-Riemannian product of a holomorphic complex
Riemannian manifold with a real interval.

\vskip 0.2in \addtocounter{subsection}{1} \setcounter{subsubsection}{0}

\noindent  {\Large\bf \thesubsection. Sasaki-like almost contact complex Riemannian manifolds}%\\[6pt]\vskip2pt}

\vskip 0.15in
%\section{Sasaki-like almost contact complex Riemannian manifolds}

In this subsection we consider the complex Riemannian cone over an
almost contact complex Riemannian manifold and we determine a
Sasaki-like almost contact complex Riemannian manifold with the
condition that its complex Riemannian cone is a holomorphic complex
Riemannian manifold.

\vskip 0.2in \addtocounter{subsubsection}{1}

\noindent  {\Large\bf{\emph{\thesubsubsection. Holomorphic complex Riemannian cone}}}

\vskip 0.15in
%\subsection{Holomorphic complex Riemannian cone}

Let $\M$ be an almost contact complex Riemannian manifold of dimension
$2n+1$. We consider the cone over $\MM$,   %$$$\mathfrak{Cone}$
%$\mathbb Cone$ %
$\mathcal{C}(\MM)=\MM\times \R^-$, with the  almost complex structure
determined in  \eqref{concom} and the complex Riemannian metric
defined by
\begin{equation}\label{barg}
\check{g}\left(\left(x,a\ddr\right),\left(y,b\ddr\right)\right)
=r^2g(x,y)+\eta(x)\eta(y)-ab,
\end{equation}
where $r$ is the coordinate on $\R^-$ and $a$, $b$ are
$C^{\infty}$ functions on $\MM\times \R^-$.

Using the general Koszul formula
\eqref{koszul},
we calculate from \eqref{barg} that the  non-zero components of
the Levi-Civita connection $\cn$ of the complex Riemannian metric
$\cg$ on $\mathcal{C}(\MM)$ are given by
\[
\begin{split}
    &\cg\left(\cn_X Y,Z\right)=r^2 g\left(\n_X Y,Z\right),\qquad
    \cg\left(\cn_X Y,\ddr\right)=-r g\left(X, Y\right), %
    \\[6pt]
    &\cg\left(\cn_X Y,\xi\right)=r^2 g\left(\n_X
    Y,\xi\right)+\frac{1}{2}\left(r^2-1\right)\D\eta(X,Y),
\\[6pt]
    &\cg\left(\cn_X \xi,Z\right)=r^2 g\left(\n_X \xi,
    Z\right)-\frac{1}{2}\left(r^2-1\right)\D\eta(X,Z),
    \\[6pt]
    &\cg\left(\cn_{\xi} Y,Z\right)=
    r^2 g\left(\n_{\xi}Y,Z\right)-\frac{1}{2}(r^2-1)\D\eta(Y,Z),
    \\[6pt]
    &\cg\left(\cn_{\xi} Y,\xi\right)= g\left(\n_{\xi}Y,\xi\right), \qquad
    \cg\left(\cn_{\xi} \xi,Z\right)= g\left(\n_{\xi}\xi,Z\right),
    \\[6pt]
    &\cg\left(\cn_X \ddr,Z\right)=r g\left(X,Z\right), \qquad
    \cg\left(\cn_{\ddr} Y,Z\right)=r g\left(Y,Z\right).
\end{split}
\]
Applying \eqref{concom} we calculate from the formulas above  that
the non-zero components of the covariant derivative $\cn \check J$
of the almost complex structure $\check J$ are given by
\[
\begin{split}
    &\cg\left(\left(\cn_X \check J\right)Y,Z\right)
    =r^2 g\left(\left(\n_X \f\right)Y,Z\right),
    \\[6pt]
    &\cg\left(\left(\cn_X  \check J\right)Y,\xi\right)
    =r^2 \left\{g\left(\left(\n_X \f\right)Y,\xi\right)
    +g(X,Y)\right\}\\[6pt]
    &\phantom{\cg\left(\left(\cn_X  \check J\right)Y,\xi\right)=}
    +\frac{1}{2}\left(r^2-1\right)\D\eta(X,\f Y),%
    \\[6pt]
%    &\phantom{\cg\left(\left(\cn_X  \check J\right)Y,\xi\right)=}
%    +\frac{1}{2}\left(r^2-1\right)\D\eta(X,\f Y),
%   \nonumber\\[6pt]
    &\cg\left(\left(\cn_X  \check J\right)Y,\ddr\right)=-r\left\{g\left(\n_X \xi, Y\right)
    +g\left(X,\f Y\right)\right\}\\[6pt]
    &\phantom{\cg\left(\left(\cn_X  \check J\right)Y,\ddr\right)=}
    +\frac{1}{2r}(r^2-1)\D\eta(X,Y),% Korigiran znamenatel ot 2r^2 na 2r
%\\[6pt]
\end{split}
\]
\[
\begin{split}
&\cg\left(\left(\cn_X  \check J\right)\xi,Z\right)=-r^2\left\{g\left(\n_X \xi,\f Y\right)
    -g\left(X,Z\right)\right\}\\[6pt]
    &\phantom{\cg\left(\left(\cn_X  \check J\right)\xi,Z\right)=}
    +\frac{1}{2}(r^2-1)\D\eta(X,\f Z),
    \\[6pt]
&\cg\left(\left(\cn_X  \check J\right)\ddr,Z\right)=-r\left\{g\left(\n_X \xi, Z\right)
    +g\left(X,\f Z\right)\right\}\\[6pt]
    &\phantom{\cg\left(\left(\cn_X  \check J\right)\ddr,Z\right)=}
    +\frac{1}{2r}(r^2-1)\D\eta(X,Z),
    \\[6pt]
    &\cg\left(\left(\cn_{\xi}  \check J\right)Y,Z\right)
    =r^2g\left(\left(\n_{\xi}\f\right)Y,Z\right)\\[6pt]
    &\phantom{\cg\left(\left(\cn_{\xi}  \check J\right)Y,Z\right)=}
    -\frac{1}{2}(r^2-1)\left\{\D\eta(\f Y, Z)-\D\eta(Y,\f Z)\right\},
    \\[6pt]
    &\cg\left(\left(\cn_{\xi}  \check J\right)Y,\xi\right)=-g(\n_{\xi}\xi,\f Y), \qquad\\[6pt]
    &\cg\left(\left(\cn_{\xi}  \check J\right)\xi,Z\right)=-g\left(\n_{\xi}\xi, \f Z\right),
    \\[6pt]
    &\cg\left(\left(\cn_{\xi}  \check J\right)Y,\ddr\right)=
    -\frac{1}{r}g\left(\n_{\xi}\xi, Y\right), \qquad \\[6pt]
    & \cg\left(\left(\cn_{\xi}
    \check J\right)\ddr,Z\right)=-\frac{1}{r}g\left(\n_{\xi}\xi, Z\right).
\end{split}
\]
Consequently, we have
\begin{prop}\label{defs}
The complex Riemannian cone $\mathcal{C}(\MM)$ over an almost
contact complex Riemannian manifold $\M$  is a holomorphic complex
Riemannian space if and only if the following conditions hold
\begin{alignat}{1}
&F(X,Y,Z)=F(\xi,Y,Z)=F(\xi,\xi,Z)=0\label{sasaki},\\[6pt]
&F(X,Y,\xi)=-g(X,Y)\label{sasaki0}.%,\\[6pt]
%    \left(\n_X\eta\right)Y&=-g(\f X, Y)=g(  \n_X
%\xi,Y)\label{sasaki1}.
\end{alignat}
\end{prop}

\begin{proof}
We obtain from the expressions above that the complex Riemannian
cone $(\mathcal{C}(\MM),\check J,\check g)$ is a holomorphic Riemannian
manifold (a K\"ahler manifold with Norden metric), i.e. $\cn \check J=0$,
if and only if  the almost contact complex Riemannian manifold
$\M$ satisfies the following conditions:
\begin{alignat}{1}
    &F(X,Y,Z)=0,
    \nonumber\\[6pt]
    &F(X,Y,\xi)=-g(X,Y)-\frac{1}{2r^2}\left(r^2-1\right)\D\eta(X,\f Y)\label{FXYxi},
\\[6pt]
%    &F(X,\f Y,\xi)=-g(X,\f
%    Y)+\frac{1}{2r^2}\left(r^2-1\right)\D\eta(X,Y),\label{FXYxi}\\[6pt]
    &F(\xi,Y,Z)=\frac{1}{2r^2}\left(r^2-1\right)\left\{\D\eta(\f Y,Z)
    -\D\eta(Y,\f Z)\right\},\label{FxiYZ}
%\\[6pt]
\end{alignat}
\begin{alignat}{1}
    &F(\xi,\xi,Z)=0, \qquad \n_{\xi}\xi=0\label{nxixi}.
\end{alignat}
The condition $\cn \check J=0$ implies the integrability of $\check J$,
hence the structure of $\M$ is normal.
%The equality  \eqref{nxixi} shows the that the integral curves of $\xi$ are geodesics.

Further, according to \eqref{FXYxi}, we get
\begin{equation}\label{nablaeta}
\left(\n_X\eta\right)Y=-g(X,\f
Y)+\frac{1}{r^2}\left(r^2-1\right)\D\eta(X,Y),
\end{equation}
yielding
$\D\eta(X,Y)=\frac{1}{2r^2}\left(r^2-1\right)\D\eta(X,Y)$ since
$g(\cdot,\f\cdot)$ is symmetric. The latter equality  shows
$\D\eta(X,Y)=0$ which together with \eqref{nablaeta} yields
\begin{equation}\label{sasaki1}
\left(\n_X\eta\right)Y=-g(X,\f  Y).
\end{equation}
From \eqref{nxixi} we get
$\D\eta(\xi,X)=(\n_{\xi}\eta)(X)-(\n_X\eta)(\xi)=0$. Hence, we have
$\D\eta=0$. We substitute $\D\eta=0$ into
\eqref{FXYxi}-\eqref{FxiYZ} to complete the proof of the
proposition.
\end{proof}

\begin{dfn}
An almost contact complex Riemannian manifold $\M$  is said to be
\emph{Sasaki-like} if the structure tensors $\f, \xi, \eta, g$
satisfy the equalities \eqref{sasaki} and \eqref{sasaki0}.
\end{dfn}

To characterize the Sasaki-like almost contact complex Riemannian
manifolds by their structure tensors, we need the general
result in \thmref{thm:FNhatN}.

The next result determines  the Sasaki-like spaces by their structure tensors.

\begin{thm}\label{ssss}
Let $\M$ be an almost contact complex Riemannian manifold.
The following conditions are equivalent:
\begin{enumerate}
\item
The manifold $\M$ is a Sasaki-like almost contact
complex Riemannian manifold; %
\item
The covariant derivative
$\nabla \f$ satisfies the equality
\end{enumerate}
\begin{equation}\label{defsl}
\begin{array}{l}
(\nabla_x\f)y=-g(x,y)\xi-\eta(y)x+2\eta(x)\eta(y)\xi;
%\\[6pt] \phantom{(\nabla_x\f)y}=g(\f x,\f y)\xi+\eta(y)\f^2 x;
\end{array}
\end{equation}
\begin{enumerate}
\item[$\mathrm{(iii)}$]
The Nijenhuis tensors $N$ and $\widehat N$ satisfy the
relations:
\end{enumerate}
\begin{equation}\label{sasnn}
N=0,\qquad  \widehat
N=-4\left(\widetilde{g}-\eta\otimes\eta\right)\otimes\xi.
\end{equation}
\end{thm}

\begin{proof}
It is easy to check using \eqref{Fxieta} and \eqref{F-prop}
that \eqref{defsl} is equivalent to the system of the equations
\eqref{sasaki} and \eqref{sasaki0} which established the equivalence between
(i) and (ii) in view of Proposition~\ref{defs}.

We substitute  \eqref{defsl} consequently into \eqref{enu} and
\eqref{enhat} to get \eqref{sasnn} which gives the implication
(ii) $\Rightarrow$ (iii).

Now, suppose \eqref{sasnn} holds. Consequently, we obtain
$\widehat N(\xi,y)=0$. Now, \eqref{defsl} follows with a
substitution of the last equality together with  \eqref{sasnn}
into \eqref{nabf} which completes the proof.
\end{proof}

\begin{cor}
Let $\M$ be a Sasaki-like almost contact complex Riemannian
manifold. Then we have %the next conditions hold
\begin{enumerate}
\item
the manifold $\M$ is normal, the fundamental
1-form $\eta$ is closed and the integral curves of
$\xi$ are
geodesics; %
\item
the 1-forms $\theta$ and $\theta^*$ satisfy the
equalities $ \theta=-2n\,\eta$ and $\theta^*=0$.
\end{enumerate}
\end{cor}

\vskip 0.2in \addtocounter{subsubsection}{1}

\noindent  {\Large\bf{\emph{\thesubsubsection. Examples}}}

\vskip 0.15in
%\subsection{Examples}

In this subsection we construct a number of examples of Sasaki-like
almost contact complex Riemannian manifolds.

\vskip 0.2in

\noindent  {\Large\it\thesubsubsection.1. Example~1}\label{ex1}%\\[6pt]\vskip2pt}

\vskip 0.15in
%\subsubsection{}
%\subsubsection{Example~1.}

We consider the solvable Lie group $G$ of
dimension $2n+1$
 with a basis of left-invariant vector fields $\{e_0,\dots, e_{2n}\}$
 defined by the commutators
\begin{equation}\label{com}
\begin{aligned}
&[e_0,e_1]=e_{n+1},\quad &\dots,\quad &[e_0,e_n]=e_{2n},\\[6pt]
&[e_0,e_{n+1}]=-e_1,\quad &\dots,\quad &[e_0,e_{2n}]=-e_n.
\end{aligned}
\end{equation}
We define an invariant almost contact complex Riemannian structure
%(an almost contact B-metric structure)
 on $G$  by
\begin{equation}\label{strEx1}
%\begin{array}{c}
\begin{array}{rl}
&g(e_0,e_0)=g(e_1,e_1)=\dots=g(e_n,e_n)=1\\[6pt]
&g(e_{n+1},e_{n+1})=\dots=g(e_{2n},e_{2n})=-1,
%\end{array}
\\[6pt]
&g(e_i,e_j)=0,\quad
i,j\in\{0,1,\dots,2n\},\; i\neq j,
\\[6pt]
&\xi=e_0, \quad \f  e_1=e_{n+1},\quad  \dots,\quad  \f e_n=e_{2n}.
\end{array}
\end{equation}

Using the Koszul formula \eqref{koszul}, we check that
\eqref{sasaki} and  \eqref{sasaki0} %and \eqref{sasaki1}
are
fulfilled, i.e. this is a Sasaki-like almost contact complex
Riemannian structure.

Let $e^0=\eta$, $e^1$, $\dots$, $e^{2n}$ be the corresponding dual
1-forms, $e^i(e_j)=\delta^i_j$. From \eqref{com} and the formula
for an arbitrary 1-form $\alpha$
\[
\D
\alpha(A,B)=A\alpha(B)-B\alpha(A)-\alpha([A,B]),
\]
it follows that the structure equations of the group are
\begin{equation}\label{comstr}
\begin{array}{llll}
 \D e^0=\D\eta=0,\; &\D e^1=e^{0}\wedge e^{n+1},\;&\dots,\; &\D
e^n=e^{0}\wedge e^{2n},\\[6pt]
&\D e^{n+1}=-e^{0}\wedge e^{1}, &\dots, &\D e^{2n}=- e^{0}\wedge
e^{n}
\end{array}
\end{equation}
and the Sasaki-like almost contact complex Riemannian structure
has the form
\begin{equation}\label{sas}
g=\sum_{i=0}^{2n}\varepsilon_i\left(e^i\right)^2,%
\qquad \f e^1= e^{n+1},\ \dots,\ \f e^n= e^{2n}.
\end{equation}

The group $G$ is the following  \textit{rank-1 solvable
extension of the Abelian group $\R^{2n}$}
\begin{equation}\label{ext1}
\begin{array}{l}
e^0=\D t, \\[6pt]
e^1=\cos t\ \D x^1+\sin t\ \D x^{n+1},\\[6pt]
\dots,\\[6pt]
e^n=\cos t\ \D x^n+\sin t\ \D x^{2n},\\[6pt]
e^{n+1}=-\sin t\ \D x^1+\cos t\ \D x^{n+1},\\[6pt]%\nonumber %
\dots, \\[6pt]
e^{2n}=-\sin t\ \D x^n+\cos t\ \D x^{2n}.
\end{array}
\end{equation}
Clearly, the 1-forms defined in \eqref{ext1} satisfy
\eqref{comstr} and the Sasaki-like almost contact complex Riemannian metric has
the form
\begin{equation}\label{sasmetric}
g=\D t^2+\cos{2t}\left(\sum_{i=1}^{2n}\varepsilon_i\left(\D
x^i\right)^2\right)-\sin{2t}\left(-2\sum_{i=1}^n\D x^i\D
x^{n+i}\right).
\end{equation}

It is known that the solvable Lie group $G$ admits a lattice
$\Gamma$ such that the quotient space $G/\Gamma$ is compact (c.f.
\cite[Chapter~3]{OT}). The invariant Sasaki-like almost contact
complex Riemannian structure $(\f, \xi, \eta, g)$  on $G$ descends
to $G/\Gamma$ which supplies a compact Sasaki-like almost contact
complex Riemannian manifold in any dimension.

It follows from \eqref{com}, \eqref{sas}, \eqref{ext1} and  \eqref{sasmetric}
that the distribution
$\HC=%\ker(\eta)=
\Span\{e_1,\dots,e_{2n}\}$ is integrable and the corresponding
integral submanifold can be considered as the holomorphic complex
Riemannian flat space $\R^{2n}=\Span\{\D x^1,\dots,\D x^{2n}\}$
with the  following holomorphic complex Riemannian  structure  %
\[
\begin{split}
&J\D x^1=\D x^{n+1},\; \dots,\;
J\D x^n=\D x^{2n};\quad \\[6pt]
&h=\sum_{i=1}^{2n}\varepsilon_i(\D x^i)^2,
\quad \widetilde h =-2\sum_{i=1}^n\D x^i\D x^{n+i}.
\end{split}
\]

\vskip 0.2in

\noindent  {\Large\it\thesubsubsection.2. \protect{$\SSS^1$}-solvable extension}%\\[6pt]\vskip2pt}

\vskip 0.15in
%\subsubsection{\protect{$\SSS^1$}-solvable extension}

Inspired by Example~1 on page~\pageref{ex1} we proposed the following more general
construction. Let $(\MM^{2n},J,h,\widetilde{h})$ be a
$2n$-dimen\-sion\-al holomorphic complex Riemannian manifold, i.e.
the almost complex structure $J$ acts as an anti-isometry on the
neutral metric $h$, $h(JX,JY)=-h(X,Y)$ and it is parallel with
respect to the Levi-Civita connection of $h$. In particular, the
almost complex structure $J$ is integrable. The associated neutral
metric $\widetilde{h}$ is defined by $\widetilde{h}(X,Y)=h(JX,Y)$
and it is also parallel with respect to the Levi-Civita connection
of $h$.

We consider the product manifold $\MM^{2n+1}=\mathbb R^+\times \MM^{2n}$.
Let $\D t$ be the coordinate 1-form on $\mathbb R^+$ and we define an
almost contact complex Riemannian structure on $\MM^{2n+1}$ as
follows
\begin{equation}\label{strsas}
\begin{split}
&\eta=\D t, \qquad \varphi |_\HC%{|_{Ker(\eta)}}
=J, \qquad \eta\circ\varphi =0,\qquad \\[6pt]
&g=\D t^2+\cos{2t}\ h-\sin{2t}\
\widetilde{h}.
\end{split}
\end{equation}

\begin{thm}\label{ext}
Let $(\MM^{2n},J,h,\widetilde{h})$ be a $2n$-dimensional holomorphic
complex Rie\-mann\-ian manifold. Then the product manifold
$\MM^{2n+1}=\mathbb R^+\times \MM^{2n}$ equipped with the almost
contact complex Riemannian structure defined in \eqref{strsas} is
a Sasaki-like almost contact complex Riemannian manifold.
If $\MM^{2n}$ is compact and $\SSS^1$ is an 1-dimensional sphere then $\MM^{2n+1}=\SSS^1\times \MM^{2n}$ with the
structure \eqref{strsas} is a compact Sasaki-like almost contact
complex Riemannian manifold.
\end{thm}
\begin{proof}
It is easy to check using \eqref{koszul}, \eqref{strsas} and the
fact that the complex structure $J$ is parallel with respect to
the Levi-Civita connection of $h$ that the structure
defined in \eqref{strsas} satisfies \eqref{sasaki} and
\eqref{sasaki0} and thus $\M$  is a  Sasaki-like almost contact
complex Riemannian manifold.

Now, suppose $\MM^{2n}$ is a compact holomorphic complex Riemannian
manifold. The equations \eqref{strsas} imply that the metric $g$
is periodic on $\R$ and therefore it descends to the compact
manifold $\MM^{2n+1}=\SSS^1\times \MM^{2n}$. Thus we obtain a compact
Sasaki-like almost contact complex Riemannian manifold.
\end{proof}
We call the Sasaki-like almost contact complex Riemannian manifold
constructed in Theorem~\ref{ext} from a holomorphic complex
Riemannian manifold  \emph{an $\SSS^1$-solvable extension of a holomorphic
complex Riemannian manifold}.

\newpage
\vskip 0.2in

\noindent  {\Large\it\thesubsubsection.3. Example~2}%\\[6pt]\vskip2pt}

\vskip 0.15in
%\subsubsection{Example~2.}

Let us consider the Lie group $G^5$ of
dimension $5$
 with a basis of left-invariant vector fields $\{e_0,\dots, e_{4}\}$
 defined by the commutators
\[
\begin{array}{ll}
[e_0,e_1] = \lm e_2 + e_3 + \mu e_4,\quad\quad &[e_0,e_2] = - \lm e_1 -
\mu e_3 + e_4,\\[6pt]
[e_0,e_3] = - e_1  - \mu e_2 + \lm e_4,\quad\quad &[e_0,e_4] = \mu e_1
- e_2 - \lm e_3,
\end{array}
\]
where $\lm,\mu\in\R$.
Let $G^5$ be equipped with an invariant almost contact complex
Riemannian structure as in \eqref{strEx1} for $n=2$. We calculate
using \eqref{koszul} that the non-zero connection 1-forms of the
Levi-Civita connection are the following
\[
\begin{array}{l}
\begin{array}{ll}
\n_{e_0} e_1 = \lm e_2 + \mu e_4,\qquad\quad & \n_{e_1}e_0 = - e_3,\\[6pt]
 \n_{e_0} e_2 = - \lm e_1 - \mu e_3, \qquad\quad & \n_{e_2} e_0 = - e_4,\\[6pt]
\n_{e_0} e_3 = - \mu e_2 + \lm e_4,\qquad\quad & \n_{e_3} e_0 = e_1, \\[6pt]
\n_{e_0} e_4 = \mu e_1 - \lm e_3, \qquad\quad & \n_{e_4}e_0 = e_2,\\[6pt]
\end{array}
\\[6pt]
\begin{array}{c}
\n_{e_1}e_3 = \n_{e_2} e_4 = \n_{e_3} e_1 = \n_{e_4}e_2 = - e_0.
\end{array}
\end{array}
\]

Similarly as in Example~1 we verify that the constructed manifold
$(G^5,\f,\xi,\eta,g)$ is a Sasaki-like almost contact complex
Riemannian manifold.

{ Take $\mu=0$ and $\lambda\not=0$. Then the structure equations of the group become
\begin{equation}\label{comstr2}
\begin{array}{ll}
\D e^0=\D\eta=0,\; & \\[6pt]
\D e^1=e^{0}\wedge e^3 +\lambda e^0\wedge e^2, \quad\quad & \D e^2=e^{0}\wedge e^4-\lambda e^0\wedge e^1,\\[6pt]
\D e^3=-e^{0}\wedge e^{1}+\lambda e^0\wedge e^4, \quad\quad &\D e^4=-e^{0}\wedge e^2 -\lambda e^0\wedge e^3.
\end{array}
\end{equation}

A basis of 1-forms satisfying \eqref{comstr2} is given by $e^0=\D
t$ and
\[
\begin{split}
e^1=\ &\cos{(1-\lambda)t}\ \D x^1-\cos{(1+\lambda)t}\ \D x^2\\[6pt]
&+\sin{(1-\lambda)t}\ \D x^3-\sin{(1+\lambda)t}\ \D x^4,\\[6pt]
e^2=\ &\sin{(1-\lambda)t}\ \D x^1+\sin{(1+\lambda)t}\ \D x^2\\[6pt]
&-\cos{(1-\lambda)t}\ \D x^3-\cos{(1+\lambda)t}\ \D x^4,\\[6pt]
e^3=\ &-\sin{(1-\lambda)t}\ \D x^1+\sin{(1+\lambda)t}\ \D x^2\\[6pt]
&+\cos{(1-\lambda)t}\ \D x^3-\cos{(1+\lambda)t}\ \D x^4,\\[6pt]
e^4=\ &\cos{(1-\lambda)t}\ \D x^1+\cos{(1+\lambda)t}\ \D x^2\\[6pt]
&+\sin{(1-\lambda)t}\ \D x^3+\sin{(1+\lambda)t}\ \D x^4.
\end{split}
\]
Then the Sasaki-like metric is of the form
\begin{equation}\label{m2}
\begin{split}
g= \D t^2%
&-4\cos{2t}\left(\D x^1\D x^2-\D x^3\D x^4\right)\\[6pt]
&-4\sin{2t}\left(\D
x^1\D x^4+\D x^2\D x^3\right).
\end{split}
\end{equation}
From \eqref{comstr2} it follows that the distribution
$\HC=%\ker(\eta)=
\Span\{e_1,\dots,e_4\}$ is integrable and the corresponding
integral submanifold  can be considered as
the holomorphic complex Riemannian
flat space $\R^{4}=\Span\{\D x^1,\dots,\D x^{4}\}$ with the
holomorphic complex Riemannian  structure  given by
\begin{equation*}
\begin{array}{ll}
J\D x^1=\D x^3, \quad & J\D x^2=\D x^4;\quad \\[6pt]
h=-4(\D x^1\D x^2-\D x^3\D x^4), \quad\quad & \widetilde h =4(\D x^1\D x^4+\D x^2\D x^3)
\end{array}
\end{equation*}
and the Sasaki-like metric \eqref{m2} takes the form
$$g=\D t^2+\cos{2t}\
h-\sin{2t}\ \widetilde{h}.$$

\vskip 0.1in \addtocounter{subsection}{1} \setcounter{subsubsection}{0}

\noindent  {\Large\bf \thesubsection. Curvature properties}%\\[6pt]\vskip2pt}

\vskip 0.15in
%\section{Curvature properties }%of the Sasaki-like almost contact complex Riemannian manifolds}

Let $\M$ be an almost contact complex Riemannian manifold. The
curva\-ture tensors of type $(1,3)$ and type $(0,4)$ are defined as usual (see Subsection 7.3, page \pageref{curvR}). The Ricci tensor $\Ric$,  the scalar
curvature $\Scal$ and the *-scalar curvature $\Scal^*$ are the known
traces of the curvature,
\[
\Ric(x,y)=\sum_{i=0}^{2n}\varepsilon_iR(e_i,x,y,e_i),\]
\[
\Scal=\sum_{i=0}^{2n}\varepsilon_i \Ric(e_i,e_i),\quad
\Scal^*=\sum_{i=0}^{2n}\varepsilon_i \Ric(e_i,\f e_i).
\]

\begin{prop}\label{vercurv}
On a Sasaki-like almost contact complex Riemannian manifold $\M$
the next formula holds
\begin{equation}\label{curf}
\begin{split}
R(x,y,\f z,u)-R(x,y,z,\f u)
&=\left[g(y,z)-2\eta(y)\eta(z)\right]g(x,\f u)\\[6pt]
&+\left[g(y,u)-2\eta(y)\eta(u)\right]g(x,\f z)\\[6pt]
&-\left[g(x,z)-2\eta(x)\eta(z)\right]g(y,\f u)\\[6pt]
&-\left[g(x,u)-2\eta(x)\eta(u)\right]g(y,\f z).
\end{split}
\end{equation}
In particular, we have
\begin{equation}\label{cur}
\begin{split}
&R(x,y)\xi=\eta(y)x-\eta(x)y, \quad \\[6pt]
& [X,\xi]\in \HC, \quad
\nabla_{\xi}X=-\f X-[X,\xi] \in \HC;
\end{split}
\end{equation}
\begin{equation}\label{ricxi}
R(\xi,X)\xi=-X, \quad \Ric(y,\xi)=2n\ \eta(y),\quad
\Ric(\xi,\xi)=2n.
\end{equation}
\end{prop}

\begin{proof}
The Ricci identity for $\f$ reads
\[
R(x,y,\f z,u)-R(x,y,z,\f
u)=g\Bigl(\left(\n_x\n_y\f\right)z,u\Bigr)-g\Bigl(\left(\n_y\n_x\f\right)z,u\Bigr).
\]
Applying \eqref{defsl} to the above equality and using
\eqref{sasaki1}, we obtain \eqref{curf} by  straightforward
calculations. Set $z=\xi$ into \eqref{curf} and using \eqref{str1},
we get the first equality in \eqref{cur}. The  rest  follows from
\eqref{sasaki1} and the condition $\D \eta=0$. The equalities
\eqref{ricxi} follow directly from the first equality in
\eqref{cur}.
\end{proof}

\vskip 0.2in \addtocounter{subsubsection}{1}

\noindent  {\Large\bf{\emph{\thesubsubsection. The horizontal curvature}}}

\vskip 0.15in
%\subsection{The horizontal curvature}

From $\D\eta=0$ it follows locally $\eta=\D x$, $\HC$ is integrable
and the manifold is locally the product
$\MM^{2n+1}=\MM^{2n}\times\mathbb R$ with $T\MM^{2n}=\HC$. The submanifold
$(\MM^{2n},J=\f|_\HC,h=g|_\HC)$ is a holomorphic complex
Riemannian manifold.
%(i.e. a K\"ahler-Norden manifold). %
Indeed, we obtain from \eqref{sasaki} that
\[
h(\DDD_X J)Y,Z)=F(X,Y,Z)=0,
\]
where
$\DDD$ is the Levi-Civita connection of $h$.

We may consider $\MM^{2n}$ as a hypersurface of $\MM$ with the unit normal
$\xi=\frac{\partial}{\partial t}$. The equality \eqref{sasaki1}
yields
\[
g(\n_X\xi,Y)=-g(\n_XY,\xi)=-g(\f X,Y)=-\widetilde{g}|_{\HC}(X,Y),
\quad \n_{\xi}\xi=0.
\]
This means that the second fundamental form is equal to
$\widetilde{g}|_{\HC}=\widetilde h$. The Gauss equation (see e.g.
\cite[Chapter VII, Proposition~4.1]{KoNo-2}) yields
\begin{equation}\label{gaus}
\begin{split}
R(X,Y,Z,U)=R^h(X,Y,Z,U)&+g(\f X,Z)g(\f Y,U)\\[6pt]
&-g(\f Y,Z)g(\f
X,U),
\end{split}
\end{equation}
where $R^h$ is the curvature tensor of the holomorphic complex
Riemannian manifold $(\MM^{2n},J,h)$.

For the horizontal Ricci tensor we obtain from \eqref{gaus} and \eqref{ricxi} that
\begin{equation}\label{ric}
\begin{split}
\Ric(Y,Z)&=\sum_{i=1}^{2n}\varepsilon_iR(e_i,Y,Z,e_i)+R(\xi,Y,Z,\xi)\\[6pt]
&=\Ric^h(Y,Z)+g(\f Y,\f Z)+g(Y,Z)=\Ric^h(Y,Z),
\end{split}
\end{equation}
where $\Ric^h$ is the Ricci tensor of $h=g|_{\HC}$.

It follows from \propref{vercurv}  that the curvature tensor in
the direction of $\xi$ on a Sasaki-like almost contact complex
Riemannian manifold is completely determined by
$\eta,\f,g,\widetilde g$. Indeed, using the properties of the
Riemannian curvature, we derive from \eqref{cur} that
$$R(x,y,\xi,z)=R(\xi,z,x,y)=\eta(y)g(x,z)-\eta(x)g(y,z).$$
Now,
the equation \eqref{gaus} implies that the Riemannian curvature of
a Sasaki-like almost contact complex Riemannian manifold is
completely determined by the curvature of the underlying
holomorphic complex Rie\-mannian manifold $(\MM^{2n},J,h)$, where $T\MM^{2n}=\HC$.

\vskip 0.2in \addtocounter{subsubsection}{1}

\noindent  {\Large\bf{\emph{\thesubsubsection. Example~3: \protect{$\SSS^1$}-solvable extension of the h-sphere}}}

\vskip 0.15in
%\subsection{Example~3: \protect{$\SSS^1$}-solvable extension of the h-sphere}

The next example
illustrates Theorem~\ref{ext}. We consider $\R^{2n+2}$ for  $n>2$ as a
flat holomorphic complex Riemannian manifold, i.e.  $\R^{2n+2}$ is
equipped with the canonical complex structure $J'$ and the
canonical Norden metrics $h'$ and $\widetilde h'$  defined by
\[
\begin{split}
h'(x',y')&=\sum_{i=1}^{n+1} \left(x^i
y^i-x^{n+i+1}y^{n+i+1}\right), %\quad \\[6pt]
\end{split}
\]
\[
\begin{split}
\widetilde h'
(x',y')&=-\sum_{i=1}^{n+1}\left(x^i
y^{n+i+1}+x^{n+i+1}y^i\right)%\\[6pt]
%&\widetilde h (x',y')=-\sum_{i=1}^{n+1}(x^i
%y^{n+i+1}+x^{n+i+1}y^i)
\end{split}
\]
for the vectors $x' = (x^1, \dots, x^{2n+2})$ and $y' = (y^1,
\dots, y^{2n+2})$ %with respect to the basis \{e_1
in $\R^{2n+2}$. Identifying the point $z' = (z^1, \dots,
z^{2n+2})$ in $\R^{2n+2}$ with the position vector $z'$, we
consider the  complex hypersurface $\SSS_h^{2n}(z'_0; a, b)$
defined by the equations
\[
h'(z'-z'_0, z'-z'_0) = a,\quad \widetilde h' (z'-z'_0, z'-z'_0) =
b,
\]
where $(0, 0)\allowbreak\neq\allowbreak (a, b) \in\R^2$.

The
codimension two submanifold $\SSS_h^{2n}(z'_0; a,b)$ is
$J'$-invariant and the restriction of $h'$ on $\SSS_h^{2n}(z'_0;
a,b)$ has rank $2n$ due to the condition $(a, b)\allowbreak\neq\allowbreak(0,0)$. The holomorphic complex
Riemannian structure on $\R^{2n+2}$ inherits a holomorphic complex
Riemannian structure $\bigl(J'|_{{\SSS_h^{2n}}},
h'|_{\SSS_h^{2n}}\bigr)$ on the complex hypersurface $\SSS_h^{2n}(z'_0;
a,b)$. The holomorphic complex Riemannian manifold
$\bigl(\SSS_h^{2n}(z'_0; a,b),\allowbreak{}
J'|_{\SSS_h^{2n}},\allowbreak{} h'|_{\SSS_h^{2n}}\bigr)$ is sometimes
called an \emph{h-sphere} with center $z'_0$ and a pair of
parameters $(a, b)$. The h-sphere $\SSS_h^{2n}(z'_0; 1, 0)$ is the
sphere of Kotel'nikov-Study \cite{v}. The curvature of an h-sphere
is given by the formula \cite{GaGrMi85}
\begin{equation}\label{Rnu-ab}
R'|_{\SSS_h^{2n}} = \frac{1}{a^2+b^2} \bigl\{a(\pi_1 - \pi_2) - b
\pi_3\bigr\},
\end{equation}
where
\[
\pi_1 = \frac{1}{2}h'|_{\SSS_h^{2n}}\owedge h'|_{\SSS_h^{2n}},\quad
\pi_2 = \frac{1}{2}\widetilde{h}'|_{\SSS_h^{2n}}\owedge
\widetilde{h}'|_{\SSS_h^{2n}},\quad
\pi_3 =-
h'|_{\SSS_h^{2n}}\owedge\widetilde{h}'|_{\SSS_h^{2n}}
\]
and $\owedge$
stands for the Kulkarni-Nomizu product of two (0,2)-tensors (see \eqref{Kulkarni}).
%; for
%example,
%\[
%\begin{aligned}
%h'\owedge\widetilde{h}'(X,Y,Z,U)&=h'(Y,Z) \widetilde h'(X,U) -
%h'(X,Z)\widetilde h'(Y,U)\\[6pt] &+\widetilde h'(Y,Z) h'(X,U) - \widetilde
%h'(X,Z) h'(Y,U).
%\end{aligned}
%\]
Consequently, we have
\begin{equation}\label{Ric-h-sph}
\begin{split}
&\Ric'|_{\SSS_h^{2n}}=\frac{2(n-1)}{a^2+b^2}\bigl(a
h'|_{\SSS_h^{2n}}+b\widetilde{h}'|_{\SSS_h^{2n}}\bigr), \qquad\\[6pt]
&\Scal'|_{\SSS_h^{2n}}=\frac{4n(n-1)a}{a^2+b^2}.
\end{split}
\end{equation}

The  product manifold $\MM^{2n+1}=\mathbb R^+\times \SSS_h^{2n}(z'_0;
a,b)$  equipped with the following  almost contact complex
Riemannian structure
\begin{equation*}
\begin{split}
&\eta=\D t, \qquad \varphi|_\HC =J'|_{\SSS_h^{2n}}, \qquad
\eta\circ\varphi =0,\quad
\\[6pt]
&g=\D t^2+\cos{2t}\
h'|_{\SSS_h^{2n}}-\sin{2t}\ \widetilde h'|_{\SSS_h^{2n}}
\end{split}
\end{equation*}
is a Sasaki-like almost contact complex Riemannian manifold according to \thmref{ext}.

The horizontal metrics on $\MM^{2n+1}=\mathbb R^+\times \SSS_h^{2n}(z'_0;
a,b)$  are
\begin{equation}\label{hh}
\begin{array}{l}
h=g|_\HC=\cos{2t}\ h'|_{\SSS_h^{2n}}-\sin{2t}\ \widetilde
h'|_{\SSS_h^{2n}},\\[6pt]
\widetilde h=\widetilde g|_\HC=\sin{2t}\ h'|_{\SSS_h^{2n}}+\cos{2t}\
\widetilde h'|_{\SSS_h^{2n}}.
\end{array}
\end{equation}
The Levi-Civita connection $\nabla'$ of the metric
$h'|_{\SSS_h^{2n}}$ coincides with the Levi-Civita connection of
$\widetilde{h}'|_{\SSS_h^{2n}}$ since $\nabla' J'=0$. Using this
fact,  the Koszul formula \eqref{koszul} together with \eqref{hh}
gives for the Levi-Civita connection $\DDD$ of $h$ the
expression
\[
\begin{split}
h(\DDD_XY,Z)&= \cos{2t}\
h'|_{\SSS_h^{2n}}\left(\nabla'_XY,Z\right)-\sin{2t}\
h'|_{\SSS_h^{2n}}\left(\nabla'_XY,JZ\right)\\[6pt]
&=h\left(\nabla'_XY,Z\right),
\end{split}
\]
which implies $\DDD_XY=\nabla'_XY$. The latter equality
together with \eqref{hh} yields for the curvature of $h$ the
formula \[R^h=\cos{2t}\ R'|_{\SSS_h^{2n}}-\sin{2t}\
\widetilde{R}'|_{\SSS_h^{2n}},\] where
$\widetilde{R}'|_{\SSS_h^{2n}}:=J'R'|_{\SSS_h^{2n}}$. The above equality
together with \eqref{Rnu-ab} implies
\begin{equation}\label{rrr}
\begin{split}
R^h&=\frac{1}{a^2+b^2}\bigl\{\cos{2t}[a(\pi_1-\pi_2)-b\pi_3]\\[6pt]
&\phantom{=\frac{1}{a^2+b^2}\bigl\{}
-\sin{2t}[-a\pi_3-b(\pi_1-\pi_2)]\bigr\}\\[6pt]
&=\frac1{a^2+b^2}\bigl\{(a\cos{2t}+b\sin{2t})(\pi_1-\pi_2)\\[6pt]
&\phantom{=\frac{1}{a^2+b^2}\bigl\{}
-(b\cos{2t}-a\sin{2t})\pi_3\bigr\}.
\end{split}
\end{equation}

Taking into account \eqref{gaus}, \eqref{hh} and \eqref{rrr}, we obtain that the
horizontal curvature $R|_\HC$ of the Sasaki-like almost contact
complex Riemannian manifold $\MM^{2n+1}=\mathbb R^+\times
\SSS_h^{2n}(z'_0; a,b)$ is given by the formula
\[
\begin{split}
R|_\HC&=R^h-(\sin{2t})^2\pi_1-(\cos{2t})^2\pi_2+\sin{2t}\cos{2t}\,\pi_3\\[6pt]
&=\left\{\frac1{a^2+b^2}(a\cos{2t}+b\sin{2t})-(\sin{2t})^2\right\}\pi_1\\[6pt]
&-\left\{\frac1{a^2+b^2}(a\cos{2t}+b\sin{2t})+(\cos{2t})^2\right\}\pi_2\\[6pt]
&-\left\{\frac1{a^2+b^2}(b\cos{2t}-a\sin{2t})-\sin{2t}\cos{2t}\right\}\pi_3.
\end{split}
\]
For the horizontal Ricci tensor we obtain from \eqref{ric},
\eqref{Ric-h-sph} and \eqref{hh} the formula
\[
\begin{split}
\Ric|_\HC=\Ric^h&=\frac{2(n-1)}{a^2+b^2}(a
h'|_{\SSS_h^{2n}}+b\widetilde{h}'|_{\SSS_h^{2n}})\\[6pt]
&=\frac{2(n-1)}{a^2+b^2}\Big[(a\cos{2t}-b\sin{2t})h+(b\cos{2t}+a\sin{2t})\widetilde
h\Big].
\end{split}
\]

\vskip 0.2in \addtocounter{subsection}{1} \setcounter{subsubsection}{0}

\noindent  {\Large\bf \thesubsection. Contact conformal (homothetic) transformations}%\\[6pt]\vskip2pt}

\vskip 0.15in
%\section{Contact conformal (homothetic) transformations}% of the Sasaki-like spaces}

In this subsection we investigate when the Sasaki-like condition is
preserved under contact conformal transformations. We recall that
a general contact conformal transformation of an almost contact
complex Riemannian manifold $\M$ is defined by \eqref{Transf} \cite{Man4,ManGri1,ManGri2}.
If the functions $u$, $v$, $w$ are constant we have a
\emph{contact homothetic transformation}.

A relation between the tensors $\overline F$ and $F$ is given in \cite{Man4}, see also \eqref{barF-F}, %\cite[(22)]{ManIv38},
\begin{subequations}\label{ff}
\begin{equation}
\begin{split}
    &2\overline{F}(x,y,z)=2e^{2u}\cos{2v}\, F(x,y,z)\\[6pt]
    &+2e^{2w}\eta(x)\left[\eta(y)\D w(\f z)+\eta(z)\D w(\f
    y)\right]\\[6pt]
    &
    +e^{2u}\sin{2v} \left[F(\f y,z,x)-F(y,\f z,x)\right.\\[6pt]
    &\phantom{+e^{2u}\sin{2v}\left[\right.}
    +F(\f z,y,x)-F(z,\f y,x)\\[6pt]
    &\phantom{+e^{2u}\sin{2v}\left[\right.}
    \left.+F(x,\f y,\xi)\eta(z)+F(x,\f
    z,\xi)\eta(y)\right]\\[6pt]
    &+(e^{2w}-e^{2u}\cos{2v})\bigl\{\left[F(x,y,\xi)+F(\f y,\f
    x,\xi)\right]\eta(z)\\[6pt]
    &\phantom{+(e^{2w}-e^{2u}\cos{2v})+}
    +\left[F(x,z,\xi)+F(\f z,\f x,\xi)\right]\eta(y)\\[6pt]
    &\phantom{+(e^{2w}-e^{2u}\cos{2v})+}
    +\left[F(y,z,\xi)+F(\f z,\f y,\xi)\right]\eta(x)\\[6pt]
    &\phantom{+(e^{2w}-e^{2u}\cos{2v})+}
    +\left[F(z,y,\xi)+F(\f y,\f z,\xi)\right]\eta(x)\bigr\}
\\[6pt]
    &-2e^{2u}\bigl\{\bigl[
    \cos{2v}\left[\D u(\f z)+\D v(z)\right]\\[6pt]
    &\phantom{-2e^{2u}\bigl\{\bigl[}
    +\sin{2v}\left[\D u(z)-\D v(\f z)\right]\bigr]g(\f x,\f y)
\\[6pt]
    &\phantom{-2e^{2u}\bigl\{\bigl[}
    +\bigl[\cos{2v}\left[\D u(\f y)+\D v(y)\right]\\[6pt]
    &\phantom{-2e^{2u}\bigl\{\bigl[\bigl[+}
    +\sin{2v}\left[\D u(y)-\D v(\f y)\right]\bigr]g(\f x,\f z)
\\[6pt]
    &\phantom{-2e^{2u}\bigl\{\bigl[}
    +\bigl[\cos{2v}\left[\D u(z)-\D v(\f z)\right]\\[6pt]
    &\phantom{-2e^{2u}\bigl\{\bigl[\bigl[+}
        -\sin{2v}\left[\D u(\f z)+\D v(z)\right]\bigr]g(x,\f y)
%\\[6pt]
\end{split}
\end{equation}
\begin{equation}
\begin{split}
    &\phantom{-2e^{2u}\bigl\{\bigl[}
    +\bigl[\cos{2v}\left[\D u(y)-\D v(\f y)\right]\\[6pt]
    &\phantom{-2e^{2u}\bigl\{\bigl[\bigl[+}
        -\sin{2v}\left[\D u(\f y)+\D v(y)\right]\bigr]g(x,\f z)\bigr\}.
\end{split}
\end{equation}
\end{subequations}

The Sasaki-like condition  \eqref{defsl} also reads as
\begin{equation}\label{defsl1}
F(x,y,z)=g(\f x,\f y)\eta(z)+g(\f x,\f z)\eta(y).
\end{equation}

We obtain the Sasaki-like condition for the metric $\overline g$
substituting \eqref{Transf} into \eqref{defsl1} which yields
\begin{equation}\label{defbar}
\begin{split}
\overline F(x,y,z)&=e^{w+2u}\Big\{
\cos{2v}\bigl[\eta(z)g(\f x,\f y)+\eta(y)g(\f x,\f z)\bigr]\\[6pt]
&\phantom{=e^{w+2u}\Big\{}
-\sin{2v}\bigl[\eta(z)g(x,\f y)+\eta(y)g(x,\f z)\bigr]\Big\}.
\end{split}
\end{equation}
We substitute \eqref{defsl1} into \eqref{ff} to get the following expression.

%\begin{subequations}
\begin{equation}\label{ff1}
\begin{split}
   \overline{F}(x,y,z)&=e^{2w}\eta(x)\left\{\eta(y)\D w(\f z)+\eta(z)\D w(\f y)\right\}
\\[6pt]
    &+e^{2u}\Bigl\{
    \cos{2v}\left[\eta(z)g(\f x,\f y)+\eta(y)g(\f x,\f z)\right]\\[6pt]
    &\phantom{-e^{2u}\Bigl\{}
    -\sin{2v}\left[\eta(z)g(x,\f y)+\eta(y)g(x,\f z)\right]%\Bigr\}
\\[6pt]
    &\phantom{-e^{2u}\Bigl\{}%-e^{2u}\Bigl\{%
    -\left\{%
    \cos{2v}\left[\D u(\f z)+\D v(z)\right]\right.\\[6pt]
    &\phantom{-e^{2u}\Bigl\{-\left\{\right.}%
    \left.+\sin{2v}\left[\D u(z)-\D v(\f z)\right]\right\}g(\f x,\f y)
\\[6pt]
    &\phantom{-e^{2u}\Bigl\{}%
    -\left\{%
    \cos{2v}\left[\D u(\f y)+\D v(y)\right]\right.\\[6pt]
    &\phantom{-e^{2u}\Bigl\{-\left\{\right.}%
    \left.+\sin{2v}\left[\D u(y)-\D v(\f y)\right]\right\}g(\f x,\f z)
\\[6pt]
    &\phantom{-e^{2u}\Bigl\{}%
    -\left\{%
    \cos{2v}\left[\D u(z)-\D v(\f z)\right]\right.\\[6pt]
    &\phantom{-e^{2u}\Bigl\{-\left\{\right.}%
    \left.-\sin{2v}\left[\D u(\f z)+\D v(z)\right]\right\}g(x,\f y)
\\[6pt]
% \end{split}
%\end{equation}
%\begin{equation}
%\begin{split}
   &\phantom{-e^{2u}\Bigl\{}%
    -\left\{%
    \cos{2v}\left[\D u(y)-\D v(\f y)\right]\right.\\[6pt]
    &\phantom{-e^{2u}\Bigl\{-\left\{\right.}%
    \left.-\sin{2v}\left[\D u(\f y)+\D v(y)\right]\right\}g(x,\f z)
    \Bigr\}.
\end{split}
\end{equation}
%\end{subequations}
The equalities \eqref{ff1} and \eqref{defbar} imply
\begin{subequations}\label{ff2}
\begin{equation}
\begin{split}
    &(1-e^w)e^{2u}\Bigl\{
    \cos{2v}\left[\eta(z)g(\f x,\f y)+\eta(y)g(\f x,\f z)\right]\\[6pt]
    &\phantom{(1-e^w)e^{2u}\Bigl\{}
    -\sin{2v}\left[\eta(z)g(x,\f y)+\eta(y)g(x,\f z)\right]\Bigr\}
\\[6pt]
    &-e^{2u}\Bigl\{%
    \left\{%
    \cos{2v}\left[\D u(\f z)+\D v(z)\right]\right.\\[6pt]
    &\phantom{-e^{2u}\Bigl\{\left\{\right.}%
    \left.+\sin{2v}\left[\D u(z)-\D v(\f z)\right]\right\}g(\f x,\f y)
\\[6pt]
    &\phantom{-e^{2u}\Bigl\{}%
    +\left\{%
    \cos{2v}\left[\D u(\f y)+\D v(y)\right]\right.\\[6pt]
        &\phantom{-e^{2u}\Bigl\{\left\{\right.+}%
    \left.+\sin{2v}\left[\D u(y)-\D v(\f y)\right]\right\}g(\f x,\f z)
\\[6pt]
    &\phantom{-e^{2u}\Bigl\{}%
    +\left\{%
    \cos{2v}\left[\D u(z)-\D v(\f z)\right]\right.\\[6pt]
        &\phantom{-e^{2u}\Bigl\{\left\{\right.+}%
    \left.-\sin{2v}\left[\D u(\f z)+\D v(z)\right]\right\}g(x,\f y)
%\\[6pt]
\end{split}
\end{equation}
\begin{equation}
\begin{split}
    &\phantom{-e^{2u}\Bigl\{}%
    +\left\{%
    \cos{2v}\left[\D u(y)-\D v(\f y)\right]\right.\\[6pt]
        &\phantom{-e^{2u}\Bigl\{\left\{\right.+}%
    \left.-\sin{2v}\left[\D u(\f y)+\D v(y)\right]\right\}g(x,\f z)
    \Bigr\}
\\[6pt]
    &%
    +e^{2w}\eta(x)\left\{\eta(y)\D w(\f z)+\eta(z)\D w(\f y)\right\}
    =0.
\end{split}
\end{equation}
\end{subequations}
We set $x=y=\xi$ into \eqref{ff2} to get
\begin{equation}\label{www}
\D w(\f z)=0.
\end{equation}
 Now, using \eqref{www} we write
\eqref{ff2} in  the form
\begin{equation}\label{ff3}
\begin{split}
&A(z)g(\f x,\f y)+B(z)g(x,\f y)\\[6pt]
+&A(y)g(\f x,\f z)+B(y)g(x,\f z)=0,
\end{split}
\end{equation}
where the 1-forms $A$ and $B$ are defined by
\begin{equation}\label{ggg}
\begin{split}
A(z)&=\cos{2v}\left[(e^w-1)\eta(z)+\D u(\f z)+\D v(z)\right]\\[6pt]
    &\phantom{=}+\sin{2v}\left[\D u(z)-\D v(\f z)\right],\\[6pt]
B(z)&=\sin{2v}\left[(e^w-1)\eta(z)+\D u(\f z)+\D v(z)\right]\\[6pt]
    &\phantom{=}-\cos{2v}\left[\D u(z)-\D v(\f z)\right].
\end{split}
\end{equation}
We take the trace of \eqref{ff3} with respect to $x=e_i$, $z=e_i$
and $y= e_i$, $z=e_i$ to get
\begin{equation}\label{ff4}
\begin{split}
 - (2n+1)&A(z)+\eta(z)A(\xi)
+B(\f z)=0, \quad\\[6pt]
   & A(z)-\eta(z)A(\xi) -B(\f z)=0.
 \end{split}
\end{equation}
We derive from \eqref{ff4} that $A(z)=0$. Similarly, we obtain
$B(z)=0$. Now, \eqref{ggg} imply
\begin{equation}\label{ggg1}
\begin{split}
&\cos{2v}\left[\D u(\f z)+\D v(z)\right]\\[6pt]
&+\sin{2v}\left[\D u(z)-\D v(\f z)\right]=(1-e^w)\cos{2v}\ \eta(z),\\[6pt]
&\sin{2v}\left[\D u(\f z)+\D v(z)\right]\\[6pt]
&-\cos{2v}\left[\D u(z)-\D v(\f z)\right]=(1-e^w)\sin{2v}\ \eta(z).
\end{split}
\end{equation}
%Using \eqref{ggg1} and \eqref{www}, we compare \eqref{defbar} and \eqref{ff1}.
Then we derive
\begin{prop}\label{prop:Sasaki}
Let $\M$ be a Sasaki-like almost contact complex Rie\-mann\-ian
manifold. Then the structure $(\f,
\overline{\xi},\overline{\eta},\overline g)$ defined by
\eqref{Transf} is Sasaki-like if and only if the smooth functions
$u,v,w$ satisfy the following conditions
\begin{equation}\label{sssl}
dw\circ\f=0,\quad \D u-\D v\circ\f=0,\quad \D u\circ\f+\D v=(1-e^w)\eta.
\end{equation}
In particular
\[
\D u(\xi)=0,\qquad \D v(\xi)=1-e^w.
\]
In the case $w=0$ the global smooth functions $u$ and $v$ does not
depend on $\xi$ and they are locally defined on the complex submanifold
$\MM^{2n}$, $T\MM^{2n}=\HC$, as well as the complex valued function $u+\sqrt{-1}\,v$
is a holomorphic function on $\MM^{2n}$.
\end{prop}
\begin{proof}
%Substitute $\alpha=e^{2u}\cos{2v}, \beta=e^{2v}\sin{2v}$ into
%\eqref{ggg1} and
Solve the %obtained
linear system \eqref{ggg1} to get the second and the third
equality into \eqref{sssl}. Now, \eqref{www} completes the proof
of \eqref{sssl}.
\end{proof}

\vskip 0.2in \addtocounter{subsubsection}{1}

\noindent  {\Large\bf{\emph{\thesubsubsection. Contact homothetic transformations}}}

\vskip 0.15in

%\subsection{Contact homothetic transformations}

Let us consider contact homothetic transformations of an almost
contact complex Riemannian manifold $\M$. Since the functions $u$,
$v$, $w$ are constant, it follows from \eqref{Transf} using the
Koszul formula \eqref{koszul} that the Levi-Civita connections
$\overline{\n}$ and $\n$ of the metrics $\overline g$ and $g$,
respectively, are connected by the formula
\begin{equation}\label{barl}
\begin{split}
\overline{\n}_xy=\n_xy&+e^{2(u-w)}\sin{2v}\ g(\f x,\f y)\xi\\[6pt]
&-\left(e^{-2w}-e^{2(u-w)}\cos{2v}\right) g(x,\f y)\xi.
\end{split}
\end{equation}

For the corresponding curvature tensors $\overline R$ and $R$   we
obtain from \eqref{barl} that
\begin{equation}\label{barRR}
\begin{aligned}
&\overline{R}(x,y)z={R}(x,y)z\\[6pt]
&+e^{2(u-w)}\sin{2v} %
\left\{g(y,\f z)\eta(x)\xi-g(\f y,\f z)\f x\right.\\[6pt]
&\phantom{+e^{2(u-w)}\sin{2v}}\left.-g(x,\f z)\eta(y)\xi+g(\f x,\f
z)\f
y\right\}\\[6pt]
 &+\left(e^{-2w}-e^{2(u-w)}\cos{2v}\right)\left\{g(\f y,\f
z)\eta(x)\xi+g(y,\f z)\f x\right.\\[6pt]
&\phantom{+\left(e^{-2w}-e^{2(u-w)}\cos{2v}\right)+}\left.-g(\f
x,\f z)\eta(y)\xi-g(x,\f z)\f y\right\}.
\end{aligned}
\end{equation}
We have
\begin{prop}\label{homric}
The Ricci tensor of an almost contact complex Riemannian manifold
is invariant under a contact homothetic transformation,
\begin{equation}\label{ri}
\overline{\Ric}=\Ric.
\end{equation}
Consequently, we obtain the following relations:
\par
\begin{equation}\label{scal}
\begin{split}
\overline{\Scal}=e^{-2u}\cos{2v}\ \Scal&-e^{-2u}\sin{2v}\ \Scal^*\\[6pt]
&+\left(e^{-2w}-e^{-2u}\cos{2v}\right)\Ric(\xi,\xi),\\[6pt]
\overline{\Scal}^*=e^{-2u}\sin{2v}\ \Scal&+e^{-2u}\cos{2v}\ \Scal^*\\[6pt]
&-e^{-2u}\sin{2v}\ \Ric(\xi,\xi).
\end{split}
\end{equation}
In particular, the scalar curvatures of a Sasaki-like almost
contact complex Riemannian manifold changes under a contact
homothetic transformation with $w=0$ as follows
\begin{equation}\label{scalsas}
\begin{split}
\overline{\Scal}=e^{-2u}\cos{2v}\ \Scal&-e^{-2u}\sin{2v}\ \Scal^*\\[6pt]
&+2n\left(1-e^{-2u}\cos{2v}\right),\\[6pt]
\overline{\Scal}^*=e^{-2u}\sin{2v}\ \Scal&+e^{-2u}\cos{2v}\ \Scal^*\\[6pt]
&-2n\ e^{-2u}\sin{2v}.
\end{split}
\end{equation}
\end{prop}
\begin{proof}
Taking the trace of \eqref{barRR} we get $\overline{\Ric}=\Ric$.

We consider the basis $\{\overline e_0=\overline\xi$, $\overline
e_1$, $\dots$, $\overline e_{n}$, $\overline e_{n+1}=\f \overline
e_1$, $\dots$, $\overline e_{2n}=\f \overline e_n\}$, where
\[
\overline e_i=e^{-u}\left\{\cos v\ e_i-\sin v\ \f
e_i\right\},\quad i=1,\dots, n.
\]
It is easy to check that this basis is orthonormal for the metric
$\overline g$. Then \eqref{ri} gives
\[
\overline{\Scal}=\sum_{i=0}^{2n}\overline\varepsilon_i \overline
\Ric(\overline e_i,\overline e_i),\quad
\overline{\Scal}^*=\sum_{i=0}^{2n}\overline\varepsilon_i \overline
\Ric(\overline e_i,\f \overline e_i),
\]
 which yield the formulas for
the scalar curvatures.

The formulas \eqref{scalsas} follow from \eqref{scal} and
\eqref{ricxi}.
\end{proof}
Consequently, we have
\begin{prop}
A Sasaki-like almost contact complex Riemannian manifold $\M$ is
Einstein if and only if the underlying holomorphic complex
Riemannian manifold $(\MM^{2n}, T\MM^{2n}=\HC,J,h)$  is an Einstein
manifold with scalar curvature not depending on the vertical
direction $\xi$.
\end{prop}
\begin{proof}
We compare \eqref{ricxi} with \eqref{ric} to see that $\M$ is an
Einstein manifold if and only if $\MM^{2n}$ is an Einstein manifold with
Einstein constant equal to $2n$, i.e. $\Ric^h=2n\,g$.

Further, we consider
a contact homothetic transformation with $w=v=0$ and we get that
$\bigl(\MM,\f,\xi,\eta,\overline
g=e^{2u}g+(1-e^{2u})\eta\otimes\eta\bigr)$ is again Sasaki-like
due to Proposition~\ref{prop:Sasaki}. Applying
Proposition~\ref{homric} and \eqref{ric}, we get the following
sequence of equalities
\[
\overline{\Ric}^{\bar h}=\overline{\Ric}{|_\HC}=\Ric|_\HC=\Ric^h
=\frac{\Scal}{2n}g|_\HC=\frac{e^{-2u}\Scal^h}{2n}\overline{g}|_\HC,
\]
which yield $\overline{\Scal}^{\bar h}=e^{-2u}\Scal^h=4n^2$ by
choosing the constant $u$ to be equal to
$e^{-2u}=\frac{4n^2}{\Scal^h}$, i.e. the Einstein constant of the
complex holomorphic Einstein manifold $\MM^{2n}$ can always be made equal
to $4n^2$ which completes the proof.
\end{proof}
Suppose we have a Sasaki-like almost contact complex Riemannian
manifold which is Einstein, i.e. $\Ric=2n\,g$, and we make a contact
homothetic transformation
\begin{equation}\label{cctt}
\begin{split}
&\overline{\eta}=\eta,\qquad \overline{\xi}=\xi,\quad\\[6pt]
&\overline{g}(x,y) = c\ g(x,y)+d\ g(x,\f y)+(1-c)\eta(x)\eta(y),
\end{split}
\end{equation}
where $c$, $d$ are constants. According to
Proposition~\ref{homric} and using \eqref{cctt}, we obtain that
\begin{equation}\label{etaein}
\begin{split}
\overline{\Ric}(x,y)&=\Ric(x,y)=2n\,g(x,y)
\\[6pt]
&=\frac{2n}{c^2+d^2}\bigl\{c\,\overline g(x,y) - d\,\overline
g(x,\f y) \\[6pt]
&\phantom{=\frac{2n}{c^2+d^2}\bigl\{c\,\overline g(x,y)\,}
+(c^2+d^2-c)\eta(x)\eta(y)\bigr.\}.
\end{split}
\end{equation}
We may call a Sasaki-like space whose Ricci tensor satisfies
\eqref{etaein}  \emph{an $\eta$-complex-Ein\-stein Sasaki-like
manifold} and if the constant $d$ vanishes, $d=0$, we have
\emph{$\eta$-Ein\-stein Sasaki-like space}. Thus, we have shown
\begin{prop}
Any $\eta$-complex-Einstein Sasaki-like space is contact
homothetic to an Einstein Sasaki-like space.
\end{prop}

\vspace{20pt}

\begin{center}
$\divideontimes\divideontimes\divideontimes$
\end{center} 

%%
%
%%%\include{Man-partB}
%\include{Man-chap2}
%%%%%%%%%%%%%%%%%%%%%%%%%%%%%%%%%%%%%%%%%%%%%%%%%%%%%%%%%%%%%%%%%%%%%%%%%%%%%%%
\newpage
\thispagestyle{empty}

%\count0 = 19 % номер начална страница

\label{chap2}

\Large{

\
\\
\bigskip\
\\
\bigskip\
\\
\bigskip

\begin{center}
\Huge{
\textbf{
Chapter II. \\[6pt] }}
\HRule
\Huge{
\textbf{\\[12pt]
                                \textsc{On manifolds with almost \\[12pt]
                                        hypercomplex structures and \\[12pt]
                                        almost contact 3-structures, \\[12pt]
                                        equipped with metrics of \\[12pt]
                                        Hermitian-Norden type  }}}
\end{center}

}

%\               премахната празна страница
%\newpage
%\thispagestyle{empty}
%
%$ \phantom{\quad} $
%

%%%\include{Man-chapIII}
%
%
%\include{Man-4n}
\newpage

\addtocounter{section}{1}\setcounter{subsection}{0}\setcounter{subsubsection}{0}

\setcounter{thm}{0}\setcounter{equation}{0}

\label{par:4n}

 \Large{

\
\\[6pt]
\bigskip

\
\\[6pt]
\bigskip

\lhead{\emph{Chapter II $|$ \S\thesection. {Almost hypercomplex manifolds with Hermitian-Norden metrics% \ldots
}}}
%\thispagestyle{empty}

%\noindent  {\Huge\bf \S\thesection. Almost hypercomplex manifolds  \\[12pt]
%\phantom{\S\thesection. }with Hermitian-Norden metrics
%}%\\[6pt]\vskip2pt

\noindent
\begin{tabular}{r"l}
  %\hline
\hspace{-6pt}{\Huge\bf \S\thesection.}  & {\Huge\bf Almost hypercomplex manifolds} \\[12pt]
                             & {\Huge\bf with Hermitian-Norden metrics}
  %\hline
\end{tabular}

\vskip 1cm

\begin{quote}
\begin{large}
In the present section, we give some facts
about the almost hypercomplex manifolds with Hermitian-Norden metrics known from
\cite{AlMa,GriMan24,GriManDim12,Man28}.
\end{large}
\end{quote}

%
%\vskip 0.2in \addtocounter{subsection}{1}
%
%\noindent  {\Large\bf \thesubsection. Introduction}

\vskip 0.15in

%\section{Introduction}\label{intro}

Let us recall the notion of \emph{almost hypercomplex structure} $H$
on a %$4n$-dimensional
manifold $\MM^{4n}$. This structure is a triad %$H=(J_\alpha)$ $(\alpha=1,2,3)$
of anticommuting almost complex structures
such that each of them is a composition of two other structures
%satisfying the property $J_3=J_1\circ J_2$
\cite{AlMa,So}.

We equip an almost hypercomplex structure $H$ with a metric structure,
generated by a pseudo-Riemannian metric $g$ of neutral signature
\cite{GriMan24,GriManDim12}. In our case, one (resp., the other
two) of the almost complex structures of $H$ acts as an isometry
(resp., act as anti-isometries) with respect to $g$ in each
tangent fibre.
Thus,  there exist three (0,2)-tensors
associated by $H$ to the metric $g$: a K\"ahler form and two metrics of the same type.
The metric $g$ is Hermitian with respect to one of almost
complex structures of $H$ and $g$ is a Norden
metric regarding the other two almost complex structures
of $H$. For this reason we call the derived almost hypercomplex
structure an \emph{almost hypercomplex structure with Hermitian-Norden metrics}.
%or briefly \emph{almost hypercomplex HN-metric structure}.

Let $(\MM,H)$ be an almost hypercomplex manifold, i.e. $\MM$ is a
$4n$-dimension\-al differentiable manifold and $H=(J_1,J_2,J_3)$
is a triad of almost complex structures on $\MM$ with the following
properties for all cyclic permutations $(\al, \bt, \gm)$ of
$(1,2,3)$:%
\begin{equation}\label{J123} %
J_\al=J_\bt\circ J_\gm=-J_\gm\circ J_\bt, \qquad
J_\al^2=-I,
\end{equation} %
where $I$ denotes the identity. %\cite{AlMa}.

Let $g$ be a neutral metric on $(\MM,H)$ with the properties
\begin{equation}\label{gJJ} %
g(\cdot,\cdot)=\ea g(J_\al \cdot,J_\al \cdot), \end{equation} %
where
\begin{equation}\label{epsiloni}
\ea=
\begin{cases}
\begin{array}{ll}
1, \quad & \al=1;\\%[4pt]
-1, \quad & \al=2;3.
\end{array}
\end{cases}
\end{equation}

Further, the index $\alpha$ runs over the range
$\{1,2,3\}$ unless otherwise stated.

The associated (K\"ahler) 2-form $\g_1$ and the
associated neutral metrics $\g_2$ and $\g_3$ are determined by
\begin{equation}\label{gJ} %
\g_\al(\cdot,\cdot)=g(J_\al \cdot,\cdot)=-\ea g(\cdot,J_\al \cdot).
\end{equation}%

Let us note that $J_1$ (resp., $J_3$, $J_2$) acts as an
isometry with respect to $g$ (resp., $\g_2$, $\g_3$)
and moreover $J_2$ and $J_3$ (resp., $J_1$ and $J_2$, $J_1$ and $J_3$) act as anti-isometries with respect to $g$ (resp., $\g_2$,
$\g_3$).
On the other hand, a quaternionic inner product $\langle\cdot,\cdot\rangle$
%over the quaternionic algebra
is generated in a natural way by the bilinear forms $g$, $\g_1$,
$\g_2$ and $\g_3$ by the following decomposition:
$\langle\cdot,\cdot\rangle=-g+i\g_1+j\g_2+k\g_3$.

We call the structure $(H,G)=(J_1,J_2,J_3;g,\g_1,\g_2,\g_3)$ on
$\MM^{4n}$ an \emph{almost hy\-per\-com\-plex
structure with Hermit\-ian-Norden metrics}. % or shortly an \emph{almost $hcHN$-structure}.
We call
the manifold $(\MM,H,G)$ an \emph{almost hypercomplex
manifold with Hermit\-ian-Norden metrics}. % or shortly an \emph{almost $hcHN$-manifold}.
These manifolds are introduced and studied in \cite{GriManDim12,Ma05,ManSek,GriMan24,Man28,ManGri32,NakHMan}.

According to \cite{GriManDim12}, the fundamental tensors of such a manifold are the following three
$(0,3)$-tensors
\begin{equation}\label{F'-al}
F_\al (x,y,z)=g\bigl( \left( \DDD_x J_\al
\right)y,z\bigr)=\bigl(\DDD_x \g_\al\bigr) \left( y,z \right),
\end{equation}
where $\DDD$ is the Levi-Civita connection generated by $g$.
These tensors have the following basic properties caused by the structures
\begin{equation}\label{FaJ-prop}
  F_{\al}(x,y,z)=-\ea F_{\al}(x,z,y)=-\ea F_{\al}(x,J_{\al}y,J_{\al}z).
\end{equation}

The following relations between the
tensors $F_\alpha$ are valid %, e.g.
%\[F_1(\cdot,\cdot,\cdot)=F_2(\cdot,J_3\cdot,\cdot)+F_3(\cdot,\cdot,J_2\cdot).\]
\begin{equation}\label{F1F2F3}
\begin{array}{l}
    F_1(x,y,z)=F_2(x,J_3y,z)+F_3(x,y,J_2z),\\[6pt]
    F_2(x,y,z)=F_3(x,J_1y,z)+F_1(x,y,J_3z),\\[6pt]
    F_3(x,y,z)=F_1(x,J_2y,z)-F_2(x,y,J_1z).
\end{array}
\end{equation}

The corresponding Lee forms $\ta_\al$ are defined by
\begin{equation}\label{theta-al}
\ta_\al(\cdot)=g^{ij}F_\al(e_i,e_j,\cdot)
\end{equation}%
for an arbitrary basis $\{e_1,e_2,\dots, e_{4n}\}$ of $T_p\MM$,
$p\in \MM$.

%\begin{remark}
In \cite{GriManDim12}, we study the so-called \emph{hyper-K\"ahler
manifolds with Hermitian-Norden metrics}%(or pseudo-hyper-K\"ahler manifolds)
, i.e. the almost hypercomplex manifold with Hermitian-Norden metrics in the class
$\KKK$, where $\DDD J_\al=0$ for all $\al$. A sufficient
condition for $(\MM,H,G)$ to be in $\KKK$ is this manifold to be of
K\"ahler-type with respect to two of the three complex structures
of $H$ \cite{GriMan24}.
%
%\end{remark}

As $g$ is an indefinite metric, there exist isotropic vectors $x$
on $\MM$, i.e.{} \(g(x,x)=0\), \(x\neq 0\). %Following \cite{GRMa}
In \cite{GriMan24}, we define the invariant square norm
\begin{equation}\label{nJ}
\nJ{\al}= g^{ij}g^{kl}g\bigl( \left( \DDD_i J_\al \right) e_k,
\left( \DDD_j J_\al \right) e_l \bigr),
\end{equation}
where $\{e_1,e_2,\dots, e_{4n}\}$ is an arbitrary basis of the
tangent space $T_p\MM$ at an arbitrary point $p\in \MM$. We
say that an almost hypercomplex manifold with Hermitian-Norden metrics is an \emph{isotropic
hyper-K\"ahler manifold with Hermitian-Norden metrics} if $\nJ{\al}=0$ for each $J_\al$ of
$H$. Clearly, if the manifold is a hyper-K\"ahler
manifold with Hermitian-Norden metrics, then it is an isotropic hyper-K\"ahler manifold with Her\-mit\-ian-Norden metrics.
The inverse statement does not hold.

Let us note that according to \eqref{gJJ} the manifold $(\MM,J_1,g)$
is almost Hermitian whereas the manifolds $(\MM,J_\al,g)$, $\al=2,3$,
are almost Norden manifolds. %\cite{GaBo}.
The basic classes of the mentioned two types of
manifolds are given in \cite{GrHe} by A.~Gray, L.M.~Hervella and in \cite{GaBo} by G.~Ganchev, A.~Borisov, respectively.
They are determined for dimension $4n$ as follows:
\begin{enumerate}[a)]
\item for $\al=1$
\begin{equation}\label{cl-H}
\begin{split}
&\W_1(J_1):\; F_1(x,y,z)=-F_1(y,x,z); \\[6pt]
&\W_2(J_1):\; \mathop{\s}_{x,y,z}\bigl\{F_1(x,y,z)\bigr\}=0; \\[6pt]
&\W_3(J_1):\; F_1(x,y,z)=F_1(J_1x,J_1y,z),\quad \ta_1=0; \\[6pt]
&\W_4(J_1):\; F_1(x,y,z)=\frac{1}{4n-2}
                \bigl\{g(x,y)\ta_1(z)-g(x,z)\ta_1(y)\\%[6pt]
&\phantom{\W_4(J_1):\; F_1(x,y,z)=\frac{1}{4n-2}\,}
                -g(x,J_1y)\ta_1(J_1z) \\%[6pt]
&\phantom{\W_4(J_1):\; F_1(x,y,z)=\frac{1}{4n-2}\,}
                +g(x,J_1z)\ta_1(J_1y)
                \bigr\}
\end{split}
\end{equation}
\item
for $\al=2$ or $3$
\begin{subequations}\label{cl-N}
\begin{equation}
\begin{split}
&\W_1(J_\al):\;
F_\al(x,y,z)=\frac{1}{4n}\bigl\{
g(x,y)\ta_\al(z)+g(x,z)\ta_\al(y)\\%[6pt]
&\phantom{\W_1(J_\al):\; F_\al(x,y,z)=\frac{1}{4n}\,} %
+g(x,J_\al y)\ta_\al(J_\al z)\\%[6pt]
\end{split}
\end{equation}
\begin{equation}
\begin{split}
&\phantom{\W_1(J_\al):\; F_\al(x,y,z)=\frac{1}{4n}\,} %
+g(x,J_\al z)\ta_\al(J_\al y)\bigr\};\\[6pt]
&\W_2(J_\al):\; \mathop{\s}_{x,y,z}
\bigl\{F_\al(x,y,J_\al z)\bigr\}=0,\quad \ta_\al=0;\\[6pt]
&\W_3(J_\al):\; \mathop{\s}_{x,y,z} \bigl\{F_\al(x,y,z)\bigr\}=0.
\end{split}
\end{equation}
\end{subequations}
%where $\s $ is the cyclic sum by three arguments $x$, $y$, $z$.
\end{enumerate}

The special class $\W_0(J_\al):$ $F_\al=0$  of the
K\"ahler-type manifolds belongs to any other class within the
corresponding classification.

%Let us note that according to \eqref{gJJ} the manifold $(\MM,J_1,g)$
%is almost Hermitian and the manifolds $(\MM,J_\al,g)$, $\al=2,3$,
%are almost complex manifolds with Norden metric \cite{GaBo}. The
%basic classes of the mentioned two types of manifolds are given in
%\cite{GrHe} and \cite{GaBo}, respectively. The special class
%$\W_0(J_\al):$ $F_\al=0$ $(\al =1,2,3)$ of the K\"ahler-type
%manifolds belongs to any other class within the corresponding
%classification.
In the 4-dimensional case, the four basic classes
of almost Hermitian manifolds with respect to
$J_1$ are restricted to two:
$\W_2(J_1)$, the class of the almost K\"ahler manifolds, and
$\W_4(J_1)$, the class of the Hermitian manifolds.
%They are determined for $\dim \MM=4$ by:
%\begin{equation}\label{cl-H-dim4}
%\begin{split}
%&\W_2(J_1):\; \mathop{\s}_{x,y,z}\bigl\{F_1(x,y,z)\bigr\}=0; \\
%&\W_4(J_1):\; F_1(x,y,z)=\dfrac{1}{2}
%                \left\{g(x,y)\ta_1(z)-g(x,J_1y)\ta_1(J_1z)\right. \\
%&\phantom{\W_4(J_1):\; F_1(x,y,z)=\dfrac{1}{2}
%                \left\{\right.}
%                \left.-g(x,z)\ta_1(y)+g(x,J_1z)\ta_1(J_1y)
%                \right\},
%\end{split}
%\end{equation}
%where $\s $ is the cyclic sum by three
%arguments $x$, $y$, $z$.
%The basic classes of the almost Norden manifolds (i.e., for $\al=2$ or $3$)
%are determined for dimension $4$ as follows:
%\begin{equation}\label{cl-N-dim4}
%\begin{split}
%&\W_1(J_\al):\; F_\al(x,y,z)=\dfrac{1}{4}\bigl\{
%g(x,y)\ta_\al(z)+g(x,J_\al y)\ta_\al(J_\al z)\bigr.\\
%&\phantom{\W_1(J_\al):\; F_\al(x,y,z)=\dfrac{1}{4}\bigl\{\bigr.} %
%\bigl.+g(x,z)\ta_\al(y)
%    +g(x,J_\al z)\ta_\al(J_\al y)\bigr\};\\
%&\W_2(J_\al):\; \mathop{\s}_{x,y,z}
%\bigl\{F_\al(x,y,J_\al z)\bigr\}=0,\qquad \ta_\al=0;\\
%&\W_3(J_\al):\; \mathop{\s}_{x,y,z} \bigl\{F_\al(x,y,z)\bigr\}=0.
%\end{split}
%\end{equation}

It is known that the class of complex manifolds with Hermitian metric for $J_1$ is
$(\W_3\oplus\W_4)(J_1)$ and
the class of complex manifolds with Norden metric for $J_{\al}$ ($\al=2,3$) is
$(\W_1\oplus\W_2)(J_{\al})$.

%
%%The introduction and the beginning of the study of the
%%Hermitian-Norden structures on an almost hypercomplex manifold is
%%given in~\cite{GriManDim12,GrMa24}.
%Given a hypercomplex manifold $(\MM^{4n},H)$ a pseudo-Riemannian
%metric of signature $(2n,2n)$ is called \emph{Hermitian-Norden} if
%it satisfies the following identities (c.f. \cite{GriManDim12})
%\begin{equation}\label{1}
%g(\cdot,\cdot)=g(J_1\cdot,J_1\cdot)=-g(J_2\cdot,J_2\cdot)=-g(J_3\cdot,J_3\cdot).
%\end{equation}
%It generates a K\"ahler 2-form $\g_1$ and two Hermitian-Norden
%metrics $\g_2$ and $\g_3$ in the following way
%\begin{equation}\label{G}
%\g_1:=g(J_1\cdot,\cdot),\qquad \g_2:=g(J_2\cdot,\cdot),\qquad
%\g_3:=g(J_3\cdot,\cdot).
%\end{equation}

By definition, an almost hypercomplex structure $H=(J_\alpha)$  is
a \emph{hypercomplex structure} if the Nijenhuis tensors $[J_\al,J_\al]$, given by %\eqref{N_al}
\begin{equation}\label{N_al} %
[J_\al,J_\al](\cdot,\cdot)\allowbreak{}= \left[J_\al \cdot,J_\al \cdot
\right]
    -J_\al\left[J_\al \cdot,\cdot \right]
    -J_\al\left[\cdot,J_\al \cdot \right]
    -\left[\cdot,\cdot \right],
\end{equation}%
vanish on $\X(\MM)$ for each $\alpha$ (\cite{Boy,AlMa}).
%
%
%(TUK E S OBRATEN ZNAK),
%\[
%[J_\al,J_\al](\cdot,\cdot)=
%    \left[\cdot,\cdot \right]
%    +J_\alpha\left[\cdot,J_\alpha \cdot \right]
%    +J_\alpha\left[J_\alpha \cdot,\cdot \right]
%    -\left[J_\alpha \cdot,J_\alpha \cdot \right]
%\]
Moreover, it is known that $H$ is
hypercomplex if and only if two of $[J_\al,J_\al]$ vanish.

%It is valid %
%\begin{equation}\label{N=0} %
%[J_\al,J_\al]=0, %\qquad \al\in\{1,2,3\}
%\end{equation} %
%for the Nijenhuis tensors $[J_\al,J_\al]$ of $J_\al$ given by

Then the class $\HCC$ of hypercomplex manifolds with Hermitian-Norden metrics is
\[
\left(\W_3\oplus\W_4\right)(J_1)\cap\left(\W_1\oplus\W_2\right)(J_2)\cap\left(\W_1\oplus\W_2\right)(J_3),
\]
which for the 4-dimensional case is restricted to
\[
\W_4(J_1)\cap\left(\W_1\oplus\W_2\right)(J_2)\cap\left(\W_1\oplus\W_2\right)(J_3).
\]

Let $R$ be the curvature tensor of the Levi-Civita connection
$\DDD$, generated by $g$.
%, be defined, as usually, by
%$R(x,y)z=\left[\DDD_x,\DDD_y\right] z -
%\DDD_{\left[x,y\right]} z$. The corresponding $(0,4)$-tensor is
%determined by the equality
%\[
%R(x,y,z,w)=g\bigl(R(x,y)z,w\bigr).
%\]
Obviously, $R$
is a \emph{K\"ahler-type tensor} on an arbitrary hyper-K\"ahler
manifold with Hermitian-Norden metrics, i.e.
\be{K-kel}%
\begin{array}{l}
R(x,y,z,w)=\ea R(x,y,J_\al z,J_\al w)=\ea R(J_\al x,J_\al y,z,w).
\end{array}
\ee

A basic property of the hyper-K\"ahler manifolds with Hermitian-Norden metrics is given in
\cite{GriManDim12} by the following
\begin{thm}[\cite{GriManDim12}]\label{thm-K=0}
Each hyper-K\"ahler manifold with Hermitian-Nor\-den metrics is a flat pseudo-Rie\-mann\-ian
ma\-ni\-fold of signature $(2n,2n)$.
\end{thm}

In \cite{Man28}, it is proved the following more general property.
\begin{thm}[\cite{Man28}]\label{th-0}
Each K\"ahler-type tensor on an almost hypercomplex
manifold with Hermitian-Nor\-den metrics is ze\-ro.
\end{thm}

\vspace{20pt}

\begin{center}
$\divideontimes\divideontimes\divideontimes$
\end{center} 

\newpage

\addtocounter{section}{1}\setcounter{subsection}{0}\setcounter{subsubsection}{0}

\setcounter{thm}{0}\setcounter{equation}{0}

\label{par:4Lie}

 \Large{

\
\\
\bigskip

\
\\
\bigskip

\lhead{\emph{Chapter II $|$ \S\thesection. Hypercomplex structures with Hermitian-Norden metrics
    on 4-dimensional \ldots %Lie algebras
}}
%\thispagestyle{empty}

%\noindent  {\Huge\bf \S\thesection. Hypercomplex structures with \\[12pt]
%\phantom{\S\thesection. }Hermitian-Nor\-den metrics on \\[12pt]
%\phantom{\S\thesection. }4-dimensional Lie algebras}%\\\vskip2pt}

\noindent
\begin{tabular}{r"l}
  %\hline
\hspace{-6pt}{\Huge\bf \S\thesection.}  & {\Huge\bf Hypercomplex structures with} \\[12pt]
                             & {\Huge\bf Hermitian-Nor\-den metrics on} \\[12pt]
                             & {\Huge\bf 4-dimensional Lie algebras}
  %\hline
\end{tabular}

\vskip 1cm

\begin{quote}
\begin{large}
In the present section, integrable hypercomplex structures with Hermitian-Norden metrics on Lie groups of dimension 4 are considered. The corresponding five types of invariant hypercomplex structures with
hyper-Hermitian metric are constructed here. The different cases
regarding the signature of the basic pseudo-Riemannian metric are considered.

The main results of this section are published in \cite{Man44}.
\end{large}\end{quote}

%\vskip 0.2in \addtocounter{subsection}{1} \setcounter{subsubsection}{0}
%
%\noindent  {\Large\bf \thesubsection. Introduction}%\\\vskip2pt}

\vskip 0.15in

The study in the present section is inspired by the work of M.L. Bar\-be\-ris \cite{Barb}
where invariant hypercomplex structures $H$ on 4-dimensional
real Lie groups are classified.
In that case the corresponding metric is positive definite and
Hermitian with respect to the triad of complex structures of $H$.
Our main goal is to classify 4-dimensional real Lie algebras
which admit hypercomplex structures with Hermitian-Norden metrics.

Let us remark that in \cite{Snow} and \cite{Ovan} are classified
the invariant complex structures on 4-dimensional solvable
simply-connected real Lie groups where the dimension of commutators
is less than three and equal to three, respectively.

A hypercomplex structure is called \emph{Abelian}, if
$[J_{\al}\cdot, J_{\al}\cdot ] = [\cdot, \cdot]$ is valid
for arbitrary vector fields and for all $\al$  \cite{BaDoMi}. Abelian hypercomplex structures are
considered in \cite{BaDo,DotFin} and
they can occur only on solvable Lie algebras \cite{FinGran}.
It is clear that the condition %
\[
[J,J](x,y) = [Jx,Jy]-J[Jx,y]-J[x,Jy]-[x,y]=0
\]
can be rewritten as
\[
[Jx,Jy]-[x,y] = J([Jx,y]+[x,Jy])\]
 for all vector fields $x, y$. Thus,
Abelian complex structures and therefore Abelian hypercomplex
structure are integrable.

%
%
%If the three almost complex structures of $H$ are parallel with
%respect to the Levi-Civita connection $\DDD$ of $g$ then such
%hypercomplex manifolds with Hermitian-Norden metrics of K\"ahler type we call
%\emph{hyper-K\"ahler manifolds with Hermitian-Norden metrics}, which are flat according to
%\cite{GriManDim12}.

In the present section we construct different types of hypercomplex
structures on Lie algebras following the Barberis classification.
The basic problem here is the existence and the geometric
characteristics of hypercomplex structures with Hermitian-Norden metrics on 4-dimensional
Lie algebras according to the Barberis classification. The
main results of this section are the construction of the different types
of the considered structures and their characterization.

%
%\vskip 0.2in \addtocounter{subsection}{1} \setcounter{subsubsection}{0}
%
%\noindent  {\Large\bf \thesubsection. Preliminaries}%\\\vskip2pt}

%\vskip 0.15in

%\section{Preliminaries}\label{sec1}%Almost hypercomplex manifolds with HN-metric structure}

%
%\vskip 0.15in \addtocounter{subsection}{1} \setcounter{subsubsection}{0}
%
%\noindent  {\Large\bf \thesection. Four-dimensional Lie algebras with such structures }%\\\vskip2pt}
%
%\vskip 0.15in
%
%%\section{}\label{sec2}

A hypercomplex structure on a Lie group is said to be \emph{invariant}
if left translations by elements of the Lie group are holomorphic with respect to
$J_{\al}$ for all $\al$. Obviously, a hypercomplex structure on the corresponding Lie algebra induces an invariant hypercomplex structure on the Lie group by left translations.

Let $\LL$ be a simply connected 4-dimensional real Lie group
admitting an invariant hypercomplex structure. A left invariant
metric on $\LL$ is called \emph{invariant hyper-Hermitian}  if it is
hyper-Hermitian with respect to some invariant hypercomplex
structure on $\LL$. It is known that all such metrics on given $\LL$
are equivalent up to homotheties.

Let $\mathfrak{l}$ denote the corresponding Lie algebra of $\LL$. Then it is known the following

\begin{thm}[\cite{Barb}]\label{thm-Barb}
The only 4-dimensional Lie algebras admitting an integrable
hypercomplex structure are the following types:
%\begin{enumerate}
%    \item[

\emph{\textbf{(hc1)}} $\mathfrak{l}$ is Abelian;$\qquad\qquad$

\emph{\textbf{(hc2)}} $\mathfrak{l}\cong\R\oplus\mathfrak{so}(3)$;
$\qquad\qquad$

\emph{\textbf{(hc3)}} $\mathfrak{l}\cong\mathfrak{aff}(\C)$;

\emph{\textbf{(hc4)}} $\mathfrak{l}$ is the solvable Lie algebra corresponding to $\R H^4$;

\emph{\textbf{(hc5)}} $\mathfrak{l}$ is the solvable Lie algebra corresponding to $\C H^2$,
%\end{enumerate}\
\\
where $\R\oplus\mathfrak{so}(3)$ is the Lie algebra of the Lie
groups $U(2)$ and $S^3\times S^1$; $\mathfrak{aff}(\C)$ is the Lie
algebra of the affine motion group of $\C$, the unique
4-dimensional Lie algebra carrying an Abelian hypercomplex
structure; $\R H^4$ is the real hyperbolic space; $\C H^2$ is the
complex hyperbolic space.
\end{thm}

Let $\{e_1,e_2,e_3,e_4\}$ be a basis of a 4-dimensional real Lie
algebra $\mathfrak{l}$ with center $\z$ and derived Lie algebra
$\mathfrak{l}'=[\mathfrak{l},\mathfrak{l}]$. A standard hypercomplex structure on $\mathfrak{l}$ is
defined as in \cite{So}:
\begin{equation}\label{JJJ}
\begin{array}{llll}
J_1e_1=e_2, \quad & J_1e_2=-e_1,  \quad &J_1e_3=-e_4, \quad &J_1e_4=e_3;
\\[6pt]
J_2e_1=e_3, &J_2e_2=e_4, &J_2e_3=-e_1, &J_2e_4=-e_2;
\\[6pt]
J_3e_1=-e_4, &J_3e_2=e_3, &J_3e_3=-e_2, &J_3e_4=e_1.\\[6pt]
\end{array}
\end{equation}

Let us introduced a pseudo-Euclidian metric $g$ of
neutral signature as follows
\begin{equation}\label{g}
g(x,y)=x^1y^1+x^2y^2-x^3y^3-x^4y^4,
\end{equation}
where $x(x^1,x^2,x^3,x^4)$, $y(y^1,y^2,y^3,y^4) \in \mathfrak{l}$. This
metric satisfies \eqref{gJJ} and \eqref{gJ}. Then the metric $g$
generates an almost hypercomplex structure with Hermitian-Norden metrics on $\mathfrak{l}$.

Let us consider the different cases of \thmref{thm-Barb}.

\vskip 0.2in \addtocounter{subsection}{1}

\noindent\Large{\textbf{\thesubsection. Hypercomplex structure of type (hc1)}}%\\\vskip2pt}

\vskip 0.15in

%\subsection{Hypercomplex structure of type (hc1)}
%\paragraph{(HC1)}  %
Obviously, in this case the considered manifold belongs to $\KKK$, the class of hyper-K\"ahler manifolds with Hermitian-Norden metrics.

\vskip 0.2in \addtocounter{subsection}{1}

\noindent\Large{\textbf{\thesubsection. Hypercomplex structure of type (hc2)}}%\\\vskip2pt}

\vskip 0.15in
%\subsection{Hypercomplex HN-metric structure of type (hc2)}
%\paragraph{(HC2)}%

Let $\mathfrak{l}$ be not solvable and let us determine it by %
\begin{equation}\label{HC2a}
[e_2,e_4]=e_3,\qquad
[e_4,e_3]=e_2,\qquad [e_3,e_2]=e_4. %
\end{equation}
In this consideration the $(+)$-unit $e_1\in\R$, i.e. $g(e_1,e_1)=1$,
is orthogonal to $\mathfrak{l}'$ with respect to $g$.

Then we compute the covariant derivatives in the basis with respect to the Levi-Civita connection $\DDD$ of $g$ and the nontrivial ones are
\begin{equation}\label{HC2a-n}
\begin{array}{lll}
\DDD_{e_2} e_3=-\dfrac{3}{2}e_4,\quad\quad & \DDD_{e_3}
e_2=-\dfrac{1}{2}e_4,\quad\quad &
\DDD_{e_4} e_2=\dfrac{1}{2}e_3, \\[6pt]
\DDD_{e_2} e_4=\dfrac{3}{2}e_3,\quad\quad & \DDD_{e_3}
e_4=-\dfrac{1}{2}e_2,\quad\quad & \DDD_{e_4} e_3=\dfrac{1}{2}e_2.
\end{array}
\end{equation}

By virtue of \eqref{HC2a-n}, \eqref{JJJ} and  \eqref{F'-al}, we obtain components
$(F_{\al})_{ijk}=F_{\al}(e_i,e_j,e_k)$, $i,j,k\in\{1,2,3,4\}$, as follows:
\begin{subequations}\label{HC2aFa}
\begin{equation}
\begin{aligned}
(F_{1})_{314}&=-(F_{1})_{323}=(F_{1})_{332}=-(F_{1})_{341}=
\\[6pt]
\phantom{(F_{1})_{314}}
&=-(F_{1})_{413}=-(F_{1})_{424}=(F_{1})_{431}=(F_{1})_{442}=\dfrac{1}{2},\\[6pt]
(F_{2})_{214}&=-(F_{2})_{223}=-(F_{2})_{232}=(F_{2})_{241}=\dfrac{3}{2},
\\[6pt]
(F_{2})_{412}&=(F_{2})_{421}=(F_{2})_{434}=(F_{2})_{443}=\dfrac{1}{2},\\[6pt]
(F_{3})_{213}&=(F_{3})_{224}=(F_{3})_{231}=(F_{3})_{242}=\dfrac{3}{2},\\[6pt]
(F_{3})_{312}&=(F_{3})_{321}=-(F_{3})_{334}=-(F_{3})_{343}=\dfrac{1}{2},%\\[6pt]
\end{aligned}
\end{equation}
\begin{equation}
\begin{aligned}
(F_{2})_{322}&=(F_{2})_{344}=-(F_{3})_{422}=-(F_{3})_{433}=-1
\end{aligned}
\end{equation}
\end{subequations}
and the others are zero.
 The only non-zero components $(\theta_{\al})_i=(\theta_{\al})(e_i)$, $i=1,2,3,4$,
 of the corresponding Lee forms are
\begin{equation}\label{HC2a-titi}
(\theta_1)_2=-1,\qquad (\theta_2)_3=-2,\qquad (\theta_3)_4=2.
\end{equation}
Using the results in \eqref{HC2aFa}, \eqref{HC2a-titi} and the classification conditions
\eqref{cl-H}, \eqref{cl-N} for dimension 4, we obtain
\begin{prop}\label{prop-HC2}
The hypercomplex manifold with Hermitian-Norden metrics on a 4-dimensional Lie algebra,
determined by \eqref{HC2a}, belongs to the largest class of the considered manifolds,
i.e. $\HCC$, as well as this manifold does not belong to neither $\W_1$ nor $\W_2$ for $J_2$ and $J_3$.
\end{prop}

The other possibility is the signature of $g$ on $\R$ to be $(-)$, e.g. $e_3\in\R$, where $g(e_3,e_3)=-1$.
By similar computations we establish the same class in the statement of \propref{prop-HC2}.

\vskip 0.2in \addtocounter{subsection}{1}

\noindent\Large{\textbf{\thesubsection. Hypercomplex structure of type (hc3)}}%\\\vskip2pt}

\vskip 0.15in
%\subsection{Hypercomplex HN-metric structure of type (hc3)}

%\paragraph{(HC3)}%

We analyze separately the cases of signature (1,1), (0,2) and (2,0) of $g$ on $\mathfrak{l}'$.

\vskip 0.15in

\noindent\Large{\textbf{\emph{\thesubsection.1. Case 1}}}\label{2.3.1.}%\\\vskip2pt}

\vskip 0.1in
%\subsubsection{}

Firstly, we consider $g$ of signature (1,1) on $\mathfrak{l}'$.

Let us determine $\mathfrak{l}$ by %
\begin{equation}\label{HC3a}
[e_2,e_3]=[e_1,e_4]=e_2,\quad\quad
[e_2,e_1]=[e_4,e_3]=e_4.
\end{equation}
Then we compute covariant derivatives
and the nontrivial ones are %
\begin{equation}\label{HC3a-n}
\begin{array}{ll}
\DDD_{e_2} e_1=\DDD_{e_4} e_3=e_4,\quad\quad &
\DDD_{e_2} e_2=-\DDD_{e_4} e_4=e_3,\\[6pt]
\DDD_{e_2} e_3=-\DDD_{e_4} e_1=e_2,\quad\quad & \DDD_{e_2} e_4=\DDD_{e_4}
e_2=e_1.
\end{array}
\end{equation}

By virtue of \eqref{HC3a}, \eqref{JJJ} and  \eqref{F'-al}, we obtain that $F_1=0$
and the other components $(F_{\al})_{ijk}$, $\al=2,3$, are as follows
\begin{equation}\label{HC3aFa}
\begin{aligned}
(F_{2})_{212}&=(F_{2})_{221}=(F_{2})_{234}=(F_{2})_{243}=
\\[6pt]
&=-(F_{2})_{414}=(F_{2})_{423}=(F_{2})_{432}=-(F_{2})_{441}=2,
\\[6pt]
(F_{3})_{211}&=-(F_{3})_{222}=-(F_{3})_{233}=(F_{3})_{244}=
\\[6pt]
&=(F_{3})_{413}=(F_{3})_{424}=(F_{3})_{431}=(F_{3})_{442}=-2
\end{aligned}
\end{equation}
and the others are zero.
 The only non-zero components of the corresponding Lee forms are
\begin{equation}\label{HC3a-titi}
(\theta_2)_1=(\theta_3)_2=4.
\end{equation}
Using that $F_1=0$, the results in \eqref{HC3aFa}, \eqref{HC3a-titi}
and the classification conditions \eqref{cl-H}, \eqref{cl-N}, we obtain

\begin{prop}\label{prop-HC3}
The hypercomplex manifold with Hermitian-Norden metrics on a 4-dimensional Lie algebra,
determined by \eqref{HC3a}, belongs to the subclass of the K\"ahler manifolds
with respect to $J_1$ of the largest class of the considered manifolds, i.e.
\[
\W_0(J_1)\cap\left(\W_1\oplus\W_2\right)(J_2)\cap\left(\W_1\oplus\W_2\right)(J_3),
\]
as well as this manifold does not belong to neither $\W_1$ nor $\W_2$ for $J_2$ and $J_3$.
\end{prop}

\vskip 0.15in

\noindent\Large{\textbf{\emph{\thesubsection.2. Case 2}}}\label{2.3.2.}%\\\vskip2pt}

\vskip 0.1in
%\subsubsection{}

Secondly, we consider $g$ of signature (2,0)  on $\mathfrak{l}'$. The case for signature (0,2) is similar.

Let us determine $\mathfrak{l}$ by \begin{equation}\label{HC3b}
[e_1,e_3]=[e_4,e_2]=e_1,\qquad [e_1,e_4]=[e_2,e_3]=e_2.
\end{equation} Then we compute covariant derivatives and the
nontrivial ones are \begin{equation}\label{HC3b-n}
\begin{array}{c}
\DDD_{e_1} e_1=\DDD_{e_2} e_2=e_3,\qquad
\DDD_{e_2} e_3=-\DDD_{e_4} e_1=e_2,\\[6pt]
\DDD_{e_1} e_3=\DDD_{e_4} e_2=e_1.
\end{array}
\end{equation}

By virtue of \eqref{HC3b}, \eqref{JJJ} and  \eqref{F'-al}, we obtain the following components of  $F_{\al}$:
\begin{equation}\label{HC3bFa}
\begin{aligned}
(F_{1})_{114}&=-(F_{1})_{123}=(F_{1})_{132}=-(F_{1})_{141}
\\[6pt]
&=(F_{1})_{213}=(F_{1})_{224}=-(F_{1})_{231}=-(F_{1})_{242}=-1,
\\[6pt]
(F_{2})_{111}&=(F_{2})_{133}=2,\\[6pt]
(F_{2})_{212}&=(F_{2})_{221}=(F_{2})_{234}=(F_{2})_{243}
\\[6pt]
&=-(F_{2})_{414}=(F_{2})_{423}=(F_{2})_{432}=-(F_{2})_{441}=1,
\\[6pt]
(F_{3})_{222}&=(F_{3})_{233}=2,\\[6pt]
-(F_{3})_{112}&=-(F_{3})_{121}=(F_{3})_{134}=(F_{3})_{143}
\\[6pt]
&=(F_{3})_{413}=(F_{3})_{424}=(F_{3})_{431}=(F_{3})_{442}=-1
\end{aligned}
\end{equation}
and the others are zero.
 The only non-zero components of the corresponding Lee forms are
\begin{equation}\label{HC3b-titi}
(\theta_1)_4=-2,\qquad (\theta_2)_1=(\theta_3)_2=4.
\end{equation}
Using the results in \eqref{HC3bFa},  \eqref{HC3b-titi}
and the classification conditions \eqref{cl-H}, \eqref{cl-N},
we obtain that the considered manifold belongs to the class \[
\W=\W_4(J_1)\allowbreak{}\cap\W_1(J_2)\cap\W_1(J_3).
\]
Remark that, according to \cite{GriManDim12}, necessary and sufficient
conditions a 4-dimensional almost hypercomplex manifold with Hermitian-Norden metrics to be
in the class $\W$ are:
\begin{equation}\label{titaJ}
\theta_2\circ J_2 =\theta_3\circ J_3 =-2\left(\theta_1\circ
J_1\right).
\end{equation}
These conditions are satisfied bearing in mind \eqref{HC3b-titi}.

Let us consider the class $\W^0=\left\{\W\;|\;\mathrm{d}\hspace{-3pt}\left(\theta_1\circ J_1\right)=0\right\}$,
which is the class of the (locally) conformally equivalent
$\KKK$-manifolds, where a conformal transformation of the metric is given by $\overline{g} = e^{2u}g$
for a differentiable function $u$ on the manifold.

Using \eqref{HC3b-titi} and \eqref{titaJ}, we establish that the considered manifold
belongs to the subclass $\W^0$.

\begin{prop}\label{prop-HC3b}
The hypercomplex manifold with Hermitian-Norden metrics on a 4-dimensional Lie algebra,
determined by \eqref{HC3b}, belongs to the class $\W^0$ of the (locally) conformally
equivalent $\KKK$-manifolds.
\end{prop}

%\newpage

\vskip 0.2in \addtocounter{subsection}{1}

\noindent\Large{\textbf{\thesubsection. Hypercomplex structure of type (hc4)}}%\\\vskip2pt}

\vskip 0.15in

%\subsection{Hypercomplex HN-metric structure of type (hc4)}
%\label{HC4}
%\paragraph{(HC4)}%

In this case,   $\mathfrak{l}$ is solvable and the derived Lie algebra $\mathfrak{l}'$ is 3-dimensional and Abelian.

\vskip 0.15in

\noindent\Large{\textbf{\emph{\thesubsection.1. Case 1}}}\label{2.4.1.}%\\\vskip2pt}

\vskip 0.1in

%\subsubsection{}

Firstly, we fix $e_1\in\mathfrak{l}$ with $g(e_1,e_1)=1$ as an
element ortho\-gonal to $\mathfrak{l}'$ with respect to $g$. Therefore $\mathfrak{l}$
is determined by \begin{equation}\label{HC4a} [e_1,e_2]=e_2,\qquad
[e_1,e_3]=e_3,\qquad [e_1,e_4]=e_4. \end{equation} Then we compute
covariant derivatives and the nontrivial ones are
\begin{equation}\label{HC4a-n}
\begin{array}{c}
\DDD_{e_2} e_1=-e_2, \qquad \DDD_{e_3} e_1=-e_3, \qquad \DDD_{e_4} e_1=-e_4, \\[6pt]
\DDD_{e_2} e_2=-\DDD_{e_3} e_3=-\DDD_{e_4} e_4=e_1.
\end{array}
\end{equation}

By similar computation as in the previous cases, the components
$(F_{\al})_{ijk}$ are as follows:
\begin{equation}\label{HC4aFa}
\begin{aligned}
(F_{1})_{314}&=-(F_{1})_{323}=(F_{1})_{332}=-(F_{1})_{341}=
\\[6pt]
&=-(F_{1})_{413}=-(F_{1})_{424}=(F_{1})_{431}=(F_{1})_{442}=1,
\\[6pt]
(F_{2})_{311}&=(F_{2})_{333}=-2,\\[6pt]
(F_{2})_{214}&=-(F_{2})_{223}=-(F_{2})_{232}=(F_{2})_{241}
\\[6pt]
&=(F_{2})_{412}=(F_{2})_{421}=(F_{2})_{434}=(F_{2})_{443}=-1,
\\[6pt]
(F_{3})_{411}&=(F_{3})_{444}=2,\\[6pt]
(F_{3})_{213}&=(F_{3})_{224}=(F_{3})_{231}=(F_{3})_{242}
\\[6pt]
&=(F_{3})_{312}=(F_{3})_{321}=-(F_{3})_{334}=-(F_{3})_{343}=-1
\end{aligned}
\end{equation}
and the others are zero.
 The only non-zero components of the corresponding Lee forms are
\begin{equation}\label{HC4a-titi}
(\theta_1)_2=-(\theta_2)_3=(\theta_3)_4=-2.
\end{equation}
The results in \eqref{HC4aFa}, \eqref{HC4a-titi} and the classification conditions
\eqref{cl-H}, \eqref{cl-N} imply

\begin{prop}\label{prop-HC4}
The hypercomplex manifold with Hermitian-Norden metrics on a 4-dimensional Lie algebra,
determined by \eqref{HC4a}, belongs to the largest class of the considered manifolds,
i.e. $\HCC$, as well as this manifold does not belong to neither $\W_1$ nor $\W_2$ for $J_2$ and $J_3$.
\end{prop}

\vskip 0.15in

\noindent\Large{\textbf{\emph{\thesubsection.2. Case 2}}}\label{2.4.2.}%\\\vskip2pt}

\vskip 0.1in

%\subsubsection{}

Secondly, we choose $e_4\in\mathfrak{l}$ with $g(e_4,e_4)=-1$ as an
element orthogonal to $\mathfrak{l}'$ with respect to $g$. Therefore, in
this case $\mathfrak{l}$ is determined by \begin{equation}\label{HC4b}
[e_4,e_1]=e_1,\qquad [e_4,e_2]=e_2,\qquad [e_4,e_3]=e_3.
\end{equation} Therefore, the nontrivial covariant derivatives are
\begin{equation}\label{HC4b-n}
\begin{array}{c}
\DDD_{e_1} e_1=\DDD_{e_2} e_2=-\DDD_{e_3} e_3=-e_4, \\[6pt]
\DDD_{e_1} e_4=-e_1, \qquad \DDD_{e_2} e_4=-e_2, \qquad \DDD_{e_3}
e_4=-e_3.
\end{array}
\end{equation}

In a similar way we obtain:
\begin{subequations}\label{HC4bFa}
\begin{equation}
\begin{aligned}
(F_{1})_{113}&=(F_{1})_{124}=-(F_{1})_{131}=-(F_{1})_{142}=
\\[6pt]
&=-(F_{1})_{214}=(F_{1})_{223}=-(F_{1})_{232}=(F_{1})_{241}=-1,
\\[6pt]
(F_{2})_{222}&=(F_{2})_{244}=-2,%\\[6pt]
\end{aligned}
\end{equation}
\begin{equation}
\begin{aligned}
(F_{2})_{112}&=(F_{2})_{121}=(F_{2})_{134}=(F_{2})_{143}
\\[6pt]
&=(F_{2})_{314}=-(F_{2})_{323}=-(F_{2})_{332}=(F_{2})_{341}=-1,
\\[6pt]
(F_{3})_{111}&=(F_{3})_{144}=2,\\[6pt]
-(F_{3})_{212}&=-(F_{3})_{221}=(F_{3})_{234}=(F_{3})_{243}
\\[6pt]
&=(F_{3})_{313}=(F_{3})_{324}=(F_{3})_{331}=(F_{3})_{342}=-1
\end{aligned}
\end{equation}
\end{subequations}
and the others are zero.
 The only non-zero components of the corresponding Lee forms are
\begin{equation}\label{HC4b-titi}
(\theta_1)_3=-2,\qquad (\theta_2)_2=-(\theta_3)_1=-4.
\end{equation}

Then, analogously to Case 2 %\ref{2.3.2.}
on page~\pageref{2.3.2.}, we obtain the following
\begin{prop}\label{prop-HC4b}
The hypercomplex manifold with Hermitian-Norden metrics on a 4-dimensional Lie algebra,
determined by \eqref{HC4b}, belongs to the class $\W^0$ of the (locally) conformally
equivalent $\KKK$-manifolds.
\end{prop}

\vskip 0.2in \addtocounter{subsection}{1}

\noindent\Large{\textbf{\thesubsection. Hypercomplex structure of type (hc5)}}%\\\vskip2pt}

\vskip 0.15in
%\subsection{Hypercomplex HN-metric structure of type (hc5)}
%\label{HC5}
%\paragraph{(HC5)}%

In this case,   $\mathfrak{l}$ is solvable and $\mathfrak{l}'$ is a 3-dimensional Heisenberg algebra.

%\newpage
\vskip 0.15in

\noindent\Large{\textbf{\emph{\thesubsection.1. Case 1}}}%\\\vskip2pt}

\vskip 0.1in

%\subsubsection{}\label{2.5.1.}

Firstly, we fix $e_1\in\mathfrak{l}$ with $g(e_1,e_1)=1$ as an
element ortho\-go\-nal to $\mathfrak{l}'$ with respect to $g$. Then $\mathfrak{l}$ is
determined by
\begin{equation}\label{HC5a}
\begin{array}{ll}
  [e_1,e_2]=e_2,\quad\quad & [e_1,e_3]=\dfrac{1}{2}e_3, \\[9pt]
  [e_1,e_4]=\dfrac{1}{2}e_4,\quad\quad & [e_3,e_4]=\dfrac{1}{2}e_2.
\end{array}
\end{equation}
Then we compute covariant
derivatives and the nontrivial ones are
\begin{equation}\label{HC5a-n}
\begin{aligned}
\DDD_{e_2} e_2=-2\DDD_{e_3} e_3=-2\DDD_{e_4} e_4=e_1,& \\[6pt]
-\DDD_{e_2} e_1=4\DDD_{e_3} e_4=-4\DDD_{e_4} e_3=e_2,& \\[6pt]
-4\DDD_{e_2} e_4=-2\DDD_{e_3} e_1=-4\DDD_{e_4} e_2=e_3,& \\[6pt]
4\DDD_{e_2} e_3=4\DDD_{e_3} e_2=-2\DDD_{e_4} e_1=e_4.&
\end{aligned}
\end{equation}

Analogously of the previous cases, we obtain the
components $(F_{\al})_{ijk}$ as follows:
\begin{subequations}\label{HC5aFa}
\begin{equation}
\begin{aligned}
(F_{1})_{314}&=-(F_{1})_{323}=(F_{1})_{332}=-(F_{1})_{341}=
\\[6pt]
&=-(F_{1})_{413}=-(F_{1})_{424}=(F_{1})_{431}=(F_{1})_{442}=\dfrac{1}{4},
%\\[6pt]
\end{aligned}
\end{equation}
\begin{equation}
\begin{aligned}
(F_{2})_{214}&=-(F_{2})_{223}=-(F_{2})_{232}=(F_{2})_{241}=-\dfrac{5}{4},
\\[6pt]
(F_{2})_{311}&=-2(F_{2})_{322}=(F_{2})_{333}=-2(F_{2})_{344}=-1,\\[6pt]
(F_{2})_{412}&=(F_{2})_{421}=(F_{2})_{434}=(F_{2})_{443}=-\dfrac{3}{4},
\\[6pt]
(F_{3})_{213}&=(F_{3})_{224}=(F_{3})_{231}=(F_{3})_{242}=-\dfrac{5}{4},
\\[6pt]
(F_{3})_{312}&=(F_{3})_{321}=-(F_{3})_{334}=-(F_{3})_{343}=-\dfrac{3}{4},
\\[6pt]
(F_{3})_{411}&=-2(F_{3})_{422}=-2(F_{3})_{433}=(F_{3})_{444}=1
\end{aligned}
\end{equation}
\end{subequations}
and the others are zero.
 The only non-zero components of the corresponding Lee forms are
\begin{equation}\label{HC5a-titi}
(\theta_1)_2=-\dfrac{1}{2},\qquad
(\theta_2)_3=-(\theta_3)_4=3.
\end{equation}

The results in \eqref{HC5aFa}, \eqref{HC5a-titi} and the classification conditions
\eqref{cl-H}, \eqref{cl-N} imply

\begin{prop}\label{prop-HC5a}
The hypercomplex manifold with Hermitian-Norden metrics on a 4-dimensional Lie algebra,
determined by \eqref{HC5a}, belongs to the largest class of the considered manifolds,
i.e. $\HCC$, as well as this manifold does not belong to neither $\W_1$ nor $\W_2$ for $J_2$ and $J_3$.
\end{prop}

\vskip 0.15in

\noindent\Large{\textbf{\emph{\thesubsection.2. Case 2}}}%\\\vskip2pt}

\vskip 0.1in

%\subsubsection{}\label{2.5.2.}

The other possibility is to choose $e_4\in\mathfrak{l}$ with
$g(e_4,e_4)=-1$ as an element orthogonal to $\mathfrak{l}'$ with respect to
$g$. We rearrange the basis in \eqref{HC5a} and then $\mathfrak{l}$ is
determined by
\begin{equation}\label{HC5b}
\begin{aligned}
&[e_1,e_2]=-\dfrac{1}{2}e_3,\quad\quad [e_1,e_4]=-\dfrac{1}{2}e_1, \quad\\[6pt]
&[e_2,e_4]=-\dfrac{1}{2}e_2,\quad\quad [e_3,e_4]=-e_3.
\end{aligned}
\end{equation}

By similar computations we establish the same statement as of \propref{prop-HC5a}
for the Heisenberg algebra introduced by \eqref{HC5b}.

\vspace{20pt}

\begin{center}
$\divideontimes\divideontimes\divideontimes$
\end{center}
%%%%%%%%%%%%%%%%%%%%%%%%%%%%%%%%%%%%%%%%%%%%%%%%%%%%%%%%%%%%%%%%%%%%%%%%%%%%%%%%%%%%%%

%\include{Man-47}
\newpage

\addtocounter{section}{1}\setcounter{subsection}{0}\setcounter{subsubsection}{0}

\setcounter{thm}{0}\setcounter{dfn}{0}\setcounter{equation}{0}

\label{par:bund}

 \Large{

\
\\[6pt]
\bigskip

\
\\[6pt]
\bigskip

\lhead{\emph{Chapter II $|$ \S\thesection. Tangent bundles with complete lift of the base metric and almost
hypercomplex
\ldots%Hermitian-Norden structure
}}
%\thispagestyle{empty}

%\noindent  {\Huge\bf \S\thesection. Tangent bundles with complete \\[12pt]
%\phantom{\S\thesection. }lift of the base metric and\\[12pt]
%\phantom{\S\thesection. }almost hypercomplex structure \\[12pt]
%\phantom{\S\thesection. }with Hermitian-Nor\-den metrics}%\\[6pt]\vskip2pt}

\noindent
\begin{tabular}{r"l}
  %\hline
\hspace{-6pt}{\Huge\bf \S\thesection.}  & {\Huge\bf Tangent bundles with complete} \\[12pt]
                             & {\Huge\bf lift of the base metric and} \\[12pt]
                             & {\Huge\bf almost hypercomplex structure}\\[12pt]
                             & {\Huge\bf with Hermitian-Nor\-den metrics}
  %\hline
\end{tabular}

\vskip 1cm

\begin{quote}
\begin{large}
In the present section, the tangent bundle of an almost Norden manifold and the complete
lift of the Norden metric are considered as a $4n$-dimensional manifold. It is
equipped with an almost hypercomplex Hermitian-Norden structure.
It is characterized geometrically. The case when the base manifold
is an h-sphere is considered.

The main results of this section are published in \cite{Man47}.
\end{large}\end{quote}

%\vskip 0.2in \addtocounter{subsection}{1}
%
%\noindent  {\Large\bf \thesubsection. Introduction}%\\[6pt]\vskip2pt}

\vskip 0.15in

The investigation of the tangent bundle $T\MM$ of a manifold $\MM$
helps us to study the manifold $\MM$. Moreover, $T\MM$ has own
structure closely related to the structure of $\MM$, which implies
mutually related geometric properties.

%The geometry of the almost hypercomplex manifolds with Hermitian
%metric is well studied (e.g.~\cite{AlMa}). A parallel direction
%including indefinite metrics is the development of the geometry of
%the almost hypercomplex manifolds with the Hermitian-Norden metric
%structure. It naturally originated from the geometry of the
%$n$-dimensional quaternionic Euclidean space.
%
%The research in this area was initiated in \cite{GriManDim12} and
%\cite{GriMan24}. More precisely, we have combined the Hermitian
%metric with the Norden metric with respect to the almost complex
%structures of a hypercomplex structure.

In this section we consider the following situation: it is given a base almost
Norden manifold and we study its tangent bundle equipped with a
metric, which is the complete lift of the base metric. Thus, we
get a manifold with an almost hypercomplex
structure and Hermitian-Norden metrics which we characterize.
%differential-geometrically this manifold bearing in mind also the
%relevant classifications of the resulted manifolds. At the end we
%construct an appropriate example.

Similar investigations are made in \cite{Ma05}. There, it is used the
diagonal lift of the base metric (known as a Sasaki metric) on the
tangent bundle. The almost hypercomplex
structure with Hermitian-Norden metrics is generated in the same manner.

Our goal is to determine an almost hypercomplex  structure with Her\-mit\-ian-Norden metrics $(H,G)$ on
$T\MM$ when the base manifold $\MM$ has an almost Norden structure
$(J,g,\widetilde{g})$.

We use the horizontal and vertical lifts of the vector fields on $\MM$ to
get the corresponding components of the considered tensor fields on $T\MM$. These
components are sufficient to describe the characteristic tensor fields on $T\MM$
in general.

%%%%%%%%%%%%%%%%%%%%%%%%%%%%%%%%%%%%%%%%%%%%%%%%%%%%%%%%%%%%%%%%%%%%%%%%%%%%%%
\vskip 0.2in \addtocounter{subsection}{1}

\noindent  {\Large\bf{\thesubsection. Almost hypercomplex structure on the tangent bundle}}%\\[6pt]\vskip2pt}

\vskip 0.15in
%\subsection{Almost hypercomplex structure on the tangent bundle}

It is well known  \cite{YaIs}, for any affine connection on
$\MM$, the induced horizontal and vertical distributions of $T\MM$ are
mutually complementary. Then we define tensor fields $J_1$, $J_2$
and $J_3$ on $T\MM$ by their action over the horizontal and vertical
lifts of
an arbitrary vector field on $\MM$:%as follows:
\begin{equation}\label{H}
\begin{array}{l}
J_1:
\left\{
\begin{aligned}
    X'&\rightarrow -(JX)'\\[6pt]
    X''&\rightarrow (JX)''
\end{aligned}\qquad
\right. \\[20pt]
J_2: \left\{
\begin{aligned}
    X'&\rightarrow X''\\[6pt]
    X''&\rightarrow -X'
\end{aligned}\qquad
\right. \\[20pt]
J_3: \left\{
\begin{aligned}
    X'&\rightarrow (JX)''\\[6pt]
    X''&\rightarrow (JX)'
\end{aligned}
\qquad
\right.
\end{array}
\end{equation}
where $J$ is the given almost complex structure on $\MM$; moreover, we denote the horizontal and vertical lifts by
$X', X'' \in \mathfrak{X}(T\MM)$ of any $X\in \mathfrak{X}(\MM)$ at $u\in T_p \MM$, $p\in\MM$, where an affine
connection $\DDD$ on $\MM$ is used.

By direct computations we get the following

\begin{prop}
There exists an almost hypercomplex structure $H$, defined
by~\eqref{H} on $T\MM$ over an almost complex manifold $(\MM,J)$ with
an affine connection $\DDD$. The constructed $4n$-dimensional
manifold is an almost hypercomplex manifold $(T\MM,H)$.
\end{prop}

Following \eqref{N_al}, let $[J_\alpha,J_\alpha]$ denote the Nijenhuis tensor of $J_\alpha$ for each
$\alpha$ and $\wh{X},\wh{Y}\in \mathfrak{X}(T\MM)$, i.e.
\begin{equation}\label{Na-def}
[J_\alpha,J_\alpha](\wh{X},\wh{Y})=[J_\alpha\wh{X},J_\alpha\wh{Y}]
-J_\alpha[J_\alpha\wh{X},\wh{Y}] -J_\alpha[\wh{X},J_\alpha\wh{Y}]
-[\wh{X},\wh{Y}].
\end{equation}

%Bearing in mind \eqref{XV-ind} and \eqref{XH-ind} we obtain by
%direct computations the following
%\begin{lem}\label{
If $\DDD$ is torsion-free and its curvature tensor is denoted by
$R$, then we have the following (see also \cite{YaIs})
\begin{equation}\label{l[]}
    \begin{array}{l}
    [X',Y']=[X,Y]'-\{R(X,Y)u\}'', \\[6pt]
    [X',Y'']=\left(\DDD_XY\right)'',\\[6pt]
    [X'',Y']=-\left(\DDD_YX\right)'',\\[6pt]
    [X'',Y'']=0.
   \end{array}
\end{equation}
%\end{lem}

Using  \eqref{H}, \eqref{Na-def} and \eqref{l[]}, we get
\begin{prop}\label{Na}
Let $(\MM,J)$ be an almost complex manifold with Nijenhuis tensor $[J,J]$, a torsion-free
affine connection $\DDD$ and its curvature tensor $R$. Then the
Nijenhuis tensors of the structure $H$ on $T\MM$ for the
corresponding horizontal and vertical lifts have the following
form
\begin{equation*}\label{N_1}
\begin{array}{l}
[J_1,J_1](X',Y')=\left([J,J](X,Y)\right)'\\[6pt]
\phantom{[J_1,J_1](X',Y')=}
-\bigl(R(JX,JY)u+JR(JX,Y)u\\[6pt]
\phantom{[J_1,J_1](X',Y')=-\bigl(}
+JR(X,JY)u-R(X,Y)u\bigr)'',\\[6pt]
[J_1,J_1](X',Y'')=-\bigl(\left(\DDD_{JX}J\right)(Y)-\left(\DDD_{X}J\right)(JY)\bigr)'',\\[6pt]
[J_1,J_1](X'',Y')=-\bigl(\left(\DDD_{Y}J\right)(JX)-\left(\DDD_{JY}J\right)(X)\bigr)'',\\[6pt]
[J_1,J_1](X'',Y'')=0;
\end{array}
\end{equation*}
\begin{equation*}\label{N_2}
\begin{array}{l}
[J_2,J_2](X',Y')=-[J_2,J_2](X'',Y'')=\bigl(R(X,Y)u\bigr)'',\\[6pt]
[J_2,J_2](X',Y'')=[J_2,J_2](X'',Y')=\bigl(R(X,Y)u\bigr)';\\[6pt]
\end{array}
\end{equation*}
\begin{equation*}\label{N_3}
\begin{array}{l}
[J_3,J_3](X',Y')=-\bigl(J\left(\DDD_{X}J\right)(Y)
-J\left(\DDD_{Y}J\right)(X)\bigr)'\\[6pt]
\phantom{[J_3,J_3](X',Y')=}
+\bigl(R(X,Y)u\bigr)'', \\[6pt]
[J_3,J_3](X',Y'')=-\bigl(J\left(\DDD_{X}J\right)(Y)
+\left(\DDD_{JY}J\right)(X)\bigr)''\\[6pt]
\phantom{[J_3,J_3](X',Y'')=}
+\bigl(JR(X,JY)u\bigr)', \\[6pt]
%
%\end{array}
%\end{equation*}
%\begin{equation*}
%\begin{array}{l}
[J_3,J_3](X'',Y')=\bigl(\left(\DDD_{JX}J\right)(Y)
+J\left(\DDD_{Y}J\right)(X)\bigr)''\\[6pt]
\phantom{[J_3,J_3](X'',Y')=}
+\bigl(JR(JX,Y)u\bigr)', \\[6pt]
[J_3,J_3](X'',Y'')=\bigl(\left(\DDD_{JX}J\right)(Y)
-\left(\DDD_{JY}J\right)(X)\bigr)'\\[6pt]
\phantom{[J_3,J_3](X'',Y'')=}
-\bigl(R(JX,JY)u\bigr)''. \\[6pt]
\end{array}
\end{equation*}
%for any $X, Y\in \mathfrak{X}(\MM)$ at $u\in T_p \MM$.
\end{prop}

The last equalities for $[J_\alpha,J_\alpha]$ imply the following necessary
and sufficient conditions for the integrability of $J_\alpha$ and
$H$.
\begin{thm}\label{thm:Na=0}
Let $T\MM$ be the tangent bundle manifold equipped with an almost
hypercomplex structure $H=(J_1,J_2,J_3)$ defined as in~\eqref{H}.
Let also $\MM$ be its base manifold with the almost complex
structure $J$. We additionally assume that the affine connection
used for $H$ be torsion-free. Then the following relations hold:
\begin{enumerate}
    \item $(T\MM,J_\alpha)$ for $\alpha=1$ or $3$ is
complex if and only if $\MM$ is flat and $J$ is parallel; %\emph{(see also \cite{To})}
    \item $(T\MM,J_2)$  is complex if and only if $\MM$ is flat;
    \item $(T\MM,H)$ is hypercomplex if and only if $\MM$ is flat and $J$ is parallel.
\end{enumerate}
\end{thm}

\begin{rem}
The assertion (ii) above is a corollary of the theorem of
Dombrowski in \cite{Dom}, where the structure $J_2$ is defined and
studied.
\end{rem}

\begin{cor}
\begin{enumerate}
    \item $(T\MM,J_1)$ is complex if and only if $(T\MM,J_3)$ is complex.
    \item If $(T\MM,J_1)$ or $(T\MM,J_3)$ is complex then $(T\MM,H)$ is hypercomplex.
\end{enumerate}
\end{cor}

%%%%%%%%%%%%%%%%%%%%%%%%%%%%%%%%%%%%%%%%%%%%%%%%%%%%%%%%%%%%%%%%%%%%%%%%%%%%%%%%%%%%%%
\vskip 0.2in \addtocounter{subsection}{1}

\noindent  {\Large\bf{\thesubsection. Complete lift of the base metric on the tangent bundle}}%\\[6pt]\vskip2pt}

\vskip 0.15in
%\subsection{Complete lift of the base metric on the tangent bundle}

Let us introduce a metric $\wh{g}$ on $T\MM$, which is the complete
lift of the base metric $g$ on $\MM$,
 by
\begin{equation}\label{5g}
\begin{array}{l}
\wh{g}(X',Y')=\wh{g}(X'',Y'')=0,\quad\\[6pt]
\wh{g}(X',Y'')=\wh{g}(X'',Y')=g(X,Y).
\end{array}
\end{equation}

It is known that $\wh{g}$, associated with a (pseudo-)Rie\-mannian
metric $g$, is a pseudo-Rie\-mann\-ian metric on $T\MM$ of signature
$(m,m)$, where $m=\dim \MM$. The metric $\wh{g}$ coincides with the
horizontal lift of $g$ with respect to its Levi-Civita connection
\cite{YaIs67}. This metric is introduced by Yano and Kobayashi
as $(T\MM,\wh{g})$ has zero scalar curvature and it is an Einstein
space if and only if $\MM$ is Ricci-flat \cite{YaIs}.

As it is known from \cite{YaKo}, whenever $\DDD$ is the
Levi-Civita connection of $\MM$ with respect to the
pseudo-Riemannian metric $g$, then the complete lift of $\DDD$
is the Levi-Civita connection of $T\MM$ generated by $\wh{g}$. Since
$\wh{\DDD}$ is the Levi-Civita connection of $\wh{g}$ on $T\MM$
and the same holds for $\DDD$ on $(\MM, g)$,
then using the Koszul formula \eqref{koszul} %of \( 2g(\DDD_X Y,Z)=Xg(Y,Z)+Y
%g(X,Z)-Z g(X,Y)+g([X,Y],Z)+g([Z,X],Y)+g([Z,Y],X), \)
we obtain the covariant derivatives of the horizontal and vertical
lifts of the vector fields on $T\MM$ at $u\in T_p \MM$ as follows (see
also \cite{YaIs})
\begin{equation}\label{nabli}
\begin{array}{ll}
\wh{\DDD}_{X'}Y'=(\DDD_XY)'+\left(R(u,X)Y\right)'',\quad\quad &
\wh{\DDD}_{X'}Y''=(\DDD_XY)'',\\[6pt]
\wh{\DDD}_{X''}Y'=0, & \wh{\DDD}_{X''}Y''=0.
\end{array}
\end{equation}

After that, we calculate the components of the curvature tensor
$\wh{R}$ of $\wh{\DDD}$ with respect to the horizontal and
vertical lifts of the vector fields on $\MM$. We obtain the
following non-zero components for the curvature tensors $R$ and
$\wh{R}$ as well as the Ricci tensors $\rho$ and $\wh{\rho}$,
corresponding to the metrics $g$ and $\wh{g}$ on $\MM$ and $T\MM$,
respectively (see also \cite{YaIs}):
\begin{subequations}\label{R}
\begin{equation}%\label{R}
\begin{array}{l}
\wh{R}(X',Y')Z'=\left\{R(X,Y)Z\right\}'+\left\{\left(\DDD_uR\right)(X,Y)Z\right\}'',\\[6pt]
\end{array}
\end{equation}
\begin{equation}%\label{R}
\begin{array}{l}
\wh{R}(X',Y')Z''=\wh{R}(X',Y'')Z'
=\left\{R(X,Y)Z\right\}'';\quad\\[6pt]
\wh{\rho}(Y',Z')=2\rho(Y,Z).
\end{array}
%for the lifts $(\cdot)', (\cdot)'' \in \mathfrak{X}(T\MM)$ of any $X, Y, Z, W
%\in \mathfrak{X}(\MM)$ at $u\in T_p \MM$.
\end{equation}
\end{subequations}

Hence, we get
\begin{cor}\label{flat}
\begin{enumerate}
    \item $(T\MM,\wh{g})$ is flat if and only if $(\MM,g)$ is flat.
    \item $(T\MM,\wh{g})$ is Ricci flat if and only if $(\MM,g)$ is Ricci
    flat.
    \item $(T\MM,\wh{g})$ is scalar flat.
\end{enumerate}
\end{cor}

\begin{rem}
The results in \eqref{nabli}, \eqref{R} and \corref{flat} are
confirmed also by \cite{YaIs}, where $g$ is a %(semi-)
Riemannian
metric.
\end{rem}

%\[
%\begin{array}{ll}
%\wh{\rho}^*_1=0;\quad &\\[6pt]
%\wh{\rho}^*_2(Y',Z')=-\rho(Y,Z),\quad &
%\wh{\rho}^*_2(Y',Z'')=-\left(\DDD_u\rho\right)(Y,Z),\\[6pt]
%\wh{\rho}^*_2(Y'',Z')=0,\quad &
%\wh{\rho}^*_2(Y'',Z'')=\rho(Y,Z);\\[6pt]
%\wh{\rho}^*_3(Y',Z')=\rho^*(Y,Z),\quad &
%\wh{\rho}^*_3(Y',Z'')=\left(\DDD_u\rho^*\right)(Y,Z)+\rho^*\left(Y,J\left(\DDD_uJ\right)Z\right),\\[6pt]
%\wh{\rho}^*_3(Y'',Z')=0,\quad &
%\wh{\rho}^*_3(Y'',Z'')=\rho^*(Y,Z).
%\end{array}
%\]

%\[
%\wh{\tau}^*_1=0;\quad %
%\wh{\tau}^*_2=\DDD_u\tau;\quad %
%\wh{\tau}^*_3=\DDD_u\tau^*.
%\]

%%%%%%%%%%%%%%%%%%%%%%%%%%%%%%%%%%%%%%%%%%%%%%%%%%%%%%%%%%%%%%%%%%%%%%%%%%%%%%%%%
\vskip 0.2in \addtocounter{subsection}{1}

\noindent  {\Large\bf{\thesubsection. Tangent bundle with almost hypercomplex  structure and Hermitian-Norden metrics}}%\\[6pt]\vskip2pt}

\vskip 0.15in
%\subsection{Tangent bundle with almost hypercomplex  structure with Hermitian-Norden metrics}

Suppose that $(\MM,J,g,\widetilde{g})$ is an almost complex manifold
with Norden metric $g$ and its associated Norden metric
$\widetilde{g}$. Suppose also that $(T\MM,H)$ is its almost
hypercomplex tangent bundle with the Hermitian-Norden metric
structure $\wh{G}=(\wh{g},\wh{g}_1,\wh{g}_2,\allowbreak\wh{g}_3)$
derived (as in~\eqref{gJ}) from the metric $\wh{g}$ on $T\MM$, the
complete lift of $g$. The generated $4n$-dimensional manifold is
denoted by $(T\MM,H,\wh{G})$.

Bearing in mind \eqref{H}, we verify immediately that $\wh{g}$
satisfies \eqref{gJJ} and therefore it is valid the following
\begin{thm}\label{Sas}
The tangent bundle $T\MM$ equipped with the almost hypercomplex
structure $H$ and the metric $\wh{g}$, defined by~\eqref{H}
and~\eqref{5g}, respectively, is an almost hypercomplex
manifold with Hermitian-Norden metrics $(T\MM,\allowbreak{}H,\wh{G})$.
\end{thm}

We use \eqref{nabli} and \eqref{H} in order to characterize the
fundamental tensors $F_\alpha$ with respect to $\wh{g}$ and
$\wh{\DDD}$ at each $u\in T_p\MM$ on $(T\MM,H,\wh{G})$. Then we
obtain the following

\begin{prop}\label{prop:Fa}
The nonzero components of $F_\alpha$ with respect to the
horizontal and vertical lifts of the vector fields depend on the
fundamental tensor $F$ and the curvature tensor $R$ of
$(\MM,J,g,\widetilde{g})$ in the following way:
\begin{subequations}\label{5F123}
\begin{equation}
\begin{array}{l}
F_1(X',Y',Z')=-R(u,X,JY,Z)-R(u,X,Y,JZ),\\[6pt]
F_1(X',Y',Z'')=-F_1(X',Y'',Z')=-F(X,Y,Z);\\[6pt]
F_2(X',Y',Z'')=-F_2(X',Y'',Z')=R(u,X,Y,Z);\\[6pt]
F_3(X',Y',Z')=F_3(X',Y'',Z'')=F(X,Y,Z),\quad\\[6pt]
\end{array}
\end{equation}
\begin{equation}%\label{5F3}
\begin{array}{l}
F_3(X',Y',Z'')=-R(u,X,Y,JZ),\quad\\[6pt]
F_3(X',Y'',Z')=R(u,X,JY,Z).
\end{array}
\end{equation}
\end{subequations}
%where $(\cdot)', (\cdot)'' \in \mathfrak{X}(T\MM)$ are the lifts of any $X,
%Y, Z \in \mathfrak{X}(\MM)$ at $u\in T_p \MM$.
\end{prop}

We use the components \eqref{5F123} of $F_\al$ and bearing in mind the definition conditions \eqref{cl-H} and \eqref{cl-N} of the basic classes
as well as their direct sums in the corresponding classifications from \cite{GrHe} and \cite{GaBo},
we obtain
\begin{prop}\label{prop:cl}
\begin{enumerate}
    \item
    The almost Hermitian manifold $(T\MM,J_1,\wh{g})$ belongs to the
    class
\[
\bigl\{\left(\W_1\oplus\W_2\oplus\W_3\oplus\W_4\right)\setminus
\left\{\W_1\cup\left(\W_3\oplus\W_4\right)\right\}\bigr\}\cup \W_0,
\]
where $\W_1\oplus\W_2\oplus\W_3\oplus\W_4$, $\W_1$, $\W_3\oplus\W_4$ and
$\W_0$ stand for the classes of the almost Hermitian,
nearly K\"ahler, Hermitian and K\"ahler manifolds, respectively.
For the 4-dimensional case, the class of $(T\MM,J_1,\wh{g})$ is
restricted to the class $\W_2$ of the almost K\"ahler
manifolds.
    \item
The almost Norden manifold $(T\MM,J_{\alpha},\wh{g})$,
$(\alpha=2,3)$, belongs to the class
\[
\bigl\{\left(\W_1\oplus\W_2\oplus\W_3\right)\setminus
\left\{\left(\W_1\oplus\W_2\right)\cup\left(\W_1\oplus\W_3\right)\right\}\bigr\}\cup
\W_0.
\]
\end{enumerate}
\end{prop}

The corresponding Lee forms are determined by \eqref{theta-al}. Hence, we can compute them
with respect to an adapted frame.

Let $\{\widehat{e}_{A}\}=
%\{\widehat{e}_{i},\widehat{e}_{\bar{i}}\}=
\{e_{i}',e_{i}''
\}$
%$\{e_i',e'_{\bar{i}},e_i'',e^{V}_{\bar{i}}\}$
be the
adapted frame
%of type $(+,\dots,+,-,\dots,-,+,\dots,+,\allowbreak{}-,\dots,-)$
at each
point of $T\MM$ derived by the orthonormal basis
$\{e_{i}\}$ %$\{e_{i},e_{\bar{i}}\}$
of signature $(n,n)$ at each point of $\MM$.
%denoting $e_{\bar{i}}=Je_i$, where ${\bar{i}}=n+i$
The indices
%$i,j,\dots$ and $\bar{i},\bar{j},\dots$ run over the ranges
%$\{1,\dots,n\}$ and $\{n+1,\dots,2n\}$, respectively,
%
$i,j,\dots$ run over the range $\{1,2,\dots,2n\}$, while the
indices $A,B,\dots$ belong to the range $\{1,2,\dots,4n\}$. We use
summation convention for these indices.

For example, we compute as follows
\[
\begin{aligned}
\theta_{3}(Z')&=\wh{g}^{AB}F_{3}(\wh{e}_A,\wh{e}_B,Z')
=g^{ij}\{F_{3}(e'_i,e''_j,Z')+F_{3}(e''_i,e'_j,Z')\}\\[6pt]
&=
g^{ij}F_{3}(e'_i,e''_j,Z')
=g^{ij}R(u,e_i,Je_j,Z)=\rho^*(u,Z).
\end{aligned}
\]
Analogously, we have \[
\theta_{3}(Z'')=g^{ij}F_{3}(e'_i,e''_j,Z'')
=g^{ij}F(e_i,e_j,Z)=\theta(Z). \]
Thus, we obtain the following nonzero components of the Lee forms
$\theta_{\alpha}$
\[
\begin{split}
\theta_1(Z')&=\theta_3(Z'')=\theta(Z),\quad \\[6pt]
\theta_2(Z')&=-\rho(u,Z),\quad \\[6pt]
\theta_3(Z')&=\rho^*(u,Z),
\end{split}
\]
where $\rho$ and $\rho^*$ are the Ricci tensor and
its associated Ricci tensor stemming from $g$ and $J$. % and $\widetilde{g}$, respectively.
Therefore, we obtain
\begin{prop}\label{prop:titi}
The following necessary and sufficient conditions are valid:
%For the Lee forms $\theta_{\alpha}$ of
%$(T\MM,J_{\alpha},\wh{g})$ as well as for the Ricci tensor
%${\rho}$, the associated Ricci tensor $\widetilde{\rho}$ and the
%Lee form $\theta$ of $(\MM,J,g,\widetilde{g})$, we have
\begin{enumerate}
    \item $\theta_1=0$ if and only if $\theta=0$;
    \item $\theta_2=0$ if and only if ${\rho}=0$;
    \item $\theta_3=0$ if and only if
        $\theta=0$ and $\rho^*=0$.
\end{enumerate}
\end{prop}

\begin{rem}
Let us recall that the vanishing condition of the corresponding
Lee form determines the class $(\W_1\oplus\W_2\oplus\W_3)(J_1)$ of the semi-K\"ahler
manifolds among the almost Hermitian manifolds and the class
$(\W_2\oplus \W_3)(J_\al)$, $(\al=2,3)$ among the almost Norden manifolds, respectively.
\end{rem}

Bearing in mind \propref{prop:Fa} and \thmref{thm:Na=0}, it is
easy to conclude the following
\begin{prop}
The following necessary and sufficient conditions are valid:
\begin{enumerate}
    \item $(T\MM,J_\alpha,\wh{g})$ for $\alpha=1$ or $3$ has a
    parallel complex structure $J_\alpha$ if and only if $(\MM,J,g,\widetilde{g})$ is flat and
        $J$ is parallel.
    \item $(T\MM,J_2,\wh{g})$ has a parallel complex structure $J_2$ if and only if
    $(\MM,J,\allowbreak{}g,\widetilde{g})$ is flat.
    \item $(T\MM,H,G)$ is a hyper-K\"ahler manifold with Hermitian-Norden metrics if and only if
        $(\MM,J,\allowbreak{}g,\widetilde{g})$ is flat and $J$ is parallel.
\end{enumerate}
\end{prop}

%\begin{rem}
%We say that a hypercomplex manifold with Hermitian-Norden metrics is a
%\emph{hyper-K\"ahler manifold with Hermitian-Norden metrics}% and denote $(\MM,H,G)$ $\in\KKK$
%, if $F_\alpha=0$ for each $\alpha\in\{1,2,3\}$, i.e. the manifold
%is of K\"ahler type with respect to each $J_\alpha$. According to
%\cite{GriManDim12}, each hyper-K\"ahler manifold with Hermitian-Norden metrics is a flat
%pseudo-Rie\-mann\-ian manifold of neutral signature.
%\end{rem}

\begin{cor}
\begin{enumerate}
    \item $J_1$ is parallel if and only if
        $J_3$ is parallel.
    \item If
        $J_1$ or $J_3$ is parallel then
        $(T\MM,H,\wh{G})$ is a hyper-K\"ahler manifold with Hermitian-Norden metrics.
\end{enumerate}
\end{cor}
\begin{cor}
\begin{enumerate}
    \item  If the manifold $(T\MM,J_\alpha,\wh{g})$ for any
                $\alpha$ is a complex manifold then it is of K\"ahler type.
    \item If the manifold $(T\MM,H,\wh{G})$ is hypercomplex then it is a hyper-K\"ahler manifold with Hermitian-Norden metrics.
\end{enumerate}
\end{cor}

\begin{cor}\label{cor:M}
\begin{enumerate}
    \item If $(\MM,J,g,\widetilde{g})$ is flat, then $T\MM$ has parallel $J_2$.
    %, i.e. $(\MM,J,g,\widetilde{g})\in\W_0$.
    \item If $(\MM,J,g,\widetilde{g})$ is flat and its Lee form $\theta$ is zero,
    i.e.
    $
    (\MM,J,g,\widetilde{g})\in\left(\W_2 \oplus \W_3\right)$,
    then
    \[
    (T\MM,H,\wh{G})\in
    (\W_1\oplus\W_2\oplus\W_3)(J_1)\cup\mathcal{W}_0(J_2)
    \cup(\mathcal{W}_2\oplus\mathcal{W}_3)(J_3);
    \]
    \item If $(\MM,J,g,\widetilde{g})$ is a K\"ahler-Norden manifold,
    then $(T\MM,J_1,\wh{g})\in\W_2(J_1)$.
\end{enumerate}
\end{cor}

\begin{rem}
The corresponding result in the Riemannian case (see \cite{To})
reads as follows. The manifold $(T\MM,J_1,\wh{g})$ is almost
K\"ahler (i.e. symplectic) for any Riemannian metric $g$ on the
base manifold when the connection used to define the horizontal
lifts is the Levi-Civita connection.
\end{rem}

%%%%%%%%%%%%%%%%%%%%%%%%%%%%%%%%%%%%%%%%%%%%%%%%%%%%%%%%%%%%%%%%%%%%%%%%%%%%%%%%%%2
\vskip 0.2in \addtocounter{subsubsection}{1}

\noindent  {\Large\bf{\emph{\thesubsubsection. Tangent bundle of an h-sphere}}}%\\[6pt]\vskip2pt}

\vskip 0.15in
%\subsection{Tangent bundle of an h-sphere}\label{sec:sect.curv.}

Let $(\MM,J,g,\g)$ be a K\"ahler-Norden manifold, $\dim \MM = 2n
\allowbreak{}\geq 4$. Let $x, y, z, w$ be arbitrary vectors in
$T_p\MM$, $p\in \MM$. The curvatu\-re ten\-sor $R$ %of type $(0, 4)$
%defined by $R(x, y, z, w)\allowbreak =\allowbreak g(R(x, y)z, w)$
is of K\"ahler type in this case. This implies that the associated tensor
$R^*$ of type $(0, 4)$ defined by the equality $R^*(x, y, z, w) = R(x, y, z, Jw)$
has the property $R^* (x, y, z, w) = -R^*(x, y, w, z)$ \cite{GaGrMi85}. Therefore $R^*$
has the properties of a curvature tensor, i.e. it is a
curva\-tu\-re-like tensor. The following tensors are essential
curvature-like (0,4)-tensors:
%\[
%\begin{array}{l}
%\pi_1(x, y, \allowbreak{}z, w) = g(y, z) g(x, w)- g(x, z) g(y,
%w),\\[6pt]
%\pi_2(x, y, z, w) = \g(y, z) \g(x, w) -\g(x, z) \g(y, w),\\[6pt]
%\pi_3(x, y, z, w) = -g(y, z) \g(x, w) + g(x, z) \g(y, w)-\g(y, z)
%g(x, w) + \g(x, z)\allowbreak{}g(y, w).
%\end{array}
%\]
%$
%\begin{array}{l}
%\pi_1 = -\dfrac{1}{2}g\owedge g,\quad \pi_2 = -\dfrac{1}{2}\g\owedge
%\g,\quad \pi_3 = g\owedge\g,
%\end{array}
%$
\[
\pi_1 = -\dfrac{1}{2}g\owedge g,\quad\quad
\pi_2 = -\dfrac{1}{2}\g\owedge\g,\quad\quad
\pi_3 = g\owedge\g.
\]
%where $\owedge$ stands for the Kulkarni-Nomizu product of two (0,2)-tensors.

Every non-degenerate 2-plane $\beta$ with respect to $g$ in
$T_p\MM$, $p \in \MM$, has the following two sectional curvatures
\[
k(\beta;p)=\frac{R(x,y,y,x)}{\pi_1(x,y,y,x)},\qquad k^*(\beta;p)=\frac{R^*(x,y,y,x)}{\pi_1(x,y,y,x)},
\]
where $\{x,y\}$ is a basis of $\beta$.
%\[
%k(\beta;p)=\dfrac{R(x,y,y,x)}{\pi_1(x,y,y,x)},\qquad
%k^*(\beta;p)=\dfrac{R^*(x,y,y,x)}{\pi_1(x,y,y,x)}.
%\]

A 2-plane $\beta$ is said to be \emph{holomorphic} (resp.,
\emph{totally real}) if $\beta= J\beta$ (resp., $\beta\perp
J\beta\neq \beta$) with respect to $g$ and $J$.

The orthonormal $J$-basis $\{e_i,e_{\bar{i}}\}$, where
$i\in\{1,2,\dots,n\}$, $\bar{i}=n+i$, $e_{\bar{i}}=Je_{i}$ of
$T_p\MM$ with respect to $g$ generates an orthogonal basis of isotropic
vectors $\{e_i',e_{\bar{i}}',e_i'',e_{\bar{i}}''\}$ of $T_u(T\MM)$
with respect to $\wh{g}$. Then, the basis
$\{\xi_i,\xi_{\bar{i}},\eta_i,\eta_{\bar{i}}\}$, where
\[
\begin{split}
\xi_i=\dfrac{1}{\sqrt{2}}(e_i'+e_i''),\quad\quad &
\xi_{\bar{i}}=\dfrac{1}{\sqrt{2}}(e_{\bar{i}}'+e_{\bar{i}}''),\\[6pt]
\eta_i=\dfrac{1}{\sqrt{2}}(e_i'-e_i''),\quad\quad
&
\eta_{\bar{i}}=\dfrac{1}{\sqrt{2}}(e_{\bar{i}}'-e_{\bar{i}}''),
\end{split}
\]
is orthonormal with respect to $\wh{g}$. Moreover, we have
\[
g(\xi_i,\xi_i)=-g(\xi_{\bar{i}},\xi_{\bar{i}})=-g(\eta_i,\eta_i)=g(\eta_{\bar{i}},\eta_{\bar{i}})=1,
\]
if $g(e_i,e_i)=-g(e_{\bar{i}},e_{\bar{i}})=1$ are valid.
Then,
using \eqref{H}, we obtain
\[
\begin{array}{lll}
J_1\xi_{\ii}=\eta_i,\quad\quad &
J_1\eta_{\ii}=\xi_i,\quad\quad &
J_2\eta_{i}=\xi_i,\quad\quad\\[6pt]
J_2\eta_{\ii}=\xi_{\ii},\quad\quad &
J_3\xi_{i}=\xi_{\ii},\quad\quad &
J_3\eta_{\ii}=\eta_i,
\end{array}
\]
 i.e. the basis
$\{\xi_i,\xi_{\bar{i}},\eta_i,\eta_{\bar{i}}\}$ is an adapted
$H$-basis.

Thus, the following basic 2-planes in $T_u(T\MM)$ are special with
respect to $H$ ($i\neq j$):
\begin{enumerate}[a)]
    \item
\emph{$J_{\alpha}$-totally-real 2-planes} ($\alpha=1,2,3$):\\[6pt]
\begin{tabular}{lllll}
$\{\xi_i,\xi_j\}$,\quad & $\{\xi_i,\xi_{\jj}\}$,\quad & $\{\xi_i,\eta_j\}$,\quad &
$\{\xi_i,\eta_{\jj}\}$,\quad & $\{\xi_{\ii},\xi_{\jj}\}$,\\[6pt]
$\{\xi_{\ii},\eta_j\}$,\quad & $\{\xi_{\ii},\eta_{\jj}\}$,\quad &
$\{\eta_i,\eta_j\}$,\quad & $\{\eta_{i},\eta_{\jj}\}$,\quad &
$\{\eta_{\ii},\eta_{\jj}\}$;
\end{tabular}
    \item
\emph{$J_1$-holomorphic and $J_{\alpha}$-totally-real 2-planes}
($\alpha=2,3$):\\[6pt]
    $\{\xi_{\ii},\eta_{i}\}$,\quad $\{\eta_{\ii},\xi_{i}\}$;
    \item
\emph{$J_2$-holomorphic and $J_{\alpha}$-totally-real 2-planes}
($\alpha=1,3$):\\[6pt]
    $\{\eta_{i},\xi_{i}\}$,\quad $\{\eta_{\ii},\xi_{\ii}\}$;
    \item
\emph{$J_3$-holomorphic and $J_{\alpha}$-totally-real 2-planes}
($\alpha=1,2$):\\[6pt]
    $\{\xi_{i},\xi_{\ii}\}$,\quad $\{\eta_{\ii},\eta_{i}\}$.
\end{enumerate}

The sectional curvatures $\wh{k}$ of these 2-planes and the
sectional curvatures $k_{ij}$, $k_{i\jj}$, $k_{\ii\jj}$ and $k_{i\ii}$
of the special basic 2-planes in $T_p\MM$:  $J$-totally-real
2-planes $\{e_{i},e_{j}\}$, $\{e_{i},e_{\jj}\}$,
$\{e_{\ii},e_{\jj}\}$ ($i\neq j$) and $J$-holomorphic 2-planes
$\{e_{i},e_{\ii}\}$, respectively, are related as follows:
\[
\begin{array}{ll}
\wh{k}(\xi_{i},\xi_{j})=\dfrac{1}{4}\left(\DDD_uk\right)_{ij}+k_{ij},\quad\quad
&
\wh{k}(\xi_{i},\xi_{\jj})=\dfrac{1}{4}\left(\DDD_uk\right)_{i\jj}+k_{i\jj},
\\[6pt]
\wh{k}(\xi_{i},\eta_{j})=-\dfrac{1}{4}\left(\DDD_uk\right)_{ij},  \quad\quad &
\wh{k}(\xi_{i},\eta_{\jj})=-\dfrac{1}{4}\left(\DDD_uk\right)_{i\jj}  \\[6pt]
\wh{k}(\xi_{\ii},\xi_{\jj})=\dfrac{1}{4}\left(\DDD_uk\right)_{\ii\jj}+k_{\ii\jj},\quad\quad
&
\wh{k}(\xi_{\ii},\eta_{j})=-\dfrac{1}{4}\left(\DDD_uk\right)_{\ii j},%\\[6pt]
\end{array}
\]
\[
\begin{array}{ll}
\wh{k}(\xi_{\ii},\eta_{\jj})=-\dfrac{1}{4}\left(\DDD_uk\right)_{\ii\jj},\quad\quad
& \wh{k}(\eta_i,\eta_j)=\dfrac{1}{4}\left(\DDD_uk\right)_{ij}-k_{ij},
\\[6pt]
\wh{k}(\eta_{i},\eta_{\jj})=\dfrac{1}{4}\left(\DDD_uk\right)_{i\jj}-k_{i\jj}, \quad\quad &
\wh{k}(\eta_{\ii},\eta_{\jj})=\dfrac{1}{4}\left(\DDD_uk\right)_{\ii\jj}-k_{\ii\jj}.
\end{array}
\]

Therefore, we obtain
\begin{prop}
The manifold $(T\MM,H,\wh{G})$ for an arbitrary almost Norden
manifold $(\MM,J,g)$ has equal sectional curvatures of the
$J_1$-holo\-morphic 2-planes and vanishing sectional curvatures of
the $J_2$-holomor\-phic 2-planes.
\end{prop}

Identifying the point $z = (x^1, ..., x^{n+1}; y^1, ..., y^{n+1})$ in
$\R^{2n+2}$ with the position vector $z$, we consider an h-sphere with  center $z_0$ and a pair of parameters
$(a,b)$. The h-sphere $\SSS^{2n}(z_0;\allowbreak{} a, b)$ is the holomorphic hypersurface
of the K\"ahler-Norden man\-i\-fold $\R^{2n+2}$ equipped with the canonical complex structure and the canonical Norden metrics $h$ and $\widetilde h$. Then $\SSS^{2n}(z_0;\allowbreak{} a, b)$ is
defined by
\[
h(z-z_0, z-z_0) = a, \qquad
\h(z-z_0, z-z_0) = b,
\]
where $(0, 0)\allowbreak\neq\allowbreak (a, b) \in\R^2$. It was also considered in Subsection 8.3.2.
Each $\SSS^{2n}$, $n \geq 2$, has vanishing holomorphic
sectional curvatures and constant totally real sectional
curvatures \cite{GaGrMi85}
\[
\nu=\dfrac{a}{a^2+b^2},\qquad \nu^*=-\dfrac{b}{a^2+b^2}.
\]
Then, according to \cite{BoGa}, the curvature
tensor of the h-sphere is
\begin{equation}\label{Rnu}
R = \nu (\pi_1 - \pi_2) + \nu^*\pi_3
\end{equation}
 and
therefore $\DDD R=0$. Moreover, we have
\[
\begin{array}{l}
\rho=2(n-1)(\nu g-\nu^*\g),\\[6pt]
\rho^*=2(n-1)(\nu^* \g+\nu g),\\[6pt]
\tau=4n(n-1)\nu,\\[6pt]
\tau^*=4n(n-1)\nu^*,
\end{array}
\]
%\[
%\begin{array}{ll}
%\rho=2(n-1)(\nu
%g-\nu^*\g),\quad\quad & \rho^*=2(n-1)(\nu^* \g+\nu g),
%\\[6pt]
%\tau=4n(n-1)\nu,\quad\quad &
%\tau^*=4n(n-1)\nu^*,
%\end{array}
%\]
where $\rho^*=\rho(R^*)$,
$\tau^*=\tau(R^*)$. Because of the form of $\rho$,  $\SSS^{2n}$ is
called \emph{almost Einstein}.

Let us consider $(T\SSS,\allowbreak{}H,\widehat{G})$, the tangent bundle with almost hypercomplex  structure with Hermitian-Norden metrics
of the h-sphere $(\SSS,J,g)$ with parameters
$(a,b)$, as its base K\"ahler-Norden manifold. Then, bearing in
mind \propref{prop:Fa}, \propref{prop:cl} and \corref{cor:M}
(iii), we get that $(T\SSS,J_1,\widehat{g})\in\W_2(J_1)$ and
$(T\SSS,J_{\alpha},\widehat{g})$ $(\alpha=2,3)$ belongs to
$\left(\W_1\oplus\W_2\oplus\W_3\right)(J_\al)\setminus
\left(\W_i\oplus\W_j\right)(J_\al)$, where $i\neq j\in\{1,2,3\}$.
Moreover, from \eqref{R} and \eqref{Rnu} we get the components of
the curvature tensor $\wh{R}$ of $(T\SSS,H,\widehat{G})$. Then we
have
\[
\begin{array}{l}
\wh{k}(\xi_{i},\xi_{j})=\wh{k}(\xi_{i},\xi_{\jj})=\wh{k}(\xi_{\ii},
\xi_{\jj})\\[6pt]
\phantom{\wh{k}(\xi_{i},\xi_{j})}
=-\wh{k}(\eta_i,\eta_j)=-\wh{k}(\eta_{i},\eta_{\jj})
=-\wh{k}(\eta_{\ii},\eta_{\jj})=\dfrac{a}{a^2+b^2},  \\[6pt]
\wh{k}(\xi_{i},\eta_{j})=\wh{k}(\xi_{i},\eta_{\jj})
=\wh{k}(\xi_{\ii},\eta_{j})=\wh{k}(\xi_{\ii},\eta_{\jj})=0.  \\[6pt]
\end{array}
\]
\begin{cor}
The manifold $(T\SSS,H,\wh{G})$ of an arbitrary h-sphere  $(\SSS,J,\allowbreak{}g)$
has constant sectional curvatures of the $J_{\alpha}$-totally-real
2-planes and vanishing sectional curvatures of the
$J_{\alpha}$-holomorphic 2-planes ($\alpha=1,2,3$).
\end{cor}

%%%%%%%%%%%%%%%%%%%%%%%%%%%%%%%%%%%%%%%%%%%%%%%%%%%%%%%%%%%%%%%%%%%%%%%%%%%%%%

\vspace{20pt}

\begin{center}
$\divideontimes\divideontimes\divideontimes$
\end{center} 

\newpage

\addtocounter{section}{1}\setcounter{subsection}{0}\setcounter{subsubsection}{0}

\setcounter{thm}{0}\setcounter{dfn}{0}\setcounter{equation}{0}

\label{par:assNij}

 \Large{

\
\\[6pt]
\bigskip

\
\\[6pt]
\bigskip

\lhead{\emph{Chapter II $|$ \S\thesection. {Associated Nijenhuis tensors on almost hypercomplex manifolds %and almost hypercontact
\ldots %structures %and Metrics of Hermitian and Norden Type
}}}
%\thispagestyle{empty}

%\noindent  {\Huge\bf \S\thesection. {Associated Nijenhuis tensors \\[12pt]
%\phantom{\S\thesection. }on manifolds with almost%and almost hypercontact
%\\[12pt]
%\phantom{\S\thesection. }hypercomplex structures and  \\[12pt]
%\phantom{\S\thesection. }Hermitian-Norden metrics}
%}%\\[6pt]\vskip2pt}

\noindent
\begin{tabular}{r"l}
  %\hline
\hspace{-6pt}{\Huge\bf \S\thesection.}  & {\Huge\bf Associated Nijenhuis tensors on } \\[12pt]
                             & {\Huge\bf almost hypercomplex manifolds } \\[12pt]
                             & {\Huge\bf with Hermitian-Norden metrics}
  %\hline
\end{tabular}

\vskip 1cm

\begin{quote}
\begin{large}
In the present section, it is introduced an associated Nijenhuis tensor of endomorphisms in the tangent bundle of an almost hypercomplex manifold with Hermitian-Norden metrics.
There are studied relations between the six associated Nijenhuis tensors of an almost hypercomplex structure as well as their vanishing.
It is given a geometric interpretation of the associated Nijenhuis tensors for an almost hypercomplex structure and Hermitian-Norden metrics.
Finally, an example of a 4-dimensional manifold of the considered type with vanishing associated Nijenhuis tensors is given.

The main results of this section are published in \cite{Man52}.
\end{large}
\end{quote}

%
%
%\vskip 0.2in \addtocounter{subsection}{1}
%
%\noindent  {\Large\bf \thesubsection. Introduction}

\vskip 0.15in
%\section{Introduction}

%The object of study in the present work are almost hypercomplex manifolds.
The vanishing of the Nijenhuis tensors as conditions for in\-te\-gra\-bility of the manifold are long-known \cite{YaAk}.
The goal here is to introduce an appropriate tensor on an almost hypercomplex manifold and to establish that its vanishing is a necessary and sufficient condition for existing of an affine connection with totally skew-symmetric torsion preserving the almost hypercomplex structure and Hermitian-Norden metrics.

The present section is organised as follows. In Subsection~\thesection.1, for endomorphisms in the tangent bundle, it is introduced a symmetric tensor, which is associated to the Nijenhuis tensor.
In Subsection~\thesection.2, it is deduced a series of relations between the six associated Nijenhuis tensors of an almost hypercomplex structure as well as it is proved that if two of these six tensors vanish, then the others also vanish.
In Subsection~\thesection.3, it is proved the main theorem for the geometric interpretation of the associated Nijenhuis tensors for an almost hypercomplex structure and
Hermitian-Norden metrics.
In Subsection~\thesection.4, it is given an example of a 4-dimensional manifold of the considered type with vanishing associated Nijenhuis tensors.

% ------------------------------------------------------------------------

\vskip 0.2in \addtocounter{subsection}{1} \setcounter{subsubsection}{0}

\noindent  {\Large\bf \thesubsection. The associated Nijenhuis tensors of endomorphisms}%\\[6pt]\vskip2pt}

\vskip 0.15in
%\section{The associated Nijenhuis tensors of endomorphisms}\label{sect-endo}

Let us consider a differentiable manifold $\MM$ and let the symmetric braces $\{\cdot,\cdot\}$ be defined by \eqref{braces}
%$\{x,y\}=\n_x y+\n_y x$
for all differentiable vector fields on $\MM$.%, i.e. $x,y\in\mathfrak{X}(\MM)$.

As it is well known \cite{KoNo-1}, the Nijenhuis tensor $[J,K]$ for arbitrary tensor fields $J$ and $K$ of type (1,1) on $\MM$ is a tensor field of type (1,2) defined by
\begin{equation}\label{N-JK}
\begin{array}{l}
  2[J,K](x,y) = (JK+KJ)[x,y]\\[6pt]
  \phantom{2[J,K](x,y) = }
  +[J x,K y]-J[K x, y]-J[x,K y]\\[6pt]
  \phantom{2[J,K](x,y) = }
  +[K x,J y]-K[J x, y]-K[x,J y].
\end{array}
\end{equation}

As a consequence, the Nijenhuis tensor $[J,J]$ for an arbitrary tensor field $J$ is determined by the following equality
\begin{equation}\label{NJ-arbitrary}
[J, J](x, y)=\left[Jx,Jy\right]+J^2\left[x,y\right]-J\left[Jx,y\right]-J\left[x,Jy\right].
\end{equation}

If $J$ is an almost complex structure, then the Nijenhuis tensor of $J$ is determined by the latter formula but with $J^2=-I$, where $I$ stands for the identity, i.e. \eqref{NJ} is valid.

Besides that, by \eqref{hatNJ} %\cite{Man48}
it is defined the %following symmetric
%(1,2)-tensor in analogy to \eqref{NJ} by
%%\[
%\{J ,J\}(x,y)=\{Jx,Jy\}-\{x,y\}-J\{Jx,y\}-J\{x,Jy\},
%\]
%where the symmetric braces $\{x,y\}=\nabla_x y+\nabla_y x$
%are used instead of the antisymmetric brackets $[x,y]=\nabla_x y-\nabla_y x$. Namely, we firstly set \eqref{braces}.
%\[
%\begin{array}{l}
%g(\{x,y\},z)=g(\nabla_x y+\nabla_y x,z)\\[6pt]
%\phantom{g(\{x,y\},z)}
%=x g(y,z)+y g(x,z)-z g(x,y)\\[6pt]
%\phantom{g(\{x,y\},z)==x g(y,z)}
%+g([z,x],y)+g([z,y],x).
%\end{array}
%\]
associated Nijenhuis tensor $\{J,J\}$ of the almost complex structure $J$ and the pseudo-Riemannian metric $g$. % (\cite{Man48}).

Bearing in mind \eqref{N-JK}, we give the following
\begin{dfn}
The (1,2)-tensor $\{J,K\}$ for (1,1)-tensors $J$ and $K$, defined by
\begin{equation}\label{1.1}
\begin{array}{l}
  2\{J,K\}(x,y) = (JK+KJ)\{x,y\}\\[6pt]
  \phantom{2\{J,K\}(x,y) = }
  +\{J x,K y\}-J\{K x, y\}-J\{x,K y\}\\[6pt]
  \phantom{2\{J,K\}(x,y) = }
  +\{K x,J y\}-K\{J x, y\}-K\{x,J y\},
\end{array}
\end{equation}
is called an \emph{associated Nijenhuis tensor of two endomorphisms on a pseudo-Riemannian  manifold}.
\end{dfn}
Obviously, it is symmetric with respect to tensor (1,1)-fields, i.e.
\begin{equation}\label{sym}
\{J,K\}(x,y)=\{K,J\}(x,y).
\end{equation}
Moreover, it is symmetric with respect to vector fields, i.e.
\[
\{J,K\}(x,y)=\{J,K\}(y,x).
\]

In the case when some of the tensor fields in $\{J,K\}$ is the identity $I$ (let us suppose $K=I$), then \eqref{1.1} yields that the corresponding associated
Nijenhuis tensor vanishes, i.e.
\begin{equation}\label{1.1a}
  \{J,I\}(x,y) = 0.
\end{equation}

For the case of two identical tensor (1,1)-fields, \eqref{1.1} implies
\begin{equation}\label{1.2}
\begin{array}{l}
  \{J,J\}(x,y) = J^2\{x,y\}+\{J x,J y\}
  -J\{J x, y\}-J\{x,J y\}.
\end{array}
\end{equation}

Let $L$ be also a tensor field of type (1,1) and $S$ be a tensor field of type (1,2).
Further, we use the following notation of Fr\"olicher-Nijenhuis from \cite{FrNi}
\begin{align}
  (S\barwedge L)(x,y)&=S(Lx,y)+S(x,Ly),\label{1.4}\\[6pt]
  (L\barwedge S)(x,y)&=L\bigl(S(x,y)\bigr).\label{1.5}
\end{align}
According to \cite{YaAk}, the following properties of the latter notation are valid
%for arbitrary tensor (1,1)-fields $J$, $K$ and a tensor (1,2)-field $S$:
\begin{align}
  \bigl(S\barwedge J\bigr)\barwedge K - \bigl(S\barwedge K\bigr)\barwedge J &=
  S \barwedge JK - S \barwedge KJ,\label{1.7}\\[6pt]
  (J\barwedge S)\barwedge K &= J\barwedge (S\barwedge K).\label{1.8}
\end{align}
%By virtue of \eqref{1.4}
%we have consecutively
%\[
%\begin{array}{l}
%  \bigl((S\barwedge J)\barwedge K\bigr) (x,y)- \bigl((S\barwedge K)\barwedge J\bigr) (x,y)\\[6pt]
%  =
%  (S\barwedge J)(Kx,y)+(S\barwedge J)(x,Ky)
%  - (S\barwedge K)(Jx,y)-(S\barwedge K)(x,Jy) \\[6pt]
%  =
%  S(JKx,y)+S(Kx,Jy)+
%  S(Jx,Ky)+S(x,JKy)\\[6pt]\quad
%  - S(KJX,y)-S(Jx,Ky)
%  - S(Kx,Jy)-S(x,KJY)\\[6pt]
% =
%  S(JKx,y)+S(x,JKy)
%  - S(KJX,y)-S(x,KJY)\\[6pt]
%  =
%  \bigl(S \barwedge JK\bigr) (x,y)- \bigl(S \barwedge KJ\bigr)(x,y).
%\end{array}
%\]
%Consequently \eqref{1.7} is proved.
%
%
%Such a way \eqref{1.8} is proved.

\begin{lem}
For arbitrary endomorphisms $J$, $K$ and $L$ in $\mathfrak{X}(\MM)$ and the relevant associated Nijenhuis tensors,  the following identity holds
\begin{equation}\label{1.6}
\begin{array}{l}
  \{J,KL\}+\{K,JL\}
  =\{J,K\}\barwedge L + J \barwedge \{K,L\} \\[6pt]
  \phantom{\{J,KL\}+\{K,JL\}
  =\{J,K\}\barwedge L}
  + K \barwedge \{J,L\}.
\end{array}
\end{equation}
\end{lem}
\begin{proof}
Using \eqref{1.1}, \eqref{1.4} and \eqref{1.5}, we get consequently
\[
\begin{array}{l}
  \bigl(\{J,K\}\barwedge L\bigr) (x,y)+ \bigl(J \barwedge \{K,L\}\bigr)(x,y) + \bigl(K \barwedge \{J,L\}\bigr)(x,y)\\[6pt]
  =\{J,K\} (Lx,y)+ \{J,K\} (x,Ly)+ J\bigl( \{K,L\}(x,y)\bigr)\\[6pt]
  \phantom{=} + K\bigl( \{J,L\}(x,y)\bigr)\\[6pt]
  =\dfrac12\bigl[(JK+KJ) \{Lx,y\} +\{JLx,Ky\} +\{KLx,Jy\}\\[6pt] \qquad
   -J\{KLx,y\}  -K\{JLx,y\}  -J\{Lx,Ky\} -K\{Lx,Jy\}\\[6pt] \qquad
   +(JK+KJ) \{x,Ly\} +\{Jx,KLy\} +\{Kx,JLy\}\\[6pt] \qquad
   -J\{Kx,Ly\}  -K\{Jx,Ly\}  -J\{x,KLy\} -K\{x,JLy\}\\[6pt] \qquad
   +J
   (KL+LK)\{x,y\}+J\{K x,L y\}+J\{L x,K y\}\\[6pt]
   \qquad
  -JK\{L x, y\}-JL\{K x, y\}-JK\{x,L y\}-JL\{x,K y\}
   \\[6pt]
   \qquad
   +K
   (JL+LJ)\{x,y\}+K\{J x,L y\}+K\{L x,J y\}\\[6pt] \qquad
  -KJ\{L x, y\}-KL\{J x, y\}-KJ\{x,L y\}-KL\{x,J y\}
   \bigl] %\\[6pt]
\end{array}
\]
\[
\begin{array}{l}
   =\dfrac12\bigl[\{JLx,Ky\} +\{KLx,Jy\}
   -J\{KLx,y\}  -K\{JLx,y\} \\[6pt] \qquad
    +\{Jx,KLy\} +\{Kx,JLy\}
   -J\{x,KLy\} -K\{x,JLy\}\\[6pt] \qquad
   +J
   (KL+LK)\{x,y\}
  -JL\{K x, y\}-JL\{x,K y\}
   \\[6pt]
   \qquad
   +K
   (JL+LJ)\{x,y\}
  -KL\{J x, y\}-KL\{x,J y\}
   \bigl] \\[6pt]
  = \{J,KL\}(x,y)+\{K,JL\}(x,y).
\end{array}
\]
Therefore, \eqref{1.6} is proved.
\end{proof}

\vskip 0.2in \addtocounter{subsection}{1} \setcounter{subsubsection}{0}

\noindent  {\Large\bf \thesubsection. The associated Nijenhuis tensors of the almost hypercomplex structure}%\\[6pt]\vskip2pt}

\vskip 0.15in
%\section{The associated Nijenhuis tensors of the almost hypercomplex structure}\label{sect-hyper}

Let $\MM$ be a differentiable manifold of dimension $4n$. It admits an almost hypercomplex structure  $(J_1,J_2,J_3)$, i.e. a triad of almost complex structures satisfying the properties \eqref{J123}.
%%
%\begin{equation}\label{J123} %
%J_\al=J_\bt\circ J_\gm=-J_\gm\circ J_\bt, \qquad
%J_\al^2=-I%
%\end{equation} %
%for all cyclic permutations $(\al, \bt, \gm)$ of $(1,2,3)$, where $I$
%denotes the identity (\cite{Boy}, \cite{AlMa}).

\vskip 0.2in \addtocounter{subsubsection}{1}

\noindent  {\Large\bf{\emph{\thesubsubsection. Relations between the associated Nijenhuis tensors}}}

\vskip 0.15in
%\subsection{Relations between the associated Nijenhuis tensors }

The presence of three almost complex structures, which form an almost hypercomplex structure, implies the existence of six associated Nijenhuis tensors: three for each almost complex structure and three for each pair of almost complex structures,  namely
\[
\{J_1,J_1\},\quad \{J_2,J_2\},\quad \{J_3,J_3\},\quad
\{J_1,J_2\},\quad \{J_2,J_3\},\quad \{J_3,J_1\}.
\]
They are called the \emph{associated Nijenhuis tensors of the almost hypercomplex structure $(J_1,J_2,J_3)$ on $(\MM,g)$}.

In the present section we examine some relations between them, which we use later.

\begin{lem}\label{lem:3.1}
The following relations between the associated Nijenhuis tensors of an almost hypercomplex manifold are valid:
\begin{gather}
  \{J_3,J_1\} =\dfrac12 \{J_1,J_1\}\barwedge J_2 + J_1 \barwedge \{J_1,J_2\},\label{2.2}
\\[6pt]
  \{J_3,J_1\}
  =-\{J_1,J_2\}\barwedge J_1 - J_1 \barwedge \{J_1,J_2\} - J_2 \barwedge \{J_1,J_1\},\label{2.3}
%\\[6pt]
%    \{J_3,J_1\}
%  =-\dfrac12\{J_1,J_2\}\barwedge J_1 -\dfrac12 J_2 \barwedge \{J_1,J_1\}
%  +\dfrac14 \{J_1,J_1\}\barwedge J_2,\label{2.4}
\\[6pt]
 \begin{split}\label{2.5}
J_2 \barwedge \{J_1,J_1\}&+\dfrac12 \{J_1,J_1\}\barwedge J_2\\[6pt]
&%\phantom{J_2 \barwedge \{J_1,J_1\}}
+ 2 J_1 \barwedge \{J_1,J_2\}  +\{J_1,J_2\}\barwedge J_1
=0,
\end{split}
\\[6pt]
  \{J_2,J_3\}
  =-\dfrac12 \{J_2,J_2\}\barwedge J_1 - J_2 \barwedge \{J_1,J_2\},\label{2.6}
%\\[6pt]
\end{gather}
\begin{gather}
  \{J_2,J_3\}
  =J_1 \barwedge \{J_2,J_2\} + \{J_1,J_2\}\barwedge J_2 + J_2 \barwedge \{J_1,J_2\},\label{2.7}
%\\[6pt]
%  \{J_2,J_3\}
%  =-\dfrac14 \{J_2,J_2\}\barwedge J_1 +\dfrac12 J_1 \barwedge \{J_2,J_2\}
%  +\dfrac12  \{J_1,J_2\}\barwedge J_2,
%\label{2.8}
\\[6pt]
 \begin{split}\label{2.9}
J_1 \barwedge \{J_2,J_2\} &+\dfrac12 \{J_2,J_2\}\barwedge J_1\\[6pt]
&%\phantom{J_1 \barwedge \{J_2,J_2\}}
  + \{J_1,J_2\}\barwedge J_2 + 2 J_2 \barwedge \{J_1,J_2\}=0,
\end{split}
\\[6pt]
\begin{split}\label{2.10}
  \{J_3,J_3\}-\{J_1,J_1\}
  =\{J_3,J_1\}\barwedge J_2 &+ J_3 \barwedge \{J_1,J_2\}\\[6pt]
  %\phantom{\{J_3,J_3\}-\{J_1,J_1\}=\{J_3,J_1\}\barwedge J_2}
  &+ J_1 \barwedge \{J_2,J_3\},
\end{split}
%\\[6pt]
%\begin{split}\label{2.11}
%  \{J_2,J_2\}-\{J_3,J_3\}
%  =\{J_2,J_3\}\barwedge J_1 &+ J_2 \barwedge \{J_3,J_1\} \\[6pt]
%  &%\phantom{\{J_2,J_2\}-\{J_3,J_3\}=\{J_2,J_3\}\barwedge J_1}
%  + J_3 \barwedge \{J_1,J_2\},
%\end{split}
\\[6pt]
\begin{split}\label{2.12}
  \{J_3,J_3\}=\dfrac12\bigl(\{J_1,J_1\}
  &+\{J_3,J_1\}\barwedge J_2 - J_2 \barwedge \{J_3,J_1\}\phantom{\bigr),}\\[6pt]
  &-\{J_2,J_3\}\barwedge J_1+ J_1 \barwedge \{J_2,J_3\}\bigr),
\end{split}
\\[6pt]
\begin{split}\label{2.13}
\{J_1,J_1\}&-\{J_2,J_2\}+\{J_3,J_1\}\barwedge J_2 + J_2 \barwedge \{J_3,J_1\}\phantom{=0.\ }\\[6pt]
&+ 2J_3 \barwedge \{J_1,J_2\}
+\{J_2,J_3\}\barwedge J_1+ J_1 \barwedge \{J_2,J_3\}=0,
\end{split}
%\\[6pt]
%  \{J_1,J_1\}\barwedge J_1 =- 2 J_1 \barwedge \{J_1,J_1\},\label{2.14}
\\[6pt]
    \{J_2,J_2\}\barwedge J_2 =- 2 J_2 \barwedge \{J_2,J_2\}.\label{2.15}
%\\[6pt]
%    \{J_3,J_3\}\barwedge J_3 =- 2 J_3 \barwedge \{J_3,J_3\}.\label{2.16}
\end{gather}
\end{lem}

\begin{proof}
The identity \eqref{1.6} for $J=K=J_1$ and $L=J_2$ has the form in \eqref{2.2}, because of $J_1J_2=J_3$ from \eqref{J123} and \eqref{sym}.

On the other hand,  \eqref{1.6} for $J=J_2$ and $K=L=J_1$ implies
% \begin{equation*}\label{2.3a}
% \begin{array}{l}
%  \{J_2,J_1J_1\}+\{J_1,J_2J_1\}
%  =\{J_2,J_1\}\barwedge J_1 + J_2 \barwedge \{J_1,J_1\} \\[6pt]
%  \phantom{\{J_2,J_1J_1\}+\{J_1,J_2J_1\}
%  =\{J_2,J_1\}\barwedge J_1}
%  + J_1 \barwedge \{J_2,J_1\}.
%\end{array}
%\end{equation*}
\[
  \{J_2,J_1J_1\}+\{J_1,J_2J_1\}
  =\{J_2,J_1\}\barwedge J_1 + J_2 \barwedge \{J_1,J_1\}
  + J_1 \barwedge \{J_2,J_1\}.
\]
Next, applying $J_1^2=-I$, $J_2J_1=-J_3$, which are corollaries of \eqref{J123}, as well as \eqref{sym} and \eqref{1.1a}, we obtain
\eqref{2.3}.

By virtue of \eqref{2.2} and \eqref{2.3} we get %the corollaries \eqref{2.4} and
\eqref{2.5}.

Next, setting $J=K=J_2$ and $L=J_1$ in \eqref{1.6}, we find
\[
  \{J_2,J_2J_1\}+\{J_2,J_2J_1\}
  =\{J_2,J_2\}\barwedge J_1 + J_2 \barwedge \{J_2,J_1\}
   + J_2 \barwedge \{J_2,J_1\},
\]
which is equivalent to \eqref{2.6}, taking into account \eqref{sym} and $J_2J_1=-J_3$.

On the other hand, setting $J=J_1$ and $K=L=J_2$ in \eqref{1.6}, we find
\[
  \{J_1,J_2J_2\}+\{J_2,J_1J_2\}
  =\{J_1,J_2\}\barwedge J_2 + J_1 \barwedge \{J_2,J_2\}
  + J_2 \barwedge \{J_1,J_2\}
\]
and applying \eqref{1.1a} and $J_1J_2=J_3$, we have \eqref{2.7}.

By virtue of \eqref{2.6} and \eqref{2.7} we get %the corollaries \eqref{2.8} and
\eqref{2.9}.

Now, setting $J=J_3$, $K=J_1$ and $L=J_2$ in \eqref{1.6} and applying $J_1J_2=J_3$, $J_3J_2=-J_1$ and \eqref{sym}, we have
\eqref{2.10}.

The equality \eqref{2.10} and the resulting equality from it by the substitutions $J_3$, $J_1$, $J_2$ for $J_1$, $J_2$, $J_3$, respectively, imply \eqref{2.12} and \eqref{2.13}.

%On the other hand, setting $J=J_3$, $K=J_2$ and $L=J_1$ in \eqref{1.6} and applying $J_2J_1=-J_3$, $J_3J_1=J_2$ and \eqref{sym}, we have \eqref{2.11}.

%Thus, \eqref{2.10} and \eqref{2.11} imply \eqref{2.12} and \eqref{2.13}.

At the end, the identity \eqref{1.6} for $J=K=L=J_2$ implies
\eqref{2.15} because of \eqref{1.1a}.
%In a similar way, we obtain \eqref{2.15}  and \eqref{2.16}.
\end{proof}

\vskip 0.2in \addtocounter{subsubsection}{1}

\noindent  {\Large\bf{\emph{\thesubsubsection. Vanishing of the associated Nijenhuis tensors}}}

\vskip 0.15in
%\subsection{Vanishing of the associated Nijenhuis tensors}

Now, we will study at least how many associated Nijenhuis tensors (and which) must be vanished to become all associated Nijenhuis tensors zeros on an almost hypercomplex manifold.

As proof steps of the relevant main theorem, we will prove a series of lemmas.

%\begin{lem}\label{thm:1}
%If $\{J_1,J_1\}=0$ and $\{J_2,J_2\}=0$, then $\{J_1,J_2\}=0$.
%\end{lem}
\begin{lem}\label{thm:2}
%If $\{J_1,J_1\}=0$ and $\{J_2,J_2\}=0$, then $\{J_1,J_2\}=0$, $\{J_2,J_3\}=0$, $\{J_3,J_1\}=0$ and $\{J_3,J_3\}=0$.
If $\{J_1,J_1\}$ and $\{J_2,J_2\}$ vanish, then $\{J_1,J_2\}$, $\{J_2,J_3\}$, $\{J_3,J_1\}$ and $\{J_3,J_3\}$ vanish.
\end{lem}
\begin{proof}
The formulae \eqref{2.2} and \eqref{2.5}, because of $\{J_1,J_1\}=0$, imply respectively
\begin{gather}
  \{J_3,J_1\} = J_1 \barwedge \{J_1,J_2\},\label{3.1}\\[6pt]
\{J_1,J_2\}\barwedge J_1
=-2 J_1 \barwedge \{J_1,J_2\}.\label{3.2}
\end{gather}

Similarly, since $\{J_2,J_2\}=0$, the equalities \eqref{2.6} and \eqref{2.9} take the corresponding form
\begin{gather}
  \{J_2,J_3\}
  =- J_2 \barwedge \{J_1,J_2\},\label{3.3}\\[6pt]
\{J_1,J_2\}\barwedge J_2 =- 2 J_2 \barwedge \{J_1,J_2\}.\label{3.4}
\end{gather}

We set $\{J_1,J_1\}=0$, $\{J_2,J_2\}=0$, \eqref{3.1} and \eqref{3.3} in  \eqref{2.13} and we obtain
\[
\begin{array}{l}
2J_3 \barwedge \{J_1,J_2\}
+\bigl(J_1 \barwedge \{J_1,J_2\}\bigr)\barwedge J_2 + J_2 \barwedge \bigl(J_1 \barwedge \{J_1,J_2\}\bigr)\\[6pt]
  \phantom{2J_3 \barwedge \{J_1,J_2\}}
-\bigl(J_2 \barwedge \{J_1,J_2\}\bigr)\barwedge J_1- J_1 \barwedge \bigl(J_2 \barwedge \{J_1,J_2\}\bigr)=0.
\end{array}
\]
The latter equality is equivalent to
\begin{equation}\label{3.5}
\begin{array}{l}
\bigl(J_1 \barwedge \{J_1,J_2\}\bigr)\barwedge J_2
=\bigl(J_2 \barwedge \{J_1,J_2\}\bigr)\barwedge J_1,
\end{array}
\end{equation}
because of the following corollaries of \eqref{1.5} and the identities $J_3=J_1J_2=-J_2J_1$ from \eqref{J123}
\begin{equation}\label{3.5'}
\begin{array}{l}
J_2 \barwedge \bigl(J_1 \barwedge \{J_1,J_2\}\bigr)=
-J_1 \barwedge \bigl(J_2 \barwedge \{J_1,J_2\}\bigr)\\[6pt]
\phantom{J_2 \barwedge \bigl(J_1 \barwedge \{J_1,J_2\}\bigr)}
=
-J_3 \barwedge \{J_1,J_2\}.
\end{array}
\end{equation}

According to \eqref{1.8}, \eqref{3.2}, \eqref{3.4} and \eqref{3.5'}, the equality \eqref{3.5} yields
\[
\begin{array}{l}
0=
\bigl(J_1 \barwedge \{J_1,J_2\}\bigr)\barwedge J_2
-\bigl(J_2 \barwedge \{J_1,J_2\}\bigr)\barwedge J_1\\[6pt]
\phantom{0}
=
J_1 \barwedge \bigl(\{J_1,J_2\}\barwedge J_2\bigr)
-J_2 \barwedge \bigl(\{J_1,J_2\}\barwedge J_1\bigr)\\[6pt]
\phantom{0}
=
-2J_1 \barwedge \bigl(J_2\barwedge \{J_1,J_2\}\bigr)
+2J_2 \barwedge \bigl(J_1\barwedge \{J_1,J_2\}\bigr)\\[6pt]
\phantom{0}
=
-4J_3 \barwedge \{J_1,J_2\},
\end{array}
\]
i.e. it is valid
\begin{equation}\label{3.6}
\begin{array}{l}
J_3 \barwedge \{J_1,J_2\}=0.
\end{array}
\end{equation}
Therefore, because of $J_3^2=-I$, we get
\[
\{J_1,J_2\}=0.
\]
%which completes the proof of the lemma.
%\end{proof}

%\begin{lem}\label{thm:2}
%If $\{J_1,J_1\}=0$ and $\{J_2,J_2\}=0$, then $\{J_1,J_2\}=0$, $\{J_2,J_3\}=0$, $\{J_3,J_1\}=0$ and $\{J_3,J_3\}=0$.
%\end{lem}
%\begin{proof}
%Bearing in mind \lemref{thm:1}, we have $\{J_1,J_2\}=0$.
Next, \eqref{3.1} and \eqref{3.3} imply
$\{J_3,J_1\}=0$ and $\{J_2,J_3\}=0$, respectively.
Finally, since $\{J_1,J_1\}$, $\{J_2,J_2\}$, $\{J_2,J_3\}$ and $\{J_3,J_1\}$ vanish, the relation
\eqref{2.12} yields $\{J_3,J_3\}=0$.
\end{proof}

%\begin{lem}\label{thm:3}
%If $\{J_1,J_1\}=0$ and $\{J_1,J_2\}=0$, then $\{J_2,J_2\}=0$.
%\end{lem}
\begin{lem}\label{thm:4}
If $\{J_1,J_1\}$ and $\{J_1,J_2\}$ vanish, then $\{J_2,J_2\}$, $\{J_3,J_3\}$, $\{J_2,J_3\}$ and $\{J_3,J_1\}$ vanish.
\end{lem}
\begin{proof}
Setting $J=K=J_2$ and $L=J_3$ in \eqref{1.6}, using $J_2J_3=J_1$ and \eqref{sym}, we obtain
\[
  \{J_1,J_2\}
  =\dfrac12 \{J_2,J_2\}\barwedge J_3 + J_2 \barwedge \{J_2,J_3\}.
\]
Since $\{J_1,J_2\}=0$, we get
\begin{equation}\label{3.7}
  \{J_2,J_2\}\barwedge J_3 =-2 J_2 \barwedge \{J_2,J_3\}.
\end{equation}

The condition $\{J_1,J_2\}=0$ and \eqref{2.7} imply
\[
  \{J_2,J_3\}
  =J_1 \barwedge \{J_2,J_2\}.
\]
The latter equality and \eqref{3.7}, using $J_2J_1=-J_3$, yield
\[
  \{J_2,J_2\}\barwedge J_3 =-2 J_2 \barwedge \bigl(J_1 \barwedge \{J_2,J_2\}\bigr)
  =2 J_3\barwedge \{J_2,J_2\},
\]
that is
\begin{equation}\label{3.8}
  \{J_2,J_2\}\barwedge J_3   =2 J_3\barwedge \{J_2,J_2\}.
\end{equation}

On the other hand, setting $S=\{J_2,J_2\}$, $J=J_1$ and $K=J_2$ in \eqref{1.7} and using $J_1J_2=-J_2J_1=J_3$, we obtain
\begin{equation}\label{3.9}
\begin{array}{l}
  2\{J_2,J_2\} \barwedge J_3=
  \bigl(\{J_2,J_2\}\barwedge J_1\bigr)\barwedge J_2 - \bigl(\{J_2,J_2\}\barwedge J_2\bigr)\barwedge J_1.
\end{array}
\end{equation}

However, since $\{J_1,J_2\}=0$, then \eqref{2.9} and \eqref{2.10} imply
\begin{equation}\label{3.10}
 \begin{array}{l}
\{J_2,J_2\}\barwedge J_1
=-2J_1 \barwedge \{J_2,J_2\}.
\end{array}
\end{equation}

Now, substituting \eqref{2.15} and \eqref{3.10} into \eqref{3.9}, we obtain
\[
\begin{array}{l}
  \{J_2,J_2\} \barwedge J_3=
  -\bigl(J_1 \barwedge \{J_2,J_2\}\bigr)\barwedge J_2 + \bigl(J_2 \barwedge \{J_2,J_2\}\bigr)\barwedge J_1
\end{array}
\]
and applying \eqref{1.8}, we have
\[
\begin{array}{l}
  \{J_2,J_2\} \barwedge J_3=
  -J_1 \barwedge \bigl(\{J_2,J_2\}\barwedge J_2\bigr) +J_2 \barwedge \bigl(\{J_2,J_2\}\barwedge J_1\bigr).
\end{array}
\]
In the latter equality, applying \eqref{2.15}, \eqref{3.10} and $J_1J_2=-J_2J_1=J_3$, we get
\[
\begin{array}{l}
  \{J_2,J_2\} \barwedge J_3=
  2J_1 \barwedge \bigl(J_2\barwedge \{J_2,J_2\}\bigr) -2J_2 \barwedge \bigl(J_1\barwedge \{J_2,J_2\}\bigr)\\[6pt]
  \phantom{\{J_2,J_2\} \barwedge J_3}
  =4J_3 \barwedge \{J_2,J_2\},
\end{array}
\]
that is
\begin{equation}\label{3.11}
  \{J_2,J_2\} \barwedge J_3  =4J_3 \barwedge \{J_2,J_2\}.
  \end{equation}

Comparing \eqref{3.8} and \eqref{3.11}, we conclude that
\[
  J_3 \barwedge \{J_2,J_2\}=0,
\]
which is equivalent to
\[
\{J_2,J_2\}=0
\]
by virtue of $J_3^2=-I$. This completes the proof of the first assertion in the lemma.
%\end{proof}

Combining it with \lemref{thm:2}, %and \lemref{thm:3},
we establish the truthfulness of the whole lemma. %following
%\begin{lem}\label{thm:4}
%If $\{J_1,J_1\}=0$ and $\{J_1,J_2\}=0$, then $\{J_2,J_2\}=0$, $\{J_3,J_3\}=0$, $\{J_2,J_3\}=0$, $\{J_3,J_1\}=0$.
%\end{lem}
\end{proof}

%Now, we prove
%\begin{lem}\label{thm:5}
%If $\{J_1,J_2\}=0$ and $\{J_3,J_1\}=0$, then $\{J_1,J_1\}=0$.
%\end{lem}
\begin{lem}\label{thm:6}
If $\{J_1,J_2\}$ and $\{J_3,J_1\}$ vanish, then $\{J_1,J_1\}$, $\{J_2,J_2\}$, $\{J_3,J_3\}$ and $\{J_2,J_3\}$ vanish.
\end{lem}
\begin{proof}
From \eqref{2.3} and the vanishing of $\{J_1,J_2\}$ and $\{J_3,J_1\}$, we get
\[
  J_2 \barwedge \{J_1,J_1\}=0,
\]
which is equivalent to
\[
  \{J_1,J_1\}=0.
\]
%\end{proof}

Now, combining the latter assertion and \lemref{thm:4}, % and \lemref{thm:5},
we have the validity of the present lemma. %following
%\begin{lem}\label{thm:6}
%If $\{J_1,J_2\}=0$ and $\{J_3,J_1\}=0$, then $\{J_1,J_1\}=0$, $\{J_2,J_2\}=0$, $\{J_3,J_3\}=0$, $\{J_2,J_3\}=0$.
%\end{lem}
\end{proof}

%Next, we prove
%\begin{lem}\label{thm:7}
%If $\{J_1,J_1\}=0$ and $\{J_2,J_3\}=0$, then $\{J_2,J_2\}=0$.
%\end{lem}
\begin{lem}\label{thm:8}
If $\{J_1,J_1\}$ and $\{J_2,J_3\}$ vanish, then $\{J_2,J_2\}$, $\{J_3,J_3\}$, $\{J_1,J_2\}$ and $\{J_3,J_1\}$ vanish.
\end{lem}
\begin{proof}
Firstly, from \eqref{2.2} and $\{J_1,J_1\}=\{J_2,J_3\}=0$, we obtain
\begin{equation}\label{3.12}
  \{J_3,J_1\} = J_1 \barwedge \{J_1,J_2\}.
\end{equation}

Secondly, setting $J=K=J_2$ and $L=J_3$ in \eqref{1.6} and using the equalities $J_2J_3=J_1$, $\{J_1,J_1\}=\{J_2,J_3\}=0$ and \eqref{sym}, we get
\begin{equation}\label{3.13}
  2\{J_1,J_2\}
  =\{J_2,J_2\}\barwedge J_3.
\end{equation}

On the other hand, setting $J=L=J_2$ and $K=J_3$ in \eqref{1.6}, using the assumptions, as well as $J_2^2=-I$, $J_3J_2=-J_1$  and \eqref{1.1a},
we find
\[
  \{J_1,J_2\}
  = -J_3 \barwedge \{J_2,J_2\},
\]
which, because of $J_3^2=-I$, is equivalent to
  \begin{equation}\label{3.14}
  \{J_2,J_2\}= J_3\barwedge \{J_1,J_2\}.
  \end{equation}

Moreover, the formula \eqref{2.13} and $\{J_1,J_1\}=\{J_2,J_3\}=0$ yield
\[
2J_3 \barwedge \{J_1,J_2\}
+\{J_3,J_1\}\barwedge J_2 + J_2 \barwedge \{J_3,J_1\}-\{J_2,J_2\}=0.
\]
The latter equality, \eqref{3.12} and \eqref{3.14}, using $J_2J_1=-J_3$, \eqref{1.5} and \eqref{1.8}, imply
\[
J_1 \barwedge \bigl(\{J_1,J_2\}\barwedge J_2\bigr) =0
\]
and therefore
\begin{equation}\label{3.15}
\{J_1,J_2\}\barwedge J_2 =0.
\end{equation}

By virtue of \eqref{3.14}, \eqref{3.15} and \eqref{1.8}, we get
\begin{equation}\label{3.16}
  \{J_2,J_2\}\barwedge J_2= 0.
\end{equation}

On the other hand, \eqref{2.15} and \eqref{3.16} imply
\[
J_2 \barwedge \{J_2,J_2\}=0,
\]
which is equivalent to
\[
\{J_2,J_2\}=0
\]
and the first assertion in the present lemma is proved.
%\end{proof}

Combining it with \lemref{thm:2}, % and \lemref{thm:7},
we obtain the validity of the rest equalities. %following
%\begin{lem}\label{thm:8}
%If $\{J_1,J_1\}=0$ and $\{J_2,J_3\}=0$, then $\{J_2,J_2\}=0$, $\{J_3,J_3\}=0$, $\{J_1,J_2\}=0$, $\{J_3,J_1\}=0$.
%\end{lem}
\end{proof}

Now, we are ready to prove the main theorem in the present section.

\begin{thm}\label{thm:9}
If two of the six associated Nijenhuis tensors
\[
\{J_1,J_1\},\quad \{J_2,J_2\},\quad \{J_3,J_3\},\quad
\{J_1,J_2\},\quad \{J_2,J_3\},\quad \{J_3,J_1\}
\]
vanish, then the others also vanish.
\end{thm}
\begin{proof}%[Proof of \thmref{thm:9}]
The truthfulness of this theorem follows from \lemref{thm:2}, \lemref{thm:4}, \lemref{thm:6} and \lemref{thm:8}.
\end{proof}

%%%%%%%%%%%%%%%%%%%%%%%%%%%%%%%%%%%%%%%%%%%%%%%%%%%%%%%%%%%%%%%%%%%%

%\section{Relations between the Nijenhuis tensors and the associated Nijenhuis tensors of an almost hypercomplex HN-metric structure}

\vskip 0.2in \addtocounter{subsection}{1} \setcounter{subsubsection}{0}

\noindent  {\Large\bf \thesubsection. Natural connections with totally skew-symmetric torsion on almost hypercomplex manifolds with Hermitian-Norden metrics}%\\[6pt]\vskip2pt}

\vskip 0.15in
%\section{Almost hypercomplex manifolds with Hermitian-Norden metrics with vanishing associated Nijenhuis tensors}\label{sect-assNije}

Let $g$ be a pseudo-Riemannian metric on an almost hypercomplex manifold $(\MM,J_1,J_2,J_3)$ defined by
\eqref{gJJ}.
%\begin{equation}\label{gJJ} %
%g(x,y)= \ea g(J_\al x,J_\al y),
%\end{equation} %
%where %
%\begin{equation}\label{ea}%
% \ea=
%\begin{cases}
%\begin{array}{ll}
%1, \quad & \al=1;\\[6pt]
%-1, \quad & \al=2;3.
%\end{array}
%\end{cases}
%\end{equation}
%
Then, we call that the almost hypercomplex manifold is equipped with Hermitian-Norden metrics.
%an HN-metric structure (HN is an abbreviation for Hermitian-Norden).
Namely, the metric $g$ is Hermitian for $\al=1$, whereas $g$ is a Norden metric in the cases  $\al=2$ and $\al=3$ \cite{GriMan24,ManGri32}.

Let us consider $(\MM,J_1,g)$ belonging to %is an almost Hermitian manifold with
%neutral metric of
$\GG_1=(\W_1\oplus\W_3\oplus\W_4)(J_1)$ (the class of cocalibrated manifolds with Hermitian
metric), %=\W_1\oplus\W_3\oplus\W_4$
according to the classification \eqref{cl-H} from \cite{GrHe}. Moreover, let $(\MM,J_\al,g)$,
$(\al=2;3)$ belong to %are almost complex manifolds with Norden metric of
$\W_3(J_\al)$ (the class of quasi K\"ahler manifolds with Norden
metric), according to the classification \eqref{cl-N} from \cite{GaBo}.
The mentioned classes are determined in
terms of the fundamental tensors $F_{\al}$, defined by \eqref{F'-al}, as follows:
\begin{gather}
\label{G1}
    \GG_1(J_1):\
    F_1 (x,x,z) = F_1 (J_1 x,J_1 x,z), \\[6pt]%[J_1,J_1](x,y,z)=-[J_1,J_1](x,z,y),\\[6pt]
\label{W3}
    \W_3(J_{\al}):\
    F_{\al} (x,y,z)+F_{\al} (y,z,x)+F_{\al} (z,x,y)=0 %\{J_\al,J_\al\}(x,y,z)=0,
\end{gather}
for $\al=2;3$.

\begin{rem}\label{rem:dim}
It is known from \cite{GrHe} that the class
$\GG_1$ of almost Hermitian manifolds
$(\MM,J_1,g)$ exists in general form when the dimension of $\MM$ is at
least 6. At dimension 4, $\GG_1$ is restricted to its subclass $\W_4$, the class of locally conformally equivalent manifolds to K\"ahler manifolds with Norden metric.
According to \cite{GaBo}, the lowest dimension for almost Norden manifolds in the class $\W_3$ is 4.
Thus, the almost hypercomplex manifold with Hermitian-Norden metrics belonging to the classes
$\GG_1(J_1)$, $\W_3(J_2)$, $\W_3(J_3)$ exists in general form when $\dim{\MM}=4n\geq 8$ holds.
\end{rem}

Let the corresponding tensors of type (0,3) with respect to $g$ of the pair of Nijenhuis tensors be denoted by
\[
\begin{array}{l}
[J_\al,J_\al](x,y,z)=g([J_\al,J_\al](x,y),z),\\[6pt]
\{J_\al,J_\al\}(x,y,z)=g(\{J_\al,J_\al\}(x,y),z).
\end{array}
\]
These tensors can be expressed by $F_{\al}$, as follows:
\begin{equation}
\begin{array}{l}
[J_\al,J_\al](x,y,z)=
F_{\al}(J_{\al} x,y,z)+\ea F_{\al}(x,y,J_{\al} z)
\label{enu-al}\\[6pt]
\phantom{[J_\al,J_\al](x,y,z)=}
-F_{\al}(J_{\al} y,x,z)-\ea F_{\al}(y,x,J_{\al} z),
\end{array}%\\[6pt]
\end{equation}
\begin{equation}
\begin{array}{l}
\{J_\al,J_\al\}(x,y,z)=
F_{\al}(J_{\al} x,y,z)+\ea F_{\al}(x,y,J_{\al} z)\\[6pt]
\phantom{\{J_\al,J_\al\}(x,y,z)=}
+F_{\al}(J_{\al} y,x,z)+\ea F_{\al}(y,x,J_{\al}z).\label{enhat-al}
\end{array}
\end{equation}
In the case $\al=2;3$, the latter formulae are given in \cite{GaBo}, as $\{J_\al,J_\al\}$ coincides with the tensor $\widetilde N$
introduced there by the second equality in \eqref{NJ-nabli}.

In \cite{GrHe}, it is given an equivalent definition of $\GG_1$ by the condition the Nijenhuis tensor $[J_1,J_1](x,y,z)$ to be a 3-form, i.e.
\begin{equation}\label{G1-3f}
  \GG_1(J_1):\quad [J_1,J_1](x,y,z)=-[J_1,J_1](x,z,y).
\end{equation}

\begin{prop}\label{prop:NN=Nhat}
For the Nijenhuis tensor and the associated Nijenhuis tensor of $(\MM,J_1,g)$ we have:
\begin{enumerate}
  \item the following relation
\begin{equation}\label{NN=NhatJ1}
\{J_1,J_1\}(x,y,z)=[J_1,J_1](z,x,y)+[J_1,J_1](z,y,x);
\end{equation}
  \item $\{J_1,J_1\}$ vanishes if and only if $[J_1,J_1]$ is a 3-form.
\end{enumerate}
\end{prop}
\begin{proof}
We compute the right hand side of \eqref{NN=NhatJ1} using \eqref{enu-al}. Applying \eqref{FaJ-prop} and their consequence
\begin{equation}\label{FffF1}
  F_1(x,y,J_1z)=F_1(x,J_1y,z),
\end{equation}
we obtain
\[
\begin{array}{l}
[J_1,J_1](z,x,y)+[J_1,J_1](z,y,x)=
-F_1(J_1x,z,y)-F_1(x,z,J_1y)\\[6pt]
\phantom{[J_1,J_1](z,x,y)+[J_1,J_1](z,y,x)=}
-F_1(J_1y,z,x)-F_1(y,z,J_1x).
\end{array}
\]
Using again \eqref{FffF1} and the first equality in \eqref{FaJ-prop}, we establish that the right hand side of the latter equality is equal to $\{J_1,J_1\}(x,y,z)$, according to \eqref{enhat-al} for $\al=1$.

The identity in (ii) follows immediately from (i).
\end{proof}

The assertion (ii) of \propref{prop:NN=Nhat} and \eqref{G1-3f} imply  the following
\begin{prop}\label{prop:G1}
The manifolds in the class $\GG_1(J_1)$ are characterized by the condition $\{J_1,J_1\}=0$.
\end{prop}

For almost Norden manifolds it is known from \cite{GaBo} that the manifolds in the class $\W_3(J_2)$ (respectively, $\W_3(J_3)$) are characterized by the condition $\{J_2,J_2\}=0$ (respectively, $\{J_3,J_3\}=0$).

From \thmref{thm:9} we have immediately that
%\begin{cor}\label{cor:aN123}
if two of associated Nijenhuis tensors $\{J_1,J_1\}$, $\{J_2,J_2\}$, $\{J_3,J_3\}$ vanish, then the third one also vanishes.
%\end{cor}
%
Thus, we establish the truthfulness of the following
\begin{thm}\label{thm:aN123}
If an almost hypercomplex manifold with Hermitian-Norden metrics belongs to two of the classes
$\GG_1(J_1)$, $\W_3(J_2)$, $\W_3(J_3)$
with respect to the corresponding almost complex structures, then it belongs also to the third class.
\end{thm}
In \cite{ManGri32}, it is proved that if
$\MM$ is in $\W_3(J_2)$ and $\W_3(J_3)$, then it
belongs to $\GG_1(J_1)$.

For the almost Hermitian manifold $(\MM,J_1,g)$ we apply Theorem 10.1 in \cite{Fri-Iv2}.
Then there exists an affine connection $\DDD^1$ with totally skew-symmetric torsion $T_1$ preserving
 $J_1$ and $g$ if and only if $(\MM,J_1,g)$ belongs to $\GG_1(J_1)$.
In this case $\DDD^1$ is unique and determined
by
\begin{equation}\label{T1D1Om}
T_1(x,y,z) = \D\g_1(J_1x, J_1y, J_1z) + [J_1,J_1](x,y,z),
\end{equation}
where $\g_1$ is the K\"ahler form determined by \eqref{gJ}.

Using properties \eqref{FaJ-prop} and \eqref{enu-al} for $\al=1$, the relation $\D\g_1=\s F_1$, where $\s$ is the cyclic sum by three arguments, we get the expression of $T_1$ in terms of $F_1$ as follows
\begin{equation}\label{T1D1F1}
T_1(x,y,z) = F_1(x,y,J_1z)-F_1(y,x,J_1z)-F_1(J_1z,x,y).
\end{equation}

On the other side, for the case $\al=2$ or $\al=3$ we dispose with an almost Norden manifold $(\MM,J_\al,g)$.
Then, according to Theorem 3.1 in \cite{Mek-10},
there exists an affine connection $\DDD^\al$  with totally skew-symmetric torsion $T_\al$ preserving
 $J_\al$ and $g$ if and only if $(\MM,J_\al,g)$ belongs to $\W_3(J_\al)$.
In this case $\DDD^\al$ is unique and determined, according to \eqref{TKT=F},
by the following expression of its torsion
\begin{equation}\label{T1D1F23}
T_\al(x,y,z) = -\frac12 \sx F_\al(x,y,J_\al z),\quad\quad \al=2,3.
\end{equation}

By virtue of \thmref{thm:aN123} and the comments above, we obtain the validity of the following
\begin{thm}\label{thm:HKT}
Let $(\MM,H,G)$ be an almost hypercomplex manifold with Hermitian-Norden metrics.
Then it admits an affine connection $\DDD^*$ with totally skew-sym\-met\-ric torsion preserving the structure $(H,G)$ if and only if two of the three associated Nijenhuis tensors $\{J_\al,J_\al\}$, $(\al=1,2,3)$ vanish and the equalities $T_1=T_2=T_3$ are valid, bearing in mind \eqref{T1D1F1} and \eqref{T1D1F23}.
If $\DDD^*$ exists, it is unique and determined by its torsion $T^*=T_1=T_2=T_3$.
\end{thm}

\newpage

\vskip 0.2in \addtocounter{subsection}{1} \setcounter{subsubsection}{0}

\noindent  {\Large\bf \thesubsection. A 4-dimensional example}%\\[6pt]\vskip2pt}

\vskip 0.15in
%\section{A 4-dimensional example}\label{sect-exa}

In \cite{GriMan24}, it is considered a connected Lie group $\LL$ with a corresponding Lie algebra $\mathfrak{l}$, %$\mathfrak{l}$
determined by the following conditions for the global basis of left invariant vector fields $\{x_1,x_2,x_3,x_4\}$:
\begin{equation}\label{[]4}
\begin{array}{ll}
[x_1,x_3]= \lambda_2 x_2+\lambda_4 x_4,\quad\quad &
[x_2,x_4]= \lambda_1 x_1+\lambda_3 x_3, \\[6pt]
[x_3,x_2]= \lambda_2 x_1+\lambda_3 x_4,\quad\quad &
[x_4,x_3]= \lambda_4 x_1-\lambda_3 x_2, \\[6pt]
[x_4,x_1]= \lambda_1 x_2+\lambda_4 x_3,\quad\quad &
[x_1,x_2]= \lambda_2 x_3-\lambda_1 x_4,
\end{array}
\end{equation}
where $\lambda_i\in \mathbb{R}$ $(i=1,2,3,4)$ and
$(\lambda_1,\lambda_2,\lambda_3,\lambda_4)\neq (0,0,0,0)$.
The pseu\-do-Riemannian metric $g$ is defined by %an invariant Norden metric under the adjoint representation of $G$ on $\mathfrak{l}$, determined by
\begin{equation}\label{g-ex}
\begin{array}{l}
  g(x_1,x_1)=g(x_2,x_2)=-g(x_3,x_3)=-g(x_4,x_4)=1, \\[6pt]
  g(x_i,x_j)=0,\quad i\neq j.%; \qquad\\[6pt]
%  g\left([x_i,x_j],x_k\right)+g\left([x_i,x_k],x_j\right)=0.
\end{array}
\end{equation}
There, it is introduced an almost hypercomplex structure $(J_1,J_2,J_3)$ on $\LL$ as follows:
\begin{equation}\label{J123-ex}
\begin{array}{llll}
J_1x_1=x_2,\quad\quad & J_2x_1=x_3,\quad\quad & J_3x_1=-x_4, \\[6pt]
J_1x_2=-x_1,\quad\quad & J_2x_2=x_4,\quad\quad & J_3x_2=x_3, \\[6pt]
J_1x_3=-x_4,\quad\quad & J_2x_3=-x_1,\quad\quad & J_3x_3=-x_2, \\[6pt]
J_1x_4=x_3,\quad\quad & J_2x_4=-x_2,\quad\quad & J_3x_4=x_1
\end{array}
\end{equation}
and then \eqref{J123} are valid.

In \cite{GrMaMe1}, it is constructed the manifold $(\LL,J_2,g)$ as an example of a 4-dimensional quasi-K\"ahler manifold with Norden metric. %, i.e. it is a $\W_3$-manifold according to the classification in \cite{GaBo}.

The conditions \eqref{g-ex} and \eqref{J123-ex} imply the properties \eqref{gJJ} and therefore the introduced structure on $\LL$ is an almost hypercomplex structure with Hermitian-Norden metrics. Hence, it follows that the constructed manifold is
an almost hypercomplex manifold with Hermitian-Norden metrics.
Moreover, in \cite{GriMan24}, it is shown that this manifold belongs to basic classes $\W_4(J_1)$, $\W_3(J_2)$, $\W_3(J_3)$ with respect to the corresponding almost complex structures.
These conclusions are made using the following nonzero basic components of $F_\al$:
\begin{equation*}\label{F1}
    \begin{array}{l}
        \frac{1}{2}\lambda_1=(F_1)_{114}=-(F_1)_{123}=(F_1)_{132}=-(F_1)_{141}\\[6pt]
        \phantom{\frac{1}{2}\lambda_1}
        =(F_1)_{213}=(F_1)_{224}=-(F_1)_{231}=-(F_1)_{242},\\[6pt]
        \frac{1}{2}\lambda_2=-(F_1)_{113}=-(F_1)_{124}=(F_1)_{131}=(F_1)_{142}\\[6pt]
        \phantom{\frac{1}{2}\lambda_2}
        =(F_1)_{214}=-(F_1)_{223}=(F_1)_{232}=-(F_1)_{241},\\[6pt]
    \end{array}
\end{equation*}
\begin{equation*}
    \begin{array}{l}        %
        \frac{1}{2}\lambda_3=-(F_1)_{314}=(F_1)_{323}=-(F_1)_{332}=(F_1)_{341}\\[6pt]
        \phantom{\frac{1}{2}\lambda_3}
        =(F_1)_{413}=(F_1)_{424}=-(F_1)_{431}=-(F_1)_{442},\\[6pt]
        \frac{1}{2}\lambda_4=-(F_1)_{313}=-(F_1)_{324}=(F_1)_{331}=(F_1)_{342}\\[6pt]
        \phantom{\frac{1}{2}\lambda_4}
        =-(F_1)_{414}=(F_1)_{423}=-(F_1)_{432}=(F_1)_{441};
    \end{array}
\end{equation*}
\begin{equation*}\label{F2}
\begin{array}{l}
\frac12\lambda_1=-\frac12(F_2)_{122}=-\frac12(F_2)_{144}=(F_2)_{212}=(F_2)_{221}=(F_2)_{234}
\\[6pt]
\phantom{\frac12\lambda_1}
=(F_2)_{243}=(F_2)_{414}=-(F_2)_{423}=-(F_2)_{432}=(F_2)_{441},
\\[6pt]
\frac12\lambda_2=(F_2)_{112}=(F_2)_{121}=(F_2)_{134}=(F_2)_{143}=-\frac12(F_2)_{211}
\\[6pt]
\phantom{\frac12\lambda_2}
=-\frac12(F_2)_{233}=-(F_2)_{314}=(F_2)_{323}=(F_2)_{332}=-(F_2)_{341},
\\[6pt]
\frac12\lambda_3=(F_2)_{214}=-(F_2)_{223}=-(F_2)_{232}=(F_2)_{241}=\frac12(F_2)_{322}
\\[6pt]
\phantom{\frac12\lambda_3}
=\frac12(F_2)_{344}=-(F_2)_{412}=-(F_2)_{421}=-(F_2)_{434}=-(F_2)_{443},
\\[6pt]
\frac12\lambda_4=-(F_2)_{114}=(F_2)_{123}=(F_2)_{132}=-(F_2)_{141}=-(F_2)_{312}
\\[6pt]
\phantom{\frac12\lambda_4}
=-(F_2)_{321}=-(F_2)_{334}=-(F_2)_{343}=\frac12(F_2)_{411}=\frac12(F_2)_{433};
\end{array}
\end{equation*}
\begin{equation*}\label{F3}
\begin{array}{l}
\frac{1}{2}\lambda_1=(F_3)_{112}=(F_3)_{121}=-(F_3)_{134}=-(F_3)_{143}=-\frac12(F_3)_{211}\\[6pt]
\phantom{\frac{1}{2}\lambda_1}
=-\frac12(F_3)_{244}=(F_3)_{413}=(F_3)_{431}=(F_3)_{424}=(F_3)_{442},\\[6pt]
\frac{1}{2}\lambda_2=\frac12(F_3)_{122}=\frac12(F_3)_{133}=-(F_3)_{212}=-(F_3)_{221}=(F_3)_{234}\\[6pt]
\phantom{\frac{1}{2}\lambda_2}
=(F_3)_{243}=-(F_3)_{313}=-(F_3)_{331}=-(F_3)_{324}=-(F_3)_{342},\\[6pt]
\frac{1}{2}\lambda_3=(F_3)_{213}=(F_3)_{231}=(F_3)_{224}=(F_3)_{242}=-(F_3)_{312}\\[6pt]
\phantom{\frac{1}{2}\lambda_3}
=-(F_3)_{321}=(F_3)_{343}=(F_3)_{334}=-\frac12(F_3)_{422}=-\frac12(F_3)_{433},\\[6pt]
\frac{1}{2}\lambda_4=-(F_3)_{113}=-(F_3)_{124}=-(F_3)_{131}=-(F_3)_{142}=\frac12(F_3)_{311}\\[6pt]
\phantom{=\frac{1}{2}\lambda_4}
=\frac12(F_3)_{344}=(F_3)_{412}=(F_3)_{421}=-(F_3)_{434}=-(F_3)_{443}.
\end{array}
\end{equation*}

Bearing in mind the discussions in the previous two subsections, this is an example of a 4-dimensional manifold with vanishing associated Nijenhuis tensors for the almost hypercomplex structure and there exist affine connections $\DDD^\al$ with totally skew-symmetric torsion $T_\al$ preserving $J_\al$ and $g$.
Using \eqref{T1D1F1}, \eqref{T1D1F23} and the components above, we compute the basic components of $T_\al$. They are determined by the following nonzero components for $\al=1,2,3$:
\[
(T_\al)_{123}=\lm_2,\quad
(T_\al)_{124}=-\lm_1,\quad
(T_\al)_{134}=\lm_4,\quad
(T_\al)_{234}=-\lm_3.
\]

Therefore the connections $\DDD^1$, $\DDD^2$ and $\DDD^3$ coincide. Then, according to \thmref{thm:HKT},  $(\LL,H,G)$ admits a unique affine connection $\DDD^*$ with totally skew-sym\-met\-ric torsion $T^*$ preserving the structure $(H,G)$ and it is determined by $T^*$ with nonzero components
\[
T^*_{123}=\lm_2,\quad
T^*_{124}=-\lm_1,\quad
T^*_{134}=\lm_4,\quad
T^*_{234}=-\lm_3.
\]

\vspace{20pt}

\begin{center}
$\divideontimes\divideontimes\divideontimes$
\end{center} 

\newpage

\addtocounter{section}{1}\setcounter{subsection}{0}\setcounter{subsubsection}{0}

\setcounter{thm}{0}\setcounter{equation}{0}

\label{par:quatK}

 \Large{

\
\\[6pt]
\bigskip

\
\\[6pt]
\bigskip

\lhead{\emph{Chapter II $|$ \S\thesection. Quaternionic K\"ahler manifolds with Her\-mit\-ian-Nor\-den
metrics}}

%\thispagestyle{empty}

%\noindent  {\Huge\bf \S\thesection. Quaternionic K\"ahler manifolds \\[12pt]
%\phantom{\S\thesection. }with Hermit\-ian-Norden metrics
%}%\\[6pt]\vskip2pt}

\noindent
\begin{tabular}{r"l}
  %\hline
\hspace{-6pt}{\Huge\bf \S\thesection.}  & {\Huge\bf Quaternionic K\"ahler manifolds} \\[12pt]
                             & {\Huge\bf with Hermit\-ian-Norden metrics}
  %\hline
\end{tabular}

\vskip 1cm

\begin{quote}
\begin{large}
In the present section, almost hypercomplex manifolds with Her\-mit\-ian-Nor\-den metrics
and more specially the corresponding quaternionic K\"ahler
manifolds are considered. Some necessary and sufficient conditions for
the studied manifolds to be isotropic hyper-K\"ahlerian and flat
are found. It is proved that the quaternionic K\"ahler manifolds
with the considered metric structure are Einstein  for dimension
at least 8. The class of the non-hyper-K\"ahler quaternionic
K\"ahler manifold of the considered type is determined.

The main results of this section are published in \cite{Man29}.
\end{large}\end{quote}

%
%\vskip 0.2in \addtocounter{subsection}{1}
%
%%\subsection{Реално векторно пространство с комплексна структура и норденова метрика}
%\noindent  {\Large\bf \thesubsection. Introduction}%\\[6pt]\vskip2pt}

\vskip 0.15in

The basic problem of this section is the existence and the geometric
characteristics of the quaternionic K\"ahler manifolds with Her\-mit\-ian-Nor\-den metrics. The
main results here is that every quaternionic K\"ahler
manifold with Her\-mit\-ian-Nor\-den metrics is Einstein for dimension at least 8 and it is not
flat hyper-K\"ahlerian only when belongs to the general class
$\W_1\oplus\W_2\oplus\W_3$ or the class $\W_1\oplus\W_3$, where
the manifold is Ricci-symmetric.

The present section is organised as follows.
%In Subsection~\thesection.1 we recall some facts about the almost
%hypercomplex manifolds with Her\-mit\-ian-Nor\-den metrics known from \cite{AlMa}, \cite{GriMan24},
%\cite{GriManDim12}, \cite{Man28}.
%
In Subsection~\thesection.1 we introduce the corresponding quaternionic
K\"ahler manifold of an almost hypercomplex manifold with Her\-mit\-ian-Nor\-den metrics. We establish that the quaternionic K\"ahler
manifolds with Her\-mit\-ian-Nor\-den metrics are Einstein for dimension $4n\geq 8$. For
comparison, it is known that the quaternionic K\"ahler manifolds
with hyper-Hermitian metric structure are Einstein for all
dimensions \cite{AlCor}.
%
%"We remark that any quaternionic K\"ahler manifold $\MM$ is an
%Einstein manifold, provided that $\dim \MM > 4$. Moreover, $\MM$ is
%irreducible (if $Ric \neq 0$) or locally hyper-K\"ahler manifold
%(if $Ric=0$) (\cite{Al}, \cite{Bes}, \cite{Ish-QK}, \cite{Sal})."
%
In Subsection~\thesection.2 we consider the location of the quaternionic
K\"ahler manifolds with Her\-mit\-ian-Nor\-den metrics in the classification of the corresponding
almost hypercomplex manifolds with respect to the covariant
derivatives of the almost complex structures. We get only one
class (except the general one) of the considered classification
where these manifolds are non-hyper-K\"ahlerian and consequently
non-flat always.
In Subsection~\thesection.3 we characterize the non-hyper-K\"ahler quaternionic K\"ahler manifolds with Her\-mit\-ian-Nor\-den metrics obtained in the previous
subsection.

\vskip 0.2in \addtocounter{subsection}{1} \setcounter{subsubsection}{0}

\noindent  {\Large\bf \thesubsection. Quaternionic K\"ahler manifolds with Her\-mit\-ian-Nor\-den metrics }%\\[6pt]\vskip2pt}

\vskip 0.15in

%\section{Quaternionic K\"ahler manifolds with Her\-mit\-ian-Nor\-den metrics }

Let us consider again only an almost hypercomplex manifold
$(\MM,H)$. The endomorphism $Q=\lm_1 J_1+\lm_2 J_2+\lm_3 J_3$,
$\lm_i\in\R$, is called a \emph{quaternionic structure} on $(\MM,H)$
with an admissible basis $H$. Let $\DDD$ be the Levi-Civita connection of a pseudo-Riemannian metric $g$ on $\MM$. A quaternionic structure with the condition $\DDD Q=0$ is called a \emph{quaternionic K\"ahler structure} on $(\MM,H)$. An almost hypercomplex manifold with
quaternionic K\"ahler structure is determined by
\begin{equation}\label{qK}
\left(\DDD_x J_\al\right)y=\om_\gm(x)J_\bt y-\om_\bt(x)J_\gm y
\end{equation}
for all cyclic permutations $(\al, \bt, \gm)$ of $(1,2,3)$, where
$\om_\al$ are local 1-forms associated to $H=(J_\al)$,
$(\al=1,2,3)$. \cite{Sal82,AlMa}

Next, we equip the quaternionic K\"ahler manifold with a structure of Hermitian-Norden metrics $G=(g,g_1,g_2,g_3)$, determined by \eqref{gJJ}
and \eqref{gJ}, and we obtain a \emph{quaternionic K\"ahler manifold
with Her\-mit\-ian-Nor\-den metrics}.

Bearing in mind \eqref{qK} and \eqref{nJ}, for a quaternionic
K\"ahler manifold with Her\-mit\-ian-Nor\-den metrics  we obtain the following form of the square
norm of $\DDD J_\al$:
\begin{equation}\label{nJ-qK}
    \nJ{\al}=4n\left\{\eb\om_\gm(\om^{\sharp}_\gm)+\eg\om_\bt(\om^{\sharp}_\bt)\right\},
\end{equation}
where $\om^{\sharp}_1$, $\om^{\sharp}_2$, $\om^{\sharp}_3$ are the corresponding
vectors of $\om_1$, $\om_2$, $\om_3$  regarding $g$, respectively. The coefficients $\varepsilon_1$, $\varepsilon_2$ and $\varepsilon_3$ are defined in \eqref{epsiloni}.

Therefore, we have immediately the following
\begin{prop}\label{prop-iqK}
A quaternionic K\"ahler manifold with Hermitian-Nor\-den metrics is an isotropic hyper-K\"ahler
manifold with Her\-mit\-ian-Nor\-den metrics if and only if  the corresponding vectors of the 1-forms $\om_1$,
$\om_2$ and $\om_3$ with respect to $g$ are isotropic vectors
regarding $g$.
\end{prop}

%%%%%%%%%%%%%%%%%%%%%%%%%%%%%%%%%%%%%%%%%%%%%%%%%%%%%%%%%%%%%%%%%%%%%%%%%%%%%%%%%%%%%

Bearing in mind \eqref{qK}, we obtain the following property of the curvature tensor $R$ of $\DDD$ for
all cyclic permutations $(\al,\bt,\gm)$ of $(1,2,3)$:
\begin{equation}\label{RJ}
      R(x,y)J_{\al} z=J_{\al} R(x,y)z-\Psi_{\bt} (x,y)J_{\gm} z+\Psi_{\gm} (x,y)J_\bt z,
\end{equation}
where
\begin{equation}\label{eta}
\Psi_{\bt}(x,y)=\dd\om_{\bt} (x,y)+\om_{\gm} (x)\om_{\al}
(y)-\om_{\al} (x)\om_{\gm} (y)
\end{equation}
are 2-forms associated to the local 1-forms $\om_1$, $\om_2$,
$\om_3$. Therefore, we have
\begin{equation}\label{RJJ}
\begin{array}{l}
    R(x,y,z,w)-\ea
    R(x,y,J_{\al}z,J_{\al}w)=\Psi_{\bt}(x,y)\g_{\bt}(z,w)\\[6pt]
    \phantom{R(x,y,z,w)-\ea
    R(x,y,J_{\al}z,J_{\al}w)=}
    +\Psi_{\gm}(x,y)\g_{\gm}(z,w).
\end{array}
\end{equation}
According to the antisymmetry of $R$ by the third and the forth
entries, we establish that $\Psi_2=\Psi_3=0$, i.e.
\begin{lem}
The local 1-forms $\om_1$, $\om_2$ and  $\om_3$, determining a
quaternionic K\"ahler manifold with Her\-mit\-ian-Nor\-den metrics, satisfy the following
identities
\begin{equation}\label{d-om}
\begin{array}{l}
    \dd\om_2(x,y)=-\om_{3} (x)\om_{1}(y)+\om_{1} (x)\om_{3} (y),\\[6pt]
    \dd\om_3(x,y)=-\om_{1} (x)\om_{2}(y)+\om_{2} (x)\om_{1} (y).\\[6pt]
\end{array}
\end{equation}
\end{lem}

Then, according to \eqref{d-om}, equations \eqref{RJJ} take the
form
\begin{gather}
    R(x,y,J_{1}z,J_{1}w)=R(x,y,z,w),\label{RJJ1}\\[6pt]
\begin{split}
    R(x,y,J_{2}z,J_{2}w)&=R(x,y,J_{3}z,J_{3}w)\\[6pt]
    %\phantom{R(x,y,J_{2}z,J_{2}w)}
    &=-R(x,y,z,w)+\Psi_{1}(x,y)\g_{1}(z,w).\label{RJJ23}
\end{split}
\end{gather}

Bearing in mind \eqref{RJJ1}, \eqref{RJJ23} and \eqref{K-kel}, we
have immediately
\begin{lem}\label{lem-Rkel}
The curvature tensor $R$ of a quaternionic K\"ahler manifold with Her\-mit\-ian-Nor\-den metrics is
of K\"ahler-type if and only if  $\Psi_1=0$, i.e. the following condition is
valid
\begin{equation}\label{om1}
\dd\om_1(x,y)=-\om_{2} (x)\om_{3} (y)+\om_{3} (x)\om_{2} (y).
\end{equation}
\end{lem}
According to \lemref{lem-Rkel} and \thmref{thm-K=0}, we have
\begin{prop}\label{prop-K=0}
The necessary and sufficient condition an arbitrary quaternionic
K\"ahler manifold with Her\-mit\-ian-Nor\-den metrics to be flat is condition \eqref{om1}.
\end{prop}

\begin{lem}\label{l-rho-eta}
The Ricci tensor $\rho$ and the 2-form $\Psi_1$, defined by \eqref{eta},
have the following relation on any quaternionic K\"ahler
manifold with Her\-mit\-ian-Nor\-den metrics:
\begin{equation}\label{rho-eta}
    \rho(x,y)=n\Psi_1(J_1x, y).
\end{equation}
\end{lem}
\begin{proof}
From \eqref{RJJ23} for $z \rightarrow e_i$, $w \rightarrow J_1e_j$
by contraction with $g^{ij}$ we have
\begin{equation}\label{K123a}
\begin{array}{l}
    -g^{ij}R(x,y,J_{2}e_i,J_{3}e_j)=g^{ij}R(x,y,J_{3}e_i,J_{2}e_j) \\[6pt]
\phantom{-g^{ij}R(x,y,J_{2}e_i,J_{3}e_j)}
    =-g^{ij}R(x,y,e_i,J_1e_j)+4n\Psi_{1}(x,y).
\end{array}
\end{equation}
Bearing in mind the antisymmetry on the second pair arguments of
$R$ and $J_1=J_2J_3$, we get
\begin{equation*}
\begin{array}{l}
    -g^{ij}R(x,y,J_{2}e_i,J_{3}e_j)=g^{ij}R(x,y,J_{3}e_i,J_{2}e_j)\\[6pt]
\phantom{-g^{ij}R(x,y,J_{2}e_i,J_{3}e_j)}
    =g^{ij}R(x,y,e_i,J_1e_j)
\end{array}
\end{equation*}
and therefore from \eqref{K123a} we have
\begin{equation}\label{a-eta}
    g^{ij}R(x,y,e_i,J_1e_j)=2n\Psi_1(x,y).
\end{equation}

After that, from \eqref{a-eta}, applying the properties of $R$, \eqref{gJJ} for $\al=1$ and \eqref{RJJ1}, we
obtain consequently
\[
\begin{split}
2n\Psi_1(x,y)&= g^{ij}R(x,y,e_i,J_1e_j)\\[6pt]
&=g^{ij}\{-R(x,e_i,J_1e_j,y)-R(x,J_1e_j,y,e_i)\}\\[6pt]
&= g^{ij}R(x,e_i,y,J_1e_j)+g^{ij}R(x,e_j,y,J_1e_i)\\[6pt]
&=2g^{ij}R(x,e_i,y,J_1e_j)= -2g^{ij}R(e_i,x,y,J_1e_j)\\[6pt]
&=2g^{ij}R(e_i,x,J_1y,e_j)=2\rho(x,J_1y),\\[6pt]
\end{split}
\]
i.e. we get
\begin{equation}\label{eta-rho}
    \Psi_1(x,y)=\frac{1}{n}\rho(x,J_1y).
\end{equation}
Because of the symmetry of $\rho$ and the antisymmetry of $\Psi_1$
we have the property
\begin{equation}\label{etaJ1}
\Psi_1(x,J_1y)=-\Psi_1(J_1x,y)
\end{equation}
and therefore
\begin{equation}\label{rJJ1}
    \rho(J_1x,J_1y)=\rho(x,y).
\end{equation}
Hence, from \eqref{eta-rho}, \eqref{etaJ1} and \eqref{rJJ1}, we
obtain \eqref{rho-eta}.
\end{proof}

\begin{prop}\label{prop-r=0}
    A quaternionic K\"ahler manifold with Hermitian-Nor\-den metrics  is
    Ricci-flat if and only if  it is flat.
\end{prop}
\begin{proof}
Using \lemref{l-rho-eta}, property \eqref{RJJ23} takes the form
\begin{equation}\label{RJJ23r}
\begin{split}
    &R(x,y,J_{2}z,J_{2}w)=R(x,y,J_{3}z,J_{3}w) \\[6pt]
    &\phantom{R(x,y,J_{2}z,J_{2}w)}
    =-R(x,y,z,w)
    -\frac{1}{n}\rho(J_1x,y)g(J_{1}z,w).
\end{split}
\end{equation}
Then, according to \eqref{RJJ23r}, \eqref{K-kel} and
\thmref{th-0}, we obtain the equivalence in the statement.
\end{proof}

\begin{thm}\label{thm-Ein}
Quaternionic K\"ahler manifolds with Her\-mit\-ian-Nor\-den metrics  are
Einstein for dimension $4n\geq 8$.
\end{thm}
\begin{proof}
By virtue of \eqref{RJJ1}, \eqref{RJJ23r}  and \eqref{rJJ1} we
obtain the following properties
\begin{gather}\label{RJJJJ1}
    R(J_{1}x,J_{1}y,J_{1}z,J_{1}w)=R(x,y,z,w), \\[6pt]
\begin{split}\label{RJJJJ23}
    R(J_{2}x,J_{2}y,J_{2}z,J_{2}w)&=R(J_{3}x,J_{3}y,J_{3}z,J_{3}w) \\[6pt]
    &=R(x,y,z,w) \\[6pt]
    &\phantom{=\,}
    -\frac{1}{n}g(x,J_1y)\rho(J_1z,w) \\[6pt]
    &\phantom{=\,}
    +\frac{1}{n}\rho(J_2x,J_3y)g(J_1z,w).
\end{split}
\end{gather}

Hence, for the Ricci tensor we have \eqref{rJJ1} and
\begin{equation}\label{rho4J23}
\begin{split}
    (n^2-1)\rho(J_{2}y,J_{2}z)&=(n^2-1)\rho(J_{3}y,J_{3}z)\\[6pt]
                                    &=-(n^2-1)\rho(y,z).
\end{split}
\end{equation}
Then for $n>1$ the Ricci tensor is hybrid with respect to $J_2$
and $J_3$, i.e.
\begin{equation*}\label{rhoJ2J3}
    \rho(J_{2}y,J_{2}z)=\rho(J_{3}y,J_{3}z)
    =-\rho(y,z).
\end{equation*}

Conditions \eqref{RJJ1}, \eqref{RJJ23} and \eqref{rho-eta} imply
for $n>1$ the following
\begin{equation*}\label{4Rrho}
    A(x,z)=-\frac{2}{n}\rho(x,x)g(z,z)=-\frac{2}{n}g(x,x)\rho(z,z),
\end{equation*}
where
\begin{equation*}\label{A}
\begin{array}{l}
A(x,z)=R(x,J_{1}x,z,J_{1}z)-R(x,J_{1}x,J_{2}z,J_{3}z)\\[6pt]
    \phantom{A(x,z)=}
        -R(J_{2}x,J_{3}x,z,J_{1}z)+R(J_{2}x,J_{3}x,J_{2}z,J_{3}z)\\[6pt]
\end{array}
\end{equation*}

Then for arbitrary non-isotropic vectors we have $\rho=\lm g$,
$\lm\in\R$.
\end{proof}

By \thmref{thm-Ein}, identity \eqref{RJJ23r} implies the following
corollary for $n\geq 2$:
\begin{equation}\label{RJJ23t}
\begin{split}
    &R(x,y,J_{2}z,J_{2}w)=R(x,y,J_{3}z,J_{3}w) \\[6pt]
    &\phantom{R(x,y,J_{2}z,J_{2}w)}
    =-R(x,y,z,w)-\frac{\tau}{4n^2}g(J_1x,y)g(J_{1}z,w).
\end{split}
\end{equation}

Therefore from \eqref{RJJ23t}, using \eqref{K-kel} and
\thmref{th-0}, we obtain the following
\begin{prop}\label{prop-t=0}
    A quaternionic K\"ahler manifold with Her\-mit\-ian-Nor\-den metrics  of dimension $4n\geq 8$ is
    scalar flat if and only if  it is flat.
\end{prop}

\begin{prop}\label{prop-d-om}
    A quaternionic K\"ahler manifold with Her\-mit\-ian-Nor\-den metrics  of dimension $4n\geq 8$
    is determined by the local 1-forms satisfying conditions
    \eqref{d-om} and
\begin{equation*}\label{om1=}
\dd\om_1(x,y)=-\om_{2} (x)\om_{3} (y)+\om_{3} (x)\om_{2}
(y)-\frac{\tau}{4n^2}g(J_1x,y).%\\[-25pt]
\end{equation*}
\end{prop}

\vskip 0.2in \addtocounter{subsection}{1} \setcounter{subsubsection}{0}

\noindent  {\Large\bf \thesubsection. Quaternionic K\"ahler manifolds with Her\-mit\-ian-Nor\-den met\-rics in a classification of almost hypercomplex
manifolds with Her\-mit\-ian-Nor\-den metrics}%\\[6pt]\vskip2pt}

\vskip 0.15in

%\section{Quaternionic K\"ahler manifolds with Her\-mit\-ian-Nor\-den metrics in a classification of almost hypercomplex
%manifolds with Her\-mit\-ian-Nor\-den metrics}

Firstly, let us consider the case when $H$ is (integrable)
hypercomplex structure, i.e. when $[J_\al,J_\al]$ vanishes for each
$\al=1,2,3$.

Taking into account \eqref{NJ-nabli} and \eqref{qK}, for the quaternionic K\"ahler
 manifolds we have
%\begin{subequations}
\begin{equation}\label{NN*}
\begin{array}{l}
[J_\al,J_\al](x,y)=
-\left[\om_\gm(x)+\om_\bt(J_\al x)\right]J_\gm y\\[6pt]
\phantom{[J_\al,J_\al](x,y)=}
-\left[\om_\bt(x)-\om_\gm(J_\al x)\right]J_\bt y\\[6pt]
\phantom{[J_\al,J_\al](x,y)=}
+\left[\om_\gm(y)+\om_\bt(J_\al y)\right]J_\gm x\\[6pt]
\phantom{[J_\al,J_\al](x,y)=}
+\left[\om_\bt(y)-\om_\gm(J_\al y)\right]J_\bt x,\\[6pt]
%\end{array}
%\end{equation}
%\begin{equation}
%\begin{array}{l}
\{J_\al,J_\al\}(x,y)=
-\left[\om_\gm(x)+\om_\bt(J_\al x)\right]J_\gm y\\[6pt]
\phantom{\{J_\al,J_\al\}(x,y)=}
-\left[\om_\bt(x)-\om_\gm(J_\al x)\right]J_\bt y\\[6pt]
\phantom{\{J_\al,J_\al\}(x,y)=}
-\left[\om_\gm(y)+\om_\bt(J_\al y)\right]J_\gm x\\[6pt]
\phantom{\{J_\al,J_\al\}(x,y)=}
-\left[\om_\bt(y)-\om_\gm(J_\al y)\right]J_\bt x.
\end{array}
\end{equation}
%\end{subequations}
The latter equations imply immediately the next two lemmas.
\begin{lem}\label{lm-N}
The tensors $[J_\al,J_\al]$ and $\{J_\al,J_\al\}$  ($\al=1,2,3$) vanish if and only if the following equality is valid for any fixed cyclic permutation
$(\al, \bt, \gm)$ of $(1,2,3)$:
\[
\om_\gm=-\om_\bt\circ J_\al.
\]
\end{lem}
\begin{lem}\label{lm-NN}
The tensors $[J_\al,J_\al]$ and $\{J_\al,J_\al\}$ ($\al=1,2,3$) vanish if and only if the following equalities hold for all cyclic permutations $(\al, \bt, \gm)$ of $(1,2,3)$:
\begin{equation}\label{qK-usl}
\om_\al=\om_\bt\circ J_\gm=-\om_\gm\circ J_\bt.
\end{equation}
\end{lem}

Now, according to \eqref{qK} and \eqref{gJ}, the fundamental
tensors and their corresponding Lee forms of the derived
quaternionic K\"ahler manifold with Hermitian-Norden metrics, defined by \eqref{F'-al}
 and \eqref{theta-al},  have the form
\begin{gather}
    F_\al(x,y,z)=\om_\gm(x)g(J_\bt
y,z)-\om_\bt(x)g(J_\gm y,z),\label{qK-F}\\[6pt]
    \ta_\al(z)=-\eb\om_\gm(J_\bt z)+\eg\om_\bt(J_\gm z).\label{qK-theta}
\end{gather}

\begin{prop}\label{prop-int}
If a quaternionic K\"ahler manifold with Her\-mit\-ian-Nor\-den metrics $(\MM,H,G)$ is integrable,
then it is a hyper-K\"ahler manifold with Her\-mit\-ian-Nor\-den metrics, i.e. the following implication is valid
\[
(\MM,H,G)\in\left(\W_3\oplus\W_4\right)(J_1)\cap\left(\W_1\oplus\W_2\right)(J_2)
\cap\left(\W_1\oplus\W_2\right)(J_3)
\]
\[
 \Rightarrow \quad
(\MM,H,G)\in\KK.
\]
\end{prop}
\begin{proof}
Let $(\MM,H,G)$ be an integrable hypercomplex manifold with Her\-mit\-ian-Nor\-den metrics, i.e.
$(\MM,H,G)$ belongs to the class $\W_3\oplus\W_4$ with respect to
$J_1$ in \eqref{cl-H} and $(\MM,H,G)$ is an element of
 $\W_1\oplus\W_2$ regarding $J_2$ and $J_3$, according to
\eqref{cl-N}.

Therefore $[J_1,J_1]=[J_2,J_2]=[J_3,J_3]=0$ hold and then, according to
\lemref{lm-NN}, conditions \eqref{qK-usl} are valid. Hence,
according to $\ep_1+\ep_2+\ep_3=-1$ and $\ep_1\ep_2\ep_3=1$, relation
\eqref{qK-theta} takes the form
\[
\ta_\al=-(1+\ea)\om_\al,
\]
which implies
\[
\ta_1=-2\om_1,\qquad \ta_2=\ta_3=0.
\]

On the other hand, $[J_\al,J_\al]=\{J_\al,J_\al\}=0$  and \eqref{NJ-nabli}  imply
\[
\left(\DDD_x J_\al\right)y=\left(\DDD_{J_\al x} J_\al\right)J_\al y
\]
and finally the fact that the manifold is hyper-K\"ahlerian with
Hermitian-Norden metrics.
\end{proof}
\begin{prop}\label{prop-23}
If an almost hypercomplex manifold with Her\-mit\-ian-Nor\-den metrics $(\MM,H,G)$ with vanishing Lee forms $\ta_2$ and $\ta_3$ is quaternionic K\"ahlerian, then it
is a hyper-K\"ahler manifold with Her\-mit\-ian-Nor\-den metrics, i.e.
the following implication is valid
\[
(\MM,H,G)\in%
\left(\W_2\oplus\W_3\right)(J_2)
\cap\left(\W_2\oplus\W_3\right)(J_3)\quad \Rightarrow \quad
(\MM,H,G)\in\KK.
\]
\end{prop}
\begin{proof}
Since $\ta_2=\ta_3=0$, we have $[J_2,J_2]=[J_3,J_3]=0$, because of
\eqref{qK-theta} and \eqref{NN*}. Consequently, $[J_1,J_1]$ vanishes,
too. Then, according to \propref{prop-int} and conditions
\eqref{cl-N}, the considered manifold belongs to the class $\KK$.
\end{proof}

Let us remark for $\al=2$ or $3$, using \eqref{cl-N}, that an almost complex manifold
with Norden metric belongs to $\W_1\oplus\W_3$ regarding $J_\al$
if and only if  the following property holds
\begin{equation}\label{W1W3}
\mathop{\s}\limits_{x,y,z}F_\al(x,y,z)=\frac{1}{2n}\mathop{\s}\limits_{x,y,z}\left\{
g(x,y)\ta_\al(z)+g(J_\al x,y)\ta_\al(J_\al z)\right\}.
\end{equation}

\begin{prop}\label{prop-13}
Let $(\MM,H,G)$ be an almost hypercomplex manifold with Hermitian-Norden metrics belonging to
the class $\W_1\oplus\W_3$ with respect to $J_2$ and $J_3$. If
$(\MM,H,G)$  is quaternionic K\"ahlerian, then it is a K\"ahler
manifold with
respect to $J_1$, i.e. the following implication is valid%
\[\begin{array}{c}
(\MM,H,G)\in\left(\W_1\oplus\W_3\right)(J_2)\cap\left(\W_1\oplus\W_3\right)(J_3)\quad\\[6pt]
\Rightarrow\quad (\MM,H,G)\in\W_0(J_1).
\end{array}
\]
Moreover, we have
\begin{equation}\label{F123}
    \begin{array}{ll}
      \left(\DDD_xJ_1\right)y=0, \\[6pt]
      \left(\DDD_xJ_2\right)y=\om_1(x)J_3 y, \quad\quad & \om_1(x)=-\ta_2(J_3 x),\\[6pt]
      \left(\DDD_xJ_3\right)y=-\om_1(x)J_2 y, \quad\quad & \om_1(x)=\ta_3(J_2 x).\\[6pt]
    \end{array}
\end{equation}
\end{prop}
\begin{proof}
From \eqref{W1W3} for $\al=2$ and $3$ we obtain
\[
\ta_2=\om_1\circ J_3,\qquad \ta_3=-\om_1\circ J_2.
\]
Then, according to \eqref{qK-theta}, we get
\begin{equation}\label{om23}
\om_2=\om_3=0
\end{equation}
and therefore we obtain \eqref{F123}.
\end{proof}

From \propref{prop-23} and \propref{prop-13} we have directly
\begin{cor}\label{cor-W3}
Let $(\MM,H,G)$ be an almost hypercomplex manifold with Hermitian-Norden metrics, belonging to
the class $\W_3$ with respect to $J_2$ and $J_3$. If $(\MM,H,G)$ is
quaternionic K\"ahlerian, then it is a hyper-K\"ahler manifold with Her\-mit\-ian-Nor\-den metrics,
i.e. the following implication is valid
\[
(\MM,H,G)\in\W_3(J_2)\cap\W_3(J_3)\quad \Rightarrow\quad
(\MM,H,G)\in\KK.
\]
%\\[-44pt]\phantom{aaaaa}

\end{cor}

Bearing in mind Propositions \ref{prop-int}--\ref{prop-13},
\corref{cor-W3} and \thmref{thm-K=0}, we give the following
\begin{thm}
Let a quaternionic K\"ahler manifold with Her\-mit\-ian-Nor\-den metrics $(\MM,H,G)$ be in some of
the classes $\W_1\oplus\W_2$ (and in particular $\W_0$, $\W_1$ and
$\W_2$), $\W_2\oplus\W_3$ and $\W_3$ with respect to both of the
structures $J_2$ and $J_3$. Then $(\MM,H,G)$ is a flat
hyper-K\"ahler manifold with Her\-mit\-ian-Nor\-den metrics. The unique class (except the class without conditions for $\DDD J_2$ and $\DDD J_3$), where $(\MM,H,G)$ is not
flat hy\-per-K\"ahlerian, is $\W_1\oplus\W_3$ and its manifolds
are determined by \eqref{F123}. %\phantom{QED}
\end{thm}

\vskip 0.2in \addtocounter{subsection}{1} \setcounter{subsubsection}{0}

\noindent  {\Large\bf \thesubsection. Non-hyper-K\"ahler quaternionic K\"ahler manifolds with \\
Her\-mit\-ian-Nor\-den metrics}%\\[6pt]\vskip2pt}

\vskip 0.15in

%\section{Non-hyper-K\"ahler quaternionic K\"ahler manifolds with Her\-mit\-ian-Nor\-den metrics}

In this subsection we characterize the manifold satisfying the
conditions of \propref{prop-13}. It is a non-hyper-K\"ahler
quaternionic K\"ahler manifold with Her\-mit\-ian-Nor\-den metrics.

We apply \eqref{om23} to \eqref{nJ-qK} and obtain the square norms
of the non-zero quantities $\DDD J_2$ and $\DDD J_3$ in the considered
case as follows:
\[
\nJ{2}=\nJ{3}=-4n\,\om_1(\om^{\sharp}_1),
\]
where $\om^{\sharp}_1$ is the corresponding vector to $\omega_1$ with
respect to $g$.
\begin{cor}
Let $(\MM,H,G)$ be a quaternionic K\"ahler manifold with Her\-mit\-ian-Nor\-den metrics, determined
by a local 1-form $\om_1$ in \eqref{F123}. It is an isotropic
hyper-K\"ahler manifold with Her\-mit\-ian-Nor\-den metrics if and only if  $\om^{\sharp}_1$
is an isotropic vector regarding
$g$.
\end{cor}

Using \eqref{om23} and \propref{prop-d-om}, for the considered
manifolds here we have
\begin{prop}\label{prop-d-om=}
    Let $\om_1$ be the local 1-form of a quaternionic K\"ahler manifold with Her\-mit\-ian-Nor\-den metrics  of dimension $4n\geq 8$,
    determined by \eqref{F123}. Then $\om_1$ satisfies the following condition
\begin{equation}\label{om1==}
\dd\om_1(x,y)=-\frac{\tau}{4n^2}\g_1(x,y).
\end{equation}
%\\[6pt][-44pt]\phantom{aaaaaaaaa}
\end{prop}

According to \eqref{F123} we have $F_1(x,y,z)=\bigl(\DDD_x
\g_1\bigr) \left( y,z \right)=0$ and then we obtain $\dd
\g_1(x,y,z)=\mathop{\s}_{x,y,z}\bigl\{F_1(x,y,z)\bigr\}=0$. Hence we establish that $\tau=\const$, using \eqref{om1==}, i.e. $(\MM,H,G)$
determined by \eqref{F123} has a constant scalar curvature. Then,
bearing in mind \thmref{thm-Ein}, we get the following
\begin{prop}\label{prop-Ric-sym}
    Quaternionic K\"ahler manifolds with Hermitian-\allowbreak{}Nor\-den metrics
    determined by \eqref{F123} for dimension $4n\geq 8$ are Ricci-sym\-met\-ric, i.e. $\DDD\rho=0$.
\end{prop}

As in \propref{prop-K=0} for an arbitrary quaternionic K\"ahler
manifold with Hermitian-Norden metrics, in the following proposition we give a necessary and
sufficient condition for the considered manifold in this subsection to be
flat.

\begin{prop}\label{prop-om1}
Let $(\MM,H,G)$ be a quaternionic K\"ahler manifold with Hermitian-Norden metrics, determined
by a non-zero local 1-form $\om_1$ in \eqref{F123}. Then $(\MM,H,G)$
is flat non-hyper-K\"ahlerian if and only if  $\om_1$ is closed.
\end{prop}
\begin{proof}
Since $\MM$ is a K\"ahler manifold with respect to $J_1$, then $R$
is of K\"ahler-type with respect to $J_1$.

Bearing in mind \eqref{om23} and \eqref{eta}, we have that
$\Psi_1=\dd\om_1$ in the considered case. Then identity
\eqref{RJJ23} takes the following form
\begin{equation*}\label{RR}
\begin{array}{l}
    R(x,y,J_2 z,J_2 w)%\\[6pt]
    =R(x,y,J_3 z,J_3 w)\\[6pt]
    \phantom{R(x,y,J_2 z,J_2 w)}
    =-R(x,y,z,w)+d\om_1(x,y)g(J_1
    z,w).
\end{array}
\end{equation*}
It is clear that $R$ is a K\"ahler-type tensor with respect to
$H=(J_\al)$ if and only if  $\om_1$ is closed. Hence, according to
\thmref{th-0}, we obtain the statement.
\end{proof}

\begin{cor}
Let $(\MM,H,G)$ be a quaternionic K\"ahler manifold with Her\-mit\-ian-Nor\-den metrics, determined
by a non-zero local 1-form $\om_1$ in \eqref{F123}. Then $(\MM,H,G)$
is flat non-hyper-K\"ahlerian  if and only if  the following identity is valid
for $\al=2$ or $\al=3$:
\[
\D\ta_\al(x,y)+\D\ta_\al(J_1 x,J_1 y)-\D\ta_\al(J_2 x,J_2
y)-\D\ta_\al(J_3 x,J_3 y)=0.
\]
\end{cor}
\begin{proof}
It follows directly from \propref{prop-om1} and the relations
\[
\om_1=-\ta_2\circ J_3=\ta_3\circ J_2
\]
in \eqref{F123}.
\end{proof}

%%%%%%%%%%%%%%%%%%%%%%%%%%%%%%%%%%%%%%%%%%%%%%%%%%%%%%%%%%%%%%%%%%%%%%%%%%%%%%%%%%%%%%

\vspace{20pt}

\begin{center}
$\divideontimes\divideontimes\divideontimes$
\end{center}

%%%%%%%%%%%%%%%%%
%}% затваря \Large
%%%%%%%%%%%%%%%%%

%
%%%\include{Man-chapIV}
%
%%
%\include{Man-53}
\newpage

\addtocounter{section}{1}\setcounter{subsection}{0}\setcounter{subsubsection}{0}

\setcounter{thm}{0}\setcounter{dfn}{0}\setcounter{equation}{0}

\label{par:4n+3}

 \Large{

\
\\[6pt]
\bigskip

\
\\[6pt]
\bigskip

\lhead{\emph{Chapter II $|$ \S\thesection. {Manifolds with almost contact 3-structure and metrics of Hermitian-Norden type% \ldots
}}}
%\thispagestyle{empty}

%\noindent  {\Huge\bf \S\thesection. Manifolds with almost contact \\[12pt]
%\phantom{\S\thesection. }3-structure and metrics of \\[12pt]
%\phantom{\S\thesection. }Hermitian-Norden type%
%}%\\[6pt]\vskip2pt

\noindent
\begin{tabular}{r"l}
  %\hline
\hspace{-6pt}{\Huge\bf \S\thesection.}  & {\Huge\bf Manifolds with almost contact} \\[12pt]
                             & {\Huge\bf 3-structure and metrics of} \\[12pt]
                             & {\Huge\bf Hermitian-Norden type}
  %\hline
\end{tabular}

\vskip 1cm

\begin{quote}
\begin{large}
In the present section, it is introduced a differentiable manifold with almost contact 3-struc\-ture which consists of an almost contact metric structure and two almost contact B-metric structures. The corresponding classifications are discussed. The product of this manifold and a real line is an almost hypercomplex manifold with Hermitian-Norden metrics. It is proved that the in\-tro\-duced manifold of cosymplectic type is flat. Some examples of the studied man\-i\-folds are given.

The main results of this section are published in \cite{Man53}.
\end{large}
\end{quote}

%
%\vskip 0.2in \addtocounter{subsection}{1}
%
%\noindent  {\Large\bf \thesubsection. Introduction}

\vskip 0.15in

%\section{Introduction}\label{intro}

Our goal here is to consider a $(4n+3)$-dimensional manifold with almost contact 3-structure and to introduce a pseudo-Riemannian metric on it having another kind of compatibility with the triad of almost contact structures. The product of this manifold of new type and a real line is a $(4n+4)$-dimensional manifold which admits an almost hypercomplex structure with Hermitian-Norden metrics. %and a Hermitian-Norden metric structure.
%In our case, one almost complex structure (respectively, the other two almost complex structures) of $H$ acts as an isometry (respectively, act as anti-isometries) with respect to the pseudo-Riemannian metric $g$ of neutral signature in each tangent fibre.
%The metric $g$ is Hermitian with respect to the former structure of $H$ and $g$ is a Norden %(also known as an anti-Hermitian) metric regarding the latter structures of $H$.
%This structure is called an \emph{almost hypercomplex structure with Hermitian-Norden metrics } and it is studied in \cite{GriManDim12,GriMan24,Man28,ManGri32}, etc.

The purpose of this development is to launch a study of the manifolds with almost contact 3-structure and metrics of a Hermitian-Norden type.

\vskip 0.2in \addtocounter{subsection}{1} \setcounter{subsubsection}{0}

\noindent  {\Large\bf \thesubsection. Almost contact metric manifolds}%\\[6pt]\vskip2pt}

\vskip 0.15in
%\section{Manifolds with almost contact 3-structure and metrics of Hermitian-Norden type }\label{sec:1}

%\section{Almost contact metric manifolds}

Let $\MM$ be an odd-dimensional smooth manifold which is compatible
with an almost contact structure $(\f_{1},\xi_{1},\eta_{1})$, i.e. $\f_1$ is an endomorphism
of the tangent bundle, $\xi_1$ is a Reeb vector field and $\eta_1$ is its dual contact 1-form
satisfying the identities \eqref{str1}, i.e.
\begin{equation}\label{str-al=1}
\begin{array}{c}
(\f_1)^2 = -I + \xi_1\otimes\eta_1,\quad
\f_1\xi_1 = o,\quad
\eta_1\circ\f_1=0,\quad \eta_1(\xi_1)=1,
\end{array}
\end{equation}
where $I$ is the identity in the Lie algebra $\X(\MM)$ and $o$ is the zero element of $\X(\MM)$.
Moreover, let $g$ be a pseudo-Riemannian metric on $\MM$ which is compatible
with $(\f_{1},\xi_{1},\eta_{1})$ as follows:
\begin{equation}\label{A-met}
\begin{array}{c}
  g(\xi_{1},\xi_{1})=-\ep, \quad \eta_{1}(x)=-\ep g(\xi_{1},x),\\[6pt]
  g(\f_{1} x,\f_{1} y)=g(x,y)+\ep\eta_{1}(x)\eta_{1}(y),
\end{array}
\end{equation}
where  $\ep=+1$ or $\ep=-1$. Then $(\f_{1},\xi_{1},\eta_{1},g)$ is called an \emph{almost contact metric structure} on $\MM$.

%\begin{remark}\label{rem:A}
Usually, one can assume that $\ep=-1$ in \eqref{A-met} without loss of generality. This is conditioned since if we put
\[
\overline\f_{1}=\f_{1}, \quad \overline\xi_{1}=-\xi_{1}, \quad \overline\eta_{1}=-\eta_{1}, \quad \overline g=-g,
\]
then $(\overline\f_{1},\overline\xi_{1},\overline\eta_{1},\overline g)$ for $\overline\ep=-\ep$ is also an almost contact metric structure on $\MM$ \cite{Taka}.
%\end{remark}
Here, we pay attention of the case when $\ep=+1$ which is in relation with our topic. We call the corresponding $(\f_{1},\xi_{1},\eta_{1},g)$-structure %(for the sake of convenience) also
an \emph{almost contact metric structure} and $g$ --- a \emph{compatible metric} on $\MM$.

Since $g$ is a Hermitian metric with respect to the almost complex structure $\f_1\vert_{\HC_1}$ on the contact distribution $\HC_1=\ker(\eta_1)$, any metric with properties \eqref{A-met} can be considered as an odd-dimensional counterpart of the corresponding pseudo-Riemannian Hermitian metric, or this compatible metric is a pseudo-Riemannian metric of Hermitian type on an odd-dimensional differentiable manifold.

A classification of the almost contact metric manifolds is given by V.~Ale\-xiev and G.~Ganchev in \cite{AlGa}. There, it is considered the vector space of the tensors of type $(0,3)$ defined by $F_1(x,y,z)=g\left(\left(\n_x\f_1\right)y,z\right)$, where $\n$ is the Levi-Civita connection generated by $g$.
They have the following basic properties
\begin{equation}\label{F1-prop}
\begin{array}{l}
  F_1(x,y,z)=- F_1(x,z,y)\\[6pt]
  \phantom{F_1(x,y,z)}
  =- F_1(x,\f_1y,\f_1z)+F_1(x,\xi_1,z)\,\eta_1(y)\\[6pt]
  \phantom{F_1(x,y,z)=- F_1(x,\f_1y,\f_1z)}
  +F_1(x,y,\xi_1)\,\eta_1(z).
\end{array}
\end{equation}
Bearing in mind \eqref{A-met}, we establish that the covariant derivatives of the structure tensors with respect to $\n$ are related by
\begin{equation}\label{Fetaxi1}
  F_1(x,\f_1y,\xi_1)=- \left(\n_x\eta_1\right)(y)=g\left(\n_x\xi_1,y\right).
\end{equation}

In \cite{AlGa}, this vector space is decomposed in 12 orthogonal and invariant subspaces with respect to the action of the structure group $\mathcal{U}\left(m\right)\times 1$, $m=\frac12(\dim\MM-1)$, where $\mathcal{U}$ is the unitary group.
In such a way, it is obtained a classification of 12 basic classes with respect to $F_1$, which we denote by $\PP_i$ $(i=1,2,\dots,12)$ in the present work. Bearing in mind the above remarks we can use the same classification for almost contact metric manifolds. The basic classes for $\dim{\MM}=4n+3$ can be determined as follows:
%\begin{subequations}
\begin{equation}\label{cl-A}
\begin{array}{rl}
  \PP_1:& F_1(x,y,z)=\eta_1(x)\left\{\eta_1(y)\om_1(z)-\eta_1(z)\om_1(y)\right\};\\[6pt]
  \PP_2:& F_1(x,y,z)=\dfrac{\ta_1(\xi_1)}{2(2n+1)}\left\{g(x,y)\eta_1(z)-g(x,z)\eta_1(y)\right\};\\[6pt]
  \PP_3:& F_1(x,y,z)=\dfrac{\ta^*_1(\xi_1)}{2(2n+1)}\left\{g(\f_1x,y)\eta_1(z)\right.\\
  &
 \phantom{F_1(x,y,z)=\dfrac{\ta^*_1(\xi_1)}{2(2n+1)}\left\{\right.}\left. -g(\f_1x,z)\eta_1(y)\right\};\\[6pt]
  \PP_{4}:& F_1(x,y,z)=F_1(x,y,\xi_1)\eta_1(z)-F_1(x,z,\xi_1)\eta_1(y),\quad \\[6pt]
         & F_1(x,y,\xi_1)= F_1(y,x,\xi_1)=F_1(\f_1 x,\f_1 y,\xi_1),\\[6pt]
         & \ta_1=0; \\[6pt]
  \PP_{5}:& F_1(x,y,z)=F_1(x,y,\xi_1)\eta_1(z)-F_1(x,z,\xi_1)\eta_1(y),\quad \\[6pt]
         & F_1(x,y,\xi_1)=-F_1(y,x,\xi_1)=F_1(\f_1 x,\f_1 y,\xi_1),\quad \\[6pt]
         & \ta_1^*=0; \\[6pt]
  \PP_{6}:& F_1(x,y,z)=F_1(x,y,\xi_1)\eta_1(z)-F_1(x,z,\xi_1)\eta_1(y),\\[6pt]
         & F_1(x,y,\xi_1)=F_1(y,x,\xi_1)=-F_1(\f_1 x,\f_1 y,\xi_1); \\[6pt]
  \PP_{7}:& F_1(x,y,z)=F_1(x,y,\xi_1)\eta_1(z)-F_1(x,z,\xi_1)\eta_1(y),\\[6pt]
         & F_1(x,y,\xi_1)=-F_1(y,x,\xi_1)=-F_1(\f_1 x,\f_1 y,\xi_1); \\[6pt]
  \PP_{8}:& F_1(x,y,z)=-\eta_1(x)F_1(\xi_1,\f_1 y,\f_1 z);\\[6pt]
  \PP_{9}:& F_1(x,y,z)=\dfrac{1}{4n}\bigl\{g(\f_1 x,y)\ta(\f_1 z)-g(\f_1 x,z)\ta(\f_1 y)\\[6pt]
         &\phantom{F_1(x,y,z)=\dfrac{1}{4n}\bigl\{}
-g(\f_1 x,\f_1 y)\ta(\f_1^2 z)\\[6pt]
         &\phantom{F_1(x,y,z)=\dfrac{1}{4n}\bigl\{}
+g(\f_1 x,\f_1 z)\ta(\f_1^2 y)\bigr\};\\[6pt]
%\end{array}
%\end{equation}
%\begin{equation}
%\begin{array}{rl}
%
  \PP_{10}:& F_1(\xi_1,y,z)=F_1(x,y,\xi_1)=0,\quad\\[6pt]
          & F_1(x,y,z)=F_1(\f_1 x,\f_1 y,z),\quad \ta_1=0;\\[6pt]
  \PP_{11}:& F_1(\xi_1,y,z)=F_1(x,y,\xi_1)=0,\quad\\[6pt]
           & F_1(x,y,z)=-F_1(y,x,z);\\[6pt]
  \PP_{12}:& F_1(\xi_1,y,z)=F_1(x,y,\xi_1)=0,\quad\\[6pt]
              &F_1(x,y,z)+F_1(y,z,x)+F_1(z,x,y)=0.
\end{array}
\end{equation}
%\end{subequations}

The class $\PP_0$ of cosymplectic metric manifolds is determined by $F_1=0$ and it is contained in any other class $\PP_i$.
There are $2^{12}$ classes at all.

Besides definitional conditions of the basic classes $\PP_i$, $i\in\{1,2,\dots,12\}$, in [4], there are given the corresponding component $P^i$ of $F_1$ for every class $\PP_i$. An almost contact metric manifold
$(\MM,\f_{1},\xi_{1},\eta_{1},g)$ belongs to some of the basic classes $\PP_i$ or their direct sum  $\PP_i\oplus\PP_j\oplus\cdots$ for $i,j\in\{1,2,\dots,12\}$, $i\neq j$, if and only if the fundamental
tensor $F_1$ on the manifold has the following form $F_1=P^i$ or $F_1=P^i+P^j+\dots$, respectively.

Some of the classes of almost contact metric manifolds are discovered before the complete classification and they are known by special names. For example, $\PP_6\oplus\PP_{11}$ is the class of almost cosymplectic metric manifolds, $\PP_2\oplus\PP_4$ is the class of quasi-Sasakian metric manifolds, $\PP_3\oplus\PP_5$ is the class of quasi-Kenmotsu metric manifolds, $\PP_3\oplus\PP_6\oplus\PP_{11}$ is the class of almost $\bt$-Kenmotsu metric manifolds, $\PP_6$ is the class of $\bt$-Kenmotsu metric manifolds and so on.

\vskip 0.2in \addtocounter{subsection}{1} \setcounter{subsubsection}{0}

\noindent  {\Large\bf \thesubsection. Almost contact B-metric manifolds}%\\[6pt][6pt]\vskip2pt}

\vskip 0.15in

%\section{Almost contact B-metric manifolds}

Let $\MM$ be equipped with another almost contact structure $(\f_2,\xi_2,\eta_2)$ and the metric $g$ is a B-metric with respect to $(\f_2,\xi_2,\eta_2)$, i.e the relations \eqref{str1} and \eqref{str2} are satisfied for the structure $(\f_2,\xi_2,\eta_2,g)$ and we have
\begin{equation}\label{B-met}
\begin{array}{c}
  g(\xi_{2},\xi_{2})=1, \quad \eta_{2}(x)=g(\xi_{2},x),\\[6pt]
  g(\f_{2} x,\f_{2} y)=-g(x,y)+\eta_{2}(x)\eta_{2}(y).
\end{array}
\end{equation}
Then $(\f_{2},\xi_{2},\eta_{2},g)$ is an \emph{almost contact B-metric structure} on $\MM$, according to \S4.

%Since $g$ is a Norden metric %(or an anti-Hermitian metric)
%with respect to the almost complex structure $\f_2$ on the contact distribution $\HC_2=\ker(\eta_2)$, any B-metric can be considered as an odd-dimensional counterpart of the corresponding Norden metric, or the B-metric is a pseudo-Riemannian metric of Norden type on an odd-dimensional differentiable manifold.

A classification of the almost contact B-metric manifolds, given by G.~Gan\-chev, V.~Mihova and K.~Gribachev, is presented in \eqref{Fi} following \cite{GaMiGr}. In the present case, the fundamental tensor is $F_2$ defined by $F_2(x,y,z)=g\left(\left(\n_x\f_2\right)y,z\right)$. Its properties \eqref{F-prop} and \eqref{Fxieta} take the following form
\begin{gather}
\begin{array}{l}\label{F2-prop}
  F_2(x,y,z)= F_2(x,z,y)\\[6pt]
  \phantom{F_2(x,y,z)}
  = F_2(x,\f_2y,\f_2z)+F_2(x,\xi_2,z)\,\eta_2(y)\\[6pt]
  \phantom{F_2(x,y,z)= F_2(x,\f_2y,\f_2z)}
    +F_2(x,y,\xi_2)\,\eta_2(z),
\end{array}
\\[6pt]\label{Fetaxi2}
  F_2(x,\f_2y,\xi_2)= \left(\n_x\eta_2\right)(y)=g\left(\n_x\xi_2,y\right).
\end{gather}

The vector space of tensors having the properties of $F_2$ is decomposed in 11 orthogonal and invariant subspaces with respect to the action of the structure group $(\mathcal{GL}(m;\C)\cap \mathcal{O}(m,m))\times 1$, $m=\frac12(\dim\MM-1)$, where $\mathcal{O}(m,m)$ is the pseudo-orthogonal group of neutral signature. Thus, the obtained basic classes $\F_i$ $(i=1,2,\dots,11)$ with respect to $F_2$ for $\dim{\MM}=4n+3$ can be determined as follows:
\begin{subequations}\label{cl-B}
\begin{equation}
\begin{array}{rl}
\F_{1}: &F_{2}(x,y,z)=\dfrac{1}{2(2n+1)}\bigl\{g(x,\f_{2} y)\ta_{2}(\f_{2} z)\\[6pt]
&\phantom{F_{2}(x,y,z)=\dfrac{1}{2(2n+1)}\bigl\{}
+g(x,\f_{2} z)\ta_{2}(\f_{2} y)\\[6pt]
&\phantom{F_{2}(x,y,z)=\dfrac{1}{2(2n+1)}\bigl\{}
+g(\f_{2} x,\f_{2}y)\ta_{2}(\f_{2}^2 z)\\[6pt]
&\phantom{F_{2}(x,y,z)=\dfrac{1}{2(2n+1)}\bigl\{}
+g(\f_{2}x,\f_{2} z)\ta_{2}(\f_{2}^2 y)\bigr\};\\[6pt]
\F_{2}: &F_{2}(\xi_{2},y,z)=F_{2}(x,y,\xi_{2})=0,\quad\\[6pt]
              &F_{2}(x,y,\f_{2} z)+F_{2}(y,z,\f_{2} x)+F_{2}(z,x,\f_{2} y)=0,\quad \\[6pt] &\ta_{2}=0;
\\[6pt]
\F_{3}: &F_{2}(\xi_{2},y,z)=F_{2}(x,y,\xi_{2})=0,\quad\\[6pt]
              &F_{2}(x,y,z)+F_{2}(y,z,x)+F_{2}(z,x,y)=0;\\[6pt]
\F_{4}: &F_{2}(x,y,z)=-\dfrac{\ta_{2}(\xi_{2})}{2(2n+1)}\bigl\{g(\f_{2} x,\f_{2} y)\eta_{2}(z)\\
&\phantom{F_{2}(x,y,z)=-\dfrac{\ta_{2}(\xi_{2})}{2(2n+1)}\bigl\{}
+g(\f_{2} x,\f_{2} z)\eta_{2}(y)\bigr\};\\[6pt]
\F_{5}: &F_{2}(x,y,z)=-\dfrac{\ta_{2}^*(\xi_{2})}{2(2n+1)}\bigl\{g( x,\f_{2} y)\eta_{2}(z)\\
&\phantom{F_{2}(x,y,z)=-\dfrac{\ta_{2}^*(\xi_{2})}{2(2n+1)}\bigl\{}
+g(x,\f_{2} z)\eta_{2}(y)\bigr\};\\[6pt]
\F_{6}: &F_{2}(x,y,z)=F_{2}(x,y,\xi_{2})\eta_{2}(z)+F_{2}(x,z,\xi_{2})\eta_{2}(y),\quad \\[6pt]
                &F_{2}(x,y,\xi_{2})= F_{2}(y,x,\xi_{2})=-F_{2}(\f_{2} x,\f_{2} y,\xi_{2}),\quad \\[6pt]
                & \ta_{2}=\ta_{2}^*=0; \\[6pt]
\F_{7}: &F_{2}(x,y,z)=F_{2}(x,y,\xi_{2})\eta_{2}(z)+F_{2}(x,z,\xi_{2})\eta_{2}(y),\quad \\[6pt]
                &F_{2}(x,y,\xi_{2})=-F_{2}(y,x,\xi_{2})=-F_{2}(\f_{2} x,\f_{2} y,\xi_{2}); \\[6pt]
\F_{8}: &F_{2}(x,y,z)=F_{2}(x,y,\xi_{2})\eta_{2}(z)+F_{2}(x,z,\xi_{2})\eta_{2}(y),\\[6pt]
                &F_{2}(x,y,\xi_{2})=F_{2}(y,x,\xi_{2})=F_{2}(\f_{2} x,\f_{2} y,\xi_{2}); \\[6pt]
\F_{9}: &F_{2}(x,y,z)=F_{2}(x,y,\xi_{2})\eta_{2}(z)+F_{2}(x,z,\xi_{2})\eta_{2}(y),\\[6pt]
                &F_{2}(x,y,\xi_{2})=-F_{2}(y,x,\xi_{2})=F_{2}(\f_{2} x,\f_{2} y,\xi_{2});
\\[6pt]
\end{array}
\end{equation}
\begin{equation}
\begin{array}{rl}
\F_{10}: &F_{2}(x,y,z)=F_{2}(\xi_{2},\f_{2} y,\f_{2} z)\eta_{2}(x); \\[6pt]
\F_{11}:
&F_{2}(x,y,z)=\eta_{2}(x)\left\{\eta_{2}(y)\om(z)+\eta_{2}(z)\om(y)\right\}.
\end{array}
\end{equation}
\end{subequations}

Obviously, the class of cosymplectic B-metric manifolds $\F_0$ is
determined by the condition $F_{2}=0$.

\vskip 0.2in \addtocounter{subsection}{1} \setcounter{subsubsection}{0}

\noindent  {\Large\bf \thesubsection. Almost contact 3-structure with metrics of Hermitian-Norden type}%\\[6pt][6pt]\vskip2pt}

\vskip 0.15in
%\section{Manifolds with almost contact 3-structure and metrics of Hermitian-Norden type}\label{sec:1}

Let $(\MM,\f_{\al},\xi_{\al},\eta_{\al})$, $(\al=1,2,3)$ be a manifold
with an almost contact 3-struc\-ture, % or an \emph{almost hypercontact manifold},
i.e. $\MM$
is a $(4n+3)$-dimensional differentiable manifold with three almost
contact structures $(\f_{\al},\xi_{\al},\eta_{\al})$, $(\al=1,2,3)$ consisting of endomorphisms
$\f_{\al}$ of the tangent bundle,  Reeb vector fields $\xi_{\al}$ and their dual contact 1-forms
$\eta_{\al}$ satisfying the following identities:
\begin{equation}\label{3str}
\begin{array}{c}
\f_{\al}\circ\f_{\bt} = -\delta_{\al\bt}I + \xi_{\al}\otimes\eta_{\bt}+\epsilon_{\al\bt\gm}\f_{\gm},\\[6pt]
\f_{\al}\xi_{\bt} = \epsilon_{\al\bt\gm}\xi_{\gm},\quad
\eta_{\al}\circ\f_{\bt}=\epsilon_{\al\bt\gm}\eta_{\gm},\quad \eta_{\al}(\xi_{\bt})=\delta_{\al\bt},
\end{array}
\end{equation}
where $\al,\bt,\gm\in\{1,2,3\}$, $I$ is the identity on the algebra $\X(\MM)$% on the smooth vector fields on $\MM$
, $\delta_{\al\bt}$ is the Kronecker delta, $\epsilon_{\al\bt\gm}$ is the Levi-Civita symbol, i.e. either the sign of the permutation $(\al,\bt,\gm)$ of $(1,2,3)$ or 0 if any index is repeated \cite{Udr,Kuo}.

% for all cyclic permutations $(\al, \bt, \gm)$ of $(1,2,3)$.

%\textcolor[rgb]{1.00,0.00,0.00}{\textbf{Za spravka.} Kuo \cite{Kuo}: Let $M^{4n+3}$ be a differentiable manifold of almost contact
%3-structure. Then the structure group of the tangent bundle is reducible to
%$Sp(n) \times I_3$. \\[6pt]
%$Sp(n)$ is the subgroup of $GL(n,\mathbb{H})$ that preserves the standard hermitian form on $\mathbb{H}^n$.
%\\[6pt]
%Gribachev-Manev \cite{GriMan24}: $GL(n,\mathbb{H}) \cap O(2n, 2n)$ for $(M^{4n},H,g)$.}

%According to \cite{Kuo},
%if a differentiable manifold admits two almost contact structures $(\f_{\al},\xi_{\al},\eta_{\al})$ $(\al=1,2)$, satisfying
%\begin{equation}\label{str12}
%\begin{array}{c}
%\f_{1}\xi_{2} = -\f_{2}\xi_{1},\quad
%\f_{1}\circ\f_{2} - \xi_{1}\otimes\eta_{2} = -\f_{2}\circ\f_{1} + \xi_{2}\otimes\eta_{1},\\[6pt]
%\eta_{1}\circ\f_{2}=-\eta_{2}\circ\f_{1},\quad \eta_{1}(\xi_{2})=\eta_{2}(\xi_{1})=0,
%\end{array}
%\end{equation}
%then it admits an almost contact 3-structure as the third structure is get by putting
%\begin{equation}\label{str3}
%\begin{array}{c}
%\xi_{3} = \f_{1}\xi_{2},\quad \eta_{3} = \eta_{1}\circ\f_{2}, \quad
%\f_{3} = \f_{1}\circ\f_{2} - \xi_{1}\otimes\eta_{2}.
%\end{array}
%\end{equation}
%Any two of these almost contact structures define the same almost contact 3-structure.

Further, the indices $\al$, $\bt$ run over the range $\{1,2,3\}$ unless otherwise stated.

In the present subsection we discuss the coherence of compatible metrics and B-metrics in an almost contact 3-structure.
In \cite{Kuo}, it is considered the case of a Riemannian metric which is compatible by equations \eqref{A-met} for the three almost contact structures.

Suppose that $\MM$ admits two almost contact structures $(\f_{\al},\xi_{\al},\eta_{\al})$, $(\al=2,3)$. If a pseudo-Riemannian metric $g$ is a B-metric for the both structures, then the property in the first line of \eqref{3str} implies the properties in the second line of the same equations.

In \cite{Kuo} for the case of Riemannian metrics (positive definite), it is proved that if the almost contact 3-structure admits two almost contact metric structures, then the third one is of the same type.
We consider the relevant cases for our investigations in the following

\begin{thm}\label{thm:3str}
Let $\MM$ admit an almost contact 3-structure $(\f_{\al},\xi_{\al},\eta_{\al})$,
and a pseudo-Riemannian metric $g$.
If one of the three structures $(\f_{\al},\xi_{\al},\allowbreak{}\eta_{\al},g)$ %(let us suppose for $\al=2$)
is an almost contact B-metric structure, then the other two ones are an almost contact metric structure and an almost contact B-metric structure.
\end{thm}
\begin{proof}
%Suppose $\MM$ admits an almost contact 3-structure. % $(\f_{\al},\xi_{\al},\eta_{\al})$, $(\al=1,2,3)$.
First we establish on $\MM$ that if the pseudo-Riemannian metric $g$ and two of the almost contact structures generate:
\begin{enumerate}[(i)]
  \item two almost contact metric structures, then the third one is an almost contact metric structure;
  \item two almost contact B-metric structures, then the third one is an almost contact metric structure;
  \item an almost contact metric structure and an almost contact B-metric structure, then the third one is an almost contact B-metric structure.
\end{enumerate}

Now, we argue for the case (ii). Let $(\f_{2},\xi_{2},\allowbreak{}\eta_{2},g)$ be an almost contact B-metric structure, i.e. \eqref{B-met} holds. Moreover, let $(\f_{3},\xi_{3},\eta_{3},g)$ be also an almost contact B-metric structure, i.e. the following properties are valid
\begin{equation}\label{B-met3}
\begin{array}{c}
  g(\xi_{3},\xi_{3})=1, \quad \eta_{3}(x)=g(\xi_{3},x),\\[6pt]
  g(\f_{3} x,\f_{3} y)=-g(x,y)+\eta_{3}(x)\eta_{3}(y).
\end{array}
\end{equation}
Then, by virtue of the relations
\[
\f_1=\f_2\circ\f_3-\xi_2\otimes\eta_3, \quad \eta_2\circ\f_2=0, \quad \eta_2\circ\f_3=\eta_1,
\]
which are consequences of \eqref{3str},  using \eqref{B-met} and \eqref{B-met3}, we obtain
\[
\begin{array}{l}
g(\f_1x,\f_1y)=g\bigl(\f_2(\f_3x)-\eta_3(x)\xi_2,\f_2(\f_3y)-\eta_3(y)\xi_2\bigr)\\[6pt]
\phantom{g(\f_1x,\f_1y)}
=g(x,y)+\eta_1(x)\eta_1(y).
\end{array}
\]
Therefore, comparing with \eqref{A-met}, the metric $g$ is a compatible metric with respect to the almost contact structure $(\f_{1},\xi_{1},\eta_{1})$.

The verifications of the other cases are similar.
\end{proof}

Since any compatible metric and any B-metric on an almost contact manifold $\MM$ are metrics corresponding to a Hermitian metric and a Norden metric on the corresponding almost complex manifold $\MM\times\R$ (or on the corresponding contact distribution $\HC=\ker(\eta)$), respectively, we said that the compatible metric and the B-metric are metrics of Hermitian type and Norden type on $\MM$, respectively. Then, we give the following
\begin{dfn}
%Let $\MM$ be equipped with an almost contact 3-structure $(\f_{\al},\xi_{\al},\eta_{\al})$, $\al=1,2,3$.
We call a pseudo-Riemannian metric $g$ a \emph{metric of Hermitian-Norden type} on a manifold with almost contact 3-structure $(\MM,\allowbreak{}\f_{\al},\allowbreak{}\xi_{\al},\allowbreak{}\eta_{\al})$, if it satisfies the identities
\begin{equation}\label{HN-met}
  g(\f_{\al}x,\f_{\al}y)=\ea g(x,y)+\eta_{\al}(x)\eta_{\al}(y)
\end{equation}
for some cyclic permutation $(\ep_1,\ep_2,\ep_3)$ of $(1,-1,-1)$.
Then, we call $(\f_{\al},\xi_{\al},\allowbreak{}\eta_{\al},g)$ an \emph{almost contact 3-structure with metrics of Hermitian-Norden type}.
% or an \emph{almost hypercontact HN-metric structure}.

Let us suppose for the sake of definiteness that the coefficients $\ea$ have values as in \eqref{epsiloni}, i.e. $(\ep_1,\ep_2,\ep_3)=(1,-1,-1)$.
%\begin{equation}\label{ea}%
% \ea=
%\begin{cases}
%\begin{array}{ll}
%1, \quad & \al=1;\\[6pt]
%-1, \quad & \al=2,3.
%\end{array}
%\end{cases}
%\end{equation}
\end{dfn}

As a sequel of \eqref{HN-met} we have the following properties:
%where $\LLL_{\xi_\al}g$ denotes the Lie derivative of $g$ along $\xi_\al$.
\begin{eqnarray}
\label{eta-g}
\eta_{\al}&=&-\ea \xi_{\al} \lrcorner\, g,
\\[6pt]
\label{HN-met-2f}
  g(\f_{\al}x,y)&=&-\ea g(x,\f_{\al}y),
%\\[6pt]
\end{eqnarray}
\begin{eqnarray}
\label{Fetaxia}
  F_\al(x,\f_\al y,\xi_\al)&=&-\ea \left(\n_x\eta_\al \right)(y)\,=\,g\left(\n_x\xi_\al,y\right),
\\[6pt]
\label{Lie-der}
%\begin{array}{rl}
  (\LLL_{\xi_\al}g)(x,y)&=& g(\n_x \xi_\al,y)+g(x,\n_y \xi_\al)\\[6pt]
                        &=& -\ea \bigl(\left(\n_x\eta_\al\right)(y)+\left(\n_y\eta_\al\right)(x)\bigr).\nonumber
%\end{array}
\end{eqnarray}

Bearing in mind \eqref{HN-met-2f}, we deduce the following.
In the case $\al=1$, the associated tensor field of type $(0,2)$ is a 2-form. Let us denote it by $\widetilde g_{1}$, i.e. $\widetilde g_{1}(x,y)=g(\f_{1}x,y)$. It is actually opposite to $\Phi(x,y)=g(x,\f_{1}y)$, known as the \emph{fundamental} %(K\"ahler)
\emph{2-form} of the almost contact metric structure.
In other two cases $\al=2$ and $\al=3$, the tensor $(0,2)$-field $g(\f_{\al}x,y)$ is symmetric as well as $\etaa\otimes\etaa$. Then, we define the following fundamental tensor $(0,2)$-fields by
\begin{equation}\label{HN-met2}
  \widetilde g_{\al}(x,y)=g(\f_{\al}x,y)+\eta_{\al}(x)\eta_{\al}(y),\quad \al=2,3.
\end{equation}
Then $\widetilde g_{2}$ and $\widetilde g_{3}$ satisfy condition \eqref{HN-met} and they are also metrics of Hermitian-Norden type, which we call \emph{associated metrics} to $g$ with respect to $(\f_{\al},\xi_{\al},\eta_{\al})$ for $\al=2$ and $\al=3$, respectively.

Bearing in mind the structure groups of the almost contact 3-structures with compatible metric (\cite{Kuo}) and the almost hypercomplex manifolds with Hermitian-Norden metrics (\cite{GriMan24}), we can conclude the following.
The structure group of the manifolds with almost contact 3-structure and metrics of Hermitian-Norden type is $(\mathcal{GL}(n,\mathbb{H})\cap \mathcal{O}(2n,2n))\times \mathcal{O}{(2,1)}$, where $\mathcal{GL}(n,\mathbb{H})$ is  the group of invertible quaternionic $(n\times n)$-matrices and $\mathcal{O}(p,q)$ is the pseudo-orthogonal group of signature $(p,q)$ for natural numbers $p$ and $q$.
%$H^{\bot}=\mathrm{span}{\xi_1,\xi_2,\xi_3}$.

The fundamental tensors of a manifold with almost contact 3-structure and metrics of Hermitian-Norden type are the three
$(0,3)$-tensors determined by
\begin{equation}\label{F}
F_\al (x,y,z)=g\bigl( \left( \n_x \f_\al
\right)y,z\bigr).
\end{equation}
%where $\n$ is the Levi-Civita connection generated by $g$.
They have the following basic properties as a generalization of \eqref{F1-prop} and \eqref{F2-prop}%caused by the structures
\begin{equation}\label{Fa-prop}
\begin{array}{l}
  F_{\al}(x,y,z)=-\ea F_{\al}(x,z,y)\\[6pt]
  \phantom{F_{\al}(x,y,z)}
  =-\ea F_{\al}(x,\f_{\al}y,\f_{\al}z)+F_{\al}(x,\xi_{\al},z)\,
  \eta_{\al}(y)\\[6pt]
  \phantom{F_{\al}(x,y,z)=-\ea F_{\al}(x,\f_{\al}y,\f_{\al}z)}
    +F_{\al}(x,y,\xi_{\al})\,\eta_{\al}(z).
\end{array}
\end{equation}

The following associated 1-forms, defined as traces of $F_{\al}$, are known as their \emph{Lee forms}:
\begin{equation}\label{ta}
\begin{array}{l}
\theta_{\al}(z)=g^{ij}F_{\al}(e_i,e_j,z),\\[6pt]
\theta^*_{\al}(z)=g^{ij}F_{\al}(e_i,\f_{\al}e_j,z),\\[6pt]
\om_{\al}(z)=F_{\al}(\xi_{\al},\xi_{\al},z),
\end{array}
\end{equation}
where $g^{ij}$ are the components of the inverse matrix of the metric $g$ with respect to an arbitrary basis of the type $\{e_1,e_2,\dots,e_{4n+2},\xi_{\al}\}$.

The simplest case  of the manifolds with almost contact 3-structure and metrics of Hermitian-Norden type is when the
structures are $\n$-parallel, i.e. $\n\f_\al=\n\xi_\al=\n\eta_\al=\n g=\n
\widetilde{g}_\al=0$, and it is determined by the condition
$F_\al=0$.
We call these structures \emph{cosymplectic 3-structure with metrics of Hermitian-Norden type}.

%The properties \eqref{3str}, \eqref{HN-met}, \eqref{eta-g}, \eqref{F}, \eqref{Fetaxi} imply the following
%\begin{prop}\label{prop:F-prop}
%The tensors $F_1$, $F_2$, $F_3$ have the following fundamental identities:
%\begin{equation}\label{F-prop}
%\begin{array}{l}
%F_{\al}(x,y,z)=\epsilon_{\al\bt\gm}\bigl\{F_{\bt}(x,\f_{\gm}y,z)-F_{\bt}(x,\f_{\bt}z,\xi_{\bt})\eta_{\gm}(y)\\[6pt]
%\phantom{F_{\al}(x,y,z)=\epsilon_{\al\bt\gm}\bigl\{}
%           -\eb F_{\gm}(x,y,\f_{\bt}z)-\eb\eg F_{\gm}(x,\f_{\gm}y,\xi_{\gm})\eta_{\bt}(z)\bigr\}
%%          =&-F_{\gm}(x,\f_{\bt}y,z)+F_{\gm}(x,\f_{\gm}z,\xi_{\gm})\eta_{\bt}(y)\\[6pt]
%%           &+\eg F_{\bt}(x,y,\f_{\gm}z)+\eb\eg F_{\bt}(x,\f_{\bt}y,\xi_{\bt})\eta_{\gm}(z),\\[6pt]
%\end{array}
%\end{equation}
%for an arbitrary permutation $(\al,\bt,\gm)$ of $(1,2,3)$.
%\end{prop}

\vskip 0.2in \addtocounter{subsection}{1} \setcounter{subsubsection}{0}

\noindent  {\Large\bf \thesubsection. Relation with pseudo-Riemannian manifolds equipped with almost complex or almost hypercomplex structures}%\\[6pt][6pt]\vskip2pt}

\vskip 0.15in
%\section{Relation with pseudo-Riemannian manifolds equipped with almost complex or almost hypercomplex structures}

We can consider each of the three $(4n+2)$-dimensional  distributions
$\HC_{\al}=\ker(\eta_{\al})$, equipped with a corresponding pair of an almost complex structure
$J_{\al}=\f_{\al}|_{\HC_{\al}}$ and a metric $h_{\al}=g|_{\HC_{\al}}$, where $\f_{\al}|_{\HC_{\al}}$, $g|_{\HC_{\al}}$ are the restrictions of $\f_{\al}$, $g$ on $\HC_{\al}$, respectively, and the metrics $h_{\al}$
are compatible with $J_{\al}$ as follows
\begin{equation}\label{herm-nor}
\begin{array}{l}
h_{\al}(J_{\al}X,J_{\al}Y)=\ea h_{\al}(X,Y) ,\quad \\[6pt]
\widetilde{h}_{\al}(X,Y):=h_{\al}(J_{\al}X,Y)=-\ea h_{\al}(X,J_{\al}Y)
\end{array}
\end{equation}
for arbitrary $X,Y\in\X(\HC_{\al})$.
Obviously, in the cases $\al=2$ and $\al=3$ the metrics $h_{\al}$ and their associated $(0,2)$-ten\-sors $\widetilde{h}_{\al}$ are Norden metrics, whereas for $\al=1$ the structure $(J_1,h_{1})$ is an almost Hermitian pseudo-Riemannian structure with K\"ahler form $\Omega=-\widetilde{h}_{1}$.
In such a way, any of the distributions $\HC_{\al}$ for $\al=2$ or $\al=3$ can be considered as a $(2n+1)$-dimensional
complex Riemannian distribution with a complex Riemannian metric
$g_{\al}^{\mathbb{C}}=h_{\al}+\widetilde h_{\al}\sqrt{-1}=g|_{\HC_{\al}}+\widetilde{g}|_{\HC_{\al}}\sqrt{-1}$.
In another point of view, the distribution $\HC_{\al}$ for $\al=2$ or $\al=3$ is a $(4n+2)$-dimensional almost complex distribution with a Norden metric $h_{\al}$ and its associated Norden metric $\widetilde h_{\al}$.
Moreover, the $4n$-dimensional distribution $\HC=\HC_1\cap \HC_2\cap \HC_3$ has an almost hypercomplex structure $(J_1,J_2,J_3)$, i.e. $J_{\al}^2=-I$, $J_3=J_1J_2=-J_2J_1$, $J_{\al}=\f_{\al}\vert_{\HC}$,
with a pseudo-Riemannian metric $h=g|_\HC$ which is Hermitian with respect to $J_1$ and a Norden metric with respect to $J_2$ and $J_3$  since
$%\begin{equation}\label{h-norden}
h(J_{\al}X,J_{\al}Y)=\ea h(X,Y)$.
%\end{equation}
%On the other hand,
%a quaternionic inner product $\langle\cdot,\cdot\rangle$ can be generated in a natural way by the bilinear
%forms $h$, $h_1$, $h_2$ and $h_3$ on $H$ by the following decomposition: $\langle\cdot,\cdot\rangle = -h+ih_1+jh_2+kh_3$\textbf{???}, where $i$, $j$ and $k$ are the imaginary units of $\mathbb{H}$, the algebra of quaternions.

%On the codimension 3 subbundle $H=\HC_{1}\cap \HC_{2}\cap \HC_{3}$, where $\HC_{\al}=\ker(\eta_{\al})$ are the contact distributions, the endomorphisms
%$\f_{\al}\vert_{H}=J_{\al}$ are almost complex structures and the metric
%$g\vert_{H}=h$ is an HN-metric on the hypercomplex manifold $(H,J_{\al})$, i.e.
%$h(X,Y)=\ea h(J_{\al}X,J_{\al}Y)$, $\al=1,2,3$.
%

%Let $\{e_1, e_2, \dots , e_{2n}; \f_3e_1, \f_3e_2, \dots , \f_3e_{2n}n;\xi_3\}$ be a $\f_3$-basis of the almost contact B-metric manifold $(M,\f_3,\xi_3,\eta_3,g)$.
Let the vector $4n$-tuple
\[
(e_1, \dots, e_n; J_1e_1, \dots, J_1e_n; J_2e_1, \dots, J_2e_n; J_3e_1,\allowbreak{} \dots,\allowbreak{} J_3e_n)
\]
be an adapt\-ed basis  (or a $J_{\al}$-basis) of the almost hypercomplex structure. %$(J_{\al})$, where $J_{\al}=\f_{\al}\vert_{H}$ on $H$.
Then, according to \eqref{HN-met}, the basis
\begin{equation}\label{f-baza}
\begin{array}{c}
(e_1, \dots , e_n; \f_1e_1, \dots , \f_1e_n; \f_2e_1, \dots , \f_2e_n; \f_3e_1, \dots , \f_3e_n;\\[6pt]
\xi_1,\xi_2,\xi_3)
\end{array}
\end{equation}
is an an \emph{adapted basis} (or a $\f_{\al}$-basis) for the almost contact 3-structure and it is orthonormal with respect to $g$, i.e.
\begin{equation}\label{g-basis}
\begin{array}{l}
  g(e_i,e_i)=\ea g(\f_{\al}e_i,\f_{\al}e_i)=-\ea g(\xi_{\al},\xi_{\al})=1,\\[6pt]
  g(e_i,e_j)=g(e_i,\f_{\al}e_i)=g(e_i,\f_{\al}e_j)\\[6pt]
  \phantom{g(e_i,e_j)}
  =g(e_i,\xi_{\al})=g(\f_{\bt}e_i,\xi_{\al})=0
\end{array}
\end{equation}
for arbitrary $i\neq j\in\{1,2,\dots,n\}$.

It is well known that an even-dimensional manifold
endowed with almost complex structure $J$ and a compatible Riemannian metric $h$, i.e. $h(J\cdot,J\cdot)=h(\cdot,\cdot)$, is an almost Hermitian manifold. There are considered also almost pseudo-Hermitian manifolds, i.e. the case when $h$ is a pseudo-Riemannian metric with the same compatibility (cf. \cite{Mats81,Mats87}). %
We recall that this manifold
equipped with a pseudo-Riemannian metric of neutral signature
satisfying the identity $h(J\cdot,J\cdot)=-h(\cdot,\cdot)$ is known as an almost complex
manifold with Norden metric (see \S1).
In the case when the almost complex structure $J$ is parallel with respect to
the Levi-Civita connection $\DDD$ of the metric $h$,
i.e. $\DDD J=0$, then the manifold is known as a K\"ahler-Norden
manifold or a holomorphic
complex Riemannian manifold. Then the almost complex
structure $J$ is integrable and the local components of the complex metric
in holomorphic coordinate system are holomorphic functions.
%A four-dimensional example of a K\"ahler manifold with Norden metric
%has been given in \cite{N1}, another approach to the
%K\"ahler manifolds with Norden metric has been used in \cite{N2}
%and in \cite{v}, there has been proved that the four-dimensional
%sphere of Kotel'nikov-Study carries a structure of a K\"ahler
%manifold with Norden metric.

From another point of view, the almost hypercomplex structure $(J_1,J_2,\allowbreak{}J_3)$ and the metric $h$ generate two almost complex structures with Norden metrics (e.g., for $\al=2,3$) and one almost complex structure with Hermitian pseudo-Riemannian metric (e.g., for $\al=1$) because of \eqref{herm-nor}, i.e. an almost hypercomplex structure with Hermitian-Norden metrics (see \S9).

%it is well known the notion of the almost hypercomplex manifold with Hermitian metric (cf. \cite{AlMa}). The case when the metric is pseudo-Riemannian of neutral signature where one of the almost complex structures acts as an anti-isometry is considered in \cite{GriManDim12,GriMan24,ManGri32}. Then the almost hypercomplex structure $(J_1,J_2,J_3)$ and the metric $h$ generate two almost complex structures with Norden metrics (e.g., for $\al=2,3$) and one almost complex structure with Hermitian pseudo-Riemannian metric (e.g., for $\al=1$) because of \eqref{herm-nor}.

%\subsection{The case of parallel structures}%Manifolds with parallel 3-structures}

The manifolds with almost contact 3-structure and metrics of Hermi\-tian-Norden type can be considered as real hypersurfaces of an almost hypercomplex manifold with Hermitian-Norden metrics.

%As $g$ is an indefinite metric, there exist isotropic vectors
%on $\MM$, i.e. $g(x,x)=0$ for a nonzero vector $x$.
%Let $\n$ be the Levi-Civita connection generated by $g$.
%In \cite{Man31}, we define the \emph{square norm of $\nabla
%\f$} by
%\begin{equation}\label{sq-norm}
%    \norm{\nabla \f}=g^{ij}g^{ks}
%    g\bigl(\left(\nabla_{e_i} \f\right)e_k,\left(\nabla_{e_j}
%    \f\right)e_s\bigr),
%\end{equation}
%where $\{e_i\}$, $i=\{1,2,\dots,\dim{M}\}$, is an arbitrary basis of the tangent
%space $T_pM$ at an arbitrary point $p\in M$. It is clear, the equality $\norm{\nabla \f}=0$ is valid if $\f$ is parallel with respect to $\n$, i.e. $(M,\f,\xi,\eta,g)$ is
%a cosymplectic manifold. But the inverse implication is not always true.
%Now, a manifold with almost contact 3-structure and metrics of Hermitian-Norden type having vanishing square norms of
%$\n\f_{\al}$ we call an \emph{isotropic cosymplectic manifold} with such a structure.

In case of cosymplectic manifolds with metrics of Hermitian-Norden type, the distribution $\HC$ is involutive. The
corresponding integral submanifold is a totally geodesic
submanifold which inherits a holomorphic hypercomplex Riemannian
structure and the manifold with almost contact 3-structure and metrics of Hermitian-Norden type is
locally a pseudo-Riemannian product of a holomorphic hypercomplex
Riemannian manifold with a 3-dimensional Lorentzian real space.

\vskip 0.2in \addtocounter{subsection}{1} \setcounter{subsubsection}{0}

\noindent  {\Large\bf \thesubsection. Curvature properties of manifolds with almost contact 3-structure and metrics of Hermitian-Norden type}%\\[6pt][6pt]\vskip2pt}

\vskip 0.15in

%\section{Curvature properties of manifolds with almost contact 3-structure and metrics of Hermitian-Norden type}% of the K\"ahler-like tensors}

Let us recall that a tensor $L$ of type $(0,4)$ with the prop\-er\-ties \eqref{curv}
%\begin{equation}\label{curv}%
%\begin{array}{l}%
%L(x,y,z,w)=-L(y,x,z,w)=-L(x,y,w,z),\\[6pt]
%L(x,y,z,w)+L(y,z,x,w)+L(z,x,y,w)=0
%\end{array}%
%\end{equation} %
is called a \emph{curvature-like tensor}.
%The last equality of
%\eqref{curv} is known as the first Bianchi identity of $L$ for a symmetric affine connection.
%
We say that a curvature-like tensor $L$ is a \emph{K\"ahler-like
tensor} on a manifold with almost contact 3-structure and metrics of Hermitian-Norden type when $L$ satisfies the properties: %
\begin{equation}\label{L-kel}%
L(x,y,z,w)=\ea L(x,y,\f_{\al} z,\f_{\al} w).
\end{equation} %where $\ea$ is determined by \eqref{ea}.
K\"ahler-like tensors on almost contact manifolds with B-metric are considered in \cite{ManIv40}.

Using \eqref{3str} and \eqref{HN-met}, we obtain that for a K\"ahler-like tensor $L$ the following properties are valid
\begin{equation}\label{L-kel2}%
\begin{array}{l}
L(x,y,z,w)=\ea L(x,\f_{\al} y,\f_{\al} z,w)=\ea L(\f_{\al} x,\f_{\al} y,z,w)\\[6pt]
L(\xi_{\al},y,z,w)=L(x,\xi_{\al},z,w)=L(x,y,\xi_{\al},w)\\[6pt]
\phantom{L(\xi_{\al},y,z,w)}
=L(x,y,z,\xi_{\al})=0.
\end{array}
\end{equation}
The latter properties show that if $L$ is a K\"ahler-like tensor on a manifold with almost contact 3-structure and metrics of Hermitian-Norden type then $L$ is a K\"ahler-like tensor on $(\HC,J_{\al}=\f_{\al}|_H,h=g|_H)$ which is a manifold with almost hypercomplex structure with Hermitian-Norden metrics. It is known from \cite{Man28} that every K\"ahler-like tensor vanishes on an almost hypercomplex manifold with Hermitian-Norden metrics. Therefore, it is valid the following
\begin{prop}\label{prop:L=0}
Every K\"ahler-like tensor vanishes on a manifold with almost contact 3-structure and metrics of Hermitian-Norden type.
\end{prop}

Let $R$ be the curvature tensor of the Levi-Civita connection
$\n$ generated by $g$.
%It is defined as usual by
%$R(x,y)z=\left[\n_x, \n_y\right] z - \n_{\left[x,y\right]} z$. The corresponding $(0,4)$-tensor,
%denoted by the same letter, is determined by
%$R(x,y,z,w)=g\left(R(x,y)z,w\right)$.

According to \cite{GriManDim12}, every hyper-K\"ahler manifold with Hermitian-Norden metrics is flat.
Since $R$ is a K\"ahler-like tensor on every manifold with cosymplectic 3-structure with metrics of Hermitian-Norden type, i.e. $\n\f_\al$ vanishes, then applying \propref{prop:L=0} we obtain
\begin{prop}\label{prop:R=0}
Every manifold with cosymplectic 3-structure with metrics of Hermitian-Norden type is flat.
\end{prop}

%
%The Ricci tensor $\rho$ for the curvature tensor $R$ and the
%scalar curvature $\tau$ for $R$ are defined by the traces
%$\rho(y,z)=g^{ij}R(e_i,y,z,e_j)$ and $\tau=g^{ij}\rho(e_i,e_j)$,
%respectively.
%%
%Similarly, the Ricci tensor and the scalar curvature are
%determined for each \emph{curvatu\-re-like tensor}.

\vskip 0.2in \addtocounter{subsection}{1} \setcounter{subsubsection}{0}

\noindent  {\Large\bf \thesubsection. Examples of manifolds with almost contact 3-structure and metrics of Hermitian-Norden type}%\\[6pt][6pt]\vskip2pt}

%\vskip 0.15in

%\section{Examples of manifolds with almost contact 3-structure and metrics of Hermitian-Norden type}

\vskip 0.2in \addtocounter{subsubsection}{1}

\noindent  {\Large\bf{\emph{\thesubsubsection. A real vector space with contact 3-structure with metrics of Hermitian-Norden type}}}%\\[6pt][6pt]\vskip2pt}

\vskip 0.15in
%\subsection{A real vector space with contact 3-structure with metrics of Hermitian-Norden type}

Let $V$ be a real $(4n+3)$-dimensional vector space and a (local) basis of $V$ is denoted by $\left\{\ddx;\ddy;\ddu;\ddv;\dda,\ddb,\ddc\right\}$, $(i=1,2,\dots,n)$ or

\[
   \begin{split}
\left\{
\dfrac{\pd}{\pd x^1},\dots,\dfrac{\pd}{\pd x^n};\right.&\dfrac{\pd}{\pd y^1},\dots,\dfrac{\pd}{\pd y^n};\dfrac{\pd}{\pd
u^1},\dots,\dfrac{\pd}{\pd u^n};\dfrac{\pd}{\pd v^1},\dots,\dfrac{\pd}{\pd v^n};\\[6pt]
&\left.\dfrac{\pd}{\pd a}, \dfrac{\pd}{\pd b}, \dfrac{\pd}{\pd c}\right\}.
   \end{split}
\]

Any vector $z$ of $V$ can be represented in
the mentioned basis as follows
   \begin{equation}\label{x}
   \begin{array}{c}
z=x^i\ddx+y^i\ddy+u^i\ddu+v^i\ddv+a\dda+b\ddb+c\ddc.
   \end{array}
   \end{equation}

A standard contact 3-structure in $V$ is defined as follows:
\begin{subequations}\label{3-str}
\begin{equation}
\begin{array}{lll}
\f_1\ddx=\ddy,\quad &\f_2\ddx=\ddu,\quad &\f_3\ddx=-\ddv,
\\[6pt]
\f_1\ddy=-\ddx,\quad &\f_2\ddy=\ddv,\quad &\f_3\ddy=\ddu,
\\[6pt]
\f_1\ddu=-\ddv,\quad &\f_2\ddu=-\ddx,\quad &\f_3\ddu=-\ddy,
\\[6pt]
\f_1\ddv=\ddu,\quad &\f_2\ddv=-\ddy,\quad &\f_3\ddv=\ddx
\\[6pt]
\end{array}
\end{equation}
\begin{equation}
\begin{array}{lll}
\f_1\dda=0,\quad &\f_2\dda=-\ddc,\quad &\f_3\dda=\ddb,
\\[6pt]
\f_1\ddb=\ddc,\quad &\f_2\ddb=0,\quad &\f_3\ddb=-\dda,
\\[6pt]
\f_1\ddc=-\ddb,\quad &\f_2\ddc=\dda,\quad &\f_3\ddc=0,
\\[6pt]
\xi_1=\dda,\quad &\xi_2=\ddb,\quad &\xi_3=\ddc,
\\[6pt]
\eta_1=\D a,\quad &\eta_2=\D b,\quad &\eta_3=\D c.
\end{array}
\end{equation}
\end{subequations}
We check immediately that the properties \eqref{3str} hold.

If $z \in V$, i.e. $z(x^i;y^i;u^i;v^i;a,b,c)$ then according to
\eqref{3-str} we have
\begin{equation}\label{x-coord}
\begin{array}{ll}
\f_1z(-y^i;x^i;v^i;-u^i;0,-c,b),\quad & \eta_1(z)=a,\\[6pt]
\f_2z(-u^i;-v^i;x^i;y^i;c,0,-a),\quad & \eta_2(z)=b,\\[6pt]
\f_3z(v^i;-u^i;y^i;-x^i;-b,a,0),\quad & \eta_3(z)=c.
\end{array}
\end{equation}

\begin{dfn}\label{d1}
The structure $(\f_{\al},\xi_{\al},\eta_{\al})$ on
$V$ is called \emph{a contact 3-structure} on
$V$.
\end{dfn}

Let us introduce a pseudo-Euclidian metric $g$ of
signature $(2n+2,2n+1)$ as follows
\begin{equation}\label{metric}
g(z,z')=\sum_{i=1}^n \left(x^ix'^i+y^i y'^i-u^iu'^i-v^iv'^i\right)-aa'+bb'+c c',
\end{equation}
where $z(x^i;y^i;u^i;v^i;a,b,c)$, $z'(x'^i;y'^i;u'^i;v'^i;a',b',c') \in V$, $(i=1,2,\allowbreak{}\dots,\allowbreak{}n)$. This metric satisfies the following properties
\begin{equation}\label{metr-sv}
\begin{array}{c}
g(\f_{\al}z,\f_{\al}z')=\ea g(z,z') +\eta_{\al}(z)\eta_{\al}(z'),
\end{array}
\end{equation}
which is actually \eqref{HN-met}.

We check immediately that $\n\f_{\al}$ vanishes for $\n$, the Levi-Civita connection of $g$. Therefore we get the following
\begin{prop}
The space $(V,\f_{\al},\xi_{\al},\eta_{\al},g)$ is %belongs to the class %$\F_0$
%of the
a manifold with cosymplectic 3-structure and metrics of Hermitian-Norden type.% for $\al=1,2,3$.
\end{prop}

%\newpage
\vskip 0.2in \addtocounter{subsubsection}{1}

\noindent  {\Large\bf{\emph{\thesubsubsection. A time-like sphere with almost contact 3-structure and metrics of Hermitian-Norden type}}}%\\[6pt][6pt]\vskip2pt}

\vskip 0.15in
%\subsection{A time-like sphere with almost contact 3-structure and metrics of Hermitian-Norden type}

It is known that any real hypersurface of an almost hypercomplex manifold carries in a natural way an almost contact 3-structure.

In a similar way it can be shown that on every real nonisotropic hypersurface of an almost hypercomplex manifold with Hermitian-Norden metrics there arises an almost contact 3-structure with metrics of Hermitian-Norden type.

Let us consider
\[
\R^{4n+4}=\bigl\{\left(x^i;y^i;u^i;v^i\right)\vert\ x^i,y^i,u^i,v^i\in\R,\ i\in\{1,2,\dots,n+1\}\bigr\},
\]
a vector space of dimension $4n+4$ with an almost hypercomplex structure $(J_{1},J_{2},J_{3})$ determined as follows \cite{GriManDim12}
\[
J_1z(-y^i;x^i;v^i;-u^i),\quad
J_2z(-u^i;-v^i;x^i;y^i),\quad
J_3z(v^i;-u^i;y^i;-x^i)
\]
for an arbitrary vector $z(x^i;y^i;u^i;v^i)$.
This space is equipped with a pseudo-Eu\-clidean metric of neutral signature, i.e. $(2n+2,2n+2)$, by
\[
g(z,z')=\sum_{i=1}^{n+1} \bigl(x^ix'^i+y^iy'^i-u^iu'^i-v^iy'^i\bigr)
\]
for arbitrary $z(x^i;y^i;u^i;v^i), z'(x'^i;y'^i;u'^i;v'^i)\in\R^{4n+4}$.

Identifying an arbitrary point $p\in\R^{4n+4}$ with its position vector $z$, we study the following hypersurface of
$\R^{4n+4}$.

Let $\SSS: g(z,z)=-1$ be the unit time-like sphere of $g$ in $\R^{4n+4}$. Then $z$ coincides with the unit normal $U$ to the tangent space $T_p\SSS$ at $p\in\SSS$.

We determine the Reeb vector fields by the equalities
\[
\xi_{\al}=\lm_{\al}\, U+\mu_{\al}\, J_{\al}U,
\]
such that $g(U,\xi_{\al})=0$ and $g(\xi_{\al},\xi_{\al})=-\ea$ are valid.

We substitute $g(U,J_{\al}U)=\tan\psi_{\al}$ for $\psi_{\al}\in\left(-\frac{\pi}{2},\frac{\pi}{2}\right)$. Then we obtain
\begin{equation*}%\label{xi-exa-}
\xi_{\al}=\sin\psi_{\al}\, U+\cos\psi_{\al}\, J_{\al}U.
\end{equation*}
Since $g(U,J_{1}U)=0$ then $\psi_{1}=0$ and therefore $\xi_{1}=J_{1}U$.
Because of $g(U,\xi_{\al})=0$ we have that $\xi_{\al}$ are in $T_p\SSS$.
The conditions $g(\xi_{\al},\xi_{\bt})=0$ for $\al\neq\bt$ are equivalent to $\psi_{2}=\psi_{3}=0$.
Therefore we obtain the following
equality for all  $\al$
\begin{equation}\label{xi-exa}
\xi_{\al}=J_{\al}U.
\end{equation}
Using the latter equality and $J_{\al}J_{\al}U=-U$, we obtain that $J_{\al}\xi_{\al}=-U$.

We define the structure endomorphisms $\f_{\al}$ and the contact 1-forms $\eta_{\al}$ in $T_p\SSS$ by the following orthonormal decomposition of $J_{\al}x$ for arbitrary $x\in T_p\SSS$
\begin{equation}\label{fi-eta-exa}
  J_{\al}x=\f_{\al}x-\eta_{\al}(x)\, U,
\end{equation}
i.e. $\f_{\al}x$ is the tangent component of $J_{\al}x$ and $-\eta_{\al}(x)U$ is the corresponding normal component.
By direct computation \eqref{fi-eta-exa} implies \eqref{3str}. Then, using \eqref{xi-exa}, we obtain \eqref{HN-met} and \eqref{eta-g}.
Thus, we equip the unit time-like sphere $\SSS$ in $\R^{4n+4}$ with an almost contact 3-structure with metrics of Hermitian-Norden type $(\f_{\al},\xi_{\al},\eta_{\al},g)$.

Let $\overline\n$ and $\n$ be the Levi-Civita connections of the metric $g$ in $\R^{4n+4}$ and $\SSS$, respectively. Since $\overline\n$ is flat, the formulae of Gauss and Weingarten have the form
\begin{equation}\label{GW}
\overline\n_xy=\n_xy+g(x,y)U,\qquad
\overline\n_x U=x.
\end{equation}
Therefore one can obtain by \eqref{xi-exa}, \eqref{fi-eta-exa} and \eqref{GW} that
\[
\n_x\xi_{\al}=\f_{\al}x,
\qquad
F_{\al}(x,y,\xi_{\al})=-g(\f_{\al}x,\f_{\al}y).
\]
Then, for the Lee forms we have
\[
\ta_{\al}(\xi_{\al})=4n+2,\qquad \ta^*_{\al}(\xi_{\al})=0,\qquad \om_\al=0.
\]
Finally, we get
\begin{equation}\label{F-S}
F_{\al}(x,y,z)=-
g(\f_{\al}x,\f_{\al}y)\eta_{\al}(z)
+\ea g(\f_{\al}x,\f_{\al}z)\eta_{\al}(y).
%\phartom{F_{\al}(x,y,z)=}
%&
%-\frac{\ta^*_{\al}(\xi_{\al})}{4n+2}
%\bigl\{
%g(\f_{\al}x,y)\eta_{\al}(z)
%-\ea g(\f_{\al}x,z)\eta_{\al}(y)
%\bigr\}.
\end{equation}

In the case $\al=1$, the equality \eqref{F-S} takes the form
\begin{equation}\label{F1-S}
F_{1}(x,y,z)=
g(x,y)\eta_{1}(z)
-g(x,z)\eta_{1}(y).
\end{equation}
i.e. by virtue of \eqref{HN-met}, \eqref{eta-g} and \eqref{F}, we have
\[
(\n_x \f_{1})y = g(x,y)\xi_{1}-g(\xi_{1},y)x.
\]
According to \cite[Theorem 6.3]{Blair}, the latter equality is a necessary and sufficient condition for a Sasakian manifold.

Similarly, in the case $\al=2$ or $\al=3$,
% the classifications of almost contact metric manifolds and almost contact B-metric manifolds, given in \cite{AlGa} and \cite{GaMiGr}, respectively, we have the following definitions of two of the basic classes for dimension $4n+3$ in the respective cases:
%\[
%\begin{split}
%    \W_2:&\quad F_1(x,y,z)=\frac{\ta_1(\xi_1)}{4n+2}\left\{g(x,y)\eta_1(z)-g(x,z)\eta_1(y)\right\},\quad \al=1;\\[6pt]
%    \F_{4}:&\quad F_{\al}(x,y,z)=-\frac{\ta_{\al}(\xi_{\al})}{4n+2}\bigl\{g(\f_{\al} x,\f_{\al} y)\eta_{\al}(z)+g(\f_{\al} x,\f_{\al} z)\eta_{\al}(y)\bigr\},\quad \al=2,3.
%\end{split}
%\]
from \eqref{F-S}, according to \thmref{ssss}, we get a necessary and sufficient condition for a Sasaki-like almost contact complex Riemannian manifold.

We recall that a Sasakian manifold (respectively, a Sasaki-like almost contact complex Riemannian manifold) is defined as an almost contact metric manifold (respectively, an almost contact B-metric manifold) which complex cone is a K\"ahler manifold  (respectively, a K\"ahler-Norden manifold) (cf. \cite{Blair} and \S8).

Thus, we obtain the following
\begin{prop}
The manifold $(\SSS,\f_{\al},\xi_{\al},\eta_{\al},g)$ is$:$
  \begin{enumerate}
    \item a Sasakian manifold for $\al=1;$
    \item a Sasaki-like almost contact complex Riemannian manifold for $\al=2,3.$
\end{enumerate}
\end{prop}

In view of \eqref{cl-A}, \eqref{cl-B}, \eqref{F-S} and \eqref{F1-S}, we obtain that
$(\SSS,\f_{\al},\xi_{\al},\allowbreak{}\eta_{\al},g)$  belongs to
the class $\PP_2$ of almost contact metric manifolds for $\al=1$
and to
the class $\F_4$ of almost contact B-metric manifolds for $\al=2,3$.

In \cite{GaMiGr}, it is considered a unit time-like sphere with almost contact B-met\-ric structure and it is proved that it belongs to the class $\F_4\oplus\F_5$, the analogue of trans-Sasakian manifold of type $(\al,\bt)$.

\vspace{20pt}

\begin{center}
$\divideontimes\divideontimes\divideontimes$
\end{center} 

\newpage

\addtocounter{section}{1}\setcounter{subsection}{0}\setcounter{subsubsection}{0}

\setcounter{thm}{0}\setcounter{dfn}{0}\setcounter{equation}{0}

\label{par:4n+3assNij}

 \Large{

\
\\[6pt]
\bigskip

\
\\[6pt]
\bigskip

\lhead{\emph{Chapter II $|$ \S\thesection. Associated Nijenhuis tensors on manifolds with almost contact 3-structure \ldots
}}
%\thispagestyle{empty}

%\noindent  {\Huge\bf \S\thesection. Associated Nijenhuis tensors on \\[12pt]
%\phantom{\S\thesection. }manifolds with almost contact  \\[12pt]
%\phantom{\S\thesection. }3-structure and metrics of   \\[12pt]
%\phantom{\S\thesection. }Hermitian-Norden type}%\\[6pt]\vskip2pt}

\noindent
\begin{tabular}{r"l}
  %\hline
\hspace{-6pt}{\Huge\bf \S\thesection.}  & {\Huge\bf Associated Nijenhuis tensors on} \\[12pt]
                             & {\Huge\bf manifolds with almost contact} \\[12pt]
                             & {\Huge\bf 3-structure and metrics of }\\[12pt]
                             & {\Huge\bf Hermitian-Norden type}
  %\hline
\end{tabular}

\vskip 1cm

\begin{quote}
\begin{large}
In the present section, it is considered a differentiable manifold equipped with a pseudo-Riemannian metric and an almost contact 3-struc\-ture so that one almost contact metric structure and two almost contact B-metric structures are generated.
There are introduced associated Nijenhuis tensors for the studied structures.
The vanishing of the Nijenhuis tensors and their associated tensors is considered.
It is given a geometric interpretation of the vanishing of associated Nijenhuis tensors for the studied structures as a necessary and sufficient condition for existence of
affine connections with totally skew-symmetric torsions preserving the
structure. An example of a 7-dimensional manifold with connections of the considered
type is given.

The main results of this section are published in \cite{Man54} and \cite{Man55}.
\end{large}\end{quote}

%\vskip 0.2in \addtocounter{subsection}{1} \setcounter{subsubsection}{0}
%
%\noindent  {\Large\bf \thesubsection. Introduction}%\\[6pt]\vskip2pt}

\vskip 0.15in

%%%%%%%%%%%%%%%%%%%%%%%%%%%%%%%%%%%%%%%%%%%%%%%%%%%%%%%%%%%%%%%%%%%%%%%%%%%0

%It is known the notion of an \emph{almost contact 3-structure} on a differentiable manifold of dimension $4n+3$ (\cite{Kuo,Udr}). The product of a man\-i\-fold with almost contact 3-structure and a real line admits an \emph{almost hypercomplex structure} (cf. \cite{Kuo,AlMa}).
%
%It is only considered the case of equipping such a manifold with a Riemannian metric compatible with each of the three structures in the given almost contact 3-structure. This is the so-called \emph{almost contact metric 3-structure}.
%
%
%In Section 14, we have introduced a pseudo-Riemannian metric which has another kind of compatibility with the triad of almost contact structures on a manifold with almost contact 3-structure.
%The product of this manifold of new type and a real line is a $(4n+4)$-dimensional manifold which admits an almost hypercomplex structure $(J_1,J_2,J_3)$ and a Hermitian-Norden metric, i.e. $J_1$ (resp., $J_2$ and $J_3$) acts as an isometry (resp., act as anti-isometries) with respect to the pseudo-Riemannian metric of neutral signature in each tangent fibre.
%The constructed structure on $(4n+3)$-dimensional manifolds we call an \emph{almost contact 3-structure with metrics of Hermitian-Norden type} (briefly, an HN-type).

In \S1, \S4 and \S12 were defined and then studied associated Nijenhuis tensors on almost complex manifold with Norden metric, almost contact manifold with B-metric and almost hypercomplex manifold with Hermitian-Norden metrics, respectively.
The goal of the present section is to introduce an appropriate associated Nijenhuis tensor on a manifold with almost contact 3-structure and metrics of Hermitian-Norden type which will be used in studying of the considered manifold.

%such that the vanishing of this tensor is a necessary and sufficient condition for existence of affine connections with totally skew-symmetric torsion preserving the almost contact 3-structure and the metric of HN-type.

%%%%%%%%%%%%%%%%%%%%%%%%%%%%%%%%%%%%%%%%%%%%%%%%%%%%%%%%%%%%%%%%%%%%%%%%%%

As it is known, for each $\al\in\{1,2,3\}$ the Nijenhuis tensor $N_{\al}$ of an almost contact manifold $(\MM,\f_\al,\allowbreak{}\xi_\al,\eta_\al)$ is defined as in \eqref{NN}.
%\[
%  N_{\al}=[\f_{\al},\f_{\al}]+\xi_{\al}\otimes\D\eta_{\al},
%\]
%where $[\f_{\al},\f_{\al}](x,y)=\f_{\al}^2[x,y]+[\f_{\al}x,\f_{\al}y]-\f_{\al}[\f_{\al}x,y]-\f_{\al}[x,\f_{\al}y]$ and $\D\eta_{\al}(x,y)=\left(\n_x\eta_{\al}\right)y-\left(\n_y\eta_{\al}\right)x$.
Moreover, if two of almost contact structures in an almost contact 3-structure are normal, then the third one is also normal \cite{Kuo,YaAk,YaIsKo}.

\newpage
%\subsection{Associated Nijenhuis tensor on $(\MM,\f_1,\xi_1,\allowbreak{}\eta_1,g)$}

%%%%%%%%%%%%%%%%%%%%%%%%%%%%%%%%%%%%%%%%%%%%%%%%%%%%%%%%%%%%%%%%%%%%%%

\vskip 0.2in \addtocounter{subsection}{1} \setcounter{subsubsection}{0}

\noindent  {\Large\bf \thesubsection. Associated Nijenhuis tensors of an almost contact 3-struc\-ture with a pseudo-Riemannian metric }%\\[6pt]\vskip2pt}

\vskip 0.15in

%We give the following %Besides the {Nijenhuis} tensor $N_1$ for a $(\f_1,\xi_1,\allowbreak{}\eta_1)$-structure

Let us consider the symmetric braces $\{\cdot,\cdot \}$ on $\X(\MM)$ introduced by \eqref{braces} for a pseudo-Riemannian metric $g$, as well as the tensors $\{\f_1 ,\f_1\}$ and $\LLL_{\xi_1}g$ determined by \eqref{[f,f]} and \eqref{Lie-der}, respectively.
%\begin{defn}
\begin{dfn}
The symmetric (1,2)-tensor $\widehat N_1$, defined by
\begin{equation}\label{hatN1}
\widehat N_1=\{\f_1,\f_1\}-\xi_1\otimes\LLL_{\xi_1}g,
\end{equation}
is called the \emph{associated Nijenhuis tensor of the almost contact metric structure $(\f_1,\xi_1,\allowbreak{}\eta_1,g)$}.
\end{dfn}
%where %$\LL$ denotes the Lie derivative, and
%%$\LLL_{\xi_1}$ denotes the Lie derivative along $\xi_1$ and
%$\{\f_1 ,\f_1\}$
%is defined as in \eqref{[f,f]}. %the symmetric tensor of type $(1,2)$ given by
%\begin{equation}\label{{f1f1}}
%\{\f_1 ,\f_1\}(x,y)=\{\f_1 x,\f_1 y\}+(\f_1)^2\{x,y\}-\f_1\{\f_1 x,y\}-\f_1\{x,\f_1
%y\}.
%\end{equation}
%We call $\widehat N_1$ an  \emph{as\-so\-ci\-ated {Nijenhuis} tensor} on $(\MM,\f_1,\xi_1,\allowbreak{}\eta_1,g)$.
%%\end{defn}

The corresponding tensors of type $(0,3)$ for $N_1$ and $\widehat N_1$ are given
by $N_1(x,y,z)=g(N_1(x,y),z)$ and $\widehat N_1(x,y,z)=g(\widehat
N_1(x,y),z)$, respectively.

By direct consequences of the definitions, we get that
$N_1$, $\widehat N_1$ and $\LLL_{\xi_1}g$ are expressed in
terms of $F_1$ as follows:
\begin{gather}
\begin{array}{l}
N_1(x,y,z)=
F_1(\f_1 x,y,z)-F_1(\f_1 y,x,z)\label{N1=F1}\\[6pt]
\phantom{N_1(x,y,z)}
+ F_1(x,y,\f_1 z)- F_1(y,x,\f_1 z)
\\[6pt]
\phantom{N_1(x,y,z)}
+F_1(x,\f_1 y,\xi_1)\,\eta_1(z)-F_1(y,\f_1x,\xi_1)\,\eta_1(z),
\end{array}\\[6pt]
\begin{array}{l}
\widehat N_1(x,y,z)=
F_1(\f_1 x,y,z)+F_1(\f_1 y,x,z)\\[6pt]
\phantom{\widehat N_1(x,y,z)}
+ F_1(x,y,\f_1 z)+ F_1(y,x,\f_1z)\\[6pt]
\phantom{\widehat N_1(x,y,z)}
+F_1(x,\f_1 y,\xi_1)\,\eta_1(z)+F_1(y,\f_1x,\xi_1)\,\eta_1(z),\label{N1hat=F1}
\end{array}\\[6pt]
\left(\LLL_{\xi_1}g\right)(x,y)=F_1(x,\f_1y,\xi_1)+F_1(y,\f_1x,\xi_1).
\label{Lxi1g=F1}
\end{gather}

%\subsection{Associated Nijenhuis tensor on $(\MM,\f_2,\xi_2,\eta_2,g)$}

%
According to \eqref{S}, the as\-so\-ci\-ated Nijenhuis tensor $\widehat N_2$ for the almost contact B-metric structure $(\f_2,\xi_2,\allowbreak{}\eta_2,g)$ is defined by
\begin{equation}\label{hatN2}
\widehat N_2=\{\f_2,\f_2\}+\xi_2\otimes\LLL_{\xi_2}g.
\end{equation}
%where %$\LL$ denotes the Lie derivative, and
%$\{\f_2 ,\f_2\}$
%is the symmetric tensor of type $(1,2)$ defined as in \eqref{{f1f1}}.

\begin{prop}\label{prop:hatN=0=>Kill}
  For the almost contact B-metric manifold $(\MM,\f_2,\allowbreak{}\xi_2,\allowbreak{}\eta_2,g)$, the vanishing of $\widehat N_2$ implies that $\xi_2$ is Killing.
\end{prop}
\begin{proof}
%For $(\MM,\f_2,\xi_2,\eta_2,g)$,
The formula for $F_2$ in terms of $N_2$ and $\widehat N_2$ is known from \eqref{nabf} of \thmref{thm:FNhatN}, %\cite{IvMaMa14},
whereas the expression of $\widehat N_2$ by $F_2$ follows from \eqref{enhat}.
%\[%\label{F=NhatN}
%\begin{split}
%F_2(x,y,z)&=-\frac14\bigl\{N_2(\f_2 x,y,z)+N_2(\f_2 x,z,y)%\\[6pt]
%%&%\phantom{=-\frac14\bigl\{ }
%+\widehat N_2(\f_2 x,y,z)+\widehat N_2(\f_2 x,z,y)\bigr\}\\[6pt]
%&\phantom{=\ }
%+\frac12\eta_2(x)\bigl\{N_2(\xi_2,y,\f_2 z)+\widehat
%N_2(\xi_2,y,\f_2 z)%\\[6pt]
%%&\phantom{=\ \,+\frac12\eta_2(x)\bigl\{N_2(\xi_2,y,\f_2 z)}
%+\eta_2(z)\widehat N_2(\xi_2,\xi_2,\f_2 y)\bigr\},
%\\[6pt]
%\widehat N_2(x,y,z)&=
%F_2(\f_2 x,y,z)- F_2(x,y,\f_2 z)+F_2(x,\f_2 y,\xi_2)\,\eta_2(z)\\[6pt]
%&%\phantom{\widehat N_2(x,y,z)=}
%\, +F_2(\f_2 y,x,z)- F_2(y,x,\f_2z)+F_2(y,\f_2x,\xi_2)\,\eta_2(z).
%\end{split}
%\]
By these relations,  \eqref{NhatN-prop2}, \eqref{Fetaxia} and \eqref{Lie-der}, we obtain %the following equality

\begin{equation*}\label{Lxig=hatN}
\begin{split}
(\LLL_{\xi_2}g)(x,y)&=-\frac12\bigl\{\widehat N_2(\f_2 x,\f_2 y,\xi_2)+\widehat N_2(\xi_2,\f_2 x,\f_2 y)\\[6pt]
&\phantom{=-\frac12\bigl\{}
+\widehat N_2(\xi_2,\f_2 y,\f_2 x)+\eta_2(x)\widehat N_2(\xi_2,\xi_2,y)\\[6pt]
&\phantom{=-\frac12\bigl\{ }
+\eta_2(y)\widehat N_2(\xi_2,\xi_2,x)\bigr\},
\end{split}
\end{equation*}
which yields the statement.
\end{proof}

Let us remark that a similar statement of \propref{prop:hatN=0=>Kill} for an almost contact metric manifold is not true whereas the corresponding proposition for the almost contact B-metric structure $(\f_3,\xi_3,\eta_3,g)$
and $\widehat N_3$ defined as in \eqref{hatN2} holds.

%Now, we shall discuss on the integrability and related topics when
Let $\MM$, $\dim \MM=4n+3$, be equipped with an almost contact 3-structure $(\f_{\al},\xi_{\al},\allowbreak{}\eta_{\al})$ %, $(\al=1,2,3)$
and then we consider the product $\MM\times\R$. % of %an almost contact manifold $\MM$ with a real line.
Let $X$ be a vector field on $\MM\times\R$ which is presented by a pair $\left(  x, a\ddt\right)$, where $x$ is a tangent vector field on $\MM$, $t$ is the coordinate on $\R$ and $a$ is a differentiable function on $\MM\times\R$ \cite[Sect. 6.1]{Blair}.
The almost complex structures $J_{\al}$ %, $(\al=1,2,3)$
are defined on the manifold $\MM\times\R$ by
\begin{equation}\label{JaX}
\begin{array}{c}
J_{\al}X=J_{\al}\left(x, a\ddt\right)=\left(
  \f_{\al}x-a\xi_{\al},
  \eta_{\al}(x)\ddt
\right).
\end{array}
\end{equation}
In such a way, an almost hypercomplex structure on $\MM\times\R$ is defined in \cite{YaIsKo} when $\MM$ has an almost contact 3-structure.

Moreover, we equip $\MM\times\R$ with the product metric $h=g-\D t^2$. By virtue of \eqref{JaX}, \eqref{HN-met} and its consequence $g(\xia,\xia)=-\ea$, we obtain
\[
h\bigl(\Ja X,\Ja Y\bigr)=\ea h\bigl(X,Y\bigr)
\]
for arbitrary
\[
X=\left(x,a\ddt\right),\; Y=\left(y,b\ddt\right)\in\X(\MM\times\R),
\]
i.e. the manifold $\MM\times\R$ admits $(\Ja,h)$, an almost hypercomplex structure with Hermitian-Norden metrics. %, $(\al=1,2,3)$.

%For $X=\left(x, a\ddt\right)$, $Y=\left(y, b\ddt\right)$ in $\X(\MM\times\R)$ their Lie bracket has the form
We introduce the braces $\{X,Y\}$ on $\MM\times\R$ defined by
\begin{equation}\label{sym-skobi}
\begin{array}{c}
\{X,Y\}=\left(\{x,y\},(x(b)+y(a))\ddt\right),
\end{array}
\end{equation}
where $\{x,y\}$ are given in \eqref{braces}. Obviously, the braces are symmetric, i.e. $\{X,Y\}=\{Y,X\}$.

It is known from \eqref{N-JK} (see also \cite{KoNo-1}), the Nijenhuis tensor of two endomorphisms $J_{\al}$ and $J_{\bt}$ has the following form:
\[
\begin{split}
2[J_{\al},J_{\bt}](X,Y)&=[J_{\al}X,J_{\bt}Y]-J_{\al}[J_{\bt}X,Y]-J_{\al}[X,J_{\bt}Y]
\\[6pt]
&%\phantom{=}
+[J_{\bt}X,J_{\al}Y]-J_{\bt}[J_{\al}X,Y]-J_{\bt}[X,J_{\al}Y]\\[6pt]
&
+(J_{\al}J_{\bt}+J_{\bt}J_{\al})[X,Y].
\end{split}
\]
Then, bearing in mind \eqref{NJ-arbitrary}, the Nijenhuis tensor of an almost complex structure $J_{\al}\equiv J_{\bt}$ is presented by
\[
[J_{\al},J_{\al}](X,Y)=[J_{\al}X,J_{\al}Y]-J_{\al}[J_{\al}X,Y]
-J_{\al}[X,J_{\al}Y]-[X,Y].
\]

Analogously of the last two equalities, using the braces \eqref{sym-skobi} instead of the Lie brackets, we define consequently the associated Nijenhuis tensors in the two respective cases as follows:
\[
\begin{split}
2\{J_{\al},J_{\bt}\}(X,Y)&=\{J_{\al}X,J_{\bt}Y\}-J_{\al}\{J_{\bt}X,Y\}-J_{\al}\{X,J_{\bt}Y\}\\[6pt]
&%\phantom{2\{J_{\al},J_{\bt}\}(X,Y)=}
+\{J_{\bt}X,J_{\al}Y\}-J_{\bt}\{J_{\al}X,Y\}-J_{\bt}\{X,J_{\al}Y\}\\[6pt]
&%\phantom{2\{J_{\al},J_{\bt}\}(X,Y)=}
+(J_{\al}J_{\bt}+J_{\bt}J_{\al})\{X,Y\},\nonumber%\label{JaJb}
\end{split}
\]
\begin{equation}\label{JaJa}
\begin{array}{l}
\{J_{\al},J_{\al}\}(X,Y)
=\{J_{\al}X,J_{\al}Y\}-J_{\al}\{J_{\al}X,Y\}-J_{\al}\{X,J_{\al}Y\}\\[6pt]
\phantom{\{J_{\al},J_{\bt}\}(X,Y)=}
-\{X,Y\}.
\end{array}
\end{equation}
%The latter tensor is given in \eqref{hatNJ} and coincides with the tensor $\tilde N$ introduced in \cite{GaBo} by an equivalent equality of \eqref{JaJa} on an almost complex manifold with Norden metric.

Recalling \cite{GrHe}, $\GG_1$-manifolds are almost Hermitian manifolds whose corresponding Nijenhuis (0,3)-tensor by the Hermitian metric is a 3-form (see \eqref{G1-3f}).
This condition is equivalent to the vanishing of the associated Nijenhuis tensor, according to \propref{prop:G1}.

As it is known from \cite{GaBo}, the class %$\W_3$
of the quasi-K\"ahler manifolds with Norden metric is the only basic class of almost Norden manifolds with non-integrable almost complex structure, because the corresponding Nijenhuis tensor is non-zero there. Moreover, this class is determined by the condition that the associated Nijenhuis tensor is zero.

According to \thmref{thm:9}, if two of its six associated Nijenhuis tensors for the almost hypercomplex structure
vanish, then the others also vanish.

We seek to express in terms of the structure tensors of $(\f_{\al},\xi_{\al},\eta_{\al},g)$ on $\MM$ a necessary and sufficient condition for vanishing of $\{J_{\al},J_{\al}\}$ on $\MM\times\R$.

For the structure $(\f_{\al},\xi_{\al},\eta_{\al},g)$ let us define the following four tensors of type (1,2), (0,2), (1,1), (0,1), respectively:
\begin{equation}\label{hatN1234}
\begin{array}{l}
\widehat N_{\al}^{(1)}(x,y)=
\{\f_{\al},\f_{\al}\}(x,y)-\ea\left(\LLL_{\xi_{\al}}g\right)(x,y)\,\xi_{\al},\\[6pt]
\widehat N_{\al}^{(2)}(x,y)=
-\ea\left(\LLL_{\xi_{\al}}g\right)(\f_{\al}x,y)
-\ea\left(\LLL_{\xi_{\al}}g\right)(x,\f_{\al}y),
%=-\ea\bigl(\left(\LLL_{\xi_{\al}}g\right)\barwedge\f_{\al}\bigr)(x,y),
\\[6pt]
\widehat N_{\al}^{(3)}x=
\{\f_{\al},\f_{\al}\}(\f_{\al} x,\xia)+\left(\LLL_{\xi_{\al}}\eta_{\al}\right)(\f_{\al}x)\,\xi_{\al}\\[6pt]
\phantom{\widehat N_{\al}^{(3)}x=}
+2\eta_{\al}(x)\,\f_{\al}\n_{\xi_{\al}}\xi_{\al},\\[6pt]
\widehat N_{\al}^{(4)}(x)=
-\left(\LLL_{\xi_{\al}}\eta_{\al}\right)(x).
\end{array}
\end{equation}
\begin{prop}\label{prop:JaJa=0_hatN1234=0}
The associated Nijenhuis tensor $\{J_{\al},J_{\al}\}$ of an almost complex structure $\Ja$ for  $(\MM\times\R,\Ja,h)$ vanishes
  if and only if the four tensors $\widehat N_{\al}^{(1)}$, $\widehat N_{\al}^{(2)}$, $\widehat N_{\al}^{(3)}$, $\widehat N_{\al}^{(4)}$ for the structure $(\f_{\al},\xi_{\al},\eta_{\al},g)$ vanish.
\end{prop}

\begin{proof}%[Proof of \propref{prop:JaJa=0_hatN1234=0}]
First of all we need of the following relations
%\begin{lem}
\begin{equation}\label{LgLetaom}
\left(\LLL_{\xi_{\al}}g\right)(\xi_{\al},x)=-\ea\left(\LLL_{\xi_{\al}}\eta_{\al}\right)(x)
%=-\ea\om_{\al}(\f_{\al} x)
=g\left(\n_{\xi_{\al}} \xi_{\al},x\right).
\end{equation}
%\end{lem}
%\begin{proof}
  These equalities follow by virtue of %\eqref{3str},
  \eqref{eta-g}, \eqref{Fetaxia}, \eqref{Lie-der}.
%\end{proof}

Since each $\{J_{\al}, J_{\al}\}$ is a tensor
of type $(1, 2)$, it suffices to compute the tensors
\[
\{J_{\al}, J_{\al}\}\bigl(X_0, Y_0\bigr),\qquad
\{J_{\al}, J_{\al}\}\bigl(X_0, Z_0\bigr),
\]
where $X_0=\left(x, 0\ddt\right)$, $Y_0=\left(y, 0\ddt\right)$, $Z_0=\left(o, \ddt\right)$
and $o$ is the zero element of $\X(\MM)$. Taking into account \eqref{[f,f]}, \eqref{JaX}, \eqref{sym-skobi}, \eqref{JaJa}, we get the equalities:
\begin{align*}
\{J_{\al},J_{\al}\}&\bigl(X_0,Y_0\bigr)=%\\[6pt]
%=&\
%    \bigl\{\left(\f_{\al} x,\etaa(x)\ddt\right), \left(\f_{\al} y,\etaa(y)\ddt\right)\bigr\}
%    -\bigl\{\left(x,0\ddt\right), \left(y,0\ddt\right)\bigr\}\\[6pt]
% &-\Ja\bigl\{\left(\f_{\al} x,\etaa(x)\ddt\right), \left(y,0\ddt\right)\bigr\}
%    -\Ja\bigl\{\left(x,0\ddt\right), \left(\f_{\al} y,\etaa(y)\ddt\right)\bigr\}\\[6pt]
%=&\ \bigl(\left\{\f_{\al} x,\f_{\al} y\right\},\ \left(\f_{\al} x(\etaa(y))+\f_{\al} y(\etaa(x))\right)\ddt\bigr)\\[6pt]
%&-\bigl(-\f_{\al}^2\left\{x,y\right\}+\etaa\left(\{x,y\}\right)\xia,\ 0\ddt\bigr)\\[6pt]
%&-\bigl(\f_{\al}\{\f_{\al} x,y\}-y\left(\etaa(x)\right)\xia,\ \etaa\left(\{\f_{\al} x,y\}\right)\ddt\bigr)\\[6pt]
%&-\bigl(\f_{\al}\{x,\f_{\al} y\}-x\left(\etaa(y)\right)\xia,\ \etaa\left(\{x,\f_{\al} y\}\right)\ddt\bigr)\\[6pt]
\left(\widehat N_{\al}^{(1)}(x,y),\ \widehat N_{\al}^{(2)}(x,y)\ddt\right),\\[6pt]
\{J_{\al},J_{\al}\}&\bigl(X_0,Z_0\bigr)=%\\[6pt]
%=&\
% \bigl\{\left(\f_{\al} x,\etaa(x)\ddt\right), \left(-\xia,0\ddt\right)\bigr\}
%   -\bigl\{\left(x,0\ddt\right), \left(o,\ddt\right)\bigr\}\\[6pt]
% &-\Ja\bigl\{\left(\f_{\al} x,\etaa(x)\ddt\right), \left(o,\ddt\right)\bigr\}
%   -\Ja\bigl\{\left(x,0\ddt\right), \left(-\xia,0\ddt\right)\bigr\}\\[6pt]
%=& -\bigl(\{\f_{\al} x,\xia\}, \xia\left(\etaa(x)\right)\ddt\bigr)
%    + \bigl(\f_{\al}\{x,\xia\}, \etaa\left(\{x,\xia\}\right)\ddt\bigr)\\[6pt]
\left(\widehat N_{\al}^{(3)}x,\ \widehat N_{\al}^{(4)}(x)\ddt\right),%\tag*{\qed}
\end{align*}
which show the correctness of the statement.
%Then, for any $\al=1,2,3$, the vanishing of $\{J_{\al},J_{\al}\}$ holds if and only if  $\widehat N_{\al}^{(1)}$, $\widehat N_{\al}^{(2)}$, $\widehat N_{\al}^{(3)}$, $\widehat N_{\al}^{(4)}$ vanish.
\end{proof}

\begin{prop}\label{prop:N1=0=>N234=0}
For an almost contact structure $(\f_{\al},\xi_{\al},\eta_{\al})$ and a pseudo-Rieman\-ni\-an metric $g$, the vanishing of
$\widehat N_{\al}^{(1)}$ implies the vanishing of $\widehat N_{\al}^{(2)}$, $\widehat N_{\al}^{(3)}$ and $\widehat N_{\al}^{(4)}$.
\end{prop}
\begin{proof}
We set $y=\xia$ in $\widehat N_{\al}^{(1)}(x,y)=0$ and apply $\etaa$. Then, using \eqref{[f,f]} and  \eqref{3str}, we obtain $\left(\LLL_{\xi_{\al}}g\right)(x,\xi_{\al})=0$ and thus $\widehat N_{\al}^{(4)}=0$, according to \eqref{LgLetaom}.

Therefore, %Bearing in mind
from the form of $\widehat N_{\al}^{(1)}$ in \eqref{hatN1234},
we get
$
\{\f_{\al},\f_{\al}\}(\f_{\al}x,\xi_{\al})=0.
$
On the other hand, bearing in mind \eqref{LgLetaom}, we have that the vanishing of $\left(\LLL_{\xi_{\al}}g\right)(x,\xi_{\al})$ is equivalent to the vanishing of $\left(\LLL_{\xi_{\al}}\eta_{\al}\right)(x)$ and $\n_{\xi_{\al}}\xi_{\al}$. Thus,
we obtain $\widehat N_{\al}^{(3)}=0$.

Finally, applying $\etaa$ to $\widehat N_{\al}^{(1)}(\f_{\al} x,y)=0$
and using \eqref{[f,f]}, we have
\[
\eta_{\al}\bigl(\{\f_{\al}^2x,\f_{\al}y\}\bigr)-\ea\left(\LLL_{\xi_{\al}}g\right)(\f_{\al}x,y)=0.
\]
The first term in the latter equality can be expressed in the following form $-\ea\left(\LLL_{\xi_{\al}}g\right)(x,\f_{\al}y)$,
using that $\LLL_{\xi_{\al}}\eta_{\al}$ vanishes. In such a way, we obtain that $\widehat N_{\al}^{(2)}(x,y)=0$.
\end{proof}

%Obviously, we have
\begin{prop}
For an almost contact structure $(\f_{\al},\xi_{\al},\eta_{\al})$ and a pseudo-Rieman\-ni\-an metric $g$, where  $\xia$ is Killing, the following assertions are valid:
\begin{enumerate}
  \item $\widehat N_{\al}^{(1)}$ vanishes if and only if $\{\f_{\al},\f_{\al}\}$ vanishes;
  \item $\widehat N_{\al}^{(2)}$ vanishes;
  \item $\widehat N_{\al}^{(3)}$ vanishes if and only if %the condition %$\{\f_{\al} x,\xia\}=\f_{\al}\{x,\xia\}$
  $\xia\,\lrcorner\,\{\f_{\al},\f_{\al}\}$  vanishes;
  \item $\widehat N_{\al}^{(4)}$ vanishes.
\end{enumerate}
\end{prop}
\begin{proof}
Taking into account that $\LLL_{\xi_{\al}}g$ vanishes, we have $\widehat N_{\al}^{(1)}=
\{\f_{\al},\f_{\al}\}$ and $\widehat N_{\al}^{(2)}=0$, i.e. (ii).
Further, we obtain $\widehat N_{\al}^{(4)}=0$, i.e. (iv), and
$\widehat N_{\al}^{(3)}x=
\{\f_{\al},\f_{\al}\}(\f_{\al} x,\xia)$,
according to \eqref{LgLetaom}. Then, (i) is obvious whereas (iii) holds, bearing in mind the assumption for $\xia$.
\end{proof}

\begin{dfn}
Let $(\MM,\f_{\al},\xi_{\al},\eta_{\al},g)$ %, $(\al=1,2,3)$
be a manifold with almost contact 3-structure and metrics of Hermitian-Norden type. The symmetric $(1,2)$-tensors defined by
\begin{equation}\label{Nhat-3}
\widehat N_{\al}=\{\f_{\al},\f_{\al}\}-\ea\xi_{\al}\otimes \LLL_{\xi_{\al}}g
\end{equation}
we call \emph{associated Nijenhuis tensors on $(\MM,\f_{\al},\xi_{\al},\eta_{\al},g)$}.
\end{dfn}

The corresponding $(0,3)$-tensors are denoted by
\[
\begin{split}
\widehat N_{\al}(x,y,z)&=g(\widehat N_{\al}(x,y),z),\\[6pt]
\{\f_{\al},\f_{\al}\}(x,y,z)&=g(\{\f_{\al},\f_{\al}\}(x,y),z).
\end{split}
\]
Then, taking into account \eqref{eta-g} and \eqref{Nhat-3}, we obtain
\[
\widehat N_{\al}(x,y,z)=\{\f_{\al},\f_{\al}\}(x,y,z)+\left(\LLL_{\xi_{\al}}g\right)(x,y)\,\etaa(z).
\]

%In such a way, we obtain the following
\begin{thm}\label{thm:Ja=0_hatN=0}
Let $(\MM,\f_{\al},\xi_{\al},\eta_{\al},g)$ %, $(\al=1,2,3)$
be a manifold
with almost contact 3-structure and metrics of Hermitian-Norden type.
For any $\al$, the associated Nijenhuis tensor
$\{J_{\al},J_{\al}\}$ of the almost complex structure $J_{\al}$  on $(\MM\times\R,\Ja,h)$ vanishes if and only if the associated Nijenhuis tensor $\widehat{N}_{\al}$ %on $\MM$
of the structure $(\f_{\al},\xi_{\al},\eta_{\al},g)$ vanishes. % and $\xi_{\al}$ is a Killing vector field.?
\end{thm}

\begin{proof}
  The statement follows from \propref{prop:JaJa=0_hatN1234=0} and  \propref{prop:N1=0=>N234=0}, bearing in mind \eqref{hatN1234} and \eqref{Nhat-3}.
\end{proof}

%Then, we establish the truthfulness of the following
\begin{thm}\label{thm:hatN}
Let $(\MM,\f_{\al},\xi_{\al},\eta_{\al},g)$ %, $(\al=1,2,3)$
be a manifold
with almost contact 3-structure and metrics of Hermitian-Norden type.
If two of the associated Nijenhuis tensors $\widehat N_{\al}$ vanish, the third also vanishes.
%: %for the associated Nijenhuis tensors of an almost contact 3-struc\-ture %and different indices $\al,\bt,\gm\in\{1,2,3\}$,
%we have:
%  \begin{enumerate}
%    \item $\widehat N_1$ vanishes;% and $\xi_1$ is Killing;
%    \item $\widehat N_2$ vanishes;
%    \item $\widehat N_3$ vanishes.
%  \end{enumerate}
\end{thm}
\begin{proof}
It follows by virtue of %\propref{cor:Nxi-al},
\thmref{thm:9} and \thmref{thm:Ja=0_hatN=0}.
\end{proof}

\vskip 0.2in \addtocounter{subsection}{1} \setcounter{subsubsection}{0}

\noindent  {\Large\bf \thesubsection. Natural connections with totally skew-symmetric torsion on a manifold with almost contact 3-structure and metrics of Hermitian-Norden type}%\\[6pt]\vskip2pt}

\vskip 0.15in
%\section{Natural connections with totally skew-symmetric torsion}

An affine connection $\n^*$ is said to be a \emph{natural connection for $(\f_\al,\xi_\al,\eta_\al,\allowbreak{}g)$},  % $\al\in\{1,2,3\}$,
if it preserves the structure, i.e.
$\n^*\f_\al=\n^*\xi_\al=\n^*\eta_\al=\n^*g=0$.

\begin{thm}\label{thm:NN=Nhat}
Let $(\MM,\f_1,\xi_1,\eta_1,g)$ be a pseudo-Riemannian manifold with an almost contact metric structure.
%For the {Nijenhuis} tensor and the associated {Nijenhuis} tensor of .
The following statements are equivalent:
\begin{enumerate}
  \item The manifold belongs to the class $\PP_2\oplus\PP_4\oplus\PP_9\oplus\PP_{10}\oplus\PP_{11}$ determined by
\begin{equation}\label{F1_W3W7}
\begin{split}
F_1 (\f_1 x,y,z)+F_1(\f_1y,x,z)&+F_1(x,y,\f_1z)\\[6pt]
&+F_1(y,x,\f_1z)=0.
\end{split}
\end{equation}
  \item The associated {Nijenhuis} tensor $\widehat N_1$ vanishes and $\xi_1$ is a Killing vector field;
  \item The tensor $\{\f_1,\f_1\}$ vanishes and $\xi_1$ is a Killing vector field;
  \item The {Nijenhuis} tensor $N_1$ is a 3-form and $\xi_1$ is a Killing vector field;
  \item There exists a natural connection $\n^1$ with totally skew-symmetric torsion for the structure $(\f_1,\xi_1,\eta_1,g)$ and this connection is unique and determined by its torsion
\end{enumerate}
\begin{equation}\label{T1}
\begin{array}{l}
T_1(x,y,z)=F_1(x,y,\f_1 z)-F_1(y,x,\f_1 z)-F_1(\f_1 z,x,y)\\[6pt]
\phantom{T_1(x,y,z)=}
+2F_1(x,\f_1 y,\xi_1)\eta_1(z).
\end{array}
\end{equation}
\end{thm}

\begin{proof}
Using \eqref{N1hat=F1} and \eqref{Lxi1g=F1}, we have that the vanishing of $\widehat N_1$ and $\LLL_{\xi_1}g$ implies the identity \eqref{F1_W3W7}. Vice versa, setting $x=y=\xi_1$ in \eqref{F1_W3W7}, we have $F_1(\xi_1,\xi_1,z)=0$. If we put $x=\f_1x$, $y=\f_1y$, $z=\xi_1$ in \eqref{F1_W3W7} and use the latter vanishing, we obtain that $\LLL_{\xi_1}g=0$ and therefore $\widehat N_1=0$. The determination of the class in (i) by \eqref{F1_W3W7} becomes under the definitions of the basic classes by \eqref{cl-A} and the form of the corresponding components $P^i(F_1)$, given in \cite{AlGa}. So, the equivalence between (i) and (ii) is valid.

Now, we need to prove the following relation
\begin{equation}\label{NN=Nhatf1}
\widehat N_1(x,y,z)=N_1(z,x,y)+N_1(z,y,x).
\end{equation}
We calculate the right hand side of \eqref{NN=Nhatf1} using \eqref{N1=F1}. By \eqref{FaJ-prop} %and their consequence
%\begin{equation}\label{FffF1}
%\begin{array}{l}
%    F_1(x,y,\f_1z)=F_1(x,\f_1y,z)+F_1(x,\xi_1,\f_1y)\eta_1(z)\\[6pt]
%    \phantom{F_1(x,y,\f_1z)=F_1(x,\f_1y,z)}
%    +F_1(x,\xi_1,\f_1z)\eta_1(y),
%\end{array}
%\end{equation}
we obtain
\[
\begin{array}{l}
N_1(z,x,y)+N_1(z,y,x)=
-F_1(\f_1x,z,y)-F_1(\f_1y,z,x)\\[6pt]
\phantom{N_1(z,x,y)+N_1(z,y,x)=}
-F_1(x,z,\f_1y)-F_1(y,z,\f_1x)\\[6pt]
\phantom{N_1(z,x,y)+N_1(z,y,x)=}
-F_1(x,\f_1z,\xi_1)\eta_1(y)-F_1(y,\f_1z,\xi_1)\eta_1(x)
\end{array}
\]
and then we establish that the right hand side of the latter equality is equal to $\widehat N_1(x,y,z)$, according to \eqref{N1hat=F1}.
Therefore, \eqref{NN=Nhatf1} is valid.

The relation \eqref{NN=Nhatf1} implies the equivalence between (ii) and (iv), where\-as
the equivalence between (iv) and (v) is given in Theorem 8.2 of \cite{Fri-Iv2}.
The equivalence between (ii) and (iii) follows from \eqref{hatN1}.

For the connection $\n^1$ from (v) we have
\begin{equation}\label{D1T1}
g\left(\n^1_{x} y,z\right)=g(\n_x y, z)+\frac12 T_1(x,y,z).
\end{equation}
According to Theorem 8.2 in \cite{Fri-Iv2}, its torsion $T_1$ is determined in our notations by
\begin{equation}\label{T1Kuch}
\begin{array}{l}
T_1=-\eta_1\wedge\D\eta_1+\D^{\f_1} \Phi+N_1-\eta_1\wedge(\xi_1 \lrcorner N_1),\\[6pt]
\end{array}
\end{equation}
where it is used the notation $\D^{\f_1} \Phi(x,y,z)=-\D \Phi(\f_1 x,\f_1 y,\f_1 z)$ for the fundamental 2-form
$\Phi$ of the almost contact metric structure, i.e. $\Phi(x,y)=g(x,\f_{1}y)$.
Since $\eta_1\wedge\D\eta_1=\s\{\eta_1\otimes\D\eta_1\}$ holds and because of \eqref{HN-met}, \eqref{FaJ-prop}
and the fact that $\xi_1$ is Killing, it is valid %the following
\begin{equation}\label{1}
(\eta_1\wedge\D\eta_1)(x,y,z)=-2\sx\{\eta_1(x)F_1(y,\f_1 z,\xi_1)\}.
\end{equation}
Moreover, from the equalities $\D \Phi(x,y,z)=-\sx F_1(x,y,z)$ and \eqref{FaJ-prop}, we get
\begin{equation}\label{2}
\D^{\f_1} \Phi(x,y,z)=-\sx\{F_1(\f_1 x,y,z)+2F_1(x,\f_1 y,\xi_1)\eta_1(z)\}.
\end{equation}
So, applying \eqref{1}, \eqref{2}, \eqref{N1=F1} and \eqref{FaJ-prop} to the equality \eqref{T1Kuch},
we obtain an expression of $T_1$ in terms of $F_1$ given in \eqref{T1}.
\end{proof}

%The equivalences in the following theorem are known from \cite{Man31} and \cite{ManIv36}.
\begin{thm}\label{thm:F3F7}
  The following statements for an almost contact B-metric manifold $(\MM,\f_2,\xi_2,\eta_2,g)$ are equivalent:
 \begin{enumerate}
 \item
 It belongs to the class $\F_3\oplus\F_7$, which is characterised by the conditions: the cyclic sum of $F_2$ by the three arguments vanishes and $\xi_2$ is Killing;
 \item
 It has a vanishing associated {Nijenhuis} tensor $\widehat N_2$;
 \item
 It has a vanishing tensor $\{\f_2,\f_2\}$ and $\xi_2$ is Killing;
 \item
 It admits the existence of a unique natural connection $\n^2$ with totally skew-sym\-met\-ric torsion determined by
\end{enumerate}
\begin{equation}\label{T37} %
T_2(x,y,z)=-\frac{1}{2} \sx\bigl\{F_2(x,y,\f_2 z)
-3\eta_2(x)F_2(y,\f_2 z,\xi_2)\bigr\}. %
\end{equation} %
\end{thm}
\begin{proof}
The equivalence of (i), (ii) and (iv) is known from \cite{Man31} and \corref{cor:fKTassN=0}. Bearing in mind \propref{prop:hatN=0=>Kill} and the definition of $\widehat N_2$, we obtain
the equivalence of (ii) and (iii).

For the natural connection $\n^2$ with totally skew-symmetric torsion $T_2$ for the structure $(\f_2,\xi_2,\eta_2,g)$ we have
\begin{equation}\label{D2T2}
g\left(\n^2_{x} y,z\right)=g(\n_x y, z)+\frac12 T_2(x,y,z),
\end{equation}
where $T_2$ is determined by
$
T_2=\eta_2\wedge \D\eta_2+\frac{1}{4}\s N_2
$
and it is expressed
in terms of $F_2$ by \eqref{T37}.
\end{proof}

%*******************************************************
%
%Verojatno za syshtestvuvaneto na svyrzanostta trjabva da imame $\{\fa,\fa\}=0$, i.e. asoc. Nijenh. vyrhu $\ker(\etaa)=0$ trjabva da e nulev, t.e. $\fa$ da e pochti kompleksna struktura s nulev asoc. Nijenhuis tensor!
%
%
%*******************************************************

%Let us remark that a similar theorem of the latter one for the
%almost contact B-metric manifold $(\MM,\f_3,\xi_3,\eta_3,g)$ with the corresponding connection $\n^3$ is valid.

Using \thmref{thm:hatN}, \thmref{thm:NN=Nhat} and \thmref{thm:F3F7}, %from \cite{Man54}, %\propref{prop:Nxi-al},
we get % immediately the following
\begin{thm}\label{thm:nat-conn}
    %Let $(\MM,\f_{\al},\xi_{\al},\eta_{\al},g)$, $(\al=1,2,3)$ be a manifold with almost contact HN-metric 3-structure.
    The existence of unique natural connections with totally skew-sym\-met\-ric torsion for two of the three structures of an almost contact 3-structure with metrics of Hermitian-Norden type implies an existence of a unique natural connection with totally skew-symmetric torsion for the remaining third structure.
\end{thm}

\begin{cor}\label{cor-G1}
Let $(\MM,\f_{\al},\xi_{\al},\eta_{\al},g)$ %$(\al=1,2,3)$
be a manifold with almost contact 3-structure and metrics of Hermitian-Norden type.
If the manifold belongs to two of the following three classes for the corresponding structure,
then the manifold belongs to the remaining third class for the corresponding structure:
%  \begin{enumerate}
%    \item
$\PP_{2}\oplus\PP_{4}\oplus\PP_{9}\oplus\PP_{10}\oplus\PP_{11}$ for $\al=1$;
%    \item
$\F_3\oplus\F_7$  for $\al=2$ and $\F_3\oplus\F_7$ for $\al=3$.%;
%    \item for.
%  \end{enumerate}
\end{cor}

%\subsection{United natural connection with totally skew-symmetric torsion}

%\textbf{Da se dopylnjat teoremite kato pri hypercomplex: koga syshtestvuva edinstvena $\f_\al$KT-sonnection
%za vsjako $\f_\al$ i koga tezi tri svyrzanosti syvpadat.  }

Now, we are interested on conditions for coincidence of these three natural connections $\n^\al$ %, $(\al=1,2,3)$
with totally skew-symmetric torsion for the particular almost contact structures with the metric $g$. Then we shall say that it exists a \emph{natural connection with totally skew-symmetric torsion for the almost contact 3-structure with metrics of Hermitian-Norden type}.

\begin{thm}\label{thm:D}
Let $(\MM,\f_{\al},\xi_{\al},\eta_{\al},g)$ be a manifold with almost contact 3-structure and metrics of Hermitian-Norden type.
Then the manifold admits an affine connection $\n^*$ with totally skew-sym\-met\-ric torsion
preserving the structure $(\f_{\al},\xi_{\al},\eta_{\al},g)$ if and only if
the associated Nijenhuis tensors $\widehat N_\al$ vanish, $\xi_1$ is Killing and
the equalities $T_1=T_2=T_3$ are valid, bearing in mind \eqref{T1} and \eqref{T37}.
%\begin{equation}\label{*}
%\begin{array}{l}
%F_1(x,y,\f_1 z)-F_1(y,x,\f_1 z)-F_1(\f_1 z,x,y)\\[6pt]
%\phantom{F_1(x,y,\f_1 z)}
%+2F_1(x,\f_1 y,\xi_1)\eta_1(z)
%\\[6pt]
%=-\frac{1}{2} \sx\bigl\{F_2(x,y,\f_2 z)
%-3\eta_2(x)F_2(y,\f_2 z,\xi_2)\bigr\} %
%\\[6pt]
%=-\frac{1}{2} \sx\bigl\{F_2(x,y,\f_3 z)
%-3\eta_3(x)F_3(y,\f_3 z,\xi_3)\bigr\}. %
%\end{array}
%\end{equation}
If $\n^*$ exists, it is unique and determined by its torsion $T^*=T_1=T_2=T_3$.
\end{thm}
\begin{proof}
According to \thmref{thm:NN=Nhat} and \thmref{thm:F3F7}, since $\widehat N_\al=\LLL_{\xi_1}g=0$ are valid then
there exist the natural connections $\n^\al$ with totally skew-sym\-metric torsion $T_\al$ for the structures $(\f_{\al},\xi_{\al},\eta_{\al},g)$, $(\al=1,2,3)$.
Bearing in mind \eqref{D1T1}, \eqref{T1}, \eqref{D2T2} and \eqref{T37}, the coincidence of $\n^1$, $\n^2$ and $\n^3$ is equivalent to the conditions to equalise of their torsions. %This completes the proof.
\end{proof}

\vskip 0.2in \addtocounter{subsection}{1} \setcounter{subsubsection}{0}

\noindent  {\Large\bf \thesubsection. A 7-dimensional Lie group as a manifold with almost contact 3-structure and metrics of Hermitian-Norden type}%\\[6pt]\vskip2pt}

\vskip 0.15in
%\section{A 7-dimensional Lie group as a manifold of the considered type}

%In this subsection we construct a 7-dimensional manifold with an almost contact 3-structure and metric of Hermitian-Norden type.

Let $\LL$ be a 7-dimensional real connected Lie group, and
$\mathfrak{l}$ be its Lie algebra with a basis
$\{e_1,e_2,e_3,e_4,e_5,e_6,e_7\}$.
%Then an arbitrary vector $x$ in $T_p\LL$ at $p\in \LL$ is presented by $x=x^{i}e_i$ $(i=1,2,\dots,7)$ and
%
Now we introduce an almost contact 3-structure and metrics of Hermitian-Norden type
$(\f_{\al},\xi_{\al},\eta_{\al},g)$ by a standard way as follows
% is equal to minus identity,
\begin{subequations}\label{str-exa}
\begin{equation}
\begin{array}{l}
\begin{array}{llllll}
\f_1e_1=e_2,\; & \f_1e_3=e_4,\; & \f_1e_6=e_7, \; & \f_1e_5=o,\\[6pt]
\f_2e_1=e_3, \; & \f_2e_4=e_2, \; & \f_2e_7=e_5, \; & \f_2e_6=o,\\[6pt]
\f_3e_1=e_4,\; & \f_3e_2=e_3, \; & \f_3e_5=e_6, \; & \f_3e_7=o, \\[6pt]
\end{array}\\[6pt]
\begin{array}{lll}
\xi_1=e_5, \quad &     \xi_2=e_6, \quad &     \xi_3=e_7, \\[6pt]
\eta_1=e^5, \quad &    \eta_2=e^6, \quad &    \eta_3=e^7,\qquad e^i(e_j)=\delta^i_j,\\[6pt]
\end{array}
\end{array}
\end{equation}
%\\[6pt]
\begin{equation}
\begin{array}{l}
\begin{array}{ll}
   g(e_i,e_j)=0,\; i\neq j; \\[6pt]
   g(e_i,e_i)=-g(e_j,e_j)=1,\ i=1,2,6,7;\ j=3,4,5.
\end{array}
\end{array}
\end{equation}
\end{subequations}
where $o$ denotes the zero vector in $T_p\LL$ at $p\in \LL$ and bearing in mind that $\f_\al^2=-I$.

%The almost contact HN-metric 3-structure
%$(\f_{\al},\xi_{\al},\eta_{\al})$ on $H$ coincides with the almost hypercomplex HN-metric structure considered in \cite{GriManDim12}. The almost hypercomplex structure is defined as in \cite{Som}.

Let us consider $(\LL,\f_{\al},\xi_{\al},\eta_{\al},g)$ with the Lie algebra $\mathfrak{l}$
determined by the following nonzero commutators for $\lm\in\R\setminus\{0\}$
\[
\left[e_1,e_2\right]=[e_3,e_4]=\lm e_7.
\]

%In terms of the almost contact 3-structure, if $e_1$ be denoted briefly as $e$, \eqref{lie} can be written as
%\begin{equation}\label{lie=}
%\left
%[e,\f_1e\right]=[\f_2e,\f_3e]=\lm \xi_3,\qquad \lm\neq 0.
%\end{equation}

By the Koszul equality \eqref{koszul}, we compute the components of $\n$ with respect to the basis and the nonzero ones of them are: %covariant derivatives of the basis vectors %by the Koszul equality
\begin{equation}\label{nabli-exa}
\begin{array}{c}
\begin{array}{c}
\n_{e_1} e_2=-\n_{e_2} e_1=\n_{e_3}e_4=-\n_{e_4} e_3=\frac12\lm e_7,
\end{array}\\[6pt]
\begin{array}{ll}
\n_{e_1} e_7=\n_{e_7} e_1=-\frac12\lm  e_2, &\quad
\n_{e_2} e_7=\n_{e_7} e_2=\frac12\lm  e_1, \\[6pt]
\n_{e_3} e_7=\n_{e_7} e_3=\frac12\lm  e_4, &\quad
\n_{e_4} e_7=\n_{e_7} e_4=-\frac12\lm  e_3.
\end{array}
\end{array}
\end{equation}

\begin{prop}\label{prop:exa-class}
Let $(\LL,\f_{\al},\xi_{\al},\eta_{\al},g)$ be the Lie group $\LL$ with almost contact 3-structure and metrics of Hermitian-Norden type depending on $\lm\in\R\setminus\{0\}$. Then this manifold belongs to the following basic
classes, according to the corresponding classification in \eqref{cl-A} and \eqref{cl-B}$:$
\begin{enumerate}
  \item
$\PP_{10}$ with res\-pect to $(\f_{1},\xi_{1},\eta_{1},g);$
  \item
$\F_{3}$ with respect to $(\f_{2},\xi_{2},\eta_{2},g);$
  \item
$\F_{7}$ with respect to $(\f_{3},\xi_{3},\eta_{3},g).$
\end{enumerate}
\end{prop}
\begin{proof}
Using \eqref{str-exa} and \eqref{nabli-exa}, we obtain the property
\[
(F_{\al})_{ijk}=-(F_{\al})_{jik}
\]
for $i\neq j$ and the following values of the basic components
\[
(F_{\al})_{ijk}\allowbreak{}=F_{\al}(e_i,e_j,e_k)
\]
of $F_{\al}$:
\begin{equation}\label{F1F2F3-L}
\begin{array}{l}
\frac12\lm  =(F_1)_{117}=(F_1)_{126}=(F_1)_{227}=(F_1)_{337}\\[6pt]
\phantom{\frac12\lm}
=(F_1)_{346}=(F_1)_{447}=(F_2)_{125}=(F_2)_{147}\\[6pt]
\phantom{\frac12\lm}
=(F_2)_{237}=(F_2)_{345}=-(F_3)_{137}=(F_3)_{247}
\end{array}
\end{equation}
and the others are zero.
%
%As a sequel we have that the Lee forms $\theta_{\al}$, $\theta^*_{\al}$, $\omega_{\al}$  $(\al=1,2,3)$ vanish.
%
From here, applying the classification conditions for the relevant classification in \cite{AlGa} or \cite{GaMiGr}, we have the classes in the statement, respectively.
\end{proof}

Bearing in mind \propref{prop:exa-class}, we deduce that the conditions (i) of \thmref{thm:NN=Nhat} and \thmref{thm:F3F7} are fulfilled and hence there exist natural connections $\n^\al$ %($\al=1,2,3$)
for the corresponding structure $(\f_\al,\xi_\al,\eta_\al,g)$ on $\LL$.
We get the basic components of their torsions $T_\al$ %($\al=1,2,3$)
by direct computations from \eqref{D1T1}, \eqref{T1}, \eqref{D2T2}, \eqref{T37} and \eqref{F1F2F3-L} as follows:
\[
\begin{array}{c}
(T_1)_{127}=(T_1)_{347}=(T_3)_{127}=(T_3)_{347}=-\lm,
\\[6pt]
(T_2)_{127}=-(T_2)_{145}=-(T_2)_{235}=(T_2)_{347}=-\frac12 \lm
\end{array}
\]
and the others are zero.

Obviously,  $\n^1$ and $\n^3$ coincides but $\n^2$ differs from them.
The condition for equality of the three torsions in \thmref{thm:D} is not fulfilled and therefore it does not exist a unique connection with totally skew-sym\-met\-ric torsion preserving the almost contact 3-structure and the metrics of Hermitian-Norden type on $(\LL,\f_{\al},\xi_{\al},\eta_{\al},g)$.

\vspace{20pt}

\begin{center}
$\divideontimes\divideontimes\divideontimes$
\end{center}

%%
%%%%%%1%%%%%%%%%%%%%%%%%%%%%%%%%%%%%%%%%%%%%%%%%%%%%%%%%%%%%%%%%%%%    12
%\include{Man-zakl}%\include{Man-contr}
\newpage

%\setcounter{section}{1}
%\addtocounter{section}{1}\setcounter{subsection}{0}\setcounter{subsubsection}{0}
%
%\setcounter{thm}{0}\setcounter{equation}{0}

\label{zakl}

 \Large{

\
\\[6pt]
\bigskip

\
\\[6pt]
\bigskip

\lhead{\emph{Conclusion. Contributions of the Dissertation
}}

\noindent  {\Huge\bf Conclusion
}%\\[6pt]\vskip2pt}

\vskip 1cm

%
%\vskip 0.2in %\addtocounter{subsection}{1}
%
%\noindent  {\Large\bf Contributions of the Dissertation}
%
%\vskip 0.15in

\vskip 0.2in

\noindent  {\huge\bf \emph{Contributions of the Dissertation}
}

\vskip 1cm

The present dissertation contains recent author's investigations on differential geometry of smooth manifolds equipped with some tensor structures (almost complex structures, almost contact structures, almost hypercomplex structures and almost contact 3-structures)
and metrics of Norden type.

According to the author, %can be formulated the following
the main contributions of the dissertation are the following:

%\vskip6pt

%\noindent
%\textbf{Main Contributions of the Dissertation:}
%The main contributions of the present dissertation are the following:
\begin{enumerate}[1.]
  \item %\S2 \S1 въвежда almost complex structures with Norden metric
  There have been introduced and studied the twin interchange of the pair of Norden metrics (the basic one and its associated metric) on almost complex manifolds as well as there have been found invariant and anti-invariant geometric objects and characteristics under this interchange. (\S2)
%A torsion-free connection and tensors with geometric interpretation are found which are invariant under the twin interchange, i.e. the swap of the counterparts of the pair of Norden metrics and the corresponding Levi-Civita connections on an almost complex manifold with Norden metric. A Lie group depending on four real parameters is considered as an example of a 4-dimensional manifold of the studied type. The mentioned invariant objects are found in an explicit form.
  %
  \item %\S3
  It has been developed further the study on the basic natural connections: the B-connection, the canonical connection and the KT-con\-nection %(i.e. metric connections with torsion satisfying a certain algebraic identity and preserving the structure tensors)
  on almost complex manifolds with Norden metric. It has been characterized all basic classes of the considered manifolds with respect to the torsions of the canonical connection, the Nijenhuis tensor and its associated one. (\S3)
  \item %\S4
  It has been introduced an associated Nijenhuis tensor on an almost contact manifold with B-metric having important geometrical characteristics. Moreover, these manifolds have been classified with respect to the Nijenhuis tensor and its associated one. (\S4)
  \item %\S5
  It has been introduced and studied a $\f$-canonical connection on the almost contact manifolds with B-metric and it has been found the relation between this connection and other two important natural connections on these manifolds -- the $\f$B-connection and the $\f$KT-connection. It has been established that the torsion of the $\f$-canonical connection is invariant with respect to a subgroup of the general conformal transformations of the almost contact B-metric structure. Thereby, the basic classes of the considered manifolds have been characterized in terms of the torsion of the $\f$-canonical connection. (\S5)
  \item %\S6
   There have been classified all affine connections on an almost contact manifold with B-metric % in 15 basic classes
   with respect to the properties of their torsions regarding the manifold's structure. Three studied natural connections have been characterized regarding this classification. (\S6)
  \item %\S7
  There have been introduced and studied a pair of associated Schouten-van Kampen affine connections adapted to the contact distribution and the almost contact B-metric structure generated by the pair of associated B-metrics and their Levi-Civita connections. By means of the constructed non-symmetric connections, there have been characterized the basic classes of the almost contact B-metric manifolds. (\S7)
  \item %\S8
  There have been introduced and studied Sasaki-like almost contact complex Riemannian manifolds. In addition, it has been presented a canonical construction (called an $\mathcal{S}^1$-solvable extension) producing such a manifold from a holomorphic complex Riemannian manifold. (\S8)
  \item %\S10     \S9 въвежда almost hypercomplex structures with Hermitian-Norden metrics
  There have been studied integrable hypercomplex structures with Her\-mitian-Norden metrics on 4-dimensional Lie groups by constructing the five types corresponding of invariant hypercomplex structures with hyper-Hermitian metric. (\S10) %The different cases regarding the signature of the basic pseudo-Riemannian metric are considered.
  \item %\S11
   It have been studied the tangent bundle of an almost complex manifold with Norden metric and the complete lift of the Norden metric as an almost hypercomplex manifold with Hermitian-Norden metrics.  (\S11)%The case when the base manifold is an h-sphere is considered.
  \item %\S12
  There have been introduced and studied the associated Nijenhuis ten\-sors of an almost hypercomplex manifold with Hermitian-Norden metrics.
  It has been found a geometric interpretation of the vanishing of these tensors
  as a necessary and sufficient condition for existence of affine connections with totally skew-symmetric torsions preserving  the manifold's structure. (\S12)
  \item %\S13
  There have been introduced and studied quaternionic K\"ahler manifolds corresponding to almost hypercomplex manifolds with Hermitian-Norden metrics. (\S13) %There have been found some necessary and sufficient conditions the investigated manifolds be isotropic hyper-K\"ahlerian and flat. It has been proved that the quaternionic K\"ahler manifolds with the considered metric structure are Einstein for dimension at least 8. It has been determined the class of the non-hyper-K\"ahler quaternionic K\"ahler manifold of the considered type.
  \item %\S14 i \S15 i \S16
  There have been introduced manifolds with almost contact 3-structure and metrics of Hermitian-Norden type
  %It has been proved that the introduced manifold of cosymplectic type is flat. %Some examples of the studied manifolds are given.%
  as well as corresponding associated Nijenhuis tensors. %and studied their vanishing.
  It has been found a geometric interpretation of the vanishing of these tensors %for the studied structures
  as a necessary and sufficient condition for existence of affine connections with totally skew-symmetric torsions preserving  the manifold's structure. (\S14, \S15)
  \item %Primeri
    There have been constructed and studied a variety of explicit examples of manifolds equipped with studied structures: almost complex structure with Norden metric, almost contact structure with B-metric, almost hypercomplex structure with Hermitian-Norden metrics and almost contact 3-structure with metrics of Hermitian-Norden type.
\end{enumerate}

%%%%%%%%%%%%%%%%%%%%%%%%%%%%%%%%%%%%%%%%%%%%%%%%%%%%%%%%%%%%%%%%%%%%%%%%%%%%%%%%%%%%%%

\vspace{20pt}

\begin{center}
$\divideontimes\divideontimes\divideontimes$
\end{center} 

\newpage

%\setcounter{section}{1}
%\addtocounter{section}{1}\setcounter{subsection}{0}\setcounter{subsubsection}{0}
%
%\setcounter{thm}{0}\setcounter{equation}{0}
\renewcommand\theenumi{\arabic{enumi}}
\renewcommand\labelenumi{{\theenumi}.}

\label{publ}

 \Large{

\
\\[6pt]
\bigskip

\
\\[6pt]
\bigskip

\lhead{\emph{Conclusion. Publications on the Dissertation
}}

\noindent  {\huge\bf \emph{Publications on the Dissertation}
}%\\[6pt]\vskip2pt}

\vskip 1cm

%
%
%\vskip 0.2in %\addtocounter{subsection}{1}
%
%\noindent  {\Large\bf Purpose of the Dissertation}
%
%\vskip 0.15in

\vskip12pt

Main results of the present dissertation are published in the following papers and preprints:
\footnote{The number in square brackets is from the general list of Bibliography.} %

\begin{enumerate}
\item \cite{Man29} %29.
%\textbf{\textcolor[rgb]{0,0,1}{([29])}}
\textsc{M. Manev}.
\emph{Quaternionic K\"ahler manifolds with Hermitian and Norden met\-rics}.
\textbf{Journal of Geometry}, vol.~103, no.~3 (2012),
519--530; ISSN:\allowbreak{}0047-2468,
DOI:\allowbreak{}10.\allowbreak{}1007/\allowbreak{}s00022-012-0139-x,
MCQ%\footnote{AMS Mathematical Citation Quotient (MCQ)}
\allowbreak{}(2012):\allowbreak{}0.16,
SJR\allowbreak{}%\footnote{SCImago Journal Rank (SJR) indicator -- www.scimagojr.com}
(2012):0.278%, H-index=15, SNIP%\footnote{Source normalized impact per paper (SNIP) -- www.journalmetrics.com}
%(2012):0.522, IPP%\footnote{Impact per publication -- www.journalmetrics.com}
%(2012):0.2\allowbreak{}03,
%arXiv:0906.\allowbreak{}50\allowbreak{}52
.

\item \cite{ManIv38} %38.
%\textbf{\textcolor[rgb]{1,0,1}{([38])}}
\textsc{M. Manev, M. Ivanova}.
\emph{Canonical-type connection on al\-most contact mani\-folds with B-metric}.
\textbf{Annals of Global Analysis and Geometry}, vol. 43, no. 4 (2013), 397--\allowbreak{}408; ISSN:0232-704X, DOI:\allowbreak{}10.\allowbreak{}1007/\allowbreak{}s10455-012-9351-z,
IF(2013):0.794,
SJR\allowbreak{}(20\allowbreak{}13):1.248.
%, H-in\-dex\allowbreak{}=23,
%EF\allowbreak{}(2013):\allowbreak{}0.004,
%AI\allowbreak{}(2013):0.9, SNIP(2013):1.166, IPP(2013):\allowbreak{}0.686,
%arXiv:1203.0137.

\item \cite{ManIv36} %36.
%\textbf{\textcolor[rgb]{1,0,1}{([36])}}
\textsc{M. Manev, M. Ivanova}.
\emph{A classification of the torsion tensors on almost con\-tact manifolds with B-metric}.
\textbf{Central European Jour\-nal of Mathematics}, vol. 12, no. 10 (2014), 1416--1432; ISSN:\allowbreak{}18\allowbreak{}95-1074, DOI:10.2478/s11533-014-0422-1,
IF%\footnote{Impact Factor (ISI Journals Citation Report)}
(2014):0.\allowbreak{}578,
MCQ(20\allowbreak{}14):\allowbreak{}0.39,
SJR(2014):0.610.
%, H-index=14,
%EF%\footnote{Eigenfactor (EF) -- %мярката на общата значимост на списанието за научната
%%общност съгласно
%%www.eigenfactor.org}
%\allowbreak{}(2014):\allowbreak{}0.005,
%AI%\footnote{Article Influence (AI) -- %мярката на средното влияние
%%на всяка от статиите в списанието през първите пет години след
%%публикуването съгласно
%%www.\allowbreak{}eigenfactor.org}
%\allowbreak{}(20\allowbreak{}14):0.5, SNIP\allowbreak{}(2014):1.003, IPP\allowbreak{}(2014):0.566,
%arXiv:1105.\allowbreak{}5715

\item \cite{Man44} %44.
%\textbf{\textcolor[rgb]{0,0,1}{([44])}}
\textsc{M. Manev}.
\emph{Hypercomplex structures with Hermitian-Norden met\-rics on four-dimen\-sional Lie algebras}.
\textbf{Journal of Geometry}, vol. 105, no. 1 (2014), 21--31; ISSN:00\allowbreak{}47-2468,
DOI:\allowbreak{}10.1007/\allowbreak{}s00022-013-0188-9,
MCQ\allowbreak{}(2014):\allowbreak{}0.26,
SJR(2014):0.345.
%, H-in\-dex=15, SNIP\allowbreak{}(2014):0.472, IPP(2014):0.291,
%arXiv:\allowbreak{}1309.0975

\item \cite{Man47} %47.
%\textbf{\textcolor[rgb]{0,0,1}{([47])}}
\textsc{M. Manev}.
\emph{Tangent bundles with complete lift of the base metric  and almost
hy\-per\-com\-plex Hermitian-Norden structure}.
\textbf{Comptes rendus de l'Academie bulgare des Sciences}, vol.~67, no.~3 (2014),
313--322;
ISSN:1310-1331, %
IF(20\allowbreak{}14):0.284,
SJR\allowbreak{}(20\allowbreak{}14):\allowbreak{}0.205.
%, H-index=12,
%EF(2014):\allowbreak{}0.001, AI\allowbreak{}(2014):0.04,
%arXiv:1309.0977

\item \cite{Man48} %48.
%\textbf{\textcolor[rgb]{1,0,1}{([48])}}
\textsc{M. Manev}.
\emph{On canonical-type connections on almost contact com\-plex Rie\-mann\-ian manifolds}.
\textbf{Filomat}, vol. 29, no. 3 (2015),
411--425; ISSN:0354-5180, DOI:10.2298/\allowbreak{}FIL1503411M, %
IF(2015):0.603, \allowbreak{}SJR(2015):0.487.
%,
%H-index=14, EF(2014):0.002,
%AI(2014):0.26, SNIP\allowbreak{}(2014):0.915, IPP(2014):0.797,
%arXiv:\allowbreak{}1407.6843

\item \cite{Man50} %50.
%\textbf{\textcolor[rgb]{1,0,1}{([50])}}
\textsc{M. Manev}.
\emph{Pair of associated Schouten-van Kampen connections adapt\-ed to an al\-most
contact B-metric structure}.
\textbf{Filomat}, vol. 29, no. 10 (2015),
2437--2446; ISSN:\allowbreak{}0354-5180,
IF(2015):0.603,
SJR\allowbreak{}(2015):0.487.
%,
%H-index=14, EF(2014):0.002,
%AI(2014):0.26, SNIP(20\allowbreak{}14):0.915, IPP(2014):0.797,
%arXiv:\allowbreak{}1504.05330

\item \cite{IvMaMa14} %45.
%\textbf{\textcolor[rgb]{1,0,1}{([45])}}
\textsc{S. Ivanov, H. Manev, M. Manev}.
\emph{Sasaki-like almost contact complex Rie\-mann\-ian manifolds}.
\textbf{Journal of Geometry and Physics}, vol. 105 (2016), 136--148; ISSN:0393-0440, DOI:10.1016/j.geom\allowbreak{}phys.\allowbreak{}2016.05.009,
IF(2015):\allowbreak{}0.752, SJR(2015):0.705.
%,
%H-\allowbreak{}index=38, EF\allowbreak{}(2014):0.0\allowbreak{}09, AI(2014):0.7, %
%SNIP(2015):1.11\allowbreak{}1,	IPP(2014):0.890,
%arX\allowbreak{}iv:\allowbreak{}1402.5426

\item \cite{Man49} %49.
%\textbf{\textcolor[rgb]{1,0,1}{([49])}}
\textsc{M. Manev}.
\emph{Invariant tensors under the twin interchange of Nor\-den metrics on almost complex manifolds}.
\textbf{Results in Math\-e\-ma\-t\-ics}, vol. 70, no. 1 (2016), 109--126; IS\allowbreak{}SN:1422-6383, DOI:10.1007/\allowbreak{}s\allowbreak{}0\allowbreak{}0\allowbreak{}025-015-0464-0,
IF(2015):0.768, MCQ(2015):0.43, \allowbreak{}SJR(2015):0.6\allowbreak{}36.
%,
%H-index\allowbreak{}=15, EF(2014):\allowbreak{}0.00\allowbreak{}3,
%AI(20\allowbreak{}14):0.4,
%SNIP(2015):\allowbreak{}0.825,IPP\allowbreak{}(2014):0.728,
%arXiv:1502.06779

\item \cite{Man52} %52
%\textbf{\textcolor[rgb]{0,0,1}{([52])}}
\textsc{M. Manev}.
\emph{Associated Nijenhuis tensors on manifolds with al\-most hypercom\-plex
struc\-tures and metrics of Hermitian-Norden type}.
\textbf{Re\-sults in Mathematics},
vol. 71 %no. ?
(2017), %;
ISSN:1422-6383, DOI:10.\allowbreak{}1\allowbreak{}007/s00025-016-0624-x,
IF(2015):0.768, MCQ(2015):0.43, SJR(201\allowbreak{}5):0.636.
%,
%H-index=\allowbreak{}15, EF\allowbreak{}(2014):0.\allowbreak{}00\allowbreak{}3,
%AI(2014):0.4,
%SNIP(2015):0.8\allowbreak{}25, IPP(20\allowbreak{}14):0.728,
%arXiv:1510.00821

\item \cite{Man53} %53
%\textbf{\textcolor[rgb]{0,0,1}{([53])}}
\textsc{M. Manev}.
\emph{Manifolds with almost contact 3-structure and metrics of Hermit\-ian-Norden type}.
%arXiv:1506.04376.%\\
\textbf{Journal of Geometry}, %vol. 108 (2017)
(accepted, 4.05.2017),  DOI:10.1007/s00022-017-0386-y,
MCQ(2015):0.27, SJR(2\allowbreak{}015):0.272.

%\item \label{Man54}%54
%\textbf{\textcolor[rgb]{0,0,1}{([54])}} \textsc{M. Manev}.
%\emph{Natural connections with totally skew-symmetric torsion on mani\-folds
%with almost contact 3-structure and metrics of Hermitian-Norden type}.
%arXiv:16\allowbreak{}04.02039.\\
%Submited on 14.09.2016, Balkan Journal of Geometry and Its Applications

\item \cite{Man54} %54(1)
%\textbf{\textcolor[rgb]{0,0,1}{([54])}}
\textsc{M. Manev}.
\emph{Associated Nijenhuis tensors on manifolds with al\-most contact 3-structure
and metrics of Hermitian-Norden type}.
\textbf{Comp\-tes rendus de l'Academie bulgare des Sciences} (accepted, 28.02.2017), IF(2015):0.233.
%\\
%Submited on 13.10.2016 to Stefan Ivanov for Compt. rend. Acad. bulg. Sci.

\item \cite{Man55} %54(2)
%\textbf{\textcolor[rgb]{0,0,1}{([55])}}
\textsc{M. Manev}.
\emph{Natural connections with totally skew-sym\-met\-ric tor\-sion on man\-i\-folds with
almost contact 3-structure and metrics of Her\-mitian-Norden type}. arXiv:16\allowbreak{}04.02039 (part 2).
%\\
%Submited on 13.10.2016 to Stefan Ivanov for Compt. rend. Acad. bulg. Sci.

\end{enumerate}

\vspace{20pt}

\begin{center}
$\divideontimes\divideontimes\divideontimes$
\end{center}

\newpage

%\setcounter{section}{1}
%\addtocounter{section}{1}\setcounter{subsection}{0}\setcounter{subsubsection}{0}
%
%\setcounter{thm}{0}\setcounter{equation}{0}

\label{dekl}

 \Large{

\
\\[4pt]
\bigskip

\
\\[4pt]
\bigskip

\lhead{\emph{Conclusion. Declaration of Originality
}}

\noindent  {\huge\bf \emph{Declaration of Originality}
}%\\[4pt]\vskip2pt}

\vskip 1cm

\noindent
by \\
\textsc{\textbf{Prof. Dr. Mancho Hristov Manev}}\\
Department of Algebra and Geometry\\
Faculty of Mathematics and Informatics\\
Paisii Hilendarski University of Plovdiv

\vskip1cm

In connection with the conducting of the procedure for award of the scientific degree of Doctor of Science in Mathematics from Paisii Hilendarski University of Plovdiv and the protection of the presented by me dissertation, I declare:

The results and the contributions of the scientific studies presented in my dissertation on the topic
\begin{quote}
\emph{\textsc{On Geometry of Manifolds with Some Tensor \\
              Structures and Metrics of Norden Type}}
\end{quote}
are original and not taken from research and publications in which I do not participate.

\vskip1cm

\noindent
1.02.2017 \hspace{6cm} Signature:\\
\noindent
Plovdiv \\
\phantom{aaa} \hfill (Prof. Dr. Mancho Manev)

\vspace{20pt}

\begin{center}
$\divideontimes\divideontimes\divideontimes$
\end{center}

%\
%\newpage
%\thispagestyle{empty}
%
%$ \phantom{\quad} $

%%%%%%%%%%%%%%%%%%%%%%%%%%%%%%%%%%%%%%%%%%%%%%%%%%%%%%%%%%%%%%%%%%%%%%%%%%%%%%%
%\include{Man-blag}
\newpage

%\setcounter{section}{1}
%\addtocounter{section}{1}\setcounter{subsection}{0}\setcounter{subsubsection}{0}
%
%\setcounter{thm}{0}\setcounter{equation}{0}

\label{blag}

 \Large{

\
\\[6pt]
\bigskip

\
\\[6pt]
\bigskip

\lhead{\emph{Conclusion. Acknowledgements
}}

\noindent  {\huge\bf \emph{Acknowledgements}
}%\\[6pt]\vskip2pt}

\vskip 1cm

In research on the topic and preparation of %the exposition of
the present dissertation,
the author was supported by a number of valuable people and therefore he expresses his appreciation to them.

At first, the author realize that his family is the bedrock of the favourable conditions for his scientific pursuits.

Invaluable was the help from his teachers, consultants and co-authors:
Kostadin Gribachev,
Dimitar Mekerov,
Georgi Ganchev,
Stefan Ivanov,
Stancho Dimiev
and
Kouei Sekigawa.

Furthermore, the author was aided by his collaborators:
Galia Nakova,
%Maria Staikova,
Miroslava Ivanova
and Hristo Manev.

The author is also grateful to all those who in one way or another
%to a lesser extent
contributed to the realisation of the presented dissertation.

%
%My departmental colleagues geometers:
%Dobrinka Gribacheva,
%Iva Dokuzova,
%Marta Teofilova,
%Asen Hristov,

\vspace{20pt}

\begin{center}
$\divideontimes\divideontimes\divideontimes$
\end{center}

%%%%%%%%%%%%%%%%%%%%%%%%%%%%%%%%%%%%%%%%%%%%%%%%%%%%%%%%%%%%%%%%%%%%%%%%%%%%%%%
%\include{Man-bib}
%%%%%%%%%%%%%%%%%%%%%%%%%%%%%%%%%%%%%%%%%%%%%%%%%%%%%%%%%%%%%%%%%%%%%%%%%%%%%%%
\newpage

%\count0 = 249

\lhead{\emph{Bibliography}}

\label{bib}

\Large{

\
\\[6pt]
\bigskip

\
\\[6pt]

\vspace{20pt}

\begin{center}
$\divideontimes\divideontimes\divideontimes$
\end{center} 

%%%%%%%%%%%%%%%%%%%%%%%%%%%%%%%%%%%%%%%%%%%%%%%%%%%%%%%%%%%%%%%%%%%%%%%%%%%%%%%
% }
%%%%%%%%%%%%%%%%%%%%%%%%%%%%%%%%%%%%%%%%%%%%%%%%%%%%%%%%%%%%%%%%%%%%%%%%%%%%%%%
\end{document}